\documentclass[11pt]{amsart}
\usepackage{amssymb, latexsym, graphicx, mathrsfs, enumerate, amsmath, amsthm}
\usepackage{color}
\input xy
\xyoption{all}

\usepackage[onehalfspacing]{setspace} 

\newlength{\margins}
\usepackage[top=\margins,bottom=\margins,left=\margins,right=\margins]{geometry}

\numberwithin{equation}{section}
\newtheorem{Thm}{Theorem}[section]

\newtheorem{Lem}[Thm]{Lemma}

\theoremstyle{definition}
\newtheorem{Def}[Thm]{Definition}

\newtheorem{Rmk}[Thm]{Remark}

\begin{document}

\title[Group schemes and local densities of ramified hermitian lattices, Part II]
{Group schemes and local densities of ramified hermitian lattices in residue characteristic 2 Part II, Expanded version}

\author[Sungmun Cho]{Sungmun Cho}
\thanks{The  author is partially supported by JSPS KAKENHI Grant No. 16F16316
and NRF 2018R1A4A 1023590.}

\address{Sungmun Cho \\ Graduate school of mathematics, Kyoto University, Kitashirakawa,
Kyoto, 606-8502, JAPAN \\
\newline
Current: Department of Mathematics, POSTECH, 77, Cheongam-ro, Nam-gu, Pohang-si, Gyeongsangbuk-do, 37673, KOREA}
\email{sungmuncho12@gmail.com}

\maketitle

\begin{abstract}
This paper is the complementary  work of \cite{C2}.
Ramified quadratic extensions $E/F$, where $F$ is a finite unramified field extension of $\mathbb{Q}_2$,
 fall into two cases that we call \textit{Case 1} and \textit{Case 2}.
In the previous work \cite{C2}, we obtained the local density formula  for a ramified hermitian lattice in \textit{Case 1}.
In this paper, we  obtain the local density formula for the remaining  \textit{Case 2},
by constructing a smooth integral group scheme model for an appropriate unitary group.
Consequently, this paper, combined with the paper \cite{GY} of W. T. Gan and J.-K. Yu and \cite{C2}, allows the computation of the mass formula for any  hermitian lattice $(L, H)$, when a base field is unramified over $\mathbb{Q}$ at a prime $(2)$.

\end{abstract}
\let\thefootnote\relax\footnote{Primary MSC 11E41, MSC 11E95, MSC 14L15, MSC 20G25; Secondary MSC 11E39, MSC 11E57}

\tableofcontents

\section{Introduction}\label{in}

\subsection{Introduction}\label{in'}
Local densities are local factors of the mass formula,
which is an essential tool in the classification of hermitian lattices over number fields.
We refer to the introduction of \cite{C2} for  history along this context.
This paper is the complementary work of \cite{C2}.

Let $F$ be a finite unramified field extension of $\mathbb{Q}_2$.
Ramified quadratic extensions $E/F$ fall into two cases that we call \textit{Case 1} and \textit{Case 2} (cf. Section \ref{Notations}),
 depending on lower ramification groups $G_i$'s of the Galois group $\mathrm{Gal}(E/F)$  as follows:
\[
\left\{
  \begin{array}{l }
 \textit{Case 1}: G_{-1}=G_{0}=G_{1}, G_{2}=0;\\
 \textit{Case 2}: G_{-1}=G_{0}=G_{1}=G_{2}, G_{3}=0.
    \end{array} \right.
\]
The paper \cite{C2} gives the local density formula of hermitian lattices in  \textit{Case 1}.

The main contribution of this paper is to get an explicit formula for the local density of a hermitian $B$-lattice $(L, h)$ in \textit{Case 2},
by explicitly constructing a certain smooth group scheme (called smooth integral model) associated to it that serves as an integral model
for the unitary group associated to $(L\otimes_AF, h\otimes_AF)$ and
  by investigating its special fiber,
where $B$ is a  ramified quadratic extension of $A$ and $A$ is an unramified finite extension of $\mathbb{Z}_2$ with $F$ as the quotient field of $A$.

In  conclusion, this paper, combined with \cite{GY}, \cite{C1}, and \cite{C2},
finally allows the computation of the mass formula for any hermitian lattice $(L, H)$
 when a base field is unramified over $\mathbb{Q}$ at a prime $(2)$.
 As the simplest case, we can compute the mass formula for an arbitrary hermitian lattice explicitly
  when  a base field is $\mathbb{Q}$.

For a brief idea and comment on  the proof, we refer to the introduction of \cite{C2}.
 The methodology and the structure of this paper are   basically the same as those of \cite{C2} and thus
we repeat a number of sentences and paragraphs from \cite{C2}  for synchronization without comment.
 But \textit{Case 2} is  more difficult and technical than \textit{Case 1}.

This paper is organized as follows.
 We first state a structure theorem for integral hermitian forms in Section \ref{sthln}.
  We then give an explicit construction of  a smooth integral model $\underline{G}$
 (in Section \ref{csm}) and study its special fiber (in Section \ref{sf}) in \textit{Case 2}.
 Finally,
 we obtain an explicit formula for the local density in Section \ref{cv} in \textit{Case 2}.
In Appendix \ref{App:AppendixB}, we provide an example to describe the smooth integral model and its special fiber
and to compute the local density for a unimodular lattice of rank 1.

The reader might want to skip to Appendix \ref{App:AppendixB} and at least go to Appendix \ref{nc} to get a first glimpse into why the case of $p=2$ is really different.
Some of the ideas behind our construction can be seen in the simple example illustrated in Appendix \ref{cfot}.

\subsection{Acknowledgements}
This paper was initially the second half of \cite{C2}. 
Due to a huge number of pages and technical difficulty, we decide to divide it into two papers.
The author  greatly thanks the referee of Algebra \& Number Theory, who read  \cite{C2}, and Professor Brian Conrad
 for  incredibly precise and helpful comments and discussions on this project.
 The author also appreciates   Professor Chia-Fu Yu's notice to point out one error in this paper and \cite{C2}. It is explained in Remark \ref{correction}.

\section{Structure theorem for hermitian lattices and notations}\label{sthln}
In this section, we explain a structure theorem for hermitian lattices.
This theorem is proved in \cite{C2}.
Thus we take necessary definitions and theorems from \cite{C2}, without providing proofs.
 \subsection{Notations}\label{Notations}
Notations and definitions in this section are taken from \cite{C1}, \cite{GY}, \cite{J}, and \cite{C2}.
\begin{itemize}
\item Let $F$ be an unramified finite extension of $\mathbb{Q}_2$ with $A$  its ring of integers and $\kappa$  its residue field.
\item  Let $E$ be a ramified quadratic field extension of $F$ with $B$ its ring of integers. 
\item Let $\sigma$ be the non-trivial element of the Galois group $\mathrm{Gal}(E/F)$.
\item The lower ramification groups $G_i$'s of the Galois group $\mathrm{Gal}(E/F)$ satisfy one of the following:
\[
\left\{
  \begin{array}{l }
 \textit{Case 1}: G_{-1}=G_{0}=G_{1}, G_{2}=0;\\
 \textit{Case 2}: G_{-1}=G_{0}=G_{1}=G_{2}, G_{3}=0.
    \end{array} \right.
\]
In \textit{Case 2}, based on Section 6 and Section 9 of \cite{J}, there is a suitable choice of a uniformizer $\pi$ of $B$ such that
\[\pi=\sqrt{2\delta}, \textit{where $\delta\in A $ and $\delta\equiv 1 \mathrm{~mod~}2$}.\]
Thus $E=F(\pi)$ and $\sigma(\pi)=-\pi$.
 From now on, we assume that $E/F$ satisfies \textit{Case 2} and
 a uniformizing element $\pi$ of $B$  and $\delta$ are fixed as explained above
 throughout this paper.
\item Set
\[\xi:=\pi\cdot\sigma(\pi).\]

\item We consider a $B$-lattice $L$ with a hermitian form $$h : L \times L \rightarrow B,$$
 where $h(a\cdot v, b \cdot w)=\sigma(a)b\cdot h(v,w)$ and $h(w,v)=\sigma(h(v,w))$. Here, $a, b \in B$ and $v, w \in L$.
 We denote by a pair $(L, h)$ a hermitian lattice.
We assume that $V=L\otimes_AF$ is nondegenerate with respect to $h$.
\item We denote by $(\epsilon)$ the $B$-lattice of rank 1 equipped with the hermitian form having Gram matrix $(\epsilon)$.
We use the symbol $A(a, b, c)$ to denote the $B$-lattice $B\cdot e_1+B\cdot e_2$
with the hermitian form having Gram matrix $\begin{pmatrix} a & c \\ \sigma (c) & b \end{pmatrix}$.
For each integer $i$,
the lattice of rank 2 having  Gram matrix $\begin{pmatrix} 0 & \pi^i \\ \sigma(\pi^i)  & 0 \end{pmatrix}$ is called the hyperbolic plane and denoted by $H(i)$.
\item A  hermitian lattice $L$ is the orthogonal sum of sublattices $L_1$ and $L_2$, written $L=L_1\oplus L_2$, if $L_1\cap L_2=0$, $L_1$ is orthogonal to $L_2$ with respect to the hermitian form $h$, and $L_1$ and $L_2$ together span $L$.
\item The ideal in $B$ generated by $h(x,x)$ as $x$ runs through $L$ will be called the norm of $L$ and written $n(L)$.
\item By the scale $s(L)$ of $L$, we mean the ideal generated by the subset $h(L,L)$ of $B$.
\item We define the dual lattice of $L$, denoted by $L^{\perp}$, as
 $$L^{\perp}=\{x \in L\otimes_A F : h(x, L) \subset B \}.$$

\end{itemize}

\begin{Def}[Definition 2.1 in \cite{C2}]\label{d1}
Let $L$ be a hermitian lattice. Then:
\begin{enumerate}
\item[(a)]
For any non-zero scalar $a$, define $aL=\{ ax|x\in L \}$.
It is also a lattice in the space $L\otimes_AF$. Call a vector $x$ of $L$ maximal in $L$
if $x$ does not lie in $\pi L$.
\item[(b)]
The lattice $L$ will be called $\pi^i$-modular if the ideal generated by the subset $h(x, L)$ of $E$ is $\pi^iB$ for every maximal vector $x$ in $L$.
Note that $L$ is $\pi^i$-modular if and only if $L^{\perp}=\pi^{-i}L$.
We can also see that $H(i)$ is $\pi^i$-modular.
\item[(c)] Assume that $i$ is even. A $\pi^i$-modular lattice $L$ is \textit{of parity type I} if $n(L)=s(L)$,
 and  \textit{of parity type II} otherwise.
The zero lattice is considered to be  \textit{of parity type II}.
We caution that we do not assign a \textit{parity type} to a $\pi^i$-modular lattice $L$ with $i$ odd.
\end{enumerate}
\end{Def}

 \begin{Rmk}[Remark 2.3 in \cite{C2}]\label{r23}
 \begin{enumerate}
 \item[(a)]
  If $L$ is $\pi^i$-modular, then $\pi^j L$ is $\pi^{i+2j}$-modular for any integer $j$.
\item[(b)] (Section 4 in \cite{J}) For a general lattice $L$, we have a Jordan splitting, namely $L=\bigoplus_i L_i$
such that $L_i$ is $\pi^{n(i)}$-modular  and such that the sequence $\{n(i)\}_i$ increases.
Two Jordan splittings $L=\bigoplus_{1\leqq i \leqq t} L_i$ and $K=\bigoplus_{1\leqq i \leqq T} K_i$ will be said to be of the same type
if $t=T$ and, for $1\leqq i \leqq T$, the following conditions are satisfied:
 $s(L_i)=s(K_i)$, rank $L_i$ = rank $K_i$, and $n(L_i)=s(L_i)$ if and only if $n(K_i)=s(K_i)$. 
Jordan splitting is not unique but partially canonical in the sense that two Jordan splittings of isometric lattices are always of the same type.
\item[(c)]
If we allow some of the $L_i$'s to be zero, then we may assume that $n(i) = i$ for all $i$.
In other words, for all $i\in \mathbb{N}\cup \{0\}$  we have $s(L_i)=(\pi^i)$, and, more precisely, $L_i$ is $\pi^i$-modular.
Then we can rephrase part (b) above  as follows.
 Let $L=\bigoplus_i L_i$ be a Jordan splitting with $s(L_i)=(\pi^i)$ for all $i\geq 0$.
Then the scale, rank and parity type of $L_i$ depend only on $L$.
We will deal exclusively with a Jordan splitting satisfying $s(L_i)=(\pi^i)$ from now on.
\end{enumerate}
 \end{Rmk}

 \subsection{Lattices}\label{lattices}[Section 2C in \cite{C2}]\label{lattices}

In this subsection, we will define several lattices and associated notation.
Fix a hermitian lattice $(L, h)$. 
We denote by $(\pi^l)$ the scale $s(L)$ of $L$.

\begin{itemize}
\item[(1)]  Define $A_i=\{x\in L \mid h(x,L) \in \pi^iB\}.$
\item[(2)] Define $X(L)$ to be the sublattice of $L$ such that
 $X(L)/\pi L$ is the radical of the symmetric bilinear form $\frac{1}{\pi^l}h$ mod $\pi$ on $L/\pi L$.
\end{itemize}

Let $l=2m$ or $l=2m-1$.
We consider the function defined over $L$ by
$$\frac{1}{2^m}q : L\longrightarrow A, x\mapsto \frac{1}{2^m}h(x,x).$$
Then $\frac{1}{2^m}q$  mod 2 defines a quadratic form $L/\pi L \longrightarrow \kappa$.
It can be easily checked that $\frac{1}{2^m}q$ mod 2 on $L/\pi L$ is an additive polynomial.
We define a lattice $B(L)$ as follows.
\begin{itemize}
\item[(3)]
$B(L)$ is defined to be the sublattice of $L$ such that
 $B(L)/\pi L$ is the kernel of the additive polynomial  $\frac{1}{2^m}q$ mod 2 on $L/\pi L$.
\end{itemize}

To define a few more lattices,  we need some preparation as follows.
Recall that $\pi\cdot\sigma(\pi)$ is denoted by $\xi$.

Assume $B(L)\varsubsetneq L$ and $l$ is even. Then the bilinear form $\xi^{-l/2}h$ mod $\pi$
on the $\kappa$-vector space $L/X(L)$ is nonsingular symmetric and nonalternating.
It is well known that there is a unique vector $e \in L/X(L)$ such that
$$(\xi^{-l/2}h(v,e))^2=\xi^{-l/2}h(v,v) \textit{ mod } \pi$$
for every vector $v \in L/X(L)$.
Let $\langle e\rangle$ denote  the 1-dimensional vector space spanned by the vector $e$ and denote by $e^{\perp}$ the 1-codimensional subspace of $L/X(L)$ which is orthogonal to the vector $e$ with respect to  $\xi^{-l/2}h$ mod $\pi$.
Then $$B(L)/X(L)=e^{\perp}.$$
If $B(L)= L$, then the bilinear form $\xi^{-l/2}h$ mod $\pi$ on the $\kappa$-vector space $L/X(L)$
is nonsingular symmetric and alternating. In this case, we put $e=0\in L/X(L)$
and note that it is characterized by the same identity.\\

The remaining lattices we need for our definition are:
\begin{itemize}
\item[(4)] Define $W(L)$ to be the sublattice of $L$ such that \[
\left\{
  \begin{array}{l l}
 \textit{$W(L)/X(L)=\langle e\rangle$} & \quad \textit{if $l$ is even};\\
 \textit{$W(L)=X(L)$} & \quad \textit{if $l$ is odd}.
    \end{array} \right.
\]
\item[(5)] Define $Y(L)$ to be the sublattice of $L$ such that $Y(L)/\pi L$ is the radical of
\[
\left\{
  \begin{array}{l l}
 \textit{the form $\frac{1}{2^{m}}h$
  mod $\pi$ on $B(L)/\pi L$} & \quad \textit{if $l=2m$};\\
 \textit{the  form $\frac{1}{\pi}\cdot\frac{1}{2^{m-1}}h$ mod $\pi$ on $B(L)/\pi L$} & \quad \textit{if $l=2m-1$}.
    \end{array} \right.
\]
\end{itemize}
Both forms are alternating and bilinear.
\begin{itemize}

 \item[(6)] Define $Z(L)$ to be the sublattice of $L$ such that  $Z(L)/\pi B(L)$   is the radical of
the quadratic form $\frac{1}{2^{m+1}}q$ mod $2$ on $B(L)/\pi B(L)$  if $l=2m$.
\end{itemize}
(see, e.g.,  page 813 of \cite{Sa} for the notion of the radical of a quadratic form on a vector space over a field of characteristic 2.)

\begin{Rmk}\label{r26} As in Remark 2.6 of \cite{C2}, 
 \begin{enumerate}
 \item[(a)]  We can associate
the 5 lattices  $(B(L), W(L), X(L), Y(L), Z(L))$ above
with $(A_i, h)$ in place of $L$.
 Let $B_i,W_i,X_i,Y_i,Z_i$ denote the resulting lattices.
 \item[(b)] As $\kappa$-vector spaces, the dimensions of $A_i/B_i$ and  $W_i/X_i$ are at most 1.
  \end{enumerate}
\end{Rmk}

Let $L=\bigoplus_i L_i$ be a Jordan splitting.
We assign a type to each $L_i$ as follows:

\begin{center}
    \begin{tabular}{| l | l | l  |}
    \hline
    parity of $i$ & type of $L_i$ & condition \\ \hline
 even &   $I$ & $L_i$ is of parity type $I$ \\
  even & $I^o$ & $L_i$ is of parity type $I$ and the rank of $L_i$ is odd \\
even &  $I^e$ & $L_i$ is of parity type $I$ and the rank of $L_i$ is even \\
even &   $II$ & $L_i$ is of parity type $II$ \\
odd &   $II$ & $A_i=B_i$ \\
odd &   $I$ & $A_i \varsupsetneq B_i$ \\ \hline
    \end{tabular}
\end{center}

In addition, we assign a subtype to $L_i$ in the following manner:

\begin{center}
    \begin{tabular}{| l | l | l  |}
    \hline
    parity of $i$ & subtype of $L_i$ & condition \\ \hline
 even &  bound of type  $I$ & $L_i$ is of type $I$ and either $L_{i-2}$ or $L_{i+2}$ is of type $I$ \\
  even & bound of type  $II$ & $L_i$ is of type $II$ and either $L_{i-1}$ or $L_{i+1}$ is of type $I$ \\
odd &   bound & either $L_{i-1}$ or $L_{i+1}$ is of type $I$ \\ \hline
    \end{tabular}
\end{center}
In all other cases, $L_i$ is called $free$.
If $L_i$ with $i$ odd is \textit{of type II}, then $L_i$ should be \textit{free}.
In other words, a lattice $L_i$ with $i$ odd cannot be \textit{bound of type II}.

Notice that the type of each $L_i$ is determined canonically regardless of the choice of a Jordan splitting.

\subsection{Sharpened Structure Theorem for Integral hermitian Forms}\label{sstm}

\begin{Thm}[Theorem 2.10 in   \cite{C2}]\label{210}
There exists a suitable choice of a  Jordan splitting of the given lattice $L=\bigoplus_i L_i$ such that $L_i=\bigoplus_{\lambda}H_{\lambda}\oplus K$,
where each $H_{\lambda}= H(i)$ and $K$ is  $\pi^i$-modular of rank 1 or 2 with the following descriptions.
Let $i=0$ or $i=1$. Then
\[
K=\left\{
  \begin{array}{l l}
  \textit{$(a)$ where $a \equiv 1$ mod 2}    & \quad  \textit{if $i=0$ and $L_0$ is \textit{of type $I^o$}};\\
  \textit{$A(1, 2b, 1)$}    & \quad  \textit{if $i=0$ and $L_0$ is \textit{of type $I^e$}};\\
  \textit{$A(2\delta, 2b, 1)$}    & \quad  \textit{if $i=0$ and $L_0$ is \textit{of type II}};\\
  \textit{$A(4a, 2\delta, \pi)$}    & \quad  \textit{if $i=1$ and $L_1$ is \textit{free of type I}};\\
  \textit{$H(1)$}    & \quad  \textit{if $i=1$, and $L_1$ is \textit{of type II} or \textit{bound of type I}}.
      \end{array} \right.
\]
Here, $a, b\in A$ and $\delta, \pi$ are explained in Section 2.1.
\end{Thm}

\begin{Rmk}\label{r211}
 \begin{enumerate}
 \item[(a)] As mentioned in Remark 2.3.(a) in \cite{C2}, If $L$ is $\pi^i$-modular, then $\pi^j L$ is $\pi^{i+2j}$-modular for any integer $j$.
   Thus, the above theorem implies its obvious generalization to the case where $i$ is allowed to be any element of $\mathbb{Z}$.

   \item[(b)]  Working with a basis furnished by the above theorem, we can describe our lattices $A_i$ through $Z_i$ more explicitly.
   We refer to Remark 2.11 in \cite{C2} for a precise description of these lattices.
       \end{enumerate}
\end{Rmk}

From now on, the pair $(L,h)$ is fixed throughout this paper.


\section{The construction of the smooth model}\label{csm}

We start with introducing the symbol $\delta_j$ according to the type of $L_j$, for the fixed hermitian lattice $(L, h)$.
We keep this symbol from now  until the end of this paper.

$$
\delta_{j} = \left\{
  \begin{array}{l l}
  1    & \quad  \textit{if $L_j$ is \textit{of type I}};\\
  0    &   \quad  \textit{if $L_j$ is \textit{of type II}}.
    \end{array} \right.
$$
Recall from the beginning of Section 2 that
we can choose a uniformizer $\pi$ in $B$ such that $\sigma(\pi)=-\pi$ and $\pi^2=2\delta$ with $\delta (\in A) \equiv 1$ mod 2.

We reproduce the beginning of Section 3 of \cite{C2} to explain our goal.
Let $\underline{G}^{\prime}$ be the naive integral model of the unitary group $\mathrm{U}(V, h)$, where $V=L\otimes_AF$, such that
for any commutative $A$-algebra $R$,
$$\underline{G}^{\prime}(R)=\mathrm{Aut}_{B\otimes_AR}(L\otimes_AR, h\otimes_AR).$$
The scheme $\underline{G}^{\prime}$ is then an (possibly non-smooth) affine group scheme over $A$ with the smooth generic fiber $\mathrm{U}(V, h)$.
Then by Proposition 3.7 in \cite{GY}, there exists a unique smooth integral model, denoted by  $\underline{G}$, with the generic fiber $\mathrm{U}(V, h)$,
characterized by
$$\underline{G}(R)=\underline{G}^{\prime}(R)$$
for any \'etale $A$-algebra $R$.
Note that every \'etale $A$-algebra is a finite product of finite unramified extensions of $A$.
This section, Section 4 and Appendix A  are devoted to gaining an explicit knowledge of the smooth integral model
$\underline{G}$ in \textit{Case 2},
 which will be used in Section 5 to compute the local density of $(L, h)$ (again, in \textit{Case 2}).
For a detailed exposition of the relation between the local density of $(L, h)$ and $\underline{G}$, see Section 3 of \cite{GY}.

In this section, we give an explicit construction of the smooth integral model $\underline{G}$
when $E/F$ satisfies \textit{Case 2}.
The construction of $\underline{G}$ is based on that of Section 5 in \cite{GY} and Section 3 in \cite{C2}.
Since the functor $R \mapsto \underline{G}(R)$ restricted to \'etale $A$-algebras $R$ determines $\underline{G}$,
we first list out some properties that are satisfied by each element of $\underline{G}(R)=\underline{G}^{\prime}(R)$.

We choose an element $g\in \underline{G}(R)$ for an  \'etale $A$-algebra $R$.
Then $g$ is an element of $\mathrm{Aut}_{B\otimes_AR}(L\otimes_AR, h\otimes_AR)$.
Here we consider $\mathrm{Aut}_{B\otimes_AR}(L\otimes_AR, h\otimes_AR)$ as a subgroup of $ \mathrm{Res}_{E/F}\mathrm{GL}_E(V)(F\otimes_AR)$.
To ease the notation, we say $g\in \mathrm{Aut}_{B\otimes_AR}(L\otimes_AR, h\otimes_AR)$ stabilizes
a lattice $M\subseteq V$ if $g(M\otimes_AR)=M\otimes_AR$.

\subsection{Main construction}\label{mc}

Let $R$ be an \'etale $A$-algebra.
    In this subsection, as mentioned above, we observe properties of elements of $\mathrm{Aut}_{B\otimes_AR}(L\otimes_AR, h\otimes_AR)$
     and their matrix interpretations.
    We choose a Jordan splitting $L=\bigoplus_iL_i$ and a basis of $L$ as explained in Theorem 2.4 and Remark 2.5.(a).
     Let $n_i=\mathrm{rank}_{B}L_i$, and $n=\mathrm{rank}_{B}L=\sum n_i$.
 Assume that $n_i=0$ unless $0\leq i < N$.
    Let $g$ be an element of $\mathrm{Aut}_{B\otimes_AR}(L\otimes_AR, h\otimes_AR)$.
 We always divide a matrix $g$ of size $n \times n$ into $N^2$ blocks such that the block in position $(i, j)$ is of size $n_i\times n_j$.
 For simplicity, the row and column numbering starts at $0$ rather than $1$.

 \begin{enumerate}
    \item[(1)] First of all, $g$ stabilizes $A_i$ for every integer $i$.
In terms of matrices, this fact means that the $(i,j)$-block
 has entries in $\pi^{max\{0,j-i\}}B\otimes_AR$.
 From now on, we write
\[g= \begin{pmatrix} \pi^{max\{0,j-i\}}g_{i,j} \end{pmatrix}.\]

\item[(2)]
   Let $i$ be even. 
   Then $g$ stabilizes $A_i, B_i, W_i, X_i$ and induces the identity on $A_i/B_i$ and $W_i/X_i$.
     We also interpret these facts in terms of matrices as described below:
\begin{itemize}
\item[(a)]      If  $L_i$ is \textit{of type II},
then $A_i=B_i$ and  $W_i=X_i$ and so there is no contribution.
\item[(b)]     If $L_i$ is \textit{of type} $\textit{I}^o$, the diagonal $(i,i)$-block $g_{i,i}$ is of the form
     \[\begin{pmatrix} s_i&\pi y_i\\ \pi v_i&1+\pi z_i \end{pmatrix}\in \mathrm{GL}_{n_i}(B\otimes_AR),\]
     where $s_i$ is an $(n_i-1) \times (n_i-1)-$matrix, etc.
\item[(c)]
     If $L_i$ is \textit{of type} $\textit{I}^e$, the diagonal $(i,i)$-block $g_{i,i}$ is of the form
     \[\begin{pmatrix} s_i&r_i&\pi t_i\\ \pi y_i&1+\pi x_i&\pi z_i\\ v_i&u_i&1+\pi w_i \end{pmatrix}\in \mathrm{GL}_{n_i}(B\otimes_AR),\]
     where $s_i$ is an $(n_i-2) \times (n_i-2)-$matrix, etc.\\
\end{itemize}

\item[(3)]
 Let $i=2m$ be even. 
   Then  $g$  stabilizes $Z_i$ and induces the identity on $W_i/(X_i\cap Z_i)$.
For the proof,   it is easy to show that $g$  stabilizes $Z_i$.
     To prove the latter, we choose an element $w$ in $W_i$.  Since  $gw-w\in X_i$ by the above step (2),
  it suffices to show that $gw-w\in Z_i$.
  Recall that $Z_i$ is the sublattice of $A_i$ such that $Z_i/\pi B_i$ is the radical of the quadratic form $\frac{1}{2^{m+1}}q$ mod $2$ on $B_i/\pi B_i$,
  where $\frac{1}{2^{m+1}}q(a)=\frac{1}{2^{m+1}}h(a,a)$ for $a\in B_i$.
  The lattice $Z_i$ can also be described as  the sublattice of $B_i$
  such that $Z_i/\pi B_i$ is the kernel of the additive polynomial $\frac{1}{2^{m+1}}q$ mod $2$ on $Y_i/\pi B_i$.
 Now it is also easy to show that $gw-w\in Y_i$. Thus our claim that  $gw-w\in Z_i$ follows from the following computation:

     $$\frac{1}{2}\cdot \frac{1}{2^m} q(gw-w)=\frac{1}{2}(2\cdot \frac{1}{2^m}q(w)- \frac{1}{2^m}(h(gw,w)+h(w,gw)))=$$
     $$\frac{1}{2^m}(q(w)-\frac{1}{2}(h(w+x,w)+h(w,w+x))) =-\frac{1}{2^m}\cdot\frac{1}{2}(h(x,w)+h(w,x))=0 \textit{ mod } 2,$$
 where $gw=w+x$ for some $x\in X_i$, since $\frac{1}{2^m}h(x,w)\in \pi B$ and $\pi + \sigma(\pi) \in 4A$.
 Recall that $\frac{1}{2^m}q(w)=\frac{1}{2^m}h(w,w)$.

We interpret this in terms of matrices.
\begin{itemize}
\item If $L_i$ is \textit{free of type II}, then $W_i=X_i=Z_i$ and so there is no contribution.
\item If $L_i$ is \textit{bound of type II}, then $W_i=X_i$ and $Z_i$ is as explained in Remark 2.11 of \cite{C2}.
Matrix interpretation in this case is covered by matrix interpretation of Steps (2) and (4) below.
Thus there is no new contribution.

\item If $L_i$ is \textit{of type I}, then
    $$z_i+\delta_{i-2}k_{i-2, i}+\delta_{i+2}k_{i+2, i} \in (\pi).$$
    Here,
    \begin{itemize}
  \item[(a)] $z_i$ is an entry of $g_{i,i}$ as described in  \textit{Step (2)}.
\item[(b)] $k_{i-2, i}$ (resp. $k_{i+2, i}$) is the $(n_{i-2}, n_i)^{th}$-entry (resp. $(n_{i+2}, n_i)^{th}$-entry) of the matrix $g_{i-2, i}$
(resp. $g_{i+2, i}$)
if $L_{i-2}$ (resp. $L_{i+2}$) is \textit{of type} $\textit{I}^o$.
\item[(c)] $k_{i-2, i}$ (resp. $k_{i+2, i}$) is the $(n_{i-2}-1, n_i)^{th}$-entry (resp. $(n_{i+2}-1, n_i)^{th}$-entry) of the matrix $g_{i-2, i}$
(resp. $g_{i+2, i}$)
if $L_{i-2}$ (resp. $L_{i+2}$) is \textit{of type} $\textit{I}^e$.\\
\end{itemize}
\end{itemize}

\item[(4)]
 Assume that $i$ is odd.
Then  $g$ induces the identity on $A_i/B_i$.
To interpret this as a matrix, we  consider the following $(1\times n_i)$-matrix:
$$
\left\{
  \begin{array}{l l}
v_i\cdot (g_{i, i}-\mathrm{Id}_{n_i})   &   \quad  \textit{if $L_i$ is \textit{free of type I}};\\
  \delta_{i-1}v_{i-1}\cdot g_{i-1, i}+\delta_{i+1}v_{i+1}\cdot g_{i+1, i}   & \quad  \textit{if $L_i$ is \textit{bound of type I}}.
    \end{array} \right.
$$

Here,
    \begin{itemize}
 \item[(a)]  $v_{i}=(0,\cdots, 0, 1)$ of size $1\times n_{i}$ and $\mathrm{Id}_{n_i}$ is the identity matrix of size $n_i \times n_i$.
   \item[(b)] $v_{i-1}=(0,\cdots, 0, 1)$ (resp. $v_{i-1}=(0,\cdots, 0, 1, 0)$) of size $1\times n_{i-1}$
if $L_{i-1}$ is \textit{of type} $\textit{I}^o$ (resp. \textit{of type} $\textit{I}^e$).
 \item[(c)] $v_{i+1}=(0,\cdots, 0, 1)$ (resp. $v_{i+1}=(0,\cdots, 0, 1, 0)$) of size $1\times n_{i+1}$
if $L_{i+1}$ is \textit{of type} $\textit{I}^o$ (resp. \textit{of type} $\textit{I}^e$).\\
 \end{itemize}

Then each entry of the above matrix  lies  in the ideal $(\pi)$.
If $L_i$ is \textit{of type II}, then $A_i=B_i$ so that there is no contribution.
\\

\item[(5)] 
Assume that $i$ is odd.
The fact that $g$ induces the identity on $A_i/B_i$  is equivalent to
  the fact that $g$ induces the identity on $B_i^{\perp}/A_i^{\perp}$.

  We give another description of this condition.
  Since the space V has a non-degenerate bilinear form $h$, V can be identified with its own dual.
 We define the adjoint $g^{\ast}$ characterized by $h(gv,w)=h(v, g^{\ast}w)$.
 Then the fact that $g$ induces the identity on $B_i^{\perp}/A_i^{\perp}$ is the same as the fact that $g^{\ast}$ induces the identity on $A_i/B_i$.

In terms of matrices, we  consider the following $(1\times n_i)$-matrix:
$$
\left\{
  \begin{array}{l l}
v_i\cdot ({}^tg_{i, i}-\mathrm{Id}_{n_i})   &   \quad  \textit{if $L_i$ is \textit{free of type I}};\\
  \delta_{i-1}v_{i-1}\cdot {}^tg_{i, i-1}+\delta_{i+1}v_{i+1}\cdot {}^tg_{i, i+1}   & \quad  \textit{if $L_i$ is \textit{bound of type I}}.
    \end{array} \right.
$$

Here,
    \begin{itemize}
 \item[(a)]  $v_{i}=(0,\cdots, 0, 1, 0)$ of size $1\times n_{i}$ and $\mathrm{Id}_{n_i}$ is the identity matrix of size $n_i \times n_i$.
  \item[(b)] $v_{i-1}$ (resp. $v_{i+1}$)$=(0,\cdots, 0, 1)$  of size $1\times n_{i-1}$ (resp. $1\times n_{i+1}$).\\
 \end{itemize}
Then each entry of the above matrix lies in the ideal $(\pi)$.
If $L_i$ is \textit{of type II}, then $A_i=B_i$ and $B_i^{\perp}=A_i^{\perp}$ so that there is no contribution.\\
\end{enumerate}

 \subsection{Construction of \textit{\underline{M}}}\label{m}

   We define a functor from the category of commutative flat $A$-algebras to the category of monoids as follows. For any commutative flat $A$-algebra $R$, let
    $$\underline{M}(R) \subset \{m \in \mathrm{End}_{B\otimes_AR}(L \otimes_A R)\}$$
    to be the set of $m \in \mathrm{End}_{B\otimes_AR}(L \otimes_A R)$ satisfying the following conditions:
    \begin{itemize}
     \item[(1)]  $m$ stabilizes $A_i\otimes_A R,B_i\otimes_A R,W_i\otimes_A R,X_i\otimes_A R$ for all $i$
     and $Z_i\otimes_A R$ for all even integer $i$.
     \item[(2)]  $m$ induces the identity on $A_i\otimes_A R/ B_i\otimes_A R$ for all $i$.
  \item[(3)] $m$ induces the identity on $W_i\otimes_A R/(X_i\cap Z_i)\otimes_A R$ for all even integer $i$.
   \item[(4)] $m$ induces the identity on $B_i^{\perp}\otimes_A R/ A_i^{\perp}\otimes_A R$ for all odd integer $i$.
\end{itemize}

\begin{Rmk}\label{r31}
We give another description for the functor $\underline{M}$.
Let us define a functor from the category of commutative flat $A$-algebras to the category of rings as follows:

For any commutative flat $A$-algebra $R$, define
$$\underline{M}^{\prime}(R) \subset \{m \in \mathrm{End}_{B\otimes_AR}(L \otimes_A R) \}$$
to be the set of $m\in  \mathrm{End}_{B\otimes_AR}(L \otimes_A R)$ satisfying the following conditions:
\begin{itemize}
     \item[(1)]  $m$ stabilizes $A_i\otimes_A R,B_i\otimes_A R,W_i\otimes_A R,X_i\otimes_A R$ for all $i$
     and $Z_i\otimes_A R$ for all even integer $i$.
 \item[(2)] $m $ maps $A_i\otimes_A R$ into $B_i\otimes_A R$ for all $i$.
 \item[(3)] $m $ maps $W_i\otimes_A R$ into $(X_i\cap Z_i)\otimes_A R$  for all even integer $i$.
 \item[(4)] $m $ maps $B_i^{\perp}\otimes_A R$ into $A_i^{\perp}\otimes_A R$  for all odd integer $i$.
\end{itemize}

Then, by  Lemma 3.1 of \cite{C1}, $\underline{M}^{\prime}$ is  represented by a unique flat $A$-algebra $A(\underline{M'})$ which is a polynomial ring over $A$ of $2n^2$ variables.
    Moreover, it is easy to see that $\underline{M}'$ has the structure of a scheme of rings since
    $\underline{M}'(R)$ is  closed under addition and  multiplication.

    Then the functor $\underline{M}$ is equivalent to the functor $1+\underline{M}^{\prime}$
    as subfunctors of $\mathrm{Res}_{B/A}\mathrm{End}_B(L)$,
    where $(1+\underline{M}^{\prime})(R)=\{1+m : m \in \underline{M}^{\prime}(R) \}$
    for a  commutative flat $A$-algebra $R$.
    For a detailed explanation about this, we refer to Remark 3.1 of \cite{C2}.
   This fact induces that
   the functor $\underline{M}$ is also represented by a unique flat $A$-algebra $A[\underline{M}]$ which is a polynomial ring over $A$ of $2n^2$ variables.
    Moreover, it is easy to see that $\underline{M}$ has the structure of a scheme of monoids since $\underline{M}(R)$ is  closed under multiplication.
\end{Rmk}

    We can therefore now talk of $\underline{M}(R)$ for any (not necessarily flat) $A$-algebra $R$.
Before describing an element of $\underline{M}(R)$ for a non-flat $A$-algebra $R$,
we explicitly, in terms of matrices, describe an element of $\underline{M}(R)$ for a flat $A$-algebra $R$.
By using our chosen basis of $L$, each element of  $\underline{M}(R)$ is written as
$$m= \begin{pmatrix} \pi^{max\{0,j-i\}}m_{i,j} \end{pmatrix} \mathrm{~together~with~}z_i^{\ast}, m_{i,i}^{\ast}, m_{i,i}^{\ast\ast}.$$
    Here,
    \begin{enumerate}
    \item[(a)] When $i\neq j$, $m_{i,j}$  is an  $(n_i \times n_j)$-matrix with entries in $B\otimes_AR$.
\item[(b)]  When $i=j$, $m_{i,i}$  is an  $(n_i \times n_i)$-matrix with entries in $B\otimes_AR$ such that
\[
m_{i,i}=\left\{
  \begin{array}{l l}
  \begin{pmatrix} s_i&\pi y_i\\ \pi v_i&1+\pi z_i \end{pmatrix}    & \quad  \textit{if $i$ is even and $L_i$ is \textit{of type $I^o$}};\\
  \begin{pmatrix} s_i&r_i&\pi t_i\\ \pi y_i&1+\pi x_i&\pi z_i\\ v_i&u_i&1+\pi w_i \end{pmatrix}    & \quad  \textit{if $i$ is even and $L_i$ is \textit{of type $I^e$}};\\
\begin{pmatrix} s_i&\pi r_i&t_i\\  y_i&1+\pi x_i& u_i\\\pi  v_i&\pi z_i&1+\pi w_i \end{pmatrix}
& \quad  \textit{if $i$ is odd and $L_i$ is \textit{free of type $I$}};\\
  m_{i,i}   & \quad  \textit{otherwise}.
      \end{array} \right.
\]
Here, $s_i$  is an  $(n_i-1 \times n_i-1)$-matrix (resp. $(n_i-2 \times n_i-2)$-matrix)
if $i$ is even and $L_i$ is \textit{of type $I^o$} (resp. if $i$ is even and $L_i$ is \textit{of type $I^e$},
or if $i$ is odd and $L_i$ is \textit{free of type $I$})
and $y_i, v_i, z_i, r_i, t_i, y_i, x_i, u_i, w_i$ are matrices of suitable sizes.

\item[(c)]
Assume that $i$ is even and that $L_i$ is \textit{of type I}.  Then
$$z_i+\delta_{i-2}k_{i-2, i}+\delta_{i+2}k_{i+2, i}=\pi z_i^{\ast}$$
such that  $z_i^{\ast}\in B\otimes_AR$.   Here,
    \begin{itemize}
  \item[(i)] $z_i$ is an entry of $m_{i,i}$ as described in the above step (b).
\item[(ii)] $k_{i-2, i}$ (resp. $k_{i+2, i}$) is the $(n_{i-2}, n_i)^{th}$-entry (resp. $(n_{i+2}, n_i)^{th}$-entry) of the matrix $m_{i-2, i}$
(resp. $m_{i+2, i}$)
if $L_{i-2}$ (resp. $L_{i+2}$) is \textit{of type} $\textit{I}^o$.
\item[(iii)] $k_{i-2, i}$ (resp. $k_{i+2, i}$) is the $(n_{i-2}-1, n_i)^{th}$-entry (resp. $(n_{i+2}-1, n_i)^{th}$-entry) of the matrix $m_{i-2, i}$
(resp. $m_{i+2, i}$)
if $L_{i-2}$ (resp. $L_{i+2}$) is \textit{of type} $\textit{I}^e$.
\end{itemize}

\item[(d)] Assume that $i$ is odd and that $L_i$ is \textit{bound of type I}. Then
$$\delta_{i-1}v_{i-1}\cdot m_{i-1, i}+\delta_{i+1}v_{i+1}\cdot m_{i+1, i}=\pi m_{i,i}^{\ast}$$
such that  $m_{i,i}^{\ast} \in M_{1\times n_i}(B\otimes_AR)$.
Here,
    \begin{itemize}
   \item[(i)] $v_{i-1}=(0,\cdots, 0, 1)$ (resp. $v_{i-1}=(0,\cdots, 0, 1, 0)$) of size $1\times n_{i-1}$
if $L_{i-1}$ is \textit{of type} $\textit{I}^o$ (resp. \textit{of type} $\textit{I}^e$).
   \item[(ii)] $v_{i+1}=(0,\cdots, 0, 1)$ (resp. $v_{i+1}=(0,\cdots, 0, 1, 0)$) of size $1\times n_{i+1}$
if $L_{i+1}$ is \textit{of type} $\textit{I}^o$ (resp. \textit{of type} $\textit{I}^e$).
 \end{itemize}

 \item[(e)] Assume that $i$ is odd and that $L_i$ is \textit{bound of type I}. Then
 $$ \delta_{i-1}v_{i-1}\cdot {}^tm_{i, i-1}+\delta_{i+1}v_{i+1}\cdot {}^tm_{i, i+1} = \pi m_{i,i}^{\ast\ast}$$
 such that $ m_{i,i}^{\ast\ast} \in M_{1\times n_i}(B\otimes_AR)$.
Here,
 $v_{i-1}$ (resp. $v_{i+1}$)$=(0,\cdots, 0, 1)$  of size $1\times n_{i-1}$ (resp. $1\times n_{i+1}$).\\

\end{enumerate}

In conclusion, each element of $\underline{M}(R)$ for a flat $A$-algebra $R$ is written as the following matrix
$$m= \begin{pmatrix} \pi^{max\{0,j-i\}}m_{i,j} \end{pmatrix} \mathrm{~together~with~}z_i^{\ast}, m_{i,i}^{\ast}, m_{i,i}^{\ast\ast}.$$

\textit{  }\\

   However, for a general $R$, the above description for  $\underline{M}(R)$ will  no longer be true.
For such $R$, we use our chosen basis of $L$ to write each element of $\underline{M}(R)$ formally.
 We describe each element of $\underline{M}(R)$  as  formal matrices
$$m= \begin{pmatrix} \pi^{max\{0,j-i\}}m_{i,j} \end{pmatrix} \mathrm{~with~}z_i^{\ast}, m_{i,i}^{\ast}, m_{i,i}^{\ast\ast}$$
such that
    \begin{enumerate}
    \item[(a)] When $i\neq j$, $m_{i,j}$  is an  $(n_i \times n_j)$-matrix with entries in $B\otimes_AR$.
\item[(b)]  When $i=j$, $m_{i,i}$  is an  $(n_i \times n_i)$-formal matrix such that
\[
m_{i,i}=\left\{
  \begin{array}{l l}
  \begin{pmatrix} s_i&\pi y_i\\ \pi v_i&1+\pi z_i \end{pmatrix}    & \quad  \textit{if $i$ is even and $L_i$ is \textit{of type $I^o$}};\\
  \begin{pmatrix} s_i&r_i&\pi t_i\\ \pi y_i&1+\pi x_i&\pi z_i\\ v_i&u_i&1+\pi w_i \end{pmatrix}    & \quad  \textit{if $i$ is even and $L_i$ is \textit{of type $I^e$}};\\
\begin{pmatrix} s_i&\pi r_i&t_i\\  y_i&1+\pi x_i& u_i\\\pi  v_i&\pi z_i&1+\pi w_i \end{pmatrix}
& \quad  \textit{if $i$ is odd and $L_i$ is \textit{free of type $I$}};\\
  m_{i,i}   & \quad  \textit{otherwise}.
      \end{array} \right.
\]
Here, $s_i$  is an  $(n_i-1 \times n_i-1)$-matrix (resp. $(n_i-2 \times n_i-2)$-matrix) with entries in $B\otimes_AR$
if $i$ is even and $L_i$ is \textit{of type $I^o$} (resp. if $i$ is even and $L_i$ is \textit{of type $I^e$},
or if $i$ is odd and $L_i$ is \textit{free of type $I$})
and $y_i, v_i, z_i, r_i, t_i, y_i, x_i, u_i, w_i$ are matrices of suitable sizes with entries in $B\otimes_AR$.
Similarly, in other cases, i.e.
if $i$ is even and $L_i$ is of type $II$, or if $i$ is odd and $L_i$ is of type $II$ or \textit{bound of type $I$},
then $m_{i,i}$ is an  $(n_i \times n_i)$-matrix with entries in $B\otimes_AR$.

\item[(c)]
Assume that $i$ is even and that $L_i$ is \textit{of type I}.  Then
$$z_i+\delta_{i-2}k_{i-2, i}+\delta_{i+2}k_{i+2, i}=\pi z_i^{\ast}$$
such that  $z_i^{\ast}\in B\otimes_AR$.
This equation is considered in $B\otimes_AR$ and $\pi$ stands for $\pi\otimes 1\in B\otimes_AR$.
 Here,
    \begin{itemize}
  \item[(i)] $z_i$ is an entry of $m_{i,i}$ as described in the above step (b).
\item[(ii)] $k_{i-2, i}$ (resp. $k_{i+2, i}$) is the $(n_{i-2}, n_i)^{th}$-entry (resp. $(n_{i+2}, n_i)^{th}$-entry) of the matrix $m_{i-2, i}$
(resp. $m_{i+2, i}$)
if $L_{i-2}$ (resp. $L_{i+2}$) is \textit{of type} $\textit{I}^o$.
\item[(iii)] $k_{i-2, i}$ (resp. $k_{i+2, i}$) is the $(n_{i-2}-1, n_i)^{th}$-entry (resp. $(n_{i+2}-1, n_i)^{th}$-entry) of the matrix $m_{i-2, i}$
(resp. $m_{i+2, i}$)
if $L_{i-2}$ (resp. $L_{i+2}$) is \textit{of type} $\textit{I}^e$.
\end{itemize}

\item[(d)] Assume that $i$ is odd and that $L_i$ is \textit{bound of type I}. Then
$$\delta_{i-1}v_{i-1}\cdot m_{i-1, i}+\delta_{i+1}v_{i+1}\cdot m_{i+1, i}=\pi m_{i,i}^{\ast}$$
such that  $m_{i,i}^{\ast} \in M_{1\times n_i}(B\otimes_AR)$.
This equation is considered in $B\otimes_AR$ and $\pi$ stands for $\pi\otimes 1\in B\otimes_AR$.
Here,
    \begin{itemize}
   \item[(i)] $v_{i-1}=(0,\cdots, 0, 1)$ (resp. $v_{i-1}=(0,\cdots, 0, 1, 0)$) of size $1\times n_{i-1}$
if $L_{i-1}$ is \textit{of type} $\textit{I}^o$ (resp. \textit{of type} $\textit{I}^e$).
   \item[(ii)] $v_{i+1}=(0,\cdots, 0, 1)$ (resp. $v_{i+1}=(0,\cdots, 0, 1, 0)$) of size $1\times n_{i+1}$
if $L_{i+1}$ is \textit{of type} $\textit{I}^o$ (resp. \textit{of type} $\textit{I}^e$).
 \end{itemize}

 \item[(e)] Assume that $i$ is odd and that $L_i$ is \textit{bound of type I}.
  Then
 $$ \delta_{i-1}v_{i-1}\cdot {}^tm_{i, i-1}+\delta_{i+1}v_{i+1}\cdot {}^tm_{i, i+1} = \pi m_{i,i}^{\ast\ast}$$
 such that $ m_{i,i}^{\ast\ast} \in M_{1\times n_i}(B\otimes_AR)$.
This equation is considered in $B\otimes_AR$ and $\pi$ stands for $\pi\otimes 1\in B\otimes_AR$.
Here,
 $v_{i-1}$ (resp. $v_{i+1}$)$=(0,\cdots, 0, 1)$  of size $1\times n_{i-1}$ (resp. $1\times n_{i+1}$).\\

\end{enumerate}

To simplify notation, each element
$$((m_{i,j})_{i\neq j}, (m_{i,i})_{\textit{$L_i$ of type II, or bound of type I with i odd}},
(s_i, y_i, v_i, z_i)_{\textit{$L_i$ of type $I^o$ with i even}},$$
$$(s_i, v_i, z_i, r_i, t_i, y_i, x_i, u_i, w_i)_{\textit{$L_i$ of type $I^e$ with i even, or free of type I with i odd}},$$
$$(z_i^{\ast})_{\textit{$L_i$ of type I with i even}}, (m_{i,i}^{\ast}, m_{i,i}^{\ast\ast})_{\textit{$L_i$ bound of type I with i odd}} )$$ of  $\underline{M}(R)$, for a $\kappa$-algebra $R$, is denoted by $(m_{i,j}, s_i \cdots w_i)$.\\

In the next section, we need a description of an element of $\underline{M}(R)$ and its multiplication for a $\kappa$-algebra $R$.
In order to prepare for this, we describe the multiplication explicitly only for a $\kappa$-algebra $R$.
   To multiply $(m_{i,j}, s_i\cdots w_i)$ and $(m_{i,j}', s_i'\cdots w_i')$,
    we form the matrices $m=\begin{pmatrix} \pi^{max\{0,j-i\}}m_{i,j} \end{pmatrix}$ and $m'=\begin{pmatrix} \pi^{max\{0,j-i\}}m_{i,j}' \end{pmatrix}$
    with $s_i\cdots w_i$ and $s_i'\cdots w_i'$
     and write the formal matrix product  $\begin{pmatrix} \pi^{max\{0,j-i\}}m_{i,j} \end{pmatrix}\cdot \begin{pmatrix} \pi^{max\{0,j-i\}}m_{i,j}' \end{pmatrix}=\begin{pmatrix} \pi^{max\{0,j-i\}}\tilde{m}_{i,j}'' \end{pmatrix}$
     with
    \[
\tilde{m}_{i,i}''=\left\{
  \begin{array}{l l}
  \begin{pmatrix} \tilde{s}_i''&\pi \tilde{y}_i''\\ \pi \tilde{v}_i''&1+\pi \tilde{z}_i'' \end{pmatrix}    & \quad  \textit{if $i$ is even and $L_i$ is \textit{of type $I^o$}};\\
  \begin{pmatrix} \tilde{s}_i''&\tilde{r}_i''&\pi \tilde{t}_i''\\ \pi \tilde{y}_i''&1+\pi \tilde{x}_i''&\pi \tilde{z}_i''\\ \tilde{v}_i''&\tilde{u}_i''&1+\pi \tilde{w}_i'' \end{pmatrix}    & \quad  \textit{if $i$ is even and $L_i$ is \textit{of type $I^e$}};\\
  \begin{pmatrix} \tilde{s}_i''&\pi \tilde{r}_i''& \tilde{t}_i''\\ \tilde{y}_i''&1+\pi \tilde{x}_i''& \tilde{u}_i''\\
  \pi \tilde{v}_i''&\pi \tilde{z}_i''&1+\pi \tilde{w}_i'' \end{pmatrix}    & \quad  \textit{if $i$ is odd and $L_i$ is \textit{free of type $I$}}.
      \end{array} \right.
\]

Let
    \[
\left\{
  \begin{array}{l l}
  (\tilde{z}_i^{\ast})''=1/\pi\left(\tilde{z}_i''+\delta_{i-2}\tilde{k}_{i-2, i}''+\delta_{i+2}\tilde{k}_{i+2, i}'' \right)    & \quad  \textit{if $i$ is even and $L_i$ is \textit{of type I}};\\
  (\tilde{m}_{i,i}^{\ast})''=1/\pi\left(\delta_{i-1}v_{i-1}\cdot \tilde{m}_{i-1, i}''+\delta_{i+1}v_{i+1}\cdot \tilde{m}_{i+1, i}''\right)
  & \quad  \textit{if $i$ is odd and $L_i$ is \textit{bound of type I}};\\
  (\tilde{m}_{i,i}^{\ast\ast})''=1/\pi\left(\delta_{i-1}v_{i-1}\cdot {}^t\tilde{m}_{i, i-1}''+\delta_{i+1}v_{i+1}\cdot {}^t\tilde{m}_{i, i+1}''\right)
  & \quad  \textit{if $i$ is odd and $L_i$ is \textit{bound of type I}}.
      \end{array} \right.
\]
Here, notations ($\tilde{z}_i''$, $\tilde{k}_{i-2, i}''$, and so on) in the right hand sides are
as explained above in the description of an element of $\underline{M}(R)$  in terms of $\tilde{m}_{i,j}''$.
These three equations should be interpreted as follows.
We formally compute the right hand sides and then they are of the form $1/\pi(\pi X)$.
Then the left hand sides are the same as $X$ in the right hand sides.
Such $X$'s are formal polynomials about $(m_{i,j}, s_i\cdots w_i)$ and $(m_{i,j}', s_i'\cdots w_i')$.
The precise description will be given below.


Let $(m_{i,j}'', s_i''\cdots w_i'')$ be formed by letting $\pi^2$ be zero
in each entry of $(\tilde{m}_{i,j}'', \tilde{s}_i''\cdots \tilde{w}_i'')$.
Then each matrix of $(m_{i,j}'', s_i''\cdots w_i'')$ has entries in $B\otimes_AR$ and so
$(m_{i,j}'', s_i''\cdots w_i'')$ is an element of $\underline{M}(R)$
and is the product  of $(m_{i,j}, s_i\cdots w_i)$ and $(m_{i,j}', s_i'\cdots w_i')$.
More precisely,
\begin{enumerate}
\item If $i\neq j$,
    if $i= j$ and $L_i$ is \textit{of type $II$}, or if $L_i$ is \textit{bound of type $I$} with odd $i$,
    $$m_{i,j}''=\sum_{k=1}^{N}\pi^{(max\{0, k-i\}+max\{0, j-k\}-max\{0, j-i\})}m_{i, k}m_{k, j}';$$

\item For $L_i$ \textit{of type $I^o$} with $i$ even, we write $m_{i, i-1}m_{i-1, i}'+m_{i, i+1}m_{i+1, i}'=\begin{pmatrix} a_i''&b_i''\\ c_i''&d_i'' \end{pmatrix}$
and $m_{i, i-2}m_{i-2, i}'+m_{i, i+2}m_{i+2, i}'=\begin{pmatrix} \tilde{a}_i''&\tilde{b}_i''\\ \tilde{c}_i''&\tilde{d}_i'' \end{pmatrix}$
where
$a_i''$ and $\tilde{a}_i''$ are  $(n_i-1) \times (n_i-1)$-matrices, etc.
Then
\[\left\{
  \begin{array}{l}
  s_i''=s_is_i'+\pi a_i'';\\
  y_i''=s_iy_i'+y_i+b_i''+\pi (y_iz_i'+ \tilde{b}_i'');\\
  v_i''=v_is_i'+v_i'+c_i''+\pi (z_iv_i'+ \tilde{c}_i'');\\
  z_i''=z_i+z_i'+d_i''+\pi (z_iz_i'+ v_iy_i'+ \tilde{d}_i'').
      \end{array} \right.
\]
\item When $L_i$ is \textit{of type $I^e$}  with $i$ even, we write $m_{i, i-1}m_{i-1, i}'+m_{i, i+1}m_{i+1, i}'=
\begin{pmatrix} a_i''&b_i''&c_i''\\ d_i''&e_i''&f_i''\\ g_i''&h_i''&k_i'' \end{pmatrix}$
and $m_{i, i-2}m_{i-2, i}'+m_{i, i+2}m_{i+2, i}'=
\begin{pmatrix} \tilde{a}_i''&\tilde{b}_i''&\tilde{c}_i''\\ \tilde{d}_i''&\tilde{e}_i''&\tilde{f}_i''\\ \tilde{g}_i''&\tilde{h}_i''&\tilde{k}_i'' \end{pmatrix}$
where
$a_i''$ and $\tilde{a}_i''$ are  $(n_i-2) \times (n_i-2)$-matrices, etc.
Then
\[\left\{
  \begin{array}{l}
  s_i''=s_is_i'+\pi (r_iy_i'+t_iv_i'+a_i'');\\
  r_i''=s_ir_i'+r_i+\pi (r_ix_i'+t_iu_i'+b_i'')   ;\\
  t_i''=s_it_i'+r_iz_i'+t_i+c_i''+\pi (t_iw_i'+\tilde{c}_i'');\\
y_i''=y_is_i'+y_i'+z_iv_i'+d_i''+\pi (x_iy_i'+\tilde{d}_i'');\\
x_i''=x_i+x_i'+z_iu_i'+y_ir_i'+e_i''+\pi (x_ix_i'+\tilde{e}_i'');\\
z_i''=z_i+z_i'+f_i''+\pi (y_it_i'+x_iz_i'+z_iw_i'+\tilde{f}_i'');\\
v_i''=v_is_i'+v_i'+\pi (u_iy_i'+w_iv_i'+g_i'');\\
u_i''=u_i+u_i'+v_ir_i'+\pi(u_ix_i'+w_iu_i'+h_i'');\\
w_i''=w_i+w_i'+v_it_i'+u_iz_i'+k_i''+\pi (w_iw_i'+\tilde{k}_i'').
      \end{array} \right.
\]
\item When $L_i$ is \textit{free of type $I$}  with $i$ odd, we write $m_{i, i-1}m_{i-1, i}'+m_{i, i+1}m_{i+1, i}'=
\begin{pmatrix} a_i''&b_i''&c_i''\\ d_i''&e_i''&f_i''\\ g_i''&h_i''&k_i'' \end{pmatrix}$
and $m_{i, i-2}m_{i-2, i}'+m_{i, i+2}m_{i+2, i}'=
\begin{pmatrix} \tilde{a}_i''&\tilde{b}_i''&\tilde{c}_i''\\ \tilde{d}_i''&\tilde{e}_i''&\tilde{f}_i''\\ \tilde{g}_i''&\tilde{h}_i''&\tilde{k}_i'' \end{pmatrix}$
where
$a_i''$ and $\tilde{a}_i''$ are  $(n_i-2) \times (n_i-2)$-matrices, etc.
Then
\[\left\{
  \begin{array}{l}
  s_i''=s_is_i'+\pi (r_iy_i'+t_iv_i'+a_i'');\\
  r_i''=s_ir_i'+r_i+t_iv_i'+b_i''+     \pi (r_ix_i'+\tilde{b}_i'')   ;\\
  t_i''=s_it_i'   +t_i+\pi (r_iu_i'+t_iw_i'+c_i'');\\
 y_i''=y_is_i'+y_i' +\pi (x_iy_i'+u_iv_i'+d_i'');\\
 x_i''=x_i+x_i'+y_ir_i'+u_iz_i'+e_i''+\pi (x_ix_i'+\tilde{e}_i'');\\
 u_i''=u_i+u_i'+y_it_i'+\pi(u_iw_i'+x_iu_i'+f_i'');\\
 v_i''=v_is_i'+z_iy_i'+v_i'+g_i''+\pi (w_iv_i'+\tilde{g}_i'');\\
 z_i''=z_i+z_i'+h_i''+\pi (v_ir_i'+   z_ix_i'+w_iz_i'+\tilde{h}_i'');\\
 w_i''=w_i+w_i'+v_it_i'+z_iu_i'+k_i''+\pi (w_iw_i'+\tilde{k}_i'').
      \end{array} \right.
\]

\item Assume that $i$ is even and $L_i$ is \textit{of type I}.
Let  $\tilde{k}_{i-2, i}''$ (resp. $\tilde{k}_{i+2, i}''$) be the $(n_{i-2}, n_i)^{th}$-entry (resp. $(n_{i+2}, n_i)^{th}$-entry) of the formal matrix $\tilde{m}_{i-2, i}''$
(resp. $\tilde{m}_{i+2, i}''$)
if $L_{i-2}$ (resp. $L_{i+2}$) is \textit{of type} $\textit{I}^o$ and
let $\tilde{k}_{i-2, i}''$ (resp. $\tilde{k}_{i+2, i}''$) be the $(n_{i-2}-1, n_i)^{th}$-entry (resp. $(n_{i+2}-1, n_i)^{th}$-entry) of the formal matrix $\tilde{m}_{i-2, i}''$
(resp. $\tilde{m}_{i+2, i}''$)
if $L_{i-2}$ (resp. $L_{i+2}$) is \textit{of type} $\textit{I}^e$.
Then the formal sum
$$\tilde{z}_i''+\delta_{i-2}\tilde{k}_{i-2, i}''+\delta_{i+2}\tilde{k}_{i+2, i}''$$
 equals
$$z_i+z_i'+(m_{i, i-1}m_{i-1, i}')^{\dag}+(m_{i, i+1}m_{i+1, i}'))^{\dag}+\delta_{i-2}\cdot(k_{i-2, i}+k_{i-2, i}'+(m_{i-2, i-1}m_{i-1, i}')^{\dag})+$$
$$\delta_{i+2}\cdot(k_{i+2, i}+k_{i+2, i}'+(m_{i+2, i+1}m_{i+1, i}')^{\dag})+\pi \tilde{z}_i''^{\ddag}$$
for some formal expansion $\tilde{z}_i''^{\ddag}$.
Here, if $m$ is a formal matrix of size $a\times b$, then $m^{\dag}$ is the $(a, b)^{th}$-entry of $m$ (resp. $(a-1, b)^{th}$-entry of $m$) if
$L_i$ is \textit{of type} $\textit{I}^o$ (resp. $\textit{I}^e$).

Note  that
\[\left \{  \begin{array}{l}
  z_i+\delta_{i-2}\cdot k_{i-2, i}+ \delta_{i+2}\cdot k_{i+2, i}=\pi z_i^{\ast};\\
  z_i'+\delta_{i-2}\cdot k_{i-2, i}'+ \delta_{i+2}\cdot k_{i+2, i}'=\pi (z_i^{\ast})';\\
  \textit{$(m_{i, i-1}m_{i-1, i}')^{\dag}+\delta_{i-2}\cdot((m_{i-2, i-1}m_{i-1, i}')^{\dag})=\pi (m_{i-1, i-1}^{\ast}\cdot m_{i-1, i}')^{\dag}$};\\
  \textit{$(m_{i, i+1}m_{i+1, i}')^{\dag}+\delta_{i+2}\cdot((m_{i+2, i+1}m_{i+1, i}')^{\dag})=\pi (m_{i+1, i+1}^{\ast}\cdot m_{i+1, i}')^{\dag}$}.
    \end{array} \right.\]
Therefore,
$$\tilde{z}_i''+\delta_{i-2}\tilde{k}_{i-2, i}''+\delta_{i+2}\tilde{k}_{i+2, i}''=
\pi\left(z_i^{\ast}+(z_i^{\ast})'+(m_{i-1, i-1}^{\ast}\cdot m_{i-1, i}')^{\dag}+m_{i+1, i+1}^{\ast}\cdot m_{i+1, i}')^{\dag}+\tilde{z}_i''^{\ddag}\right)$$
for some formal expansion $\tilde{z}_i''^{\ddag}$.
Then
$$(z_i^{\ast})''=z_i^{\ast}+(z_i^{\ast})'+(m_{i-1, i-1}^{\ast}\cdot m_{i-1, i}')^{\dag}+m_{i+1, i+1}^{\ast}\cdot m_{i+1, i}')^{\dag}+\tilde{z}_i''^{\ddag}$$
as an equation in $B\otimes_AR$.\\

\item   Assume that $i$ is odd and $L_i$ is \textit{bound of type I}.
Then the following formal sum
$$\delta_{i-1}v_{i-1}\cdot \tilde{m}_{i-1, i}''+\delta_{i+1}v_{i+1}\cdot \tilde{m}_{i+1, i}''$$
equals
$$\delta_{i-1}v_{i-1}\cdot (m_{i-1, i}m_{i,i}'+m_{i-1, i-1}m_{i-1,i}')+\delta_{i+1}v_{i+1}\cdot (m_{i+1, i}m_{i,i}'+m_{i+1, i+1}m_{i+1,i}')+\pi \tilde{z}_i''^{\dag}=$$
$$\left(\delta_{i-1}v_{i-1}\cdot (m_{i-1, i}m_{i,i}')+\delta_{i+1}v_{i+1}\cdot (m_{i+1, i}m_{i,i}')\right)+$$
$$\left(\delta_{i-1}v_{i-1}\cdot (m_{i-1, i-1}m_{i-1,i}')+\delta_{i+1}v_{i+1}\cdot (m_{i+1, i+1}m_{i+1,i}')\right)
+\pi \tilde{z}_i''^{\dag}   $$
for some formal expansion $\tilde{z}_i''^{\dag}$.
Here, $\delta_{j}v_{j}$ is as explained in Step (d) of the above description  of an element of $\underline{M}(R)$.
Note that
\[\left \{  \begin{array}{l}
\delta_{i-1}v_{i-1}\cdot (m_{i-1, i}m_{i,i}')+\delta_{i+1}v_{i+1}\cdot (m_{i+1, i}m_{i,i}')=\pi m_{i,i}^{\ast}\cdot m_{i,i}';\\
\delta_{i-1}v_{i-1}\cdot (m_{i-1, i-1}m_{i-1,i}')+\delta_{i+1}v_{i+1}\cdot (m_{i+1, i+1}m_{i+1,i}')
=\pi (m_{i,i}^{\ast})'+\pi \tilde{z}_i''^{\dag\dag}
    \end{array} \right.\]
for some formal expansion $\tilde{z}_i''^{\dag\dag}$.
Therefore,
$$\delta_{i-1}v_{i-1}\cdot \tilde{m}_{i-1, i}''+\delta_{i+1}v_{i+1}\cdot \tilde{m}_{i+1, i}''=
\pi\left(m_{i,i}^{\ast}\cdot m_{i,i}'+(m_{i,i}^{\ast})'+ \tilde{z}_i''^{\dag\dag}+\tilde{z}_i''^{\dag}  \right).$$
Then
$$(m_{i,i}^{\ast})''=m_{i,i}^{\ast}\cdot m_{i,i}'+(m_{i,i}^{\ast})'+ \tilde{z}_i''^{\dag\dag}+\tilde{z}_i''^{\dag}$$
as an equation in $B\otimes_AR$.\\

\item   Assume that $i$ is odd and $L_i$ is \textit{bound of type I}.
Then the following formal sum
$$\delta_{i-1}v_{i-1}\cdot {}^t\tilde{m}_{i, i-1}''+\delta_{i+1}v_{i+1}\cdot {}^t\tilde{m}_{i, i+1}''$$
equals
$$\delta_{i-1}v_{i-1}\cdot ({}^tm_{i,i-1}'\cdot {}^tm_{i, i} +{}^tm_{i-1,i-1}'\cdot {}^tm_{i, i-1})+
\delta_{i+1}v_{i+1}\cdot ({}^tm_{i,i+1}'\cdot {}^tm_{i, i}+ {}^tm_{i+1,i+1}'\cdot {}^tm_{i, i+1})+\pi \tilde{z}_i''^{\dag}=$$
$$\left(\delta_{i-1}v_{i-1}\cdot ({}^tm_{i,i-1}'\cdot {}^tm_{i, i})+\delta_{i+1}v_{i+1}\cdot ({}^tm_{i,i+1}'\cdot {}^tm_{i, i})\right)+$$
$$\left(\delta_{i-1}v_{i-1}\cdot ({}^tm_{i-1,i-1}'\cdot {}^tm_{i, i-1})+\delta_{i+1}v_{i+1}\cdot ({}^tm_{i+1,i+1}'\cdot {}^tm_{i, i+1})\right)
+\pi \tilde{z}_i''^{\dag}   $$
for some formal expansion $\tilde{z}_i''^{\dag}$.
Here, $\delta_{j}v_{j}$ is as explained in Step (e) of the above description  of an element of $\underline{M}(R)$.
Note that
\[\left \{  \begin{array}{l}
\delta_{i-1}v_{i-1}\cdot ({}^tm_{i,i-1}'\cdot {}^tm_{i, i})+\delta_{i+1}v_{i+1}\cdot ({}^tm_{i,i+1}'\cdot {}^tm_{i, i})
=\pi (m_{i,i}^{\ast\ast})'\cdot {}^tm_{i, i};\\
\delta_{i-1}v_{i-1}\cdot ({}^tm_{i-1,i-1}'\cdot {}^tm_{i, i-1})+\delta_{i+1}v_{i+1}\cdot ({}^tm_{i+1,i+1}'\cdot {}^tm_{i, i+1})
=\pi m_{i,i}^{\ast\ast}+\pi \tilde{z}_i''^{\dag\dag}
    \end{array} \right.\]
for some formal expansion $\tilde{z}_i''^{\dag\dag}$.
Therefore,
$$\delta_{i-1}v_{i-1}\cdot {}^t\tilde{m}_{i, i-1}''+\delta_{i+1}v_{i+1}\cdot {}^t\tilde{m}_{i, i+1}''=
\pi\left((m_{i,i}^{\ast\ast})'\cdot {}^tm_{i, i}+m_{i,i}^{\ast\ast}+ \tilde{z}_i''^{\dag\dag}+ \tilde{z}_i''^{\dag} \right).$$
Then
$$(m_{i,i}^{\ast\ast})''=(m_{i,i}^{\ast\ast})'\cdot {}^tm_{i, i}+m_{i,i}^{\ast\ast}+ \tilde{z}_i''^{\dag\dag}+ \tilde{z}_i''^{\dag} $$
as an equation in $B\otimes_AR$.
\end{enumerate}

\textit{  }\\

\begin{Rmk}\label{r32}
 We define a functor $\underline{M}^{\ast}$ from the category of commutative $A$-algebras to the category of groups as follows.
For a commutative $A$-algebra $R$,
 set
 $$\underline{M}^{\ast}(R)=\{ m \in \underline{M}(R) :  \textit{there exists $m'\in \underline{M}(R)$ such that $m\cdot m'=m'\cdot m=1$}\}.$$
Then $\underline{M}^{\ast}$ is an open subscheme of $\underline{M}$
with generic fiber $M^{\ast}=\mathrm{Res}_{E/F}\mathrm{GL}_E(V)$,
and that $\underline{M}^{\ast}$ is smooth over $A$. Moreover, $\underline{M}^{\ast}$ is a group scheme since $\underline{M}$ is a scheme in monoids.
The proof of this is similar to that of Remark 3.2 in \cite{C2} and so we skip it.
\end{Rmk}

 \subsection{Construction of \textit{\underline{H}}}\label{h}

 Recall that the pair $(L, h)$ is fixed throughout this paper
  and the lattices $A_i$, $B_i$, $W_i$, $X_i$, $Z_i$ only depend on the hermitian pair $(L, h)$.
  For any flat $A$-algebra $R$, let $\underline{H}(R)$ be the set of hermitian forms $f$ on $L\otimes_{A}R$ (with values in $B\otimes_AR$)
  such that  $f$ satisfies the following conditions:
  \begin{enumerate}
 \item[(a)]  $f(L\otimes_{A}R,A_i\otimes_{A}R) \subset \pi^iB\otimes_AR$ for all $i$.
 \item[(b)]
 $\xi^{-m}f(a_i,a_i)$ mod 2 = $\xi^{-m}h(a_i, a_i)$ mod 2, where $a_i \in A_i \otimes_{A}R$, and $i=2m$ or $i=2m-1$.
 \item[(c)]
 $\pi^{-i}f(a_i,w_i) = \pi^{-i}h(a_i, w_i)$ mod $\pi$, where $a_i \in A_i\otimes_{A}R$ and $w_i \in W_i\otimes_{A}R$, and $i=2m$.
 \item[(d)]
  $\xi^{-m} f(w_i,w_i)-\xi^{-m}h(w_i,w_i) \in (4)$,   where $w_i \in W_i\otimes_{A}R$, and $i=2m$.
\item[(e)] $f(B_i, B_i^{\perp})\in B\otimes_{A}R$ and $f(a_i, b_i^{\prime})-h(a_i, b_i^{\prime})\in B\otimes_{A}R$,
  where $a_i\in A_i\otimes_{A}R$ and $b_i^{\prime}\in B_i^{\perp}\otimes_{A}R$, and $i$ is odd.
\end{enumerate}
\textit{ }

We interpret the above conditions in terms of matrices.
The matrix forms are taken with respect to the basis of $L$ fixed in Theorem \ref{210} and Remark \ref{r211}.(a).
A matrix form of the given hermitian form $h$ is described in Remark \ref{r33}.(1) below.
We use $\sigma$ to mean the automorphism of $B\otimes_AR$ given by $b\otimes r \mapsto \sigma(b)\otimes r$.
For a flat $A$-algebra $R$, $\underline{H}(R)$ is
the set of hermitian matrices $$\begin{pmatrix}\pi^{max\{i,j\}}f_{i,j}\end{pmatrix}$$ of size $n\times n$ satisfying the following:
 \begin{enumerate}
\item[(1)]  $f_{i,j}$ is an $(n_i\times n_j)$-matrix with entries in $B\otimes_AR$.
\item[(2)] If $i$ is even and $L_i$ is \textit{of type} $\textit{I}^o$,
then $\pi^if_{i,i}$ is of the form $$\xi^{i/2}\begin{pmatrix} a_i&\pi b_i\\ \sigma(\pi \cdot {}^tb_i) &1+2\gamma_i +4c_i \end{pmatrix}.$$
 Here, the diagonal entries of $a_i$ are divisible by $2$, where $a_i$ is an $(n_i-1) \times (n_i-1)$-matrix with entries
 in $B\otimes_AR$, etc. and $\gamma_i (\in A)$ is
 as chosen in Remark \ref{r33}.(1) below.
\item[(3)]  If $i$ is even and $L_i$ is \textit{of type} $\textit{I}^e$, then $\pi^i f_{i,i}$ is of the form
  $$\xi^{i/2}\begin{pmatrix} a_i&b_i&\pi e_i\\ \sigma({}^tb_i) &1+2f_i&1+\pi d_i \\ \sigma(\pi \cdot {}^te_i) &\sigma(1+\pi d_i) &2\gamma_i+4c_i \end{pmatrix}.$$
 Here,  the diagonal entries of $a_i$  are divisible by $2$, where $a_i$ is an $(n_i-2) \times (n_i-2)$-matrix with entries
 in $B\otimes_AR$, etc. and $\gamma_i (\in A)$ is
 as chosen in Remark \ref{r33}.(1) below.
\item[(4)]  Assume that $i$ is even and that $L_i$ is \textit{of type} $\textit{II}$.
The diagonal entries of $f_{i,i}$   are divisible by $2$.
\item[(5)] Assume that $i$ is odd.
The diagonal entries of $\pi f_{i,i}-\pi h_i$  are  are divisible by $4$.
\item[(6)] Assume that $i$ is odd and that $L_i$ is \textit{of type I}.
If $L_i$ is \textit{free of type I}, then
$\pi^i f_{i,i}$ is of the form
$$\xi^{(i-1)/2}\cdot \pi\begin{pmatrix} a_i&\pi b_i& e_i\\ -\sigma(\pi \cdot {}^tb_i) &\pi^3f_i&1+\pi d_i \\
-\sigma({}^te_i) &-\sigma(1+\pi d_i) &\pi+\pi^3c_i \end{pmatrix}.$$
  Here,  the diagonal entries of $a_i$ are  are divisible by $\pi^3$, where $a_i$ is an $(n_i-2) \times (n_i-2)$-matrix with entries
 in $B\otimes_AR$, etc.

 If $L_i$ is \textit{bound of type I}, then each entry of
 $$ \delta_{i-1}(0,\cdots, 0, 1)\cdot f_{i-1,i}+\delta_{i+1}(0,\cdots, 0, 1)\cdot f_{i+1,i} $$
lies in the ideal $(\pi)$.
\item[(7)] Since $\begin{pmatrix}\pi^{max\{i,j\}}f_{i,j}\end{pmatrix}$ is a hermitian matrix, its diagonal entries are fixed by the nontrivial Galois action over $E/F$ and hence belong to $R$. \\
\end{enumerate}

The functor $\underline{H}$ is represented by a flat $A$-scheme
  which is isomorphic to an affine space of dimension $n^2$.
  The proof of this is similar to that in Section 3C of \cite{C2} and so we skip it.

    Now suppose that $R$ is any (not necessarily flat) $A$-algebra.
We again use $\sigma$ to mean the automorphism of $B\otimes_AR$ given by $b\otimes r \mapsto \sigma(b)\otimes r$.
Note that  $\sigma(\pi)=-\pi$.
By choosing a $B$-basis of  $L$ as explained in Theorem \ref{210} and Remark \ref{r211}.(a),
we describe each element of $\underline{H}(R)$ formally as a  matrix
    $\begin{pmatrix}\pi^{max\{i,j\}}f_{i,j}\end{pmatrix}$ with  matrices $f_{i,i}^{\ast}$ satisfying the following:
    \begin{enumerate}
\item When $i\neq j$, $f_{i,j}$  is an  $(n_i \times n_j)$-matrix with entries in $B\otimes_AR$ and  $(-1)^{max\{i,j\}}\sigma({}^tf_{i,j})=f_{j,i}$.
\item   Assume that $i=j$ is even. Then
 \[
\pi^if_{i,i}=\left\{
  \begin{array}{l l}
  \xi^{i/2}\begin{pmatrix} a_i&\pi b_i\\ \sigma(\pi\cdot {}^t b_i) &1+2\gamma_i +4c_i \end{pmatrix}    & \quad  \textit{if  $L_i$ is \textit{of type $I^o$}};\\
  \xi^{i/2}\begin{pmatrix} a_i&b_i&\pi e_i\\ \sigma({}^tb_i) &1+2f_i&1+\pi d_i \\ \sigma(\pi \cdot {}^te_i) &\sigma(1+\pi d_i) &2\gamma_i+4c_i \end{pmatrix}    & \quad  \textit{if  $L_i$ is \textit{of type $I^e$}};\\
  \xi^{i/2}a_i  & \quad  \textit{if  $L_i$ is \textit{of type $II$}}.
      \end{array} \right.
\]
Here, $a_i$  is a formal  $(n_i-1 \times n_i-1)$-matrix (resp. $(n_i-2 \times n_i-2)$-matrix or $(n_i \times n_i)$-matrix)
when $L_i$ is \textit{of type $I^o$} (resp. \textit{of type $I^e$} or \textit{of type $II$}).
Non-diagonal entries of $a_i$ are  in $B\otimes_AR$ and
the $j$-th diagonal entry of $a_i$ is of the form $2x_i^j$ with $x_i^j \in R$. In addition, for non-digonal entries of $a_i$, we have the relation  $\sigma({}^ta_i)=a_i$.
And $b_i, d_i, e_i$ are matrices of suitable sizes with entries in $B\otimes_AR$ and $c_i, f_i$ are elements in $R$.

\item   Assume that $i=j$ is odd. Then
 \[
\pi^i f_{i,i}=\left\{
  \begin{array}{l l}
  \xi^{(i-1)/2}\cdot \pi\begin{pmatrix} a_i&\pi b_i& e_i\\ -\sigma(\pi \cdot {}^tb_i) &\pi^3f_i&1+\pi d_i \\
-\sigma({}^te_i) &-\sigma(1+\pi d_i) &\pi+\pi^3c_i \end{pmatrix}    & \quad  \textit{if  $L_i$ is \textit{free of type $I$}};\\
  \xi^{(i-1)/2}\cdot \pi a_i  & \quad  \textit{otherwise}.
      \end{array} \right.
\]
Here, $a_i$  is a formal  $(n_i-2 \times n_i-2)$-matrix (resp. $(n_i \times n_i)$-matrix) if $L_i$ is \textit{free of type I} (resp. otherwise).
Non-diagonal entries of $a_i$ are in $B\otimes_AR$
and the $j$-th diagonal entry of $a_i$ is of the form $\pi^3 x_i^j$ with $x_i^j \in R$.
In addition, for non-digonal entries of $a_i$, we have the relation  $\sigma({}^ta_i)=-a_i$.
And $b_i, d_i, e_i$ are matrices of suitable sizes with entries in $B\otimes_AR$ and $c_i, f_i$ are elements in $R$.

\item Assume that $i$ is odd and that $L_i$ is \textit{bound of type I}.
Then $$ \delta_{i-1}(0,\cdots, 0, 1)\cdot f_{i-1,i}+\delta_{i+1}(0,\cdots, 0, 1)\cdot f_{i+1,i} = \pi f_{i, i}^{\ast}$$
for a matrix $f_{i, i}^{\ast}$ of size $(1 \times n_i)$ with entries in $B\otimes_AR$.
\end{enumerate}
\textit{   }

To simplify notation, each element
$$((f_{i,j})_{i< j}, (a_i, x_i^j)_{\textit{$L_i$ of type II}}, (a_i, x_i^j, b_i, c_i)_{\textit{$L_i$ of type $I^o$}},
(a_i, x_i^j, b_i, c_i, d_i, e_i, f_i)_{\textit{$L_i$ of type $I^e$}},$$
$$(a_i, x_i^j, b_i, c_i, d_i, e_i, f_i)_{\textit{$L_i$ free of type $I$ with $i$ odd}}, 
(a_i, x_i^j, f_{i,i}^{\ast})_{\textit{$L_i$ bound of type $I$ with $i$ odd}})   $$
 of  $\underline{H}(R)$
is denoted by $(f_{i,j}, a_i \cdots f_i)$.

\begin{Rmk}\label{r33}
\begin{enumerate}\item
Recall that $\delta$ is a unit element in $A$ such that $\delta\equiv 1 \mathrm{~mod~}2$ and $\pi=\sqrt{2\delta}$.
Note that the given hermitian form $h$ is an element of $\underline{H}(A)$.
We represent the given hermitian form $h$ by a hermitian matrix $\begin{pmatrix} \pi^{i}\cdot h_i\end{pmatrix}$
whose $(i,i)$-block is $\pi^i\cdot h_i$ for all $i$, and all of whose remaining blocks are $0$.
Then:
\begin{enumerate}
\item If $i$ is even and $L_i$ is \textit{of type} $\textit{I}^o$,
then $\pi^i\cdot h_i$ has the following form (with $\gamma_i\in A$):
$$\xi^{i/2}\begin{pmatrix} \begin{pmatrix} 0&1\\1&0\end{pmatrix}& & & \\ &\ddots & & \\ & &\begin{pmatrix} 0&1\\1&0\end{pmatrix}& \\ & & & 1+2\gamma_i \end{pmatrix}.$$

\item  If $i$ is even and $L_i$ is \textit{of type} $\textit{I}^e$,
then $\pi^i\cdot h_i$ has the following form (with $\gamma_i\in A$):
$$\xi^{i/2}\begin{pmatrix} \begin{pmatrix} 0&1\\1&0\end{pmatrix}& & & \\ &\ddots & & \\ & &\begin{pmatrix} 0&1\\1&0\end{pmatrix}& \\ & & & \begin{pmatrix} 1&1\\1&2\gamma_i\end{pmatrix} \end{pmatrix}.$$

\item If $i$ is even and $L_i$ is \textit{of type} $\textit{II}$,
then $\pi^i\cdot h_i$ has the following form (with $\gamma_i\in A$):
$$\xi^{i/2}\begin{pmatrix} \begin{pmatrix} 0&1\\1&0\end{pmatrix}& & & \\ &\ddots & & \\ & &\begin{pmatrix} 0&1\\1&0\end{pmatrix}& \\ & & & \begin{pmatrix} 2\delta&1\\1&2\gamma_i\end{pmatrix} \end{pmatrix}.$$

\item If $i$ is odd and $L_i$ is \textit{free of type I},
then $\pi^i\cdot h_i$ has the following form (with $\gamma_i\in A$):
$$\xi^{(i-1)/2}\begin{pmatrix} \begin{pmatrix} 0&\pi\\ \sigma(\pi)&0\end{pmatrix}& & & \\ &\ddots & & \\ & &\begin{pmatrix} 0&\pi\\ \sigma(\pi)&0\end{pmatrix}& \\ & & & \begin{pmatrix} 4\gamma_i&\pi\\ \sigma(\pi)&2\delta\end{pmatrix} \end{pmatrix}.$$

\item If $i$ is odd and $L_i$ is \textit{bound of type I or of type II},
then $\pi^i\cdot h_i$ has the following form:
$$\xi^{(i-1)/2}\begin{pmatrix} \begin{pmatrix} 0&\pi\\ \sigma(\pi)&0\end{pmatrix}& & & \\ &\ddots & & \\ & &\begin{pmatrix} 0&\pi\\ \sigma(\pi)&0\end{pmatrix}& \\ & & & \begin{pmatrix} 0&\pi\\ \sigma(\pi)&0 \end{pmatrix} \end{pmatrix}.$$
\end{enumerate}
\textit{ }

\item Let $R$ be a $\kappa$-algebra.
  We also denote by $h$ the element of $\underline{H}(R)$ which is the image of $h\in \underline{H}(A)$
   under the natural map from $\underline{H}(A)$ to $\underline{H}(R)$.
   Recall that we denote each element of $\underline{H}(R)$ by  $(f_{i, j}, a_i\cdots f_i)$.
 Then the tuple $(f_{i, j}, a_i\cdots f_i)$ denoting $h\in \underline{H}(R)$ is defined by the conditions:
  \begin{enumerate}
\item  If $i\neq j$, then  $f_{i,j}=0$.
\item If $i$ is even and $L_i$ is \textit{of type I}, then $$a_i=\begin{pmatrix} \begin{pmatrix} 0&1\\1&0\end{pmatrix}& &  \\ &\ddots &  \\ & & \begin{pmatrix} 0&1\\1&0\end{pmatrix}\end{pmatrix}, \textit{thus $x_i^j=0$},$$
    $$b_i=0, d_i=0, e_i=0, f_i=0, c_i=0.$$
    If $i$ is even and $L_i$ is \textit{of type II}, then $$a_i=\begin{pmatrix} \begin{pmatrix} 0&1\\1&0\end{pmatrix}& & & \\ &\ddots & & \\ & &\begin{pmatrix} 0&1\\1&0\end{pmatrix}& \\ & & & \begin{pmatrix} 2\cdot 1&1\\1&2\cdot\bar{\gamma}_i\end{pmatrix} \end{pmatrix}, \textit{thus $x_i^j=0$ with $j\leq n_i-2$, $x_i^{n_i-1}=1$, $x_i^{n_i}=\bar{\gamma}_i$}.$$
\item If $i$ is odd, then
$$a_i=\begin{pmatrix} \begin{pmatrix} 0&1\\-1&0\end{pmatrix}& &  \\ &\ddots &  \\ & & \begin{pmatrix} 0&1\\-1&0\end{pmatrix}\end{pmatrix}, \textit{thus $x_i^j=0$},$$
    $$b_i=0, d_i=0, e_i=0,  c_i=0, f_{i,i}^{\ast}=0, f_i=\bar{\gamma}_i.$$
Here, $\bar{\gamma}_i\in \kappa$ is the reduction of $\gamma_i$ mod $2$.

  \end{enumerate}
\end{enumerate}
\end{Rmk}

 \subsection{The smooth affine group scheme \textit{\underline{G}}}\label{sags}

\begin{Thm}\label{t34}
    For any flat $A$-algebra $R$, the group $\underline{M}^{\ast}(R)$ acts on $\underline{H}(R)$ on the right by
    $f\circ m = \sigma({}^tm)\cdot f\cdot m$. This action is represented by an action morphism
     \[\underline{H} \times \underline{M}^{\ast} \longrightarrow \underline{H} .\]
\end{Thm}

\begin{proof}
 We start with any $m\in \underline{M}^{\ast}(R)$ and $f\in \underline{H}(R)$.
 In order to show that  $\underline{M}^{\ast}(R)$ acts on the right of $\underline{H}(R)$ by $f\circ m = \sigma({}^tm)\cdot f\cdot m$,
it suffices to show that $f\circ m$ satisfies  conditions (a) to (e) given in the beginning of  Section \ref{h}.
The proof that $f\circ m$ satisfies  conditions (a) to  (c) is  similar to the proof of Theorem 3.4 in \cite{C2} and so we skip it.\\

For condition (d), 
it suffices to show that
$$\xi^{-m} f(mw_i, mw_i)-\xi^{-m}h(w_i,w_i) \in (4),$$ where $w_i\in W_i$.
Since $m$ induces the identity on $W_i/(X_i\cap Z_i)$, we can write $mw_i=w_i+x_i$ where $x_i\in X_i\cap Z_i$.
Thus it suffices to show that $$\xi^{-m} \left( f(w_i, x_i)+f(x_i, w_i)+f(x_i, x_i)   \right) \in (4).$$
Since $\xi^{-m}  f(w_i, x_i)\equiv \xi^{-m}  h(w_i, x_i)$ mod $\pi$ by condition (c) and
$\xi^{-m}  h(w_i, x_i) \in (\pi)$ by the definition of $X_i$, we can see that
$$\xi^{-m}  f(w_i, x_i) \in (\pi) \textit{ and so }
\xi^{-m} \left( f(w_i, x_i)+f(x_i, w_i)   \right) \in (4).$$
Furthermore, since $x_i\in Z_i$ and clearly $x_i\in W_i$, we can see that
$$\xi^{-m} f(x_i, x_i) - \xi^{-m}  h(x_i, x_i) \in (4) \textit{ and }
\xi^{-m}  h(x_i, x_i) \in (4).$$
This completes the proof of condition (d).\\

For  condition (e), it suffices to show that $$f(ma_i, mb_i^{\prime})-h(a_i, b_i^{\prime})\in B,
  \textit{ where $a_i\in A_i$ and $b_i^{\prime}\in B_i^{\perp}$.}$$
We write $ma_i=a_i+b_i$ and $mb_i^{\prime}=b_i^{\prime}+a_i^{\prime}$, where $b_i\in B_i$ and $a_i^{\prime} \in A_i^{\perp}$.
Hence it suffices to show $$f(a_i+b_i, a_i^{\prime})+f(b_i, b_i^{\prime})\in B.$$
Since $a_i+b_i, b_i\in A_i$ and $a_i', b_i'\in B_i^{\perp}$,
we can see that
$$\left(f(a_i+b_i, a_i^{\prime})+f(b_i, b_i^{\prime})\right)-\left(h(a_i+b_i, a_i^{\prime})+h(b_i, b_i^{\prime})\right)\in B
\textit{ and } h(a_i+b_i, a_i^{\prime})+h(b_i, b_i^{\prime})\in B.$$
This completes the proof of condition (e).

The proof of the representability of this action is similar to that of the proof of Theorem 3.4 in \cite{C2} and we may skip it.
\end{proof}

\begin{Rmk}\label{r35}
Let $R$ be a $\kappa$-algebra.
We explain the above action morphism in terms of $R$-points.
Choose an element  $(m_{i,j}, s_i\cdots w_i)$ in $ \underline{M}^{\ast}(R) $ as explained in Section \ref{m}
and  express this element formally as a matrix $m=\begin{pmatrix}\pi^{max\{0,j-i\}}m_{i,j}\end{pmatrix}$
with $z_i^{\ast}, m_{i, i}^{\ast}, m_{i, i}^{\ast\ast}$.
We also choose an element $(f_{i,j}, a_i \cdots f_i)$ of $\underline{H}(R)$ and express this element formally as a matrix
 $f=\begin{pmatrix}\pi^{max\{i,j\}}f_{i,j}\end{pmatrix}$ with $f_{i, i}^{\ast}$ as explained in Section \ref{h}.

We then compute the formal matrix product $\sigma({}^tm)\cdot f\cdot m$  and denote it by the formal matrix $\begin{pmatrix}\pi^{max\{i,j\}}\tilde{f}_{i,j}'\end{pmatrix}$ with  $(\tilde{f}_{i,j}', \tilde{a}_i' \cdots \tilde{f}_i')$.
 Here, the description of the formal matrix $\begin{pmatrix}\pi^{max\{i,j\}}\tilde{f}_{i,j}'\end{pmatrix}$ with
 $(\tilde{f}_{i,j}', \tilde{a}_i' \cdots \tilde{f}_i')$ is as explained in Section \ref{h}.
We emphasize that in the above formal computation,
we distinguish $-1$ from $1$ formally.
If $L_i$ is \textit{bound of type I} with $i$ odd, then
let
$$\pi (\tilde{f}_{i, i}^{\ast})'=\delta_{i-1}(0,\cdots, 0, 1)\cdot \tilde{f}_{i-1,i}'+\delta_{i+1}(0,\cdots, 0, 1)\cdot \tilde{f}_{i+1,i}'$$
 formally.
 We formally compute the right hand side and then it is of the form $\pi X$.
Here, $X$ involves $m_{i,i}^{\ast}, m_{i,i}^{\ast\ast}$.
Then $(\tilde{f}_{i, i}^{\ast})'$ is defined as $X$.


 We now let $\pi^2$ be zero in each entry of the formal matrices
$$(\tilde{f_{i,j}'})_{i< j},  (\tilde{b_i'})_{\textit{$L_i$ of type $I^o$}},
(\tilde{b_i'}, \tilde{d_i'}, \tilde{e_i'})_{\textit{$L_i$ of type $I^e$ or free of type I with i odd}},
(\tilde{(f_{i,i}^{\ast})'})_{\textit{$L_i$ bound of type $I$ with $i$ odd}}  $$
and in each nondiagonal entry of the formal matrix  $(\tilde{a}_i')$.
Then these entries are elements in $B\otimes_AR$.

We also let $\pi^2$ be zero in
$(\tilde{x}_i^j)', (\tilde{c}_i')_{\textit{$L_i$ of type $I^o$}},$
 $(\tilde{f}_i', \tilde{c}_i')_{\textit{$L_i$ of type $I^e$ or free of type $I$ with $i$ odd}}$.
 Note that $(\tilde{x}_i^j)'$ is a  diagonal entry of a formal matrix $\tilde{a}_i'$.
 Then these entries are elements in $R$.

 Let $(f_{i,j}', a_i' \cdots f_i')$ be the reduction of $(\tilde{f}_{i,j}', \tilde{a}_i' \cdots \tilde{f}_i')$ as explained above,
 i.e. by letting $\pi^2$ be zero in the entries of formal matrices as described above.
 Then $(f_{i,j}', a_i' \cdots f_i')$ is an element of $\underline{H}(R)$
and the composition $(f_{i,j}, a_i \cdots f_i)\circ (m_{i,j}, s_i\cdots w_i)$ is $(f_{i,j}', a_i' \cdots f_i')$.

We can also write $(f_{i,j}', a_i' \cdots f_i')$ explicitly in terms of $(f_{i,j}, a_i \cdots f_i)$ and $(m_{i,j}, s_i\cdots w_i)$
like  the product of $(m_{i,j}, s_i\cdots w_i)$ and $(m_{i,j}', s_i'\cdots w_i')$ explained in Section \ref{m}.
However, this is complicated and we do not use it in this generality.
On the other hand, we  explicitly calculate  $(f_{i,j}, a_i \cdots f_i)\circ (m_{i,j}, s_i\cdots w_i)$ when
$(f_{i,j}, a_i \cdots f_i)$ is the given hermitian form $h$ and $(m_{i,j}, s_i\cdots w_i)$ satisfies certain  conditions on each block.
This explicit calculation  will be done in Appendix \ref{App:AppendixA}.
\end{Rmk}

 \begin{Thm}\label{t36}
  Let $\rho$ be the morphism $\underline{M}^{\ast} \rightarrow \underline{H}$ defined by $\rho(m)=h \circ m$,
  which is induced by the  action morphism  of Theorem \ref{t34}.
  Then $\rho$ is smooth of relative dimension dim $\mathrm{U}(V, h)$.
  \end{Thm}

  \begin{proof}
  The theorem follows from Theorem 5.5 in \cite{GY} and the following lemma.
  \end{proof}

      \begin{Lem}\label{l37}
      The morphism $\rho \otimes \kappa : \underline{M}^{\ast}\otimes \kappa \rightarrow \underline{H}\otimes \kappa$
      is smooth of relative dimension $\mathrm{dim~} \mathrm{U}(V, h)$.
      \end{Lem}

\begin{proof}
 The proof is based on Lemma 5.5.2 in \cite{GY} and is parallel to Lemma 3.7 of \cite{C2}.
    It is enough to check the statement over the algebraic closure $\bar{\kappa}$ of $\kappa$.
    By \cite{H}, III.10.4, it suffices to show that, for any $m \in \underline{M}^{\ast}(\bar{\kappa})$,
    the induced map on the Zariski tangent space $\rho_{\ast, m}:T_m \rightarrow T_{\rho(m)}$ is surjective.

   We define the two functors from the category of commutative flat $A$-algebras to the category of abelian groups as follows:
 \[T_1(R)=\{m-1 : m\in\underline{M}(R)\},\]
     \[T_2(R)=\{f-h : f\in\underline{H}(R)\}.\]

   The functor $T_1$ (resp. $T_2$) is representable by a flat $A$-algebra which is a polynomial ring over $A$ of $2n^2$ (resp. $n^2$) variables
   by Lemma 3.1 of \cite{C1}.
    Moreover, each of them is represented by a commutative group scheme since  they are closed under addition.
        In fact, $T_1$ is the same as the functor $\underline{M}^{\prime}$ in Remark \ref{r31}.

   We still need to introduce another functor on flat $A$-algebras.
   Define $T_3(R)$ to be the set of all $(n \times n)$-matrices $y$ over $B\otimes_AR$ satisfying the following conditions:
   \begin{enumerate}
  \item[(a)] The $(i,j)$-block $y_{i,j}$ of $y$ has entries in $\pi^{max(i,j)}B\otimes_AR$ so that
   $$y=\begin{pmatrix} \pi^{max(i,j)}y_{i,j}\end{pmatrix}.$$
      Here, the size of $y_{i,j}$ is $n_i\times n_j$.
\item[(b)] Assume that $i$ is even.
\begin{itemize}
\item[(i)] If $L_i$ is \textit{of type} $\textit{I}^o$, then $y_{i,i}$ is of the form
   \[\begin{pmatrix} s_i&\pi y_i\\ \pi v_i&\pi z_i \end{pmatrix}\in \mathrm{M}_{n_i}(B\otimes_AR)\]
     where $s_i$ is an $(n_i-1) \times (n_i-1)$ matrix, etc.
\item[(ii)]  If  $L_i$ is \textit{of type} $\textit{I}^e$, then $y_{i,i}$ is of the form
   \[\begin{pmatrix} s_i&r_i&\pi t_i\\ y_i&x_i&\pi w_i\\ \pi v_i&\pi u_i&\pi z_i \end{pmatrix}\in \mathrm{M}_{n_i}(B\otimes_AR)\]
     where $s_i$ is an $(n_i-2) \times (n_i-2)$-matrix, etc.
\end{itemize}
\item[(c)] Assume that  $i$ is even and that $L_i$ is \textit{of type I}. Then
     $$z_i+\delta_{i-2}k_{i-2, i}+\delta_{i+2}k_{i+2, i} \in (\pi).$$
    Here,
    \begin{itemize}
    \item[(i)]  $z_i$ is in the $(n_i\times n_i)^{th}$-entry of $y_{i,i}$ as described in the above  Step (b).
\item[(ii)] $k_{i-2, i}$ (resp. $k_{i+2, i}$) is the $(n_{i-2}\times n_i)^{th}$-entry
(resp. $(n_{i+2}\times n_i)^{th}$-entry) of the matrix $y_{i-2, i}$ (resp. $y_{i+2, i}$)
if $L_{i-2}$ (resp. $L_{i+2}$) is \textit{of type} $\textit{I}$.
\end{itemize}
\item[(d)] Assume that $i$ is odd.
Consider the following $(1\times n_i)$-matrix:
$$
\left\{
  \begin{array}{l l}
v_i\cdot y_{i, i}   &   \quad  \textit{if $L_i$ is \textit{free of type I}};\\
  \delta_{i-1}v_{i-1}\cdot y_{i-1, i}+\delta_{i+1}v_{i+1}\cdot y_{i+1, i}   & \quad  \textit{if $L_i$ is \textit{bound of type I}}.
    \end{array} \right.
$$
Here,
    \begin{itemize}
 \item[(i)]  $v_{i}=(0,\cdots, 0, 1, 0)$ of size $1\times n_{i}$.
  \item[(ii)] $v_{i-1}$ (resp. $v_{i+1}$)$=(0,\cdots, 0, 1)$  of size $1\times n_{i-1}$ (resp. $1\times n_{i+1}$).\\
 \end{itemize}

Then each entry of the above matrix  lies  in the ideal $(\pi)$.\\
\item[(e)]
Assume that $i$ is odd.
Consider the following $(1\times n_i)$-matrix:
$$
\left\{
  \begin{array}{l l}
v_i\cdot {}^ty_{i, i}   &   \quad  \textit{if $L_i$ is \textit{free of type I}};\\
  \delta_{i-1}v_{i-1}\cdot {}^ty_{i, i-1}+\delta_{i+1}v_{i+1}\cdot {}^ty_{i, i+1}   & \quad  \textit{if $L_i$ is \textit{bound of type I}}.
    \end{array} \right.
$$
Here, $v_{i}, v_{i-1}, v_{i+1}$ are as described in the above Step (d).
Then each entry of the above matrix lies in the  ideal $(\pi)$.\\
     \end{enumerate}

     The functor $T_3$ is represented by a flat $A$-scheme which is isomorphic to an affine space by Lemma 3.1 of \cite{C1}.
     Moreover it is represented by  a commutative group scheme since it is closed under addition.
So far, we have defined three functors $T_1, T_2, T_3$ and these are represented by schemes. Therefore, we can talk about their $\bar{\kappa}$-points.

We identify $T_m$ with $T_1(\bar{\kappa})$ and $T_{\rho(m)}$ with $T_2(\bar{\kappa})$. 
   The map $\rho_{\ast, m}:T_m \rightarrow T_{\rho(m)}$ is then $X \mapsto \sigma(m^t)\cdot h\cdot X + \sigma(X^t)\cdot h\cdot m$.
For an explanation of these identifications and the map, we refer to the argument explained from the second half of page 475 to
the top of page 477 in \cite{C2}.
An explanation of the explicit computation of the map is also explained in \cite{C2} and we reproduce it here.

We explain how to compute  $X\mapsto  \sigma(m)^t\cdot h \cdot X+\sigma(X)^t\cdot h \cdot m$ explicitly.
Recall that for a $\kappa$-algebra $R$, we denote an element $m$  of $\underline{M}(R)$ by $(m_{i,j}, s_i\cdots w_i)$ with a formal matrix interpretation
$m=(\pi^{max\{0, j-i\}}m_{i,j}) \mathrm{~together~with~}z_i^{\ast}, m_{i,i}^{\ast}, m_{i,i}^{\ast\ast}$ (cf.  Section \ref{m})
and  we denote an element $f$ of $\underline{H}(R)$  by $(f_{i,j}, a_i\cdots f_i)$   with a formal matrix interpretation
$f=(\pi^{max\{i,j\}}f_{i,j}) \mathrm{~together~with~}f_{i,i}^{\ast}$ (cf. Section \ref{h}).
Similarly, we can also denote an element $X$ of $T_1(\bar{\kappa})$ by $(m_{i,j}', s_i'\cdots w_i')$ with a formal matrix interpretation $X=(\pi^{max\{0, j-i\}}m_{i,j}') \mathrm{~together~with~}(z_i')^{\ast}, (m_{i,i}')^{\ast}, (m_{i,i}')^{\ast\ast}$ and an element $Z$ of $T_2(\bar{\kappa})$ by $(f_{i,j}', a_i'\cdots f_i')$ with a formal matrix interpretation  $Z=(\pi^{max\{i,j\}}f_{i,j}')\mathrm{~together~with~}(f_{i,i}')^{\ast}$.
Then we formally compute $X \mapsto \sigma(m^t)\cdot h\cdot X + \sigma(X^t)\cdot h\cdot m$
and consider the reduction of the formal matrix $ \sigma(m^t)\cdot h\cdot X + \sigma(X^t)\cdot h\cdot m$
 in a manner similar to that of the reduction   explained in Remark \ref{r35}.
We denote  this reduction by  $(f_{i,j}'', a_i''\cdots f_i'')$ with a formal matrix interpretation $(\pi^{max\{i,j\}}f_{i,j}'')\mathrm{~together~with~}(f_{i,i}'')^{\ast}$.
This $(f_{i,j}'', a_i''\cdots f_i'')$ may and shall be identifed with an element of $T_2(\bar{\kappa})$ in the manner just described.
Then $\rho_{\ast, m}(X)$ is the element $Z=(f_{i,j}'', a_i''\cdots f_i'')$ of $T_2(\bar{\kappa})$.\\

   To prove the surjectivity of $\rho_{\ast, m}:T_1(\bar{\kappa}) \rightarrow T_2(\bar{\kappa})$, it suffices to show the following three statements:
       \begin{itemize}
   \item[(1)] $X \mapsto h\cdot X $ defines a bijection $T_1(\bar{\kappa}) \rightarrow T_3(\bar{\kappa})$;
   \item[(2)] for any $m \in \underline{M}^{\ast}(\bar{\kappa})$, $Y \mapsto \sigma({}^t m) \cdot Y$ defines a bijection from $T_3(\bar{\kappa})$ to itself;
   \item[(3)] $Y \mapsto \sigma({}^t Y) + Y$ defines a surjection $T_3(\bar{\kappa}) \rightarrow T_2(\bar{\kappa})$.
\end{itemize}
Here, all the above maps are  interpreted as in Remark \ref{r35} (if they are well-defined).
Then $\rho_{\ast, m}$ is the composite of these three.
 Condition  (3) is direct from the construction of $T_3(\bar{\kappa})$. Hence we provide the proof of (1) and (2).\\

For (1), by using the argument explained in the last two paragraphs of page 477 in \cite{C2},
it suffices to show that
the two functors $T_1(R)\longrightarrow T_3(R), X\mapsto h\cdot X (\in \mathrm{M}_{n\times n}(B\otimes_AR))$ and
  $T_3(R) \longrightarrow T_1(R), Y \mapsto h^{-1}\cdot Y (\in \mathrm{M}_{n\times n}(B\otimes_AR))$
   are  well-defined for all flat $A$-algebras $R$.
  In other words, we only need to show that $h\cdot X \in T_3(R)$ and $h^{-1}\cdot Y\in T_1(R)$.
We represent $h$ by a hermitian block matrix $\begin{pmatrix} \pi^{i}\cdot h_i\end{pmatrix}$
   with a matrix $(\pi^{i}\cdot h_i)$ for the $(i,i)$-block and $0$ for the remaining blocks as in Remark \ref{r33}.(1).

   For the first functor, it suffices to show that  $h\cdot X$ satisfies the five conditions defining the functor $T_3$.
    Here, $X\in T_1(R)$ for a flat $A$-algebra $R$.
  We express $$X=\begin{pmatrix} \pi^{max\{0,j-i\}}x_{i,j} \end{pmatrix}.$$
  Then
   $$h\cdot X = \begin{pmatrix} \pi^{max(i,j)}y_{i,j}\end{pmatrix}.$$
  Here, $y_{i,i}=h_i\cdot x_{i,i}$.
The proof that $h\cdot X$ satisfies conditions (a) and (b) is similar to that of Lemma 3.7 of \cite{C2} and so we skip it.

For condition (c), let $L_i$ be \textit{of type I} with \textit{i} even.
Recall that we denote $X$ by $(m_{i,j}', s_i'\cdots w_i')$.
Then  the $(n_i\times n_i)^{th}$-entry of $y_{i,i}$ is $\pi (1+2\gamma_i) z_i'$ or $\pi (z_i'+2\gamma_i w_i')$ if $L_i$ is \textit{of type $I^o$} or \textit{of type $I^e$},
respectively.
The  $(n_{i-2}\times n_i)^{th}$-entry (resp. $(n_{i+2}\times n_i)^{th}$-entry) of
the matrix $y_{i-2, i}$ (resp. $y_{i+2, i}$) is $k_{i-2, i}'+2 k_{i-2}'$ (resp. $k_{i+2, i}'+2 k_{i+2}'$), for some $k_{i-2}', k_{i+2}' \in B\otimes_AR$
if $L_{i-2}$ (resp. $L_{i+2}$) is \textit{of type} $\textit{I}$.
Therefore, $$(1+2\gamma_i) z_i' ~(\textit{or } z_i'+2\gamma_i w_i') + \delta_{i-2}(k_{i-2, i}'+2 k_{i-2}')+\delta_{i+2}(k_{i+2, i}'+2 k_{i+2}') \in (\pi)$$
since  $z_i'+\delta_{i-2}k_{i-2, i}'+\delta_{i+2}k_{i+2, i}'\in (\pi)$.

For conditions (d) and (e),
assume that $L_i$ is \textit{free of type I} with \textit{i} odd.
Then by observing $h_i\cdot x_{i,i} ~(=y_{i,i})$, we can easily see that each entry of
$$(0,\cdots, 0, 1, 0)\cdot y_{i,i} \textit{ and } (0,\cdots, 0, 1, 0)\cdot {}^ty_{i,i}$$
is contained in the ideal $(\pi)$.\\
We now assume that $L_i$ is \textit{bound of type I} with \textit{i} odd.
    If $L_{i-1}$ is \textit{of type $I^o$} (resp. $I^e$), then
    $(0,\cdots, 0, 1)\cdot y_{i-1, i}= (0,\cdots, 0, 1)\cdot x_{i-1, i}$ mod 2
    (resp. $(0,\cdots, 0, 1)\cdot y_{i-1, i} = (0,\cdots, 0, 1, 0)\cdot x_{i-1, i}$ mod 2).
Thus, $$\delta_{i-1}(0,\cdots, 0, 1)\cdot y_{i-1, i}+\delta_{i+1}(0,\cdots, 0, 1)\cdot y_{i+1, i} =
\delta_{i-1}v_{i-1}\cdot x_{i-1, i}+\delta_{i+1}v_{i+1}\cdot x_{i+1, i} = 0 \textit{ mod }\pi.$$
Here, $v_{i-1}$ or $v_{i+1}$ is $(0,\cdots, 0, 1)$ or $(0,\cdots, 0, 1, 0)$, as explained in Step (d)
of the description of an element of $\underline{M}(R)$ for a flat $A$-algebra $R$ in Section \ref{m}.
The proof of condition (e) is similar to the above and so we skip it.

   For the second functor, we express $Y=\begin{pmatrix} \pi^{max(i,j)}y_{i,j}\end{pmatrix}$ and
   $h^{-1}=\begin{pmatrix} \pi^{-i}\cdot h_i^{-1}\end{pmatrix}$.
   Then we have the following:
   $$h^{-1}\cdot Y = \begin{pmatrix} \pi^{max\{0,j-i\}}x_{i,j} \end{pmatrix}.$$
  Here, $x_{i,i}=h_i^{-1}\cdot y_{i,i}$.

Then it suffices to show that $h^{-1}\cdot Y = \begin{pmatrix} \pi^{max\{0,j-i\}}x_{i,j} \end{pmatrix}$ satisfies the conditions
defining $T_1(R)$ for a flat $A$-algebra $R$.
Indeed, we do not describe the conditions defining $T_1(R)$ explicitly in this paper.
However, these conditions can be read off from the conditions defining $\underline{M}(R)$ (cf. Section \ref{mc})
because of the definition of the functor $T_1$.
The proof that $h^{-1}\cdot Y$ satisfies the first two conditions
is similar to that of Lemma 3.7 of \cite{C2} and the rest is similar to the above case.
Thus we skip them.

For (2), by using the argument explained from the last paragraph of page 479 to the first paragraph of page 480  in \cite{C2},
it suffices to show that the functor
$$\underline{M}^{\ast}(R)\times T_3(R)\longrightarrow T_3(R), ~~~~~ (m, Y)\mapsto \sigma({}^tm)\cdot Y, $$
 for a flat $A$-algebra $R$, is well-defined.
 In other words, we only need to show that $\sigma({}^tm)\cdot Y\in T_3(R)$.

For a flat $A$-algebra, we choose an element $m\in \underline{M}^{\ast}(R)$ and $Y\in T_3(R)$
and    we again express $m=\begin{pmatrix} \pi^{max\{0,j-i\}}m_{i,j} \end{pmatrix}$ and $Y=\begin{pmatrix} \pi^{max(i,j)}y_{i,j}\end{pmatrix}$.
The proof that  $\sigma({}^t m) \cdot Y $  satisfies the conditions (a) and (b) in the definition of $T_3(R)$
is similar to that of Lemma 3.7 of \cite{C2} and the rest is similar to the above case (1).
Thus we skip them.
\end{proof}

 Let $\underline{G}$ be the stabilizer of $h$ in $\underline{M}^{\ast}$. It is an affine group subscheme of $\underline{M}^{\ast}$, defined over $A$.
 Thus we have the following theorem.
 \begin{Thm}\label{t38}
 The group scheme $\underline{G}$ is smooth, and $\underline{G}(R)=\mathrm{Aut}_{B\otimes_AR}(L\otimes_A R,h\otimes_A R)$ for any \'{e}tale $A$-algebra $R$.
 \end{Thm}
\begin{proof}
The proof is similar to that of
Theorem 3.8 in \cite{C2} and so we skip it.
\end{proof}

As in \textit{Case 1} mentioned in the paragraph following Theorem 3.8 of \cite{C2},
 in the theorem, the equality holds only for  an \'{e}tale $A$-algebra $R$
since we obtain conditions defining $\underline{M}$ by considering properties of elements of
$\mathrm{Aut}_{B\otimes_AR}(L\otimes_A R,h\otimes_A R)$ for an \'{e}tale $A$-algebra $R$ (cf. Section \ref{mc}).
For example, let $(L, h)$ be the hermitian lattice of rank 1 as given in Appendix \ref{App:AppendixB}.
For simplicity, let $\pi^2=2$.
As a set, $\mathrm{Aut}_{B\otimes_AR}(L\otimes_A R,h\otimes_A R)$ is the same as
$\{(a, b):a,b\in R \textit{ and } a^2-2b^2=1\}$ for a flat $A$-algebra $R$.
Thus we cannot guarantee that $a-1$ is contained in the ideal $(2)$, which should be necessary
in order that $(a, b)$ is an element of $\underline{G}(R)$.

\section{The special fiber of the smooth integral model}\label{sf}

In this section, we will determine the structure of the special fiber $\tilde{G}$ of $\underline{G}$
by determining the maximal reductive quotient and the component group when $E/F$ satisfies \textit{Case 2},
 by adapting the approach of Section 4 of \cite{C1} and Section 4 of \cite{C2}.
From this section to the end, the identity matrix is denoted by id.

    \subsection{The reductive quotient of the special fiber}\label{red}

Assume that $i=2m$ is even.
Recall that $Z_i$ is the sublattice of $B_i$ such that $Z_i/\pi B_i$ is the radical of the quadratic form $\frac{1}{2^{m+1}}q$ mod 2 on $B_i/\pi B_i$,
where $\frac{1}{2^{m+1}}q(x)=\frac{1}{2^{m+1}}h(x,x)$.

\begin{Thm}\label{t43}
Assume that $i=2m$ is even.
Let $\bar{q}_i$ denote the nonsingular quadratic form $\frac{1}{2^{m+1}}q$ mod 2 on $B_i/Z_i$.
Then there exists a unique morphism of algebraic groups
$$\varphi_i:\tilde{G}\longrightarrow \mathrm{O}(B_i/Z_i, \bar{q}_i)_{\mathrm{red}}$$
defined over $\kappa$, where $\mathrm{O}(B_i/Z_i, \bar{q}_i)_{\mathrm{red}}$ is the reduced subgroup scheme of $\mathrm{O}(B_i/Z_i, \bar{q}_i)$,
such that for all \'{e}tale local $A$-algebras $R$ with residue field $\kappa_R$ and
every $\tilde{m} \in \underline{G}(R)$ with reduction $m\in \tilde{G}(\kappa_R)$,
$\varphi_i(m)\in \mathrm{GL}(B_i\otimes_AR/Z_i\otimes_AR)$ is induced by the action of $\tilde{m}$ on $L\otimes_AR$
(which preserves $B_i\otimes_AR$ and $Z_i\otimes_AR$ by the construction of $\underline{M}$).
\end{Thm}
\begin{proof}
The proof of this theorem is similar to that of Theorem 4.3 of \cite{C2}. 
Thus we only provide the image of an element $m$ of $\tilde{G}(\kappa_R)$ in
$\mathrm{O}(A_i/Z_i, \bar{q}_i)_{\mathrm{red}}(\kappa_R)$, where $R$ is an \'{e}tale local $A$-algebra with $\kappa_R$ as its residue field.

Recall that $\underline{G}$ is a closed subgroup scheme of $\underline{M}^{\ast}$ and $\tilde{G}$ is a closed subgroup scheme of $\tilde{M}$,
 where $\tilde{M}$ is the special fiber of $\underline{M}^{\ast}$.
 Thus we may consider an element of  $\tilde{G}(\kappa_R)$ as an element of $\tilde{M}(\kappa_R)$.
 Based on Section \ref{m}, an element $m$ of $\tilde{G}(\kappa_R)$ may be written as, say,
  $(m_{i,j}, s_i \cdots w_i)$ and it has the following formal matrix description:
  $$m= \begin{pmatrix} \pi^{max\{0,j-i\}}m_{i,j} \end{pmatrix} \textit{ together with } z_i^{\ast}, m_{i,i}^{\ast}, m_{i,i}^{\ast\ast}.$$
Here, if $i$ is even and $L_i$ is \textit{of type} $\textit{I}^o$ or \textit{of type} $\textit{I}^e$, then
 $$m_{i,i}=\begin{pmatrix} s_i&\pi y_i\\ \pi v_i&1+\pi z_i \end{pmatrix} \textit{or}
\begin{pmatrix} s_i&r_i&\pi t_i\\ \pi y_i&1+\pi x_i&\pi z_i\\ v_i&u_i&1+\pi w_i \end{pmatrix},$$
 respectively,
where
 $s_i\in M_{(n_i-1)\times (n_i-1)}(B\otimes_A\kappa_R)$ (resp.  $s_i\in M_{(n_i-2)\times (n_i-2)}(B\otimes_A\kappa_R)$), etc.

We can write  $m_{i, i}=(m_{i, i})_1+\pi\cdot (m_{i, i})_2$ when $L_i$ is \textit{of type II} and
for each block of $m_{i,i}$ when $L_i$ is \textit{of type I},  $s_i=(s_i)_1+\pi\cdot (s_i)_2$ and so on.
We can also write $m_{i, j}=(m_{i, j})_1+\pi\cdot (m_{i, j})_2$ when $i\neq j$.
Here,
$(m_{i, i})_1, (m_{i, i})_2\in M_{n_i\times n_i}(\kappa_R) \subset M_{n_i\times n_i}(B\otimes_A\kappa_R)$ when $L_i$ is \textit{of type II}
 and so on, and $\pi$ stands for $\pi\otimes 1\in B\otimes_A\kappa_R$.
Note that the description of the multiplication in $\tilde{M}(\kappa_R)$ given in Section \ref{m} forces   $(m_{i,i})_1$
(when $L_i$ is \textit{of type II}) and $(s_i)_1$ to be invertible.

Then an element $m$ of $\tilde{G}(\kappa_R)$ maps to

\[
\left\{
  \begin{array}{l l}
  \begin{pmatrix} (s_i)_1&0\\
\mathcal{X}_i&1 \end{pmatrix} & \quad \textit{if $L_i$ is \textit{of type} $\textit{I}^o$};\\
  \begin{pmatrix} (s_i)_1&0\\
\mathcal{Y}_i&1 \end{pmatrix} & \quad \textit{if $L_i$ is \textit{of type} $\textit{I}^e$};\\
  \begin{pmatrix} (m_{i,i})_1&0\\
\mathcal{Z}_i&1 \end{pmatrix} & \quad \textit{if $L_i$ is \textit{bound of type II}};\\
(m_{i,i})_1   & \quad \textit{if $L_i$ is \textit{free of type II}}.
    \end{array} \right.
\]

Here,
\[
\left\{
  \begin{array}{l}
 \mathcal{X}_i=(v_i)_1+(\delta_{i-2}e_{i-2}\cdot (m_{i-2, i})_1+\delta_{i+2}e_{i+2}\cdot (m_{i+2, i})_1)\tilde{e_i};\\
 \mathcal{Y}_i=((y_i)_1+\sqrt{\bar{\gamma}_i}(v_i)_1)+(\delta_{i-2}e_{i-2}\cdot (m_{i-2, i})_1
 +\delta_{i+2}e_{i+2}\cdot (m_{i+2, i})_1)\tilde{e_i};\\
 \mathcal{Z}_i=\delta_{i-1}^{\prime}e_{i-1}\cdot (m_{i-1, i})_1+\delta_{i+1}^{\prime}e_{i+1}\cdot (m_{i+1, i})_1+
 \delta_{i-2}e_{i-2}\cdot (m_{i-2, i})_1+\delta_{i+2}e_{i+2}\cdot (m_{i+2, i})_1,
 \end{array} \right.
\]

where

\begin{enumerate}
\item When $j$ is even, $e_{j}=(0,\cdots, 0, 1)$ (resp. $e_j=(0,\cdots, 0, 1, 0)$) of size $1\times n_{j}$
if $L_{j}$ is \textit{of type} $\textit{I}^o$ (resp. \textit{of type} $\textit{I}^e$).
\item $\tilde{e_i}=\begin{pmatrix} \mathrm{id}\\0 \end{pmatrix}$ of size $n_i\times (n_{i}-1)$ (resp. $n_i\times (n_{i}-2)$), where $\mathrm{id}$ is the identity matrix of size $(n_i-1)\times (n_{i}-1)$ (resp. $(n_i-2)\times (n_{i}-2)$) if $L_{i}$ is \textit{of type} $\textit{I}^o$ (resp. \textit{of type} $\textit{I}^e$).
\item 
 $\bar{\gamma}_i$ is as explained in Remark \ref{r33}.(2).
\item $
\delta_{j}^{\prime} = \left\{
  \begin{array}{l l}
  1    & \quad  \textit{if $j$ is odd and $L_j$ is \textit{free of type I}};\\
  0    &   \quad  \textit{otherwise}.
    \end{array} \right.
$
\item When $j$ is odd,   $e_{j}=(0,\cdots, 0, 1)$  of size $1\times n_{j}$.
 \end{enumerate}
\end{proof}

 Note that  if the dimension of $B_i/Z_i$ is even and positive,
 then $\mathrm{O}(B_i/Z_i, \bar{q}_i)_{\mathrm{red}} (= \mathrm{O}(B_i/Z_i, \bar{q}_i))$  is disconnected.
  If the dimension of $B_i/Z_i$ is odd, then
    $\mathrm{O}(B_i/Z_i, \bar{q}_i)_{\mathrm{red}} (= \mathrm{SO}(B_i/Z_i, \bar{q}_i))$  is connected.
The dimension of $B_i/Z_i$, as a $\kappa$-vector space, is as follows:
\[
\left\{
  \begin{array}{l l}
  n_i-1 & \quad  \textit{if $L_i$ is \textit{of type} $\textit{I}^e$};\\
  n_i & \quad  \textit{if $L_i$ is \textit{of type} $\textit{I}^o$ or \textit{free of type II}};\\
  n_i+1 & \quad \textit{if $L_i$ is \textit{bound of type II}}.
    \end{array} \right.
\]
\textit{   }

We next assume that $i=2m-1$ is odd.
Recall that $Y_i$ is the sublattice $B_i$ such that $Y_i/\pi A_i$ is the radical of the alternating bilinear form
$\frac{1}{\pi}\cdot\frac{1}{2^{m-1}} h$ mod $\pi$ on $B_i/\pi A_i$.
\begin{Lem}\label{l42}
Let $i$ be odd.
Then  each element of $\underline{M}(R)$, for a flat $A$-algebra $R$, preserves $Y_i\otimes_A R$.
\end{Lem}
\begin{proof}
We claim that  $Y_i=\pi^{i+1}B_i^{\perp}\cap B_i$. Then the lemma follows from this directly.
The inclusion $Y_i \subseteq \pi^{i+1}B_i^{\perp}\cap B_i$ is clear by the definition of $Y_i$.
For the other direction, 
we choose $a=\pi^{i+1}b^{\perp} \in \pi^{i+1}B_i^{\perp}\cap B_i$ where $b^{\perp}\in B_i^{\perp}$.
Then $h(a, b')=\pi^{i+1}h(b^{\perp}, b')\in \pi^{i+1}B$ for any $b'\in B_i$.
This completes the proof.
\end{proof}

\begin{Thm}\label{t44}
Assume that $i=2m-1$ is odd.
Let $h_i$ denote the nonsingular alternating bilinear form $\frac{1}{\pi}\cdot\frac{1}{2^{m-1}} h$ mod $\pi$ on $B_i/Y_i$.
Then there exists a unique morphism of algebraic groups
$$\varphi_i:\tilde{G}\longrightarrow \mathrm{Sp}(B_i/Y_i, h_i)$$
defined over $\kappa$ such that for all \'{e}tale local $A$-algebras $R$ with residue field $\kappa_R$ and
every $\tilde{m} \in \underline{G}(R)$ with reduction $m\in \tilde{G}(\kappa_R)$,
$\varphi_i(m)\in \mathrm{GL}(B_i\otimes_AR/Y_i\otimes_AR)$ is induced by the action of $\tilde{m}$ on $L\otimes_AR$
(which preserves $B_i\otimes_AR$ and $Y_i\otimes_AR$ by Lemma \ref{l42}).
\end{Thm}

\begin{proof}
As in the above theorem,
 we only provide the image of an element $m$ of $\tilde{G}(\kappa_R)$ into
$\mathrm{Sp}(B_i/Y_i, h_i)$, where $R$ is an \'{e}tale local $A$-algebra with $\kappa_R$ as its residue field.

As in Theorem \ref{t43},  an element $m$ of $\tilde{G}(\kappa_R)$ may be written as, say,
  $(m_{i,j}, s_i \cdots w_i)$ and it has the following formal matrix description:
$$m= \begin{pmatrix} \pi^{max\{0,j-i\}}m_{i,j} \end{pmatrix}  \textit{ together with } z_i^{\ast}, m_{i,i}^{\ast}, m_{i,i}^{\ast\ast}.$$
Here, if $i$ is odd and $L_i$ is \textit{free of type I}, then
 $$m_{i,i}=\begin{pmatrix} s_i&\pi r_i&t_i\\  y_i&1+\pi x_i& u_i\\\pi  v_i&\pi z_i&1+\pi w_i \end{pmatrix},$$
where
 $s_i\in M_{(n_i-2)\times (n_i-2)}(B\otimes_A\kappa_R)$, etc.
If $i$ is odd and  $L_i$ is \textit{of type II} or \textit{bound of type I}, then
$m_{i,i}\in M_{n_i\times n_i}(B\otimes_A\kappa_R)$.

We can write  $m_{i, i}=(m_{i, i})_1+\pi\cdot (m_{i, i})_2$ when $L_i$ is \textit{of type II} or \textit{bound of type I} and
for each block of $m_{i,i}$ when $L_i$ is \textit{free of type I},  $s_i=(s_i)_1+\pi\cdot (s_i)_2$ and so on.
Here,
$(s_i)_1, (s_i)_2\in M_{(n_i-2)\times (n_i-2)}(\kappa_R) \subset M_{(n_i-2)\times (n_i-2)}(B\otimes_A\kappa_R)$ when $L_i$ is \textit{free of type I}
 and so on, and $\pi$ stands for $\pi\otimes 1\in B\otimes_A\kappa_R$.
Note that the description of the multiplication in $\tilde{M}(\kappa_R)$ given in Section \ref{m} forces   $(m_{i,i})_1$
(when $L_i$ is \textit{of type II} or  \textit{bound of type I}) and $(s_i)_1$ to be invertible.

Then $m$   maps to 
\[
\left\{
  \begin{array}{l l}
  (m_{i,i})_1 & \quad  \textit{if $L_i$ is \textit{of type II} or  \textit{bound of type I}};\\
  (s_i)_1 & \quad \textit{if $L_i$ is \textit{free of type I}}.
    \end{array} \right.
\]
\end{proof}

Note that the dimension of $B_i/Y_i$, as a $\kappa$-vector space, is as follows:
\[
\left\{
  \begin{array}{l l}
  n_i & \quad  \textit{if $L_i$ is \textit{of type II} or  \textit{bound of type I}};\\
  n_i-2 & \quad \textit{if $L_i$ is \textit{free of type I}}.
    \end{array} \right.
\]

\begin{Thm}\label{t45}
The morphism $\varphi$ defined by
$$\varphi=\prod_i \varphi_i : \tilde{G} ~ \longrightarrow  ~\prod_{i:even} \mathrm{O}(B_i/Z_i, \bar{q}_i)_{\mathrm{red}}\times \prod_{i:odd} \mathrm{Sp}(B_i/Y_i, h_i)$$
is surjective.
\end{Thm}

\begin{proof}
 Let us first prove the theorem under the assumption that
 \begin{equation}\label{e41}
\textit{dim $\tilde{G}$ = dim $\mathrm{Ker~}\varphi$   + $\sum_{i:\mathrm{even}}$ (dim $\mathrm{O}(B_i/Z_i, \bar{q}_i)_{\mathrm{red}}$)
  + $\sum_{i:\mathrm{odd}}$ (dim $\mathrm{Sp}(B_i/Y_i, h_i)$).}
    \end{equation}
  This equation will be proved in Appendix \ref{App:AppendixA}.
Thus $\mathrm{Im~}\varphi$ contains the identity component of
$\prod_{i:even} \mathrm{O}(B_i/Z_i, \bar{q}_i)_{\mathrm{red}}\times \prod_{i:odd} \mathrm{Sp}(B_i/Y_i, h_i)$.
Here $\mathrm{Ker~}\varphi$ denotes the kernel of $\varphi$ and $\mathrm{Im~}\varphi$ denotes the image of $\varphi.$
Note that it is well known that the image of a homomorphism of algebraic groups is a closed subgroup.

Recall from Section \ref{m} that a matrix form of an element of $\tilde{G}(R)$ for a $\kappa$-algebra $R$ is written $(m_{i,j}, s_i \cdots w_i)$
with the formal matrix interpretation
$$m= \begin{pmatrix} \pi^{max\{0,j-i\}}m_{i,j} \end{pmatrix}  \textit{ together with } z_i^{\ast}, m_{i,i}^{\ast}, m_{i,i}^{\ast\ast}.$$
We represent the given hermitian form  $h$ by a hermitian matrix $\begin{pmatrix} \pi^{i}\cdot h_i\end{pmatrix}$
   with $\pi^{i}\cdot h_i$ for the $(i,i)$-block and $0$ for the remaining blocks, as in Remark \ref{r33}.(1).

 Let $\mathcal{H}$ be the set of even integers $i$ such that $\mathrm{O}(B_i/Z_i, \bar{q_i})_{\mathrm{red}}$ is disconnected.
Notice that $\mathrm{O}(B_i/Z_i, \bar{q_i})_{\mathrm{red}}$ is disconnected exactly when $L_i$ with $i$ even is \textit{free of type II}.
We  first prove that $ \varphi_i$, for such an even integer $i$, is surjective.
We prove this by a series of reductions, after which we will be able to assume that $L$ is of rank two.

For such an even integer $i$  with a \textit{free of type II} lattice $L_i$, we define the closed  scheme $H_i$ of $\tilde{G}$ by the equations $m_{j,k}=0$ if $ j\neq k$, and $m_{j,j}=\mathrm{id}$ if $j \neq i$.
An element of $H_i(R)$ for a $\kappa$-algebra $R$ can be represented by a matrix of the form
$$\begin{pmatrix} id&0& & \ldots& & &0\\ 0&\ddots&& & & &\\ & &id& & & & \\  \vdots & & &m_{i,i} & & &\vdots
\\ & & & & id & & \\  & & & & &\ddots &0 \\ 0& & &\ldots & &0 &id \end{pmatrix}
\textit{ with $z_j^{\ast}=0, m_{j,j}^{\ast}=0, m_{j,j}^{\ast\ast}=0.$}$$
Obviously, $H_i$ has a group scheme structure.
We claim that $\varphi_i$ is surjective from  $H_i$ to
$\mathrm{O}(B_i/Z_i, \bar{q_i})_{\mathrm{red}}$ (recall that $B_i=A_i$ and $Z_i=X_i$ since $L_i$ is \textit{free of type II}).
Note that equations defining $H_i$ are induced by the formal matrix equation $$\sigma({}^tm_{i,i})(\pi^{i}\cdot h_i)m_{i,i}=\pi^{i}\cdot h_i$$
 which is interpreted as in Remark \ref{r35}.
 We emphasize that, in this formal matrix equation, we work with $m_{i,i}$, not $m$, because of the description of $H_i$.
Note that none of the congruence conditions mentioned in Section \ref{mc} involves any entry from $m_{i,i}$.

On the other hand, let us consider the hermitian lattice $L_i$ independently as a $\pi^i$-modular lattice.
Since there is only one non-trivial Jordan component for this lattice and $i$ is even, the smooth integral model associated to $L_i$ is determined by  the following formal matrix equation which is interpreted as in Remark \ref{r35}:
$$\sigma({}^tm)(\pi^{i}\cdot h_i)m=\pi^{i}\cdot h_i,$$
where $m$ is an $(n_i\times n_i)$-matrix and is not subject to any congruence condition.

We consider the map from $H_i$ to the special fiber  of the smooth integral model
associated to the hermitian lattice $L_i$ such that $m_{i,i}$ maps to $m$.
Since $m_{i,i}$ and $m$ are subject to the same set of equations,
this map is an isomorphism as algebraic groups.
In addition, this map induces compatibility between
the morphism $\varphi_i$ from  $H_i$ to $\mathrm{O}(B_i/Z_i, \bar{q_i})_{\mathrm{red}}$
and the morphism from the special fiber of the smooth integral model associated to  $L_i$ to $\mathrm{O}(B_i/Z_i, \bar{q_i})_{\mathrm{red}}$.
Thus, in order to show that $\varphi_i$ is surjective from $H_i$ to
$\mathrm{O}(B_i/Z_i, \bar{q_i})_{\mathrm{red}}$, we may and do assume that $L=L_i$ and in this case $B_i=A_i=L_i$ and $Z_i=X_i=\pi L_i$.
For simplicity, we can also assume that $i=0$.

Because of Equation (\ref{e41}) stated at the beginning of the proof,
the dimension of the image of $\varphi_i$, as a $\kappa$-algebraic group, is the same as that of
$\mathrm{O}(B_i/Z_i, \bar{q_i})_{\mathrm{red}} (=\mathrm{O}(L_i/\pi L_i, \bar{q_i}))$.
Therefore, the image of $\varphi_i$ contains the identity component of $\mathrm{O}(L_i/\pi L_i, \bar{q_i})$, namely $\mathrm{SO}(L_i/\pi L_i, \bar{q_i})$.
Since $\mathrm{O}(L_i/\pi L_i, \bar{q_i})$ has  two connected components, we only need to show the surjectivity of $\varphi_i$
at the level of $\kappa$-points and it suffices to show that the image of $\varphi_i(\kappa)$ contains at least one element which is not contained in $\mathrm{SO}(L_i/\pi L_i, \bar{q_i})(\kappa)$, where $\mathrm{SO}(L_i/\pi L_i, \bar{q_i})(\kappa)$ is the group of $\kappa$-points of the algebraic group $\mathrm{SO}(L_i/\pi L_i, \bar{q_i})$.

Recall that $L_i=\bigoplus_{\lambda}H_{\lambda}\oplus A(2\delta, 2b, 1)$
for a certain $b\in A$ and $\delta (\in A) \equiv 1 \mathrm{~mod~}2$,  cf. Theorem \ref{210}.
We  consider the orthogonal group associated to the quadratic $\kappa$-space $A(2\delta, 2b, 1)/\pi A(2\delta, 2b, 1)$ of dimension $2$.
Then this group is embedded into $\mathrm{O}(L_i/\pi L_i, \bar{q_i})(\kappa)$ as a closed subgroup
and we denote the embedded group by $\mathrm{O}(A(2\delta, 2b, 1)/\pi A(2\delta, 2b, 1), \bar{q_i})(\kappa)$.

We express an element $m_{i,i}\in H_i(R)$, for a $\kappa$-algebra $R$, as $\begin{pmatrix} x&y\\ z&w \end{pmatrix}$
such that $x=(x)_1+\pi \cdot(x)_2$ and so on, where
$(x)_1, (x)_2 \in M_{(n_i-2)\times(n_i-2)}(R)\subset M_{(n_i-2)\times(n_i-2)}(R\otimes_AB)$ and $\pi$ stands for $1\otimes \pi\in R\otimes_AB$.
Consider the closed subscheme of $H_i$ defined by the equations  $x=id, y=0, z=0$.
An argument similar to one used above to reduce to the case where $L = L_i$ shows that this subscheme is isomorphic to the special fiber
of the smooth integral model associated to the hermitian lattice $A(2\delta, 2b, 1)$ of rank $2$.
Then under the map $\varphi_i(\kappa)$, an element of this subgroup maps to an element of
$\mathrm{O}(A(2\delta, 2b, 1)/\pi A(2\delta, 2b, 1), \bar{q_i})(\kappa)$ of the form $\begin{pmatrix} id&0\\ 0&(w)_1 \end{pmatrix}$.
Note that $\mathrm{O}(A(2\delta, 2b, 1)/\pi A(2\delta, 2b, 1), \bar{q_i})(\kappa)$ is not contained in $\mathrm{SO}(L_i/\pi L_i, \bar{q_i})(\kappa)$.
Thus it suffices to show that the restriction of $\varphi_i(\kappa)$ to the above subgroup of $H_i(\kappa)$,
which is given by letting  $x=id, y=0, z=0$, is surjective onto
$\mathrm{O}(A(2\delta, 2b, 1)/\pi A(2\delta, 2b, 1), \bar{q_i})(\kappa)$ and we may and do assume that $L=L_i=A(2\delta, 2b, 1)$ of rank $2$.

Let $m_{i,i}=\begin{pmatrix} r&s\\ t&u \end{pmatrix}$ be an element of $H_i(\kappa)$ such that $r=(r)_1+\pi \cdot(r)_2$ and so on,
where $(r)_1, (r)_2 \in R\subset R\otimes_AB$ and $\pi$ stands for $1\otimes \pi\in R\otimes_AB$.
Recall that $\pi=\sqrt{2\delta}$ for a certain unit $\delta\in A$ such that $\delta\equiv 1 \mathrm{~mod~}2$ so that
$\sigma(\pi)=-\pi$, as mentioned in Section \ref{Notations}.
Let $\bar{b}\in \kappa$ be the reduction of $b$ modulo $\pi$.
Then the equations defining  $H_i(\kappa)$  are
 $$(r)_1^2+(r)_1(t)_1+\bar{b}(t)_1^2=1, (r)_1(u)_1+(t)_1(s)_1=1,$$
 $$ (s)_1^2+(s)_1(u)_1+\bar{b}(u)_1^2=\bar{b}, (r)_1(u)_2+(r)_2(u)_1+(t)_1(s)_2+(t)_2(s)_1=0.$$
Under the map $\varphi_i(\kappa)$, $m_{i,i}$ maps to  $\begin{pmatrix} (r)_1&(s)_1\\ (t)_1&(u)_1 \end{pmatrix}$.
Note that the quadratic form $\bar{q_i}$ restricted to $A(2\delta, 2b, 1)/\pi A(2\delta, 2b, 1)$ is given by the matrix
$\begin{pmatrix} 1&1\\ 0&b \end{pmatrix}$.

We now choose an element of $H_i(\kappa)$ by setting
\[(r)_1=(s)_1=(u)_1=1, (t)_1=0, (r)_2=(s)_2=(t)_2=(u)_2=0.  \]
Then under the morphism $\varphi_i(\kappa)$,
this element maps to  $\begin{pmatrix} 1&1\\ 0&1 \end{pmatrix} \in \mathrm{O}(A(2\delta, 2b, \pi)/\pi A(2\delta, 2b, \pi), \bar{q}_i)(\kappa)$
whose Dickson invariant is nontrivial so that it is not contained   in $\mathrm{SO}(A(2\delta, 2b, \pi)/\pi A(2\delta, 2b, \pi), \bar{q}_i)(\kappa)$.

Therefore, $\varphi_i(\kappa)$ induces a surjection from $H_i(\kappa)$ to
$\mathrm{O}(A(2\delta, 2b, 1)/\pi A(2\delta, 2b, 1), \bar{q_i})(\kappa)$ for $i\in \mathcal{H}$.

The proof to show that $\varphi=\prod_i \varphi_i$ is surjective is similar to that of Theorem 4.5 in \cite{C2}
explained from the last paragraph of page 485 to the first paragraph of page 486 and so we skip it.

Now it suffices to prove Equation (\ref{e41}) made at the beginning of the proof, which is the next lemma.
\end{proof}

\begin{Lem}\label{l46}

$\mathrm{Ker~}\varphi $ is  smooth and unipotent  of dimension $l$.
In addition, the number of connected components of $\mathrm{Ker~}\varphi $ is $2^\beta$.
 Here,
 \begin{itemize}
\item   $l$ is such that
\[\textit{$l$  + $\sum_{i:\mathrm{even}}$ (dim $~\mathrm{O}(B_i/Z_i, \bar{q}_i)_{\mathrm{red}}$)
+ $\sum_{i:\mathrm{odd}}$ (dim $~\mathrm{Sp}(B_i/Y_i, h_i)$) = dim $\tilde{G} ~(=n^2)$.}\]
 \item   $\beta$ is the number of  integers $j$ such that $L_j$ is of type I and $L_{j+2}, L_{j+3}, L_{j+4}$ (resp. $L_{j-1}, L_{j+1},$
 $L_{j+2},  L_{j+3}$) are of type II if $j$ is even (resp. odd).
\end{itemize}

\end{Lem}

Recall that the zero lattice with $i$ even is \textit{of type II}.
If $i$ is odd, then the zero lattice is \textit{of type II} only when both $L_{i-1}$ and $L_{i+1}$ are \textit{of type II}.
The proof is postponed to Appendix \ref{App:AppendixA}.

\begin{Rmk}\label{r47}
We summarize the description of Im $\varphi_i$ as follows.
     \[
      \begin{array}{c|c|c}
      \mathrm{type~of~lattice~}  L_i &  i & \mathrm{Im~}  \varphi_i \\
      \hline
      \textit{II, free}& even & \mathrm{O}(n_i, \bar{q}_i)\\
      \textit{II, bound}& even & \mathrm{SO}(n_i+1, \bar{q}_i)\\
      \textit{I}^o & even & \mathrm{SO}(n_i, \bar{q}_i)\\
      \textit{I}^e & even & \mathrm{SO}(n_i-1, \bar{q}_i)\\
      \textit{II} & odd & \mathrm{Sp}(n_i, h_i)\\
      \textit{I, bound} & odd & \mathrm{Sp}(n_i, h_i)\\
      \textit{I, free} & odd & \mathrm{Sp}(n_i-2, h_i)\\
      \end{array}
    \]
        Let $i$ be even and $L_i$ be \textit{free of type II}.
Then $B_i/Z_i=L_i/\pi L_i$ is a $\kappa$-vector space with even dimension.
We now consider the question of whether the orthogonal group $\mathrm{O}(B_i/Z_i, \bar{q}_i) ~(=\mathrm{O}(n_i, \bar{q}_i))$
is  split or nonsplit.


By Theorem \ref{210}, we have that
$L_i=\bigoplus_{\lambda}H_{\lambda}\oplus A(2\delta, 2b_i, 1)$ for certain $b_i\in A$ and $\delta (\in A) \equiv 1 \mathrm{~mod~}2$.
Thus the orthogonal group $\mathrm{O}(B_i/Z_i, \bar{q}_i) ~(=\mathrm{O}(n_i, \bar{q}_i))$ is split
if and only if the quadratic space  $A(2\delta, 2b_i, 1)/\pi A(2\delta, 2b_i, 1)$ is isotropic.
Recall that $\pi=-\sigma(\pi)$. 
Using this, the quadratic form on $A(2\delta, 2b_i, 1)/\pi A(2\delta, 2b_i, 1)$  is $q(x, y)=x^2+xy+\bar{b}_iy^2$,
where $\bar{b}_i$ is the reduction of $b_i$ in $\kappa$.

We consider the identity $q(x, y)=x^2+xy+\bar{b}_iy^2=0$.
If $y=0$, then $x=0$. Assume that $y\neq 0$.
Then we have that $\bar{b}_i=(x/y)^2+x/y$.

Thus we can see that there exists a solution of the equation $z^2+z=\bar{b}_i$ over $\kappa$ if and only if $q(x, y)$ is isotropic if and only if
$\mathrm{O}(B_i/Z_i, \bar{q}_i) ~(=\mathrm{O}(n_i, \bar{q}_i))$ is split.
\end{Rmk}

\subsection{The construction of component groups}\label{cg}

The purpose of this subsection is to define a surjective morphism from $\tilde{G}$ to $(\mathbb{Z}/2\mathbb{Z})^{\beta}$,
where $\beta$  is the number of  integers $j$ such that $L_j$ is \textit{of type I} and $L_{j+2}, L_{j+3}, L_{j+4}$ (resp. $L_{j-1}, L_{j+1},$
 $L_{j+2},  L_{j+3}$) are \textit{of type II} if $j$ is even (resp. odd), as defined in Lemma \ref{l46}.

We start with reproducing the definitions of the sublattices $L^i$ and $C(L)$ of $L$ given in Definitions 4.8 and 4.9 of \cite{C2}.
   \begin{Def}\label{d48}
    We set $L^0=L$ and inductively define, for positive integers $i$,
    \[L^i:=\{x\in L^{i-1} | h(x, L^{i-1})\subset (\pi^i)\}.\]
    When $i=2m$ is even,
      $$L^{2m}=\pi^m(L_0\oplus L_1)\oplus\pi^{m-1}(L_2\oplus L_3)\oplus \cdots \oplus \pi(L_{2m-2}\oplus L_{2m-1})\oplus \bigoplus_{i\geq 2m}L_i.$$
    \end{Def}
     We choose a Jordan splitting for the hermitian lattice $(L^{2m}, \xi^{-m}h)$ as follows:
     $$L^{2m}=\bigoplus_{i \geq 0} M_i,$$ where
     $$M_0=\pi^mL_0\oplus\pi^{m-1}L_2\oplus \cdots \oplus \pi L_{2m-2}\oplus L_{2m},$$
     $$M_1=\pi^mL_1\oplus\pi^{m-1}L_3\oplus \cdots \oplus \pi L_{2m-1}\oplus L_{2m+1}$$
     $$\mathrm{and}~ M_k=L_{2m+k} \mathrm{~if~} k\geq 2.$$
Here, $M_i$ is $\pi^i$-modular.
We caution that the hermitian form we use on $L^{2m}$ is not $h$, but its rescaled version $\xi^{-m}h$.
Thus   $M_i$ is $\pi^i$-modular, not $\pi^{2m+i}$-modular.
\begin{Def}\label{d49}
We define $C(L)$ to be the sublattice of $L$ such that $$C(L)=\{x\in L \mid h(x,y) \in (\pi) \ \ \mathrm{for}\ \ \mathrm{all}\ \ y \in B(L)\}.$$
    \end{Def}

We  choose any even integer $j=2m$ such that $L_{j}$ is \textit{of type I} and $L_{j+2}, L_{j+3}, L_{j+4}$ are \textit{of type II}
(possibly zero by our convention), and consider the Jordan splitting
$$L^{j}=\bigoplus_{i \geq 0} M_i$$
defined above.
The reason that we require  $L_{j+2}, L_{j+3}, L_{j+4}$ to be \textit{of type II} is explained in  Step (1)
which will be stated below.
    We stress that
    \[
\left\{
  \begin{array}{l }
  \textit{$M_0$ is nonzero and \textit{of type I}, since it contains $L_{j}$ as a direct summand};\\
     \textit{$M_1$ is \textit{bound of type I}, and all of $M_2 (=L_{j+2}), M_3 (=L_{j+3}), M_4 (=L_{j+4})$ are \textit{of type II}.}
         \end{array} \right.
\]

 That $M_1$ is \textit{bound of type I} does not guarantee that the norm of $M_1$ (=$n(M_1)$) is the ideal $(4)$
 since $M_1=\pi^{j/2}L_1\oplus\pi^{j/2-1}L_3\oplus \cdots \oplus \pi L_{j-1}\oplus L_{j+1}$.
 If $n(M_1)=(2)$, then we choose a suitable basis of both $M_0$ and $M_1$  such that
 the associated Jordan splitting for $M_0\oplus M_1$ is $M_0'\oplus M_1'$ with  $n(M_1')=(4)$
 (cf. Lemma 2.9 and the following  paragraph in \cite{C2}).
 Thus we may and do assume that  $n(M_1)=(4)$. 

   Choose a basis $(\langle e_i\rangle, e)$ (resp. $(\langle e_i\rangle, a, e)$) for $M_0$
     so that $M_0=\bigoplus_{\lambda}H_{\lambda}\oplus K$
      when the rank of $M_0$ is odd (resp. even).
      Here, we follow the notation from Theorem \ref{210}.
     Note that the Gram matrix associated to $(a, e)$, when the rank of $M_0$ is even, is
     $\begin{pmatrix} 1&1\\1&2b\end{pmatrix}$ with $b\in A$.
 Then $B(L^{j})$ is spanned by $$(\langle e_i\rangle, \pi e) ~(resp.~ (\langle e_i\rangle, \pi a, e)) \textit{~and~}  M_1 \oplus (\bigoplus_{i\geq 2} M_i)$$
     and  $C(L^{j})$ is spanned by
     $$(\langle \pi e_i\rangle, e) ~(resp.~ (\langle \pi e_i\rangle, \pi a, e)) \textit{~and~}  M_1 \oplus (\bigoplus_{i\geq 2} M_i).$$\\

We now construct a morphism $\psi_j : \tilde{G} \rightarrow \mathbb{Z}/2\mathbb{Z}$ as follows.
(There are 3 cases.)\\

(1) Firstly, we assume that $M_0$ is \textit{of type} $\textit{I}^e$.
We choose a Jordan splitting for the hermitian lattice $(C(L^j), \xi^{-m}h)$ as follows:
$$C(L^j)=\bigoplus_{i \geq 1} M_i^{\prime},$$
where $$M_1^{\prime}=(\pi)a\oplus Be\oplus M_1, ~~~ M_2^{\prime}=(\oplus_i(\pi)e_i)\oplus M_2, ~~~ \mathrm{and}~ M_k^{\prime}=M_k \mathrm{~if~} k\geq 3.$$
Here, $M_i^{\prime}$ is $\pi^i$-modular and $(\pi)$ is the ideal of $B$ generated by a uniformizer $\pi$.
Notice that $M_2^{\prime}$ is \textit{of type II}, since both $\oplus_i(\pi)e_i$ and $M_2$ are \textit{of type II},
and that both $M_3^{\prime}$ and $M_4^{\prime}$ are \textit{of type II} as well.
The lattice $M_1^{\prime}$ is \textit{of type I}, since the Gram matrix associated to $(\pi)a\oplus Be$ is
$\begin{pmatrix} -2\delta&\pi\\-\pi&2b\end{pmatrix}$ with $\delta (\in A) \equiv 1 \mathrm{~mod~}2$.
Thus $M_1^{\prime}$ is \textit{free of type I} since the adjacent two lattices $M_0^{\prime}$ (which is empty) and $M_2^{\prime}$ are \textit{of type II}.

Then consider  the sublattice $Y(C(L^j))$ of $C(L^j)$ and choose a Jordan splitting for the hermitian lattice
$(Y(C(L^j)), \xi^{-(m+1)}h)$ as follows:
$$Y(C(L^j))=\bigoplus_{i \geq 0} M_i^{\prime\prime},$$
where $M_i^{\prime\prime}$ is $\pi^i$-modular.
We explain the above Jordan splitting precisely.
Since $C(L^j)=\bigoplus_{i \geq 1} M_i^{\prime}$ and $M_1'$ is \textit{free of type I},
$Y(C(L^j))=Y(M_1')\oplus \bigoplus_{i \geq 2} M_i^{\prime}$.

\begin{enumerate}
\item[(i)] If $b\in (2)$, then $Y(M_1')= (2)a\oplus Be\oplus \pi M_1$.
The lattice $(2)a\oplus Be$ is $\pi^2$-modular \textit{of type II} and  $\pi M_1$ is $\pi^3$-modular \textit{of type II}.
Since we rescale $Y(C(L^j))$ by $\xi^{-1}$,
we have that
$$M_0''=\left((2)a\oplus Be\right)\oplus M_2'=\left((2)a\oplus Be\right)\oplus (\oplus_i(\pi)e_i)\oplus M_2,$$
$$M_1''=\pi M_1\oplus M_3'=\pi M_1\oplus M_3, ~~~M_2''=M_4'=M_4.$$
Thus, $M_0''$ is \textit{free of type II} as both $M_1''$ and $M_2''$ are \textit{of type II}.

\item[(ii)] If $b\in A$ is a unit, then let $\sqrt{b}$ be an element of $A$ such that $\sqrt{b}^2\equiv b$ mod $2$.
We choose a basis $(\pi a, \pi a+1/\sqrt{b}\cdot e)$ for the component $(\pi)a\oplus Be$ of $M_1'$
whose associated Gram matrix is
$\begin{pmatrix} -2\delta& -2\delta+\pi/\sqrt{b} \\-2\delta-\pi/\sqrt{b}& -2\delta+2b/\sqrt{b}^2 \end{pmatrix}$.
Here, the $(2,2)$-component $-2\delta+2b/\sqrt{b}^2$ is contained in the ideal $(4)$ since  $\delta\equiv 1$ mod $2$.
Thus, as in the above case (i), we have that
$$Y(M_1')= (2)a\oplus B\left( \pi a+1/\sqrt{b}\cdot e  \right) \oplus \pi M_1.$$
Here, $(2)a\oplus B\left( \pi a+1/\sqrt{b}\cdot e  \right)$ is $\pi^2$-modular \textit{of type II} and
 $\pi M_1$ is $\pi^3$-modular \textit{of type II}.
Since we rescale $Y(C(L^j))$ by $\xi^{-1}$,
we have that
$$M_0''=\left((2)a\oplus B\left( \pi a+1/\sqrt{b}\cdot e  \right)\right) \oplus M_2'=\left((2)a\oplus B\left( \pi a+1/\sqrt{b}\cdot e  \right)\right) \oplus (\oplus_i(\pi)e_i)\oplus M_2,$$
$$M_1''=\pi M_1\oplus M_3'=\pi M_1\oplus M_3, ~~~M_2''=M_4'=M_4.$$
Thus, $M_0''$ is \textit{free of type II} as both $M_1''$ and $M_2''$ are \textit{of type II}.\\
\end{enumerate}

Therefore, we conclude that the $\pi^0$-modular  Jordan component of $Y(C(L^j))$, which is  $M_0''$,  is \textit{free of type II}.
The reason of our assumption that $L_{j+2}, L_{j+3}, L_{j+4}$ are \textit{of type II}, while $L_j$  is \textit{of type I},
is to make $M_0''$  \textit{free of type II}.

Let $G_j$ denote the special fiber of the smooth integral model associated
to the hermitian lattice $(Y(C(L^j)), \xi^{-(m+1)}h)$.
If $m$ is an element of the group of $R$-points of the naive integral model
associated to the hermitian lattice $L$, for a flat $A$-algebra $R$, then $m$ stabilizes the hermitian lattice $(Y(C(L^j))\otimes_AR, \xi^{-(m+1)}h\otimes 1)$ as well.
This fact induces a morphism from $\tilde{G}$ to $G_j$ (cf. the second paragraph of page 488 in \cite{C2}).
Moreover, since $M_0^{\prime\prime}$ is \textit{free of type II} and nonzero, we have a morphism from $G_j$ to the even orthogonal group associated to $M_0^{\prime\prime}$
 as explained in Section \ref{red}.
Thus, the Dickson invariant of this orthogonal group induces the morphism
$$\psi_j : \tilde{G} \longrightarrow \mathbb{Z}/2\mathbb{Z}.$$
\textit{ }

(2) We next assume that $M_0$ is \textit{of type} $\textit{I}^o$.
We choose a Jordan splitting for the hermitian lattice $(C(L^j), \xi^{-m}h)$ as follows:
$$C(L^j)=\bigoplus_{i \geq 0} M_i^{\prime},$$
where $$M_0^{\prime}= Be, ~~~M_1^{\prime}=M_1, ~~~ M_2^{\prime}=(\oplus_i(\pi)e_i)\oplus M_2, ~~~ \mathrm{and}~ M_k^{\prime}=M_k \mathrm{~if~} k\geq 3.$$
Here, $M_i^{\prime}$ is $\pi^i$-modular and $(\pi)$ is the ideal of $B$ generated by a uniformizer $\pi$.
Notice that the rank of the $\pi^0$-modular lattice $M_0^{\prime}$ is 1 and
that all of the lattices $M_2^{\prime}, M_3^{\prime}, M_4^{\prime}$ are \textit{of type II}.
If $G_j$ denotes the special fiber of the smooth integral model associated to the hermitian lattice $(C(L^j), \xi^{-m}h)$,
then we  have a morphism from $\tilde{G}$ to $G_j$ as in the above argument (1).

We now consider the new hermitian lattice $M_0^{\prime}\oplus C(L^j)$.
Then for a flat $A$-algebra $R$, there is a natural  embedding from the group of $R$-points of the naive integral model
associated to the hermitian lattice $(C(L^j), \xi^{-m}h)$
to that of  the hermitian lattice $M_0^{\prime}\oplus C(L^j)$ such that
$m$ maps to $\begin{pmatrix} 1&0 \\ 0&m \end{pmatrix}$,
where $m$ is an element of the former group.
This fact induces a morphism from the smooth integral model associated to the hermitian lattice $(C(L^j), \xi^{-m}h)$
to the smooth integral model associated to the hermitian lattice $M_0^{\prime}\oplus C(L^j)$
(cf. from the last paragraph of page 488 to the first paragraph of page 489 in \cite{C2}).
In Remark \ref{r410}, we describe this morphism explicitly in terms of matrices.

Thus we have a morphism from the special fiber $G_j$ of the smooth integral model associated to $C(L^j)$  to the special fiber
$G_j'$ of the smooth integral model associated to  $M_0^{\prime}\oplus C(L^j)$.
Note that $(M_0^{\prime}\oplus M_0^{\prime})\oplus \bigoplus_{i \geq 1} M_i^{\prime}$ is a Jordan splitting of the hermitian lattice $M_0^{\prime}\oplus C(L^j)$.
Let $G_j''$ be the special fiber of the smooth integral model associated to
$Y(C((M_0^{\prime}\oplus M_0^{\prime})\oplus \bigoplus_{i \geq 1} M_i^{\prime}))$.
Since the $\pi^0$-modular lattice $M_0^{\prime}\oplus M_0^{\prime}$ is \textit{of type} $\textit{I}^e$,
we have a morphism $G_j'\rightarrow \mathbb{Z}/2\mathbb{Z}$ obtained by factoring through $G_j''$
and the corresponding even orthogonal group with the Dickson invariant as constructed in argument (1).
$\psi_j$ is defined to be the composite
$$\psi_j : \tilde{G} \rightarrow G_j \rightarrow G_j' \rightarrow  \mathbb{Z}/2\mathbb{Z}.$$

\begin{Rmk}\label{r410}
In this remark, we describe the morphism from the smooth integral model $\underline{G}_j$ associated to the hermitian lattice
$(C(L^j), \xi^{-m}h)$
to the smooth integral model $\underline{G}_j'$
associated to the hermitian lattice $M_0^{\prime}\oplus C(L^j)$ as given in argument (2) above, in terms of matrices.
Let $R$ be a flat $A$-algebra.
We choose an element in $\underline{G}_j(R)$
and express it as a  matrix $m= \begin{pmatrix} \pi^{max\{0,j-i\}}m_{i,j} \end{pmatrix}$.
Then $m_{0,0}=\begin{pmatrix} 1+2 z_0^{\ast} \end{pmatrix}$ since $M_0'$ is \textit{of type I} with rank 1 and
$M_2'$ is \textit{of type II}
so that we may and do write $m$ as
$m= \begin{pmatrix} 1+2z_0^{\ast} &m_1\\m_2&m_3 \end{pmatrix}$.
We consider a morphism from $\underline{G}_j$ to $\mathrm{Aut}_{B}(M_0^{\prime}\oplus C(L^j))$
such that $m$ maps to
$$T=\begin{pmatrix} 1&0 \\ 0&m \end{pmatrix}=\begin{pmatrix} 1&0 &0\\ 0&1+2 z_0^{\ast} &m_1\\0& m_2&m_3\end{pmatrix},$$
where the set of $R$-points of the group scheme $\mathrm{Aut}_{B}(M_0^{\prime}\oplus C(L^j))$ is
the automorphism group of $(M_0^{\prime}\oplus C(L^j)) \otimes_A R$ by ignoring the hermitian form.
Then the image of this morphism is represented by an affine group scheme which is isomorphic to $\underline{G}_j$.
Note that  $T$ preserves the hermitian form attached to the lattice $M_0^{\prime}\oplus C(L^j)$.

We claim that $\begin{pmatrix} 1&0 \\ 0&m \end{pmatrix}$ is contained in $\underline{G}_j'(R)$.
If this is true, then the above matrix description defines the morphism from $\underline{G}_j$ to $\underline{G}_j'$
we want to describe (cf. the last paragraph of page 489 in \cite{C2}).

We rewrite the hermitian lattice $M_0^{\prime}\oplus C(L^j)$ as $(M_0^{\prime}\oplus M_0^{\prime})\oplus (\bigoplus_{i \geq 1} M_i^{\prime})$.
Let $(e_1, e_2)$ be a basis for $(M_0^{\prime}\oplus M_0^{\prime})$
so that the corresponding Gram matrix of $(M_0^{\prime}\oplus M_0^{\prime})$  is   $\begin{pmatrix} a&0 \\ 0&a \end{pmatrix}$, where $a \equiv 1$ mod 2.
Then the hermitian lattice $(M_0^{\prime}\oplus M_0^{\prime})$ has Gram matrix $\begin{pmatrix} a&a \\ a&2a \end{pmatrix}$
with respect to the basis $(e_1, e_1+e_2)$.
The lattice $(M_0^{\prime}\oplus M_0^{\prime})$ is \textit{unimodular of type $I^e$} with rank 2.
With this basis, $T$ becomes

$$\tilde{T}=\begin{pmatrix} 1&-2 z_0^{\ast} &-m_1\\ 0&1+2 z_0^{\ast} &m_1\\0& m_2&m_3\end{pmatrix}.$$

On the other hand, an element of $\underline{G}_j'(R)$,
with respect to a basis for $M_0^{\prime}\oplus C(L^j)$ obtained by putting together
the basis $(e_1, e_1+e_2)$ for $(M_0^{\prime}\oplus M_0^{\prime})$ and a basis for $C(L^j)$, is given by an expression
$$\begin{pmatrix} 1+\pi x_0'&-2 z_0'^{\ast} & m_1'  \\ u_0'&1+\pi w_0' &m_1'' \\ m_2' &m_2''  & m_3''\end{pmatrix},$$ cf. Section \ref{mc}.
Then we can easily see that the congruence conditions on $m_1, m_2, m_3$ are the same as those of $m_1', m_2'', m_3''$, respectively,
and that the congruence conditions on $m_1'$ are included in  those of $m_1''$.
We caution that the congruence conditions on $m_1'$ are not the same as  those of $m_1''$
because of the condition (d) of the description of an element of $\underline{M}(R)$ mentioned in the argument following Remark \ref{r31}.
Thus $\tilde{T}$ is an element of $\underline{M}_j^{\ast}(R)$,
where $\underline{M}_j^{\ast}$ is the group scheme in Section \ref{m} associated to $M_0^{\prime}\oplus C(L^j)$
so that $\underline{G}_j'$ is defined as the closed subgroup scheme of $\underline{M}_j^{\ast}$ stabilizing the hermitian form
on $M_0^{\prime}\oplus C(L^j)$.

In conclusion, $\tilde{T}$ is an element of $\underline{M}_j^{\ast}(R)$ preserving
the hermitian form on $M_0^{\prime}\oplus C(L^j)$.
Therefore, it is an element of $\underline{G}_j'(R)$.

To summarize, if $R$ is a nonflat $A$-algebra, then
we can write an element of $\underline{G}_j(R)$ formally as $m= \begin{pmatrix} 1+2 z_0^{\ast} &m_1\\m_2&m_3 \end{pmatrix}$.
Then the image of $m$ in $\underline{G}_j'(R)$ is $\tilde{T}$
with respect to a basis as explained above.
\end{Rmk}

\textit{ }

(3) We choose any odd integer $j$ such that $L_{j}$ is \textit{of type I} and $L_{j-1}, L_{j+1}, L_{j+2}, L_{j+3}$ are \textit{of type II}
 (possibly zero, by our convention).
Note that the lattice $L_{j}$ is then \textit{free of type I}.
Consider the Jordan splitting
$$L^{j-1}=\bigoplus_{i \geq 0} M_i,$$
where
     $$M_0=\pi^{(j-1)/2}L_0\oplus\pi^{(j-1)/2-1}L_2\oplus \cdots \oplus \pi L_{j-3}\oplus L_{j-1},$$
     $$M_1=\pi^{(j-1)/2}L_1\oplus\pi^{(j-1)/2-1}L_3\oplus \cdots \oplus \pi L_{j-2}\oplus L_{j}$$
     $$\mathrm{and}~ M_k=L_{j-1+k} \mathrm{~if~} k\geq 2.$$
    We stress that
    \[
\left\{
  \begin{array}{l }
  \textit{$M_1$ is nonzero and \textit{of type I}, since it contains $L_{j}$ (of type I) as a direct summand};\\
     \textit{All of $M_2 (=L_{j+1}), M_3 (=L_{j+2}), M_4 (=L_{j+3}$) are \textit{of type II}.}
         \end{array} \right.
\]

We now follow the arguments of the above two cases.
\begin{enumerate}
\item[(i)] If $M_0$ is \textit{of type $I^e$}, then we follow the argument (1) with $j-1$.
We briefly summarize it below.
Since $M_0$ is \textit{of type $I$} and $n(M_1)=(2)$, we choose another basis for $M_0\oplus M_1$ whose associate Jordan splitting is
$M_0'\oplus M_1'$ with $n(M_1')=(4)$.
Consider the lattice $Y(C(L^{j-1}))=\bigoplus_{i \geq 0} M_i^{\prime\prime}$.
Then $M_0''$ is \textit{free of type II} and so we have a morphism from $G_{j-1}$ to the even orthogonal group associated to $M_0^{\prime\prime}$.
Here, $G_{j-1}$ is the special fiber of the smooth integral model associated to $Y(C(L^{j-1}))$.
Thus, the Dickson invariant of this orthogonal group induces the morphism
$$\psi_j : \tilde{G} \longrightarrow \mathbb{Z}/2\mathbb{Z}.$$

\item[(ii)] If $M_0$ is \textit{of type II}, then we follow the argument (1) with $j-1$.
Namely, if we consider the lattice $Y(C(L^{j-1}))=\bigoplus_{i \geq 0} M_i^{\prime\prime}$,
then it is easy to show that $M_0''$ is \textit{free of type II} by using the similar argument used in Step (1).
As in the above case,  we have a morphism from $G_{j-1}$ to the even orthogonal group associated to $M_0^{\prime\prime}$.
Here, $G_{j-1}$ is the special fiber of the smooth integral model associated to $Y(C(L^{j-1}))$.
Thus, the Dickson invariant of this orthogonal group induces the morphism
$$\psi_j : \tilde{G} \longrightarrow \mathbb{Z}/2\mathbb{Z}.$$

\item[(iii)] If  $M_0$ is \textit{of type $I^o$}, then we follow the argument (2) with $j-1$.
We briefly summarize it below.
As in the above case, since $M_0$ is \textit{of type $I$} and $n(M_1)=(2)$, we choose another basis for $M_0\oplus M_1$ whose associate Jordan splitting is
$M_0'\oplus M_1'$ with $n(M_1')=(4)$.
Consider two lattices $C(L^{j-1})=\bigoplus_{i \geq 0} M_i^{\prime}$ and $M_0^{\prime}\oplus C(L^{j-1})$.
Here the rank of the $\pi^0$-modular lattice $M_0^{\prime}$ is 1.
Then we can assign  the even orthogonal group  to the $\pi^0$-modular Jordan component of $Y\left(C(M_0^{\prime}\oplus C(L^{j-1}))\right)$.
Thus, the Dickson invariant of this orthogonal group induces the morphism
$$\psi_j : \tilde{G} \longrightarrow \mathbb{Z}/2\mathbb{Z}.$$
\end{enumerate}



(4) Combining all cases, we have the morphism $$\psi=\prod_j \psi_j : \tilde{G} \longrightarrow (\mathbb{Z}/2\mathbb{Z})^{\beta},$$
where $\beta$ is the number of  integers $j$ such that $L_j$ is \textit{of type I} and $L_{j+2}, L_{j+3}, L_{j+4}$ (resp. $L_{j-1},$ $L_{j+1}, L_{j+2}, L_{j+3}$) are \textit{of type II} (possibly zero, by our convention) if $j$ is even (resp. odd).

\begin{Thm}\label{t411}
 The morphism $$\psi=\prod_j \psi_j : \tilde{G} \longrightarrow (\mathbb{Z}/2\mathbb{Z})^{\beta}$$ is surjective.

Moreover, the morphism $$\varphi \times \psi : \tilde{G} \longrightarrow  \prod_{i:even} \mathrm{O}(B_i/Z_i, \bar{q}_i)_{\mathrm{red}} \times \prod_{i:odd} \mathrm{Sp}(B_i/Y_i, h_i)\times (\mathbb{Z}/2\mathbb{Z})^{\beta}$$
is also surjective.

\end{Thm}

\begin{proof}
We first show that $\psi_j$ is surjective.
Recall that for such an integer $j$, $L_j$ is \textit{of type I} and $L_{j+2}, L_{j+3}, L_{j+4}$ (resp. $L_{j-1}, L_{j+1},$
 $L_{j+2},  L_{j+3}$) are \textit{of type II} if $j$ is even (resp. odd).
We define  the closed subgroup scheme $F_j$ of $\tilde{G}$ defined by the following equations:
\begin{itemize}
\item $m_{i,k}=0$ \textit{if $i\neq k$};
\item $m_{i,i}=\mathrm{id}, z_i^{\ast}=0, m_{i,i}^{\ast}=0, m_{i,i}^{\ast\ast}=0$ \textit{if $i\neq j$};
\item and for $m_{j,j}$,
\[\left \{
  \begin{array}{l l}
  s_j=\mathrm{id~}, y_j=0, v_j=0, z_j=\pi z_j^{\ast} & \quad  \textit{if $i$ is even and $L_i$ is \textit{of type} $\textit{I}^o$};\\
  s_j=\mathrm{id~}, r_j=t_j=y_j=v_j=u_j=w_j=0, z_j=\pi z_j^{\ast}  & \quad \textit{if $i$ is even and  $L_i$ is \textit{of type} $\textit{I}^e$};\\
  s_j=\mathrm{id~}, r_j=t_j=y_j=v_j=u_j=w_j=0 & \quad \textit{if $i$ is odd and  $L_i$ is \textit{free of type I}}.\\
    \end{array} \right.\]
\end{itemize}
A formal matrix form of an element of $F_j(R)$ for a $\kappa$-algebra $R$ is then
\[\begin{pmatrix} id&0& & \ldots& & &0\\ 0&\ddots&& & & &\\ & &id& & & & \\  \vdots & & &m_{j,j} & & &\vdots
\\ & & & & id & & \\  & & & & &\ddots &0 \\ 0& & &\ldots & &0 &id \end{pmatrix}$$
such that $$m_{j,j}=\left\{
\begin{array}{l l}
\begin{pmatrix}id&0\\0&1+2 z_j^{\ast} \end{pmatrix} & \quad \textit{if $j$ is even and $L_j$ is of type $I^o$};\\
\begin{pmatrix}id&0&0\\0&1+\pi x_j&2 z_j^{\ast}\\0&0&1 \end{pmatrix} & \quad \textit{if $j$ is even and  $L_j$ is of type $I^e$};\\
\begin{pmatrix}id&0&0\\0&1+\pi x_j&0\\0&\pi z_j&1 \end{pmatrix} & \quad \textit{if $j$ is odd and  $L_j$ is free of type $I$}.
\end{array}\right.\]
We emphasize that we have $2z_j^{\ast}$, not $\pi z_j$, when $j$ is even.

In Lemma \ref{la9}, we will show  that   $F_j$ is isomorphic to $ \mathbb{A}^{1} \times \mathbb{Z}/2\mathbb{Z}$ as a $\kappa$-variety
so that it has exactly two connected components,
by enumerating equations defining $F_j$ as a closed subvariety of an affine space of dimension $2$ (resp. $4$)
if $j$ is even and $L_j$ is \textit{of type $\textit{I}^o$} (resp. otherwise).
Here,  $\mathbb{A}^{1}$ is an affine space of dimension $1$.
These equations are necessary in this theorem and thus we state them in Equation (\ref{e42}) below.
We  refer to Lemma \ref{la9} for the proof.
We write $x_j=(x_j)_1+\pi \cdot(x_j)_2$, $z_j=(z_j)_1+\pi \cdot(z_j)_2$,  and $z_j^{\ast}=(z_j^{\ast})_1+\pi \cdot(z_j^{\ast})_2$,
where $(x_j)_1, (x_j)_2, (z_j)_1, (z_j)_2, (z_j^{\ast})_1, (z_j^{\ast})_2 \in R \subset R\otimes_AB$ and $\pi$ stands for $1\otimes \pi\in R\otimes_AB$.
Then the equations defining $F_j$ as a closed subvariety of an affine space of dimension $2$ (resp. $4$),
if $j$ is even and $L_j$ is \textit{of type $I^o$} (resp.  otherwise), are
\begin{equation}\label{e42}
\left\{
\begin{array}{l l}
(z_j^{\ast})_1+(z_j^{\ast})_1^2=0 & \quad \textit{if $j$ is even and $L_j$ is of type $I^o$};\\
(x_j)_1=0, (x_j)_2+(z_j^{\ast})_1=0, (z_j^{\ast})_1+(z_j^{\ast})_1^2=0 & \quad \textit{if $j$ is even and $L_j$ is of type $I^e$};\\
(z_j)_1+(z_j)_1^2=0,   (x_j)_1=0,  (z_j)_1+(x_j)_2=0 & \quad \textit{if $j$ is odd and  $L_j$ is free of type $I$}.
\end{array}\right.
\end{equation}

The proof of the surjectivity of $\psi_j$ is given below.
The main idea is to show that $\psi_j|_{F_j}$ is surjective.
First assume that $j$ is even.
There are 4 cases according to the types of $M_0$ and $L_j$.
Recall that $\bigoplus_{i \geq 0} M_i$ is a Jordan splitting of  a rescaled hermitian lattice $(L^{j}, \xi^{-j/2}h)$
and that $M_0=\pi^{j/2}L_0\oplus\pi^{j/2-1}L_2\oplus \cdots \oplus \pi L_{j-2}\oplus L_{j}$.

\begin{enumerate}
\item
Assume that both $M_0$ and $L_j$ are \textit{of type $I^e$}.
In this case and the next case, we will describe $\psi_j|_{F_j} : F_j \rightarrow \mathbb{Z}/2\mathbb{Z}$ explicitly in terms of a formal matrix.
To do that, we will first describe a morphism from $F_j$ to the special fiber of the smooth integral model associated to $L^j$.
Then we will describe a morphism from $F_j$ to the even orthogonal group associated to $M_0''$,
where $M_0''$ is a Jordan component of $Y(C(L^j))=\bigoplus_{i \geq 0} M_i''$,
and compute the Dickson invariant of the image of an element of $F_j$ in this orthogonal group.


We write $M_0=N_0\oplus L_j$, where $N_0$ is unimodular with even rank. Thus $N_0$ is either \textit{of type II} or \textit{of type $I^e$}.
First we assume that $N_0$ is \textit{of type $I^e$}.
Then we can write $N_0=(\oplus_{\lambda'}H_{\lambda'})\oplus A(1, 2b, 1)$ and $L_j=(\oplus_{\lambda''}H_{\lambda''})\oplus A(1, 2b', 1)$ by Theorem \ref{210},
where  $H_{\lambda'}=H(0)=H_{\lambda''}$ and $b, b'\in A$.
Thus we write  $M_0=(\oplus_{\lambda}H_{\lambda})\oplus A(1, 2b, 1)\oplus A(1, 2b', 1)$, where $H_{\lambda}=H(0)$.
For this choice of a basis of $L^j=\bigoplus_{i \geq 0} M_i$,
the image of a fixed element of $F_j$ in the special fiber of the smooth integral model associated to $L^j$ is
$$\begin{pmatrix} id&0 &0\\ 0 &\begin{pmatrix} 1+\pi x_j & 2 z_j^{\ast}\\ 0 & 1 \end{pmatrix}  &0 \\ 0& 0 &id \end{pmatrix}.$$
Here, $id$ in the $(1,1)$-block  corresponds to the direct summand $(\oplus_{\lambda}H_{\lambda})\oplus A(1, 2b, 1)$ of $M_0$ and
the diagonal block $\begin{pmatrix} 1+\pi x_j & 2 z_j^{\ast}\\ 0 & 1 \end{pmatrix} $  corresponds to the direct summand $A(1, 2b', 1)$ of $M_0$.

Let $(e_1, e_2, e_3, e_4)$ be a basis for the direct summand $A(1, 2b, 1)\oplus A(1, 2b', 1)$ of $M_0$.
Since this is  \textit{unimodular of type $I^e$}, we can choose another basis based on Theorem \ref{210}.
With the basis $(-2be_1+e_2, (2b'-1)e_1+e_3-e_4, e_3, e_2+e_4)$, denoted by $(e_1', e_2', e_3', e_4')$,
$A(1, 2b, 1)\oplus A(1, 2b', 1)$ becomes
$A(2b(2b-1), 2b'(2b'-1), -(2b-1)(2b'-1))\oplus A(1, 2(b+b'), 1)$.
Here,  $A(2b(2b-1), 2b'(2b'-1), -(2b-1)(2b'-1))$ is \textit{unimodular of type II}. 
Thus we can write
$$M_0=(\oplus_{\lambda}H_{\lambda})\oplus \left( Be_1'\oplus Be_2' \right)
\oplus \left(Be_3'\oplus Be_4'\right).$$
For this basis, the image of a fixed element of $F_j$ in the special fiber of the smooth integral model associated to $L^j$ is
$$\begin{pmatrix} id&0 &0\\ 0 &\begin{pmatrix} 1&0&0&0 \\ 0&1&0&0 \\ 0&\pi x_j-2z_j^{\ast} &1 +\pi x_j & 2 z_j^{\ast}\\0&0& 0 & 1 \end{pmatrix}  &0 \\ 0& 0 &id \end{pmatrix}.$$
Here, $id$ in the $(1,1)$-block corresponds to the direct summand $(\oplus_{\lambda}H_{\lambda})$ of $M_0$ and
the diagonal block $\begin{pmatrix} 1&0&0&0 \\ 0&1&0&0 \\ 0&\pi x_j-2z_j^{\ast} &1 +\pi x_j & 2 z_j^{\ast}\\0&0& 0 & 1 \end{pmatrix}$ corresponds to
$A(2b(2b-1), 2b'(2b'-1), -(2b-1)(2b'-1))\oplus A(1, 2(b+b'), 1)$ with a basis $(e_1', e_2', e_3', e_4')$.

We now describe  the image of a fixed element of $F_j$ in the even orthogonal group associated to $M_0''$,
where $M_0''$ is a Jordan component of $Y(C(L^j))=\bigoplus_{i \geq 0} M_i''$.
There are 3 cases depending on whether $b+b'$ is a unit or not, and whether $M_1=\oplus H(1)$ (possibly empty)
or $M_1=A(4b'', 2\delta, \pi) \oplus (\oplus H(1))$ with $b''\in A$.
We will see in Step (iii) below that the case $M_1=A(4b'', 2\delta, \pi) \oplus (\oplus H(1))$ with $b''\in A$ is reduced to the case
$M_1=\oplus H(1)$ (possibly empty). \\

\begin{enumerate}
\item[(i)] Assume that $M_1=\oplus H(1)$ and  $b+b'\in (2)$. Then
$$M_0''=   \left( (\pi) e_1'\oplus (\pi) e_2' \right)\oplus
  \left((2)e_3'\oplus Be_4'\right) \oplus (\oplus_{\lambda}\pi H_{\lambda})\oplus M_2,$$
as explained in the argument (i) of  Step (1) in the construction of $\psi_j$.
For this basis,  the image of a fixed element of $F_j$ in the orthogonal group associated to $M_0''/\pi M_0''$ is
$$T_1=\begin{pmatrix} \begin{pmatrix} 1&0&0&0 \\ 0&1&0&0 \\ 0& (x_j)_1 &1&  (z_j^{\ast})_1\\0&0& 0 & 1 \end{pmatrix} &0 \\ 0 & id  \end{pmatrix}.$$
Here,  $(z_j^{\ast})_1$ (resp. $(x_j)_1$) is in $R$ such that
$z_j^{\ast}=(z_j^{\ast})_1+\pi \cdot (z_j^{\ast})_2$ (resp. $x_j=(x_j)_1+\pi \cdot (x_j)_2$) as explained in the paragraph before Equation (\ref{e42}).
  The Dickson invariant of $T_1$
is the same as that of  $\begin{pmatrix} 1&0&0&0 \\ 0&1&0&0 \\ 0& (x_j)_1 &1&  (z_j^{\ast})_1\\0&0& 0 & 1 \end{pmatrix}$.
Here, we consider $\begin{pmatrix} 1&0&0&0 \\ 0&1&0&0 \\ 0& (x_j)_1 &1&  (z_j)_1\\0&0& 0 & 1 \end{pmatrix}$ as an element of the orthogonal group associated to
$\left( (\pi) e_1'\oplus (\pi) e_2' \right)\oplus  \left((2)e_3'\oplus Be_4'\right)$.
On the other hand, by Equation (\ref{e42}), the equations defining $F_j$ are
$$(x_j)_2=(z_j^{\ast})_1, ~~~ (x_j)_1=0, ~~~ (z_j^{\ast})_1+(z_j^{\ast})_1^2=0.$$
Since $(x_j)_1=0$, the Dickson invariant of $\begin{pmatrix} 1&0&0&0 \\ 0&1&0&0 \\ 0& (x_j)_1 &1&  (z_j^{\ast})_1\\0&0& 0 & 1 \end{pmatrix}$
is the same as  that of  $\begin{pmatrix} 1 &  (z_j^{\ast})_1\\ 0 & 1 \end{pmatrix}$.
In order to compute the Dickson invariant, we use the scheme-theoretic description of the Dickson invariant explained in Remark 4.4 of \cite{C1}.
The Dickson invariant of an orthogonal group of the quadratic space with dimension 2 is explicitly given at the end of the proof of Lemma 4.5 in \cite{C1}.
Based on this, the Dickson invariant of $\begin{pmatrix} 1 &  (z_j^{\ast})_1\\ 0 & 1 \end{pmatrix}$ is $(z_j^{\ast})_1$.
Note that $(z_j^{\ast})_1$ is indeed an element of $\mathbb{Z}/2\mathbb{Z}$ by Equation (\ref{e42}).

In conclusion, $(z_j^{\ast})_1$ is the image of a fixed element of $F_j$ under the map $\psi_j$.
Since $(z_j^{\ast})_1$ can be either $0$ or $1$,
 $\psi_j|_{F_j}$ is surjective onto $\mathbb{Z}/2\mathbb{Z}$ and thus $\psi_j$ is surjective.\\


\item[(ii)] Assume that $M_1=\oplus H(1)$ and that $b+b'$ is a unit.
Then,  $$M_0''=  \left( (\pi)e_1'\oplus (\pi)e_2' \right)\oplus $$
$$\left((2)e_3'\oplus B\left(\pi e_3'+1/\sqrt{b+b'}e_4'\right)\right) \oplus (\oplus_{\lambda}\pi H_{\lambda})\oplus M_2,$$
as explained in the argument (ii) of  Step (1) in the construction of $\psi_j$.
For this basis,  the image of a fixed element of $F_j$ in the orthogonal group associated to $M_0''/\pi M_0''$ is
$$T_1=\begin{pmatrix} \begin{pmatrix} 1&0&0&0 \\ 0&1&0&0 \\ 0& (x_j)_1 &1&  (x_j)_1+1/\sqrt{b+b'}(z_j^{\ast})_1\\0&0& 0 & 1 \end{pmatrix} &0 \\ 0 & id  \end{pmatrix}.$$
Since $(x_j)_1=0$ by  Equation (\ref{e42}),  the Dickson invariant of $T_1$ is the same as that of
 $\begin{pmatrix} 1&\sqrt{b+b'}(z_j^{\ast})_1\\0& 1 \end{pmatrix}$, which turns to be
$(z_j^{\ast})_1$ by using the similar argument used in  the above case (i).

In conclusion, $(z_j^{\ast})_1$ is the image of a fixed element of $F_j$ under the map $\psi_j$.
Since $(z_j^{\ast})_1$ can be either $0$ or $1$ by Equation (\ref{e42}),
 $\psi_j|_{F_j}$ is surjective onto $\mathbb{Z}/2\mathbb{Z}$ and thus $\psi_j$ is surjective.\\

\item[(iii)] Assume that $M_1=A(4b'', 2\delta, \pi) \oplus (\oplus H(1))$ with $b''\in A$.
Let $(e_5, e_6)$ be a basis for $A(4b'', 2\delta, \pi)$.
Recall that $(e_3', e_4')$ is a basis for the direct summand $A(1, 2(b+b'), 1)$ of $M_0$.
We choose another basis $(e_3'-e_4', e_3')$ for $A(1, 2(b+b'), 1)$ whose associated Gram matrix is
$\begin{pmatrix} -1+2(b+b')& 0 \\ 0 & 1 \end{pmatrix}$.
Then choose a basis $(e_3'-e_4', e_3'-e_5, e_5-\frac{2b''\pi}{\delta}e_6, e_6+\pi e_3')$
for the lattice spanned by $(e_3'-e_4', e_3', e_5, e_6)$ whose  associated Gram matrix is
$$\begin{pmatrix} 1+2(b+b')&0&0&0\\ 0&1+4b''& 0&0 \\ 0&0 & -4b''(1+4b'')&-\pi(1+4b'')\\0&0&\pi(1+4b'')  &0    \end{pmatrix}$$
(cf. Lemma 2.9 and the following paragraph of loc. cit. in \cite{C2}).
Note that this lattice is the same as $A(1, 2(b+b'), 1)\oplus  A(4b'', 2\delta, \pi)$.
Since the lattice spanned by  $(e_5-\frac{2b''\pi}{\delta}e_6, e_6+\pi e_3')$ is $\pi^1$-modular with the norm $(4)$,
it is isometric to $H(1)$ by Theorem 2.2 of \cite{C2}.
Now choose another basis $(e_3'-e_5, e_4'-e_5, e_5-\frac{2b''\pi}{\delta}e_6, e_6+\pi e_3')$ for the above lattice
$A(1, 2(b+b'), 1)\oplus  A(4b'', 2\delta, \pi)$   such that the associated Gram matrix is
$$\begin{pmatrix} 1+4b''&1+4b''&0&0\\ 1+4b''&2(1+b+b'+2b'')& 0&0 \\ 0&0 & -4b''(1+4b'')&-\pi(1+4b'')\\0&0&\pi(1+4b'')  &0    \end{pmatrix}.$$

Let
\[\left\{
\begin{array}{l}
\tilde{M}_0=(\oplus_{\lambda}H_{\lambda})\oplus \left( Be_1'\oplus Be_2' \right)\\
  \textit{        } \oplus \left( B(e_3'-e_5)\oplus B(e_4'-e_5) \right);\\
\tilde{M}_1=B(e_5-\frac{2b''\pi}{\delta}e_6)\oplus B(e_6+\pi e_3') \oplus (\oplus H(1)).
\end{array}\right.
\]
Then
$\tilde{M}_0\oplus\tilde{M}_1\oplus(\oplus_{i\geq 2}M_i)$ is another Jordan splitting of $L^j$,
where $\tilde{M}_0$ (resp. $\tilde{M}_1$) is $\pi^0$-modular (resp. $\pi^1$-modular)
and $\tilde{M}_1$ is isometric to $\oplus H(1)$.

For this choice of a basis,
the block associated to $\tilde{M}_0\oplus M_2$ of the image of a fixed element of $F_j$ in the special fiber of the smooth integral model associated to $L^j$ is
$$\begin{pmatrix} id&0 &0\\ 0 &\begin{pmatrix} 1&0&0&0 \\ 0&1&0&0 \\ 0&\frac{1}{1+4b''}(\pi x_j-2z_j^{\ast})
&\frac{1}{1+4b''}(1 +\pi x_j) & \frac{1}{1+4b''}(2 z_j^{\ast})\\0&0& 0 & 1 \end{pmatrix}  &0 \\ 0& 0 &id \end{pmatrix}.$$
Here, $id$ in the $(1,1)$-block  corresponds to the direct summand $(\oplus_{\lambda}H_{\lambda})$ of $\tilde{M}_0$
and $id$ in the $(3,3)$-block  corresponds to $M_2$.

We now apply the argument of  Steps (i) and (ii) to the above Jordan splitting $\tilde{M}_0\oplus\tilde{M}_1\oplus(\oplus_{i\geq 2}M_i)$ of $L^j$ since $\tilde{M}_1$ is isometric to $\oplus H(1)$.
As explained in Steps (i) and (ii), in order to describe the image of a fixed element of $F_j$ in the orthogonal group associated to $M_0''$,
we only need the above block associated to $\tilde{M}_0\oplus M_2$.
Note that  $\frac{1}{1+4b''}$ is a unit and $(x_j)_1=0$. 
Then $(z_j^{\ast})_1$ is the image of a fixed element of $F_j$ under the map $\psi_j$.
Since $(z_j^{\ast})_1$ can be either $0$ or $1$ by Equation (\ref{e42}),
 $\psi_j|_{F_j}$ is surjective onto $\mathbb{Z}/2\mathbb{Z}$ and thus $\psi_j$ is surjective.\\
\end{enumerate}

If $N_0$ is \textit{of type II}, then the proof of the surjectivity of $\psi_j$ is similar  to
and simpler than   that of the above cases when $N_0$ is \textit{of type $I^e$} and so we skip it.\\

\item Assume that $M_0$  is \textit{of type $I^e$} and   $L_j$ is  \textit{of type $I^o$}.
We write $M_0=N_0\oplus L_j$, where $N_0$ is unimodular with odd rank so that it is  \textit{of type $I^o$}.
Then we can write $N_0=(\oplus_{\lambda'}H_{\lambda'})\oplus (a)$ and $L_j=(\oplus_{\lambda''}H_{\lambda''})\oplus (a')$ by Theorem \ref{210},
where $H_{\lambda'}=H(0)=H_{\lambda''}$ and $a, a'\in A$ such that $a, a' \equiv 1$ mod 2.
Thus we write  $M_0=(\oplus_{\lambda}H_{\lambda})\oplus (a)\oplus (a')$, where $H_{\lambda}=H(0)$.
For this choice of a basis of $L^j=\bigoplus_{i \geq 0} M_i$,
the image of a fixed element of $F_j$ in the special fiber of the smooth integral model associated to $L^j$ is
$$\begin{pmatrix} id&0 &0\\ 0 &\begin{pmatrix} 1+ 2 z_j^{\ast}  \end{pmatrix}  &0 \\ 0& 0 &id \end{pmatrix}.$$
Here, $id$ in the $(1,1)$-block  corresponds to the direct summand $(\oplus_{\lambda}H_{\lambda})\oplus (a)$ of $M_0$ and
the diagonal block $\begin{pmatrix} 1+ 2 z_j^{\ast} \end{pmatrix} $  corresponds to the direct summand $(a')$ of $M_0$.

Let $(e_1, e_2)$ be a basis for the direct summand $(a)\oplus (a')$ of $M_0$.
Since this is  \textit{unimodular of type $I^e$},  we can choose another basis $(e_1, e_1+e_2)$
such that the associated Gram matrix is $A(a, a+a', a)$, where $a+a'\in (2)$, so that 
\[M_0=(\oplus_{\lambda}H_{\lambda})\oplus A(a, a+a', a).\]

For this basis, the image of a fixed element of $F_j$ in the special fiber of the smooth integral model associated to $L^j$ is
\begin{equation}\label{e4.3}
\begin{pmatrix} id&0 &0\\ 0 &\begin{pmatrix} 1 & -2 z_j^{\ast}\\ 0 & 1+2 z_j^{\ast} \end{pmatrix}  &0 \\ 0& 0 &id \end{pmatrix}.
\end{equation}
Here, the diagonal block $\begin{pmatrix}  1 & -2 z_j^{\ast}\\ 0 & 1+2 z_j^{\ast} \end{pmatrix} $ corresponds to $A(a, a+a', a)$ with a basis $(e_1, e_1+e_2)$
and $id$ in the $(1,1)$-block  corresponds to the direct summand $\oplus_{\lambda}H_{\lambda}$ of $M_0$.

We now describe  the image of a fixed element of $F_j$ in the even orthogonal group associated to $M_0''$,
where $M_0''$ is a Jordan component of $Y(C(L^j))=\bigoplus_{i \geq 0} M_i''$.
As in the above case (1),
there are 3 cases depending on whether $(a+a')/2$ is a unit or not, and whether $M_1=\oplus H(1)$ (possibly empty)
or $M_1=A(4b'', 2\delta, \pi) \oplus (\oplus H(1))$ with $b''\in A$.
We will see in Step (iii) below that the case $M_1=A(4b'', 2\delta, \pi) \oplus (\oplus H(1))$ with $b''\in A$ is reduced to the case
$M_1=\oplus H(1)$ (possibly empty). \\

\begin{enumerate}
\item[(i)] Assume that $M_1=\oplus H(1)$ and  $(a+a')/2\in (2)$. Then
$$M_0''=   \left((2)e_1\oplus B(e_1+e_2)\right) \oplus (\oplus_{\lambda}\pi H_{\lambda})\oplus M_2,$$
as explained in the argument (i) of  Step (1) in the construction of $\psi_j$.
For this basis,  the image of a fixed element of $F_j$ in the orthogonal group associated to $M_0''/\pi M_0''$ is
$$T_1=\begin{pmatrix} \begin{pmatrix}  1 &  (z_j^{\ast})_1\\ 0 & 1 \end{pmatrix} &0 \\ 0 & id  \end{pmatrix}.$$
Here,  $(z_j^{\ast})_1$ is in $R$ such that  $z_j^{\ast}=(z_j^{\ast})_1+\pi\cdot (z_j^{\ast})_2$
as explained in the paragraph before Equation (\ref{e42}).
Then by using a similar argument used in the above case (1),
 the Dickson invariant of $T_1$ is $(z_j^{\ast})_1$.

 In conclusion, $(z_j^{\ast})_1$ is the image of a fixed element of $F_j$ under the map $\psi_j$.
Since $(z_j^{\ast})_1$ can be either $0$ or $1$ by Equation (\ref{e42}),
 $\psi_j|_{F_j}$ is surjective onto $\mathbb{Z}/2\mathbb{Z}$ and thus $\psi_j$ is surjective.\\

\item[(ii)] Assume that $M_1=\oplus H(1)$ and  that $(a+a')/2$ is a unit.
Then  $$M_0''= (2)e_1\oplus B\left(\pi e_1+1/\sqrt{(a+a')/2}\cdot (e_1+e_2)\right) \oplus (\oplus_{\lambda}\pi H_{\lambda})\oplus M_2,$$
as explained in the argument (ii) of  Step (1) in the construction of $\psi_j$.
For this basis,  the image of a fixed element of $F_j$ in the orthogonal group associated to $M_0''/\pi M_0''$ is
$$T_1=\begin{pmatrix} \begin{pmatrix}  1 &  1/\sqrt{(a+a')/2}(z_j^{\ast})_1\\ 0 & 1 \end{pmatrix} &0 \\ 0 & id  \end{pmatrix}.$$
As in the  case (ii) of the above step (1),
 the Dickson invariant of $T_1$
is the same as that of $\begin{pmatrix}  1 &  1/\sqrt{(a+a')/2}(z_j^{\ast})_1\\ 0 & 1 \end{pmatrix}$,
which turns to be $(z_j^{\ast})_1$.

 In conclusion, $(z_j^{\ast})_1$ is the image of a fixed element of $F_j$ under the map $\psi_j$.
Since $(z_j^{\ast})_1$ can be either $0$ or $1$ by Equation (\ref{e42}),
 $\psi_j|_{F_j}$ is surjective onto $\mathbb{Z}/2\mathbb{Z}$ and thus $\psi_j$ is surjective.\\

\item[(iii)]  Assume that $M_1=A(4b'', 2\delta, \pi) \oplus (\oplus H(1))$ with $b''\in A$.
Let $(e_3, e_4)$ be a basis for $A(4b'', 2\delta, \pi)$.
Recall that $e_2$ is a basis for the direct summand $(a')$ of $M_0$.
Then choose a basis $(e_2-e_3, e_3-\frac{2b''\pi}{\delta}e_4, e_4+\pi e_2)$, denoted by $(e_2', e_3', e_4')$, 
for the lattice spanned by $(e_2, e_3, e_4)$ such that the associated Gram matrix is
$$\begin{pmatrix} a'(1+4a'b'')&0&0\\ 0 & -4b''(1+4b'')&-\pi(1+4b'')\\ 0&\pi(1+4b'')  &2\delta(1-a') \end{pmatrix}$$
(cf. Lemma 2.9 and the following paragraph of loc. cit. in \cite{C2}).
Note that this lattice is the same as $(a')\oplus  A(4b'', 2\delta, \pi)$.
Since the lattice spanned by  $(e_3', e_4')$ is $\pi^1$-modular with the norm $(4)$,
it is isometric to $H(1)$ by Theorem 2.2 of \cite{C2}.
Now choose another basis $(e_1, e_1+e_2', e_3', e_4')$ for the  lattice
$(a)\oplus (a')\oplus  A(4b'', 2\delta, \pi)$ such that the associated Gram matrix is
$$\begin{pmatrix} a&a&0&0\\ a&a+a'(1+4a'b'')& 0&0 \\ 0&0 & -4b''(1+4b'')&-\pi(1+4b'')\\0&0&\pi(1+4b'')  &2\delta(1-a')    \end{pmatrix}.$$

Let
\[\left\{
\begin{array}{l}
\tilde{M}_0=(\oplus_{\lambda}H_{\lambda})\oplus  \left( Be_1\oplus B(e_1+e_2') \right);\\
\tilde{M}_1=Be_3'\oplus Be_4' \oplus (\oplus H(1)).
\end{array}\right.
\]
Then
$\tilde{M}_0\oplus\tilde{M}_1\oplus(\oplus_{i\geq 2}M_i)$ is another Jordan splitting of $L^j$,
where $\tilde{M}_0$ (resp. $\tilde{M}_1$) is $\pi^0$-modular (resp. $\pi^1$-modular)
and $\tilde{M}_1$ is isometric to $\oplus H(1)$.

For this choice of a basis,
the block associated to $\tilde{M}_0\oplus M_2$ of the image of a fixed element of $F_j$ in the special fiber of the smooth integral model associated to $L^j$ is
$$\begin{pmatrix} id&0 &0\\ 0 &\begin{pmatrix} 1&\frac{1}{1+4b''}(-2z_j^{\ast})\\
0 & 1+\frac{1}{1+4b''}(2 z_j^{\ast})\end{pmatrix}  &0 \\ 0& 0 &id \end{pmatrix}.$$
Here,  $id$ in the $(1,1)$-block corresponds to the direct summand $(\oplus_{\lambda}H_{\lambda})$ of $\tilde{M}_0$
and $id$ in the $(3,3)$-block  corresponds to $M_2$.

We now apply the argument of  Steps (i) and (ii) to the above Jordan splitting $\tilde{M}_0\oplus\tilde{M}_1\oplus(\oplus_{i\geq 2}M_i)$ of $L^j$ since $\tilde{M}_1$ is isometric to $\oplus H(1)$.
As explained in Steps (i) and (ii), in order to describe the image of a fixed element of $F_j$ in the orthogonal group associated to $M_0''$,
we only need the above block associated to $\tilde{M}_0\oplus M_2$.
Note that  $\frac{1}{1+4b''}$ is a unit. 
Then $(z_j^{\ast})_1$ is the image of a fixed element of $F_j$ under the map $\psi_j$.
Since $(z_j^{\ast})_1$ can be either $0$ or $1$ by Equation (\ref{e42}),
 $\psi_j|_{F_j}$ is surjective onto $\mathbb{Z}/2\mathbb{Z}$ and thus $\psi_j$ is surjective.\\
\end{enumerate}

\item Assume that both $M_0$ and $L_j$ are \textit{of type $I^o$}.
In this case, we will describe $\psi_j|_{F_j} : F_j \rightarrow \mathbb{Z}/2\mathbb{Z}$ explicitly in terms of a formal matrix.
To do that, we will first describe a morphism from $F_j$ to the special fiber of the smooth integral model associated to $L^j$
and  then to $G_j$. Recall that $G_j$ is  the special fiber of the smooth integral model associated to $C(L^j)=\bigoplus_{i \geq 0} M_i^{\prime}$.
Then we will describe a morphism from $F_j$ to the special fiber of the smooth integral model associated to $M_0'\oplus C(L^j)$
and to the special fiber of the smooth integral model associated to $Y(C(M_0'\oplus C(L^j)))$.
Finally,  we will describe a morphism from $F_j$ to a certain even orthogonal group associated to $Y(C(M_0'\oplus C(L^j)))$
 and compute the Dickson invariant of the image of an element of $F_j$ in this orthogonal group.

We write $M_0=N_0\oplus L_j$, where $N_0$ is unimodular with even rank.
Thus $N_0$ is either \textit{of type II} or \textit{of type $I^e$}.
First we assume that $N_0$ is \textit{of type $I^e$}.
Then we can write $N_0=(\oplus_{\lambda'}H_{\lambda'})\oplus A(1, 2b, 1)$ and $L_j=(\oplus_{\lambda''}H_{\lambda''})\oplus (a)$ by Theorem \ref{210},
where  $H_{\lambda'}=H(0)=H_{\lambda''}$, $b\in A$, and  $a (\in A) \equiv 1$ mod 2.
Thus we write  $M_0=(\oplus_{\lambda}H_{\lambda})\oplus A(1, 2b, 1)\oplus (a)$, where $H_{\lambda}=H(0)$.
For this choice of a basis of $L^j=\bigoplus_{i \geq 0} M_i$,
the image of a fixed element of $F_j$ in the special fiber of the smooth integral model associated to $L^j$ is
$$\begin{pmatrix} id&0 &0\\ 0 &\begin{pmatrix} 1+2 z_j^{\ast}  \end{pmatrix}  &0 \\ 0& 0 &id \end{pmatrix}.$$
Here, $id$ in the $(1,1)$-block corresponds to the direct summand $(\oplus_{\lambda}H_{\lambda})\oplus A(1, 2b, 1)$ of $M_0$ and
the diagonal block $\begin{pmatrix} 1+2 z_j^{\ast}\end{pmatrix}$ corresponds to the direct summand $(a)$ of $M_0$.

Let $(e_1, e_2, e_3)$ be a basis for the direct summand $A(1, 2b, 1)\oplus (a)$ of $M_0$.
Since this is  \textit{unimodular of type $I^o$}, we can choose another basis based on Theorem 2.2 of \cite{C2}.
Namely, if we choose $(-2be_1+e_2, -ae_1+e_3,  e_2+e_3)$ as another basis, then
$A(1, 2b, 1)\oplus (a)$ becomes $A(2b(2b-1), a(a+1), a(2b-1))\oplus (a+2b)$.
Here, $A(2b(2b-1), a(a+1), a(2b-1))$ is \textit{unimodular of type II}.
Thus we can write that
$$M_0=(\oplus_{\lambda}H_{\lambda})\oplus \left(B(-2be_1+e_2)\oplus B(-ae_1+e_3)\right)\oplus B(e_2+e_3).$$
For this basis, the image of a fixed element of $F_j$ in the special fiber of the smooth integral model associated to $L^j$ is
$$\begin{pmatrix} id&0 &0\\ 0 &
\begin{pmatrix}1&\frac{-2a}{a+2b}z_j^{\ast} & \frac{-2a}{a+2b}z_j^{\ast}\\
0&1+\frac{4(a+b)}{a+2b}z_j^{\ast} &\frac{4b}{a+2b}z_j^{\ast}\\0&\frac{-2a}{a+2b}z_j^{\ast}&1+\frac{2a}{a+2b} z_j^{\ast} \end{pmatrix}
&0 \\ 0& 0 &id \end{pmatrix}.$$
Here,
the $(2,2)$-block of the above formal matrix  corresponds to $B(-2be_1+e_2)\oplus$ $B(-ae_1+e_3)\oplus$ $B(e_2+e_3)$ with the Gram matrix  $A(2b(2b-1), a(a+1), a(2b-1))\oplus (a+2b)$
and $id$ in the $(1,1)$-block corresponds to the direct summand $(\oplus_{\lambda}H_{\lambda})$ of $M_0$.
The above formal matrix can be simplified by observing a formal matrix description of an element of $\underline{M}(R)$ for a $\kappa$-algebra $R$,
explained in Section \ref{m}.
Since $A(2b(2b-1), a(a+1), a(2b-1))\oplus (a+2b)$ is \textit{of type $I^o$},
the $(2,2)$-block of the above formal matrix turns to be
\begin{equation}\label{e4.4}
\begin{pmatrix}1&0 & \frac{-2a}{a+2b}z_j^{\ast}\\
0&1&0\\0&\frac{-2a}{a+2b}z_j^{\ast}&1+\frac{2a}{a+2b} z_j^{\ast} \end{pmatrix}.
\end{equation}

Then the direct summand $M_0'$ of $C(L^j)=\oplus_{i\geq 0}M_i'$ is $B(e_2+e_3)$ of rank 1.
Recall that $$M_0^{\prime}= B(e_2+e_3), ~~~M_1^{\prime}=M_1, ~~~ M_2^{\prime}=\left( \left(\pi B(-2be_1+e_2)\oplus \pi B(-ae_1+e_3)\right)\oplus(\oplus_{\lambda}\pi H_{\lambda})\right)\oplus M_2,$$
$$ \mathrm{and}~ M_k^{\prime}=M_k \mathrm{~if~} k\geq 3.$$
The image of a fixed element of  $F_j$ in the special fiber of the smooth integral model associated to $C(L^j)$ is then
$$\begin{pmatrix}
\begin{pmatrix}1+\frac{2a}{a+2b} z_j^{\ast}&0 & \frac{-2\pi a}{a+2b}z_j^{\ast}\\
\frac{-\pi a}{a+2b}z_j^{\ast}&1&0\\0&0&1 \end{pmatrix} &0 \\
0& id\end{pmatrix}.$$
Here, the $(1,1)$-block  corresponds to $B(e_2+e_3)\oplus \left(\pi B(-2be_1+e_2)\oplus \pi B(-ae_1+e_3)\right)$.

We now describe the image of a fixed element of $F_j$ in  the special fiber of the smooth integral model associated to
$M_0'\oplus C(L^j)=(M_0^{\prime}\oplus M_0^{\prime})\oplus (\bigoplus_{i \geq 1} M_i^{\prime})$.
If $(e_1', e_2')$ is a basis for $(M_0^{\prime}\oplus M_0^{\prime})$, then
we choose another basis $(e_1', e_1'+e_2')$  for $(M_0^{\prime}\oplus M_0^{\prime})$.
For this basis, based on the description of the morphism from   the smooth integral model associated to
$C(L^j)$ to  the smooth integral model associated to
$M_0'\oplus C(L^j)$ explained in Remark \ref{r410},
the image of a fixed element of $F_j$ in the special fiber of the smooth integral model associated to $M_0'\oplus C(L^j)$ is
$$\begin{pmatrix}\begin{pmatrix}1&\frac{-2a}{a+2b} z_j^{\ast} &0 & \frac{2\pi a}{a+2b}z_j^{\ast}\\
0&1+\frac{2a}{a+2b} z_j^{\ast}&0 & \frac{-2\pi a}{a+2b}z_j^{\ast}\\
0&\frac{-\pi a}{a+2b}z_j^{\ast}&1&0\\0&0&0&1 \end{pmatrix}&0
\\ 0&id\end{pmatrix}.$$
Here,   the $(1,1)$-block  corresponds to  $(M_0^{\prime}\oplus M_0^{\prime})\oplus \left(\pi B(-2be_1+e_2)\oplus \pi B(-ae_1+e_3)\right)$.

We now follow Step (2) with $M_0'\oplus C(L^j)$ such that the above formal matrix corresponds to the formal matrix (\ref{e4.3}).
Recall that a Jordan splitting of $M_0'\oplus C(L^j)$ is
$$M_0'\oplus C(L^j)=(M_0^{\prime}\oplus M_0^{\prime})\oplus (\bigoplus_{i \geq 1} M_i^{\prime}),$$
where, $M_0^{\prime}\oplus M_0^{\prime}$ with a basis $(e_1', e_1'+e_2')$ is $\pi^0$-modular, $M_i'$ is $\pi^i$-modular,
and   $M_2'=\left(\pi B(-2be_1+e_2)\oplus \pi B(-ae_1+e_3)\right)\oplus (\oplus_{\lambda}\pi H_{\lambda})\oplus M_2$.
We consider a Jordan splitting $Y(C(M_0'\oplus C(L^j)))=\bigoplus_{i \geq 0} M_i''$.
Then $M_0''$, which is the only lattice needed in the desired orthogonal group, can be described by using $M_0^{\prime}\oplus M_0^{\prime}$, 
$M_1^{\prime}$, and  $M_2'$.
The only difference between the above formal matrix and the formal matrix (\ref{e4.3}) of Step (2) is the appearance of $\frac{2\pi a}{a+2b}z_j^{\ast}$ in the $(1, 4)$ and $(2, 4)$-entries,
and $\frac{-\pi a}{a+2b}z_j^{\ast}$ in the $(3, 2)$-entry of the above formal matrix.
However, these entries will be zero after reduction to the orthogonal group associated to $M_0''$
since $M_0''$ is \textit{free of type II} so that the  diagonal block associated to $M_0''$ has no congruence condition.
Thus in the corresponding orthogonal group, all entries having $\pi$ as a factor become zero.

 Then by using the result of Step (2),
 $(z_j^{\ast})_1$ is the image of a fixed element of $F_j$ under the map $\psi_j$.
Since $(z_j^{\ast})_1$ can be either $0$ or $1$ by Equation (\ref{e42}),
 $\psi_j|_{F_j}$ is surjective onto $\mathbb{Z}/2\mathbb{Z}$ and thus $\psi_j$ is surjective.\\

If $N_0$ is \textit{of type II}, then the proof of the surjectivity of $\psi_j$ is similar  to
and simpler than   that of the above case when $N_0$ is \textit{of type $I^e$} and so we skip it.\\

 \item Assume that  $M_0$ is \textit{of type $I^o$} and that $L_j$ is \textit{of type $I^e$}.
We write $M_0=N_0\oplus L_j$, where $N_0$ is unimodular with odd rank so that it is  \textit{of type $I^o$}.
Then we can write $N_0=(\oplus_{\lambda'}H_{\lambda'})\oplus (a)$ and $L_j=(\oplus_{\lambda''}H_{\lambda''})\oplus A(1, 2b, 1)$ by Theorem \ref{210},
where  $H_{\lambda'}=H(0)=H_{\lambda''}$ and $a, b \in A$ such that $a \equiv 1$ mod 2.
We  write  $M_0=(\oplus_{\lambda}H_{\lambda})\oplus (a)\oplus A(1, 2b, 1)$, where $H_{\lambda}=H(0)$.
For this choice of a basis of $L^j=\bigoplus_{i \geq 0} M_i$,
the image of a fixed element of $F_j$ in the special fiber of the smooth integral model associated to $L^j$ is
$$\begin{pmatrix} id&0 &0\\ 0 &\begin{pmatrix} 1+\pi x_j & 2 z_j^{\ast}\\ 0 & 1 \end{pmatrix}  &0 \\ 0& 0 &id \end{pmatrix}.$$
Here, $id$ in the $(1, 1)$-block corresponds to the direct summand $(\oplus_{\lambda}H_{\lambda})\oplus (a)$ of $M_0$ and
the diagonal block $\begin{pmatrix} 1+\pi x_j & 2 z_j^{\ast}\\ 0 & 1 \end{pmatrix} $  corresponds to the direct summand $A(1, 2b, 1)$ of $M_0$.

Let $(e_1, e_2, e_3)$ be a basis for the direct summand $(a)\oplus A(1, 2b, 1)$ of $M_0$.
Since this is  \textit{unimodular of type $I^o$}, we can choose another basis based on Theorem \ref{210}.
Namely, if we choose $(-2be_2+e_3, e_1-ae_2,  e_1+e_3)$ as another basis, then
$(a)\oplus A(1, 2b, 1)$ becomes $A(2b(2b-1), a(a+1), a(2b-1))\oplus (a+2b)$.
Here,  $A(2b(2b-1), a(a+1), a(2b-1))$ is \textit{unimodular of type II}. 
Thus we can write that $M_0=(\oplus_{\lambda}H_{\lambda})\oplus A(2b(2b-1), a(a+1), a(2b-1))\oplus (a+2b)$.
For this basis, the image of a fixed element of $F_j$ in the special fiber of the smooth integral model associated to $L^j$ is
$$\begin{pmatrix} id&0 &0\\ 0 &
\begin{pmatrix} 1+\frac{2}{a+2b}(b\pi x_j-z_j^{\ast})&\frac{a\pi x_j}{a+2b} &\frac{-2}{a+2b}z_j^{\ast}\\
\frac{2}{a+2b}(b\pi x_j-z_j^{\ast}) &1+\frac{a\pi x_j}{a+2b} &\frac{-2}{a+2b}z_j^{\ast} \\
 \frac{-2}{a+2b}(b\pi x_j-z_j^{\ast})& \frac{-a\pi x_j}{a+2b} &1+\frac{2}{a+2b}z_j^{\ast} \end{pmatrix}
   &0 \\ 0& 0 &id \end{pmatrix}.$$
Here, the $(2,2)$-block  corresponds to $A(2b(2b-1), a(a+1), a(2b-1))\oplus (a+2b)$.

As in  Case (3), the  above formal matrix
can be simplified by observing a formal matrix description of an element of $\underline{M}(R)$ for a $\kappa$-algebra $R$,
explained in Section \ref{m}.
Since $A(2b(2b-1), a(a+1), a(2b-1))\oplus (a+2b)$ is \textit{of type $I^o$} and $(x_j)_1=0$ (by Equation (\ref{e42})),
the $(2,2)$-block of the above formal matrix turns to be
$$\begin{pmatrix} 1&0 &\frac{-2}{a+2b}z_j^{\ast}\\
0&1 &\frac{-2}{a+2b}z_j^{\ast} \\
 \frac{2}{a+2b}z_j^{\ast}& \frac{-a\pi x_j}{a+2b} &1+\frac{2}{a+2b}z_j^{\ast} \end{pmatrix}.$$

 We now follow Step (3) with $L^j$ such that the above formal matrix corresponds to the formal matrix (\ref{e4.4}).
If we switch the order of the  first two vectors in the basis of  $A(2b(2b-1), a(a+1), a(2b-1))$, then 
 the only difference between the above formal matrix and the formal matrix (\ref{e4.4}) of Step (3) is the appearance of
  $\frac{-2}{a+2b}z_j^{\ast}$  in the $(2, 3)$-entry 
and $\frac{-a\pi x_j}{a+2b}$ in the $(3, 2)$-entry of the above formal matrix.
However, these entries will be zero after reduction to the orthogonal group associated to $M_0''$,
where $M_0''$ is the $\pi^0$-modular Jordan component of $Y(C(M_0'\oplus C(L^j)))$,
since $M_0''$ is \textit{free of type II} so that the  diagonal block associated to $M_0''$ has no congruence condition.
Thus in the corresponding orthogonal group, all entries having $\pi$ as a factor become zero.

 Then by using the result of Step (3),
 $(z_j^{\ast})_1$ is the image of a fixed element of $F_j$ under the map $\psi_j$.
Since $(z_j^{\ast})_1$ can be either $0$ or $1$ by Equation (\ref{e42}),
 $\psi_j|_{F_j}$ is surjective onto $\mathbb{Z}/2\mathbb{Z}$ and thus $\psi_j$ is surjective.\\

\end{enumerate}

So far, we showed that $\psi_j$ is surjective when $j$ is even.
We now show that $\psi_j$ is surjective when $j$ is odd.
Recall that  when $j$ is odd, $L_j$ is \textit{free of type I} and $L_{j-1}, L_{j+1}, L_{j+2}, L_{j+3}$  are \textit{of type II}.
Recall that $\bigoplus_{i \geq 0} M_i$ is a Jordan splitting of  a rescaled hermitian lattice $(L^{j-1}, \xi^{-(j-1)/2}h)$
and that $M_1=\pi^{(j-1)/2}L_1\oplus\pi^{(j-1)/2-1}L_3\oplus \cdots \oplus \pi L_{j-2}\oplus L_{j}$.
We can also let $L_j=\left(\oplus H(1)\right)\oplus A(4b, 2\delta, \pi)\textit{  with $b\in A$}$ by Theorem \ref{210} so that
$n(L_j)=(2)$.

We write $M_1=L_j\oplus N_1$, where $N_1$ is $\pi^1$-modular so that $n(N_1)=(2)$  or $n(N_1)=(4)$.
If  $n(N_1)=(4)$, then the proof of the surjectivity of $\psi_j$ is similar  to
and simpler than   that of the  case $n(N_1)=(2)$.
Thus  we assume that $n(N_1)=(2)$.
Then by Theorem \ref{210}, $N_1=\left(\oplus H(1)\right)\oplus A(4b', 2\delta, \pi)$  with $b'\in A$.
Let $(e_1, e_2, e_3, e_4)$ be a basis for $A(4b, 2\delta, \pi)\oplus A(4b', 2\delta, \pi)$.
Then the  diagonal block (associated to $(e_1, e_2, e_3, e_4)$) of
the image of a fixed element of $F_j$ in the special fiber of the smooth integral model associated to $L^{j-1}$ is
$\begin{pmatrix}1+\pi x_j&0&0&0\\\pi z_j&1&0&0\\0&0&1&0\\0&0&0&1 \end{pmatrix}$.
We choose another basis $$(e_1-e_3, \pi e_3+e_4, e_2+e_4, e_1-2b\pi/\delta\cdot e_2), 
\textit{ denoted by $(e_1', e_2', e_3', e_4')$},$$
for the lattice $A(4b, 2\delta, \pi)\oplus A(4b', 2\delta, \pi)$
so that the associated Gram matrix is
$$A(4(b+b'), -2\delta(1+4b'), \pi(1+4b'))\oplus A(4\delta, -4b(1+4b), \pi(1+4b)).$$
Here, the former lattice $A(4(b+b'), -2\delta(1+4b'), \pi(1+4b'))$ is $\pi^1$-modular with the norm $(2)$ and
the latter lattice $A(4\delta, -4b(1+4b), \pi(1+4b))$ is $\pi^1$-modular with the norm $(4)$ (so that it is isometric to $H(1)$ by Theorem \ref{210}).
Then by observing a formal matrix description of an element of $\underline{M}(R)$ for a $\kappa$-algebra $R$,
explained in Section \ref{m}, the above formal matrix turns to be
$$\begin{pmatrix}1-2z_j&0&0&0\\\pi z_j&1&0&\pi z_j\\\pi z_j&0&1&0\\\pi x_j+2 z_j&0&0&1 \end{pmatrix}.$$

We now follow the argument of the first case dealing with $j-1$ even.
Namely, for a Jordan splitting $Y(C(L^{j-1}))=\oplus_{i\geq 0}M_i''$,
 we describe the image of a fixed element of $F_j$ in the orthogonal group associated to $M_0''$.
There are three cases depending on the types of $M_0$.\\

\begin{enumerate}
\item Assume that $M_0$ is \textit{of type II}.
In this case,
$$M_0''=Be_1'\oplus (\pi)e_2'\oplus M_2\oplus \pi M_0.$$
Since $M_0''$ is \textit{free of type II},  the image of a fixed element of $F_j$ in the orthogonal group associated to $M_0''$ is
$$\begin{pmatrix}1&0&0\\ (z_j)_1&1&0\\0&0&id\\ \end{pmatrix}.$$
Here, $z_j=(z_j)_1+\pi \cdot (z_j)_2$ and $id$ is associated to $M_2\oplus \pi M_0$.
 Then the Dickson invariant of the above matrix is $(z_j)_1$.

 In conclusion, $(z_j)_1$ is the image of a fixed element of $F_j$ under the map $\psi_j$.
Since $(z_j)_1$ can be either $0$ or $1$ by Equation (\ref{e42}),
 $\psi_j|_{F_j}$ is surjective onto $\mathbb{Z}/2\mathbb{Z}$ and thus $\psi_j$ is surjective.\\

\item Assume that $M_0$ is \textit{of type $I^e$}.
Let $M_0=\left(\oplus H(0)\right)\oplus A(1, 2a, 1)$ and let $(e_5, e_6)$ be a basis for $A(1, 2a, 1)$.
We consider the lattice spanned by $(e_5, e_6, e_1', e_2')$ with the Gram matrix $A(1, 2a, 1)\oplus A(4(b+b'), -2\delta(1+4b'), \pi(1+4b'))$.
    Then by Theorem \ref{210},
    there is a suitable basis for this lattice such that the norm of the $\pi^1$-modular Jordan component is the ideal $(4)$.
    Namely, we choose
    $$(e_5-e_1', e_6-e_1', e_1'-\frac{2\pi(b+b') }{\delta(1+4b')}e_2', \pi e_5+\frac{1}{1+4b'}e_2').$$
    Here, a method to find the above basis follows from the argument used in Case (iii) of Case (1) with $j$ even.
Then the lattice spanned by the latter two vectors is $\pi^1$-modular with the norm $(4)$
(so that it is isometric to $H(1)$ by Theorem \ref{210}) and the lattice spanned by the former two vectors is
$A(1+4(b+b'), 2a+4(b+b'), 1+4(b+b'))$, which is $\pi^0$-modular and \textit{of type $I^e$}.
Let
\[\left\{
\begin{array}{l}
\tilde{M}_0=\left(\oplus H(0)\right)\oplus \left( B(e_5-e_1')\oplus B(e_6-e_1') \right);\\
\tilde{M}_1=\left(\oplus H(1)\right)\oplus \left( Be_3'\oplus Be_4'  \right) 
 \oplus \left( B(e_1'-\frac{2\pi(b+b') }{\delta(1+4b')}e_2')\oplus B(\pi e_5+\frac{1}{1+4b'}e_2')  \right).
\end{array}\right.
\]
Then
$\tilde{M}_0\oplus\tilde{M}_1\oplus(\oplus_{i\geq 2}M_i)$ is another Jordan splitting of $L^{j-1}$,
where $\tilde{M}_0$  is $\pi^0$-modular and \textit{of type $I^e$} and
 $\tilde{M}_1$ is isometric to $\oplus H(1)$.

For $\tilde{M}_0\oplus M_2$, the associated diagonal  block of the image of a fixed element of $F_j$
in the special fiber of the smooth integral model associated to $L^{j-1}$ is

$$\begin{pmatrix}id&0&0 \\ 0&\begin{pmatrix}1+2z_j&2z_j\\ 0&1 \end{pmatrix}&0
\\ 0&0&id   \end{pmatrix}.$$
Here, the $(2,2)$-block corresponds to $\left( B(e_5-e_1')\oplus B(e_6-e_1') \right)$.


We now  follow the argument used in Steps (i) and (ii) of Step (1) in even case.
Namely, the lattice $M_0''$ can be constructed by using $\tilde{M}_0$ and $M_2$.
Then we can easily check that the Dickson invariant of the image of a fixed element of $F_j$ in the orthogonal group associated to $M_0''$
is  $(z_j)_1$.

 In conclusion, $(z_j)_1$ is the image of a fixed element of $F_j$ under the map $\psi_j$.
Since $(z_j)_1$ can be either $0$ or $1$ by Equation (\ref{e42}),
 $\psi_j|_{F_j}$ is surjective onto $\mathbb{Z}/2\mathbb{Z}$ and thus $\psi_j$ is surjective.\\

 \item Assume that $M_0$ is \textit{of type $I^o$}.
Let $M_0=\left(\oplus H(0)\right)\oplus (a)$ with $a\equiv 1$ mod $2$  and let $(e_5)$ be a basis for $(a)$.
We consider the lattice spanned by $(e_5, e_1', e_2')$ with the Gram matrix $(a)\oplus A(4(b+b'), -2\delta(1+4b'), \pi(1+4b'))$.
    Then by Theorem \ref{210},
    there is a suitable basis  for this lattice such that the norm of the $\pi^1$-modular Jordan component is the ideal $(4)$.
    Namely, we choose
    $$(e_5-e_1', e_1'-\frac{2\pi(b+b') }{\delta(1+4b')}e_2', \pi e_5+\frac{a}{1+4b'}e_2').$$
     Here, a method to find the above basis follows from the argument used in Case (iii) of Case (1) with $j$ even.
Then the lattice spanned by the latter two vectors is $\pi^1$-modular with the norm $(4)$
(so that it is isometric to $H(1)$ by Theorem \ref{210}) and the lattice spanned by the first vector is $(a+4(b+b'))$.
Let
\[\left\{
\begin{array}{l}
\tilde{M}_0=\left(\oplus H(0)\right)\oplus B(e_5-e_1');\\
\tilde{M}_1=\left(\oplus H(1)\right)\oplus \left( Be_3'\oplus Be_4'  \right) 
  \oplus \left( B(e_1'-\frac{2\pi(b+b') }{\delta(1+4b')}e_2')\oplus B(\pi e_5+\frac{a}{1+4b'}e_2')  \right).
\end{array}\right.
\]
Then
$\tilde{M}_0\oplus\tilde{M}_1\oplus(\oplus_{i\geq 2}M_i)$ is another Jordan splitting of $L^{j-1}$,
where $\tilde{M}_0$  is $\pi^0$-modular and \textit{of type $I^o$} and
 $\tilde{M}_1$ is isometric to $\oplus H(1)$.

For $\tilde{M}_0\oplus M_2$, the associated diagonal  block of the image of a fixed element of $F_j$
in the special fiber of the smooth integral model associated to $L^{j-1}$ is
$$\begin{pmatrix}id&0&0 \\ 0&1+2z_j&0
\\ 0&0&id   \end{pmatrix}.$$
Here, the $(2,2)$-block corresponds to $B(e_5-e_1')$.

We now  follow the argument used in  Step (3) in even case.
Then we can easily check that the Dickson invariant of the image of a fixed element of $F_j$ in the orthogonal group associated to $M_0''$
is  $(z_j)_1$.

 In conclusion, $(z_j)_1$ is the image of a fixed element of $F_j$ under the map $\psi_j$.
Since $(z_j)_1$ can be either $0$ or $1$ by Equation (\ref{e42}),
 $\psi_j|_{F_j}$ is surjective onto $\mathbb{Z}/2\mathbb{Z}$ and thus $\psi_j$ is surjective.\\

\end{enumerate}

So far, we have proved that $\psi_j$ is surjective.
 Let $\mathcal{B}$ be the set of  integers $j$ such that $L_j$ is \textit{of type I} and $L_{j+2}, L_{j+3}, L_{j+4}$
 (resp. $L_{j-1}, L_{j+1},$ $L_{j+2}, L_{j+3}$) are \textit{of type II} if $j$ is even (resp. odd).
 Choose   $j, j' \in \mathcal{B}$ with $j<j'$.
 By using the fact that $j'-j\geq 5$ if $j$ is even and $j'-j\geq 4$ if $j$ is odd,
 the proof of the surjectivity of the morphism $\psi=\prod_{j\in \mathcal{B}}\psi_j$
 is similar to that of Theorem 4.11 in \cite{C2} (cf. from 7th line of page 497 to the first paragraph of page 498 in \cite{C2}).
 The proof of the surjectivity of  $\varphi \times \psi$ is similar to that of Theorem 4.11 in \cite{C2}
explained in  the second paragraph of page 498.
Thus we skip them.
\end{proof}

     \subsection{The maximal reductive quotient}\label{mred}

     We finally have the structure theorem for the algebraic group $\tilde{G}$.

     \begin{Thm}\label{t412}
     The morphism $$\varphi \times \psi : \tilde{G} \longrightarrow \prod_{i:even} \mathrm{O}(B_i/Z_i, \bar{q}_i)_{\mathrm{red}} \times \prod_{i:odd} \mathrm{Sp}(B_i/Y_i, h_i)\times (\mathbb{Z}/2\mathbb{Z})^{\beta}$$
      is surjective and the kernel is unipotent and connected.
     Consequently, $$ \prod_{i:even} \mathrm{O}(B_i/Z_i, \bar{q}_i)_{\mathrm{red}} \times \prod_{i:odd} \mathrm{Sp}(B_i/Y_i, h_i)\times (\mathbb{Z}/2\mathbb{Z})^{\beta}$$ is the maximal reductive quotient.
 Here, $\mathrm{O}(B_i/Z_i, \bar{q}_i)_{\mathrm{red}}$ and $\mathrm{Sp}(B_i/Y_i, h_i)$ are explained in Section \ref{red} (especially Remark \ref{r47})
 and $\beta$ is defined in Lemma \ref{l46}.
     \end{Thm}

     \begin{proof}
The proof is similar to that of Theorem 4.12 in \cite{C2} and so we skip it.
 \end{proof}

\section{Comparison of volume forms and final formulas}\label{cv}

This section is based on Section 5 of \cite{C2}.
We refer to loc. cit. and Section 3.2 of \cite{GY} for a detailed explanation.
Let $H$ be the $F$-vector space of hermitian forms on $V=L\otimes_AF$.
Let $M'=\mathrm{End}_{B}(L)$  
and let $H'=\{\textit{f : f is a hermitian form on $L$}\}$. 
Regarding  $\mathrm{End}_EV$ and $H$ as varieties over $F$,
let $\omega_M$ and $\omega_H$ be nonzero,  translation-invariant forms on   $\mathrm{End}_EV$ and $H$, respectively, 
with normalization
$$\int_{M'}|\omega_M|=1 \mathrm{~and~}  \int_{H'}|\omega_H|=1.$$

We choose another nonzero,  translation-invariant forms $\omega^{\prime}_M$ and $\omega^{\prime}_H$ on  $\mathrm{End}_EV$ and $H$, respectively,
with normalization
$$\int_{\underline{M}(A)}|\omega^{\prime}_M|=1 \mathrm{~and~}  \int_{\underline{H}(A)}|\omega^{\prime}_H|=1 $$
(cf. the second paragraph of page 499 in \cite{C2}).
By Theorem \ref{t36}, we have an exact sequence of locally free sheaves on $\underline{M}^{\ast}$:
\[ 0\longrightarrow \rho^{\ast}\Omega_{\underline{H}/A} \longrightarrow \Omega_{\underline{M}^{\ast}/A}
\longrightarrow \Omega_{\underline{M}^{\ast}/\underline{H}} \longrightarrow 0. \]
Put $\omega^{\mathrm{can}}=\omega^{\prime}_M/\rho^{\ast}\omega^{\prime}_H$.
For a detailed explanation of what $\omega_M'/\rho^{\ast}\omega_H'$ means, we refer to Section 3.2 of \cite{GY}.
It follows that
$\omega^{\mathrm{can}}$ is a differential of top degree on $\underline{G}$, which is  invariant under the generic fiber of $\underline{G}$,
and which has nonzero reduction on the special fiber.

Recall that $2$ is a uniformizer of $A$.
\begin{Lem}\label{l51}
We have:
  $$|\omega_M|=|2|^{N_M}|\omega_M^{\prime}|, \ \ \ \  N_M=\sum_{\textit{$L_i$:type I}}2n_i+\sum_{i<j}(j-i)\cdot n_i\cdot n_j-a,$$
  $$|\omega_H|=|2|^{N_H}|\omega_H^{\prime}|, \ \ \ \ N_H=\sum_{L_i:\textit{type I}}n_i+\sum_{i<j}j\cdot n_i\cdot n_j+\sum_{\textit{i:even}} \frac{i+2}{2} \cdot n_i
  +\sum_{\textit{i:odd}} \frac{i+3}{2} \cdot n_i+\sum_i d_i-a,$$
  $$|\omega^{\mathrm{ld}}|=|2|^{N_M-N_H}|\omega^{\mathrm{can}}|.$$
  Here,
  \begin{itemize}
  \item $a$ is the total number of  $L_i$'s such that $i$ is odd and $L_i$ is \textit{free of type I}.
  \item $d_i=i\cdot n_i\cdot (n_i-1)/2$.
  \end{itemize}
\end{Lem}
\begin{proof}
The proof is similar to that of Lemma 5.1 in \cite{C2} and so we skip it.
\end{proof}

Let $f$ be the cardinality of $\kappa$.
The local density is defined as
\[\beta_L= \frac{1}{[G:G^{\circ}]}\cdot \lim_{N\rightarrow \infty} f^{-N~dim G}\#\underline{G}'(A/\pi^N A). \]
Here, $\underline{G}'$ is the naive integral model described at the beginning of Section \ref{csm} and $G$ is the generic fiber of $\underline{G}'$
and $G^{\circ}$ is the identity component of $G$.
In our case, $G$ is the unitary group $\mathrm{U}(V, h)$, where $V=L\otimes_AF$.
Since $\mathrm{U}(V, h)$ is connected, $G^{\circ}$ is the same as $G$ so that $[G:G^{\circ}]=1$.

Then based on Lemma 3.4 and Section 3.9 of \cite{GY}, we finally have the following local density formula.

\begin{Thm}\label{t52}
Let $f$ be the cardinality of $\kappa$.
The local density of ($L,h$) is

$$\beta_L=f^N \cdot f^{-\mathrm{dim~} \mathrm{U}(V, h)} \#\tilde{G}(\kappa),$$

where $$N=N_H-N_M=\sum_{i<j}i\cdot n_i\cdot n_j+\sum_{\textit{i:even}} \frac{i+2}{2} \cdot n_i
  +\sum_{\textit{i:odd}} \frac{i+3}{2} \cdot n_i+\sum_i d_i-\sum_{L_i:\textit{type I}}n_i.$$
Here, $\#\tilde{G}(\kappa)$ can be computed explicitly based on Remark 5.3.(1) of \cite{C2} and Theorem \ref{t412}.
\end{Thm}

\begin{Rmk}\label{r53}

As  in Remark 7.4 of \cite{GY}, although we have assumed that $n_i=0$ for $i<0$,
it is easy to check that the formula in the preceding theorem remains true without this assumption.
\end{Rmk}

 \appendix
 \section{The proof of  Lemma \ref{l46}} \label{App:AppendixA}

The proof of Lemma 4.5 is based on  Proposition 6.3.1 in \cite{GY} and Appendix A in \cite{C2}.\\ 
We first state a theorem of Lazard which is repeatedly used in this paper.
Let $U$ be a group scheme of finite type over $\kappa$ which is isomorphic to an affine space as an algebraic variety.
Then $U$ is connected smooth unipotent group
(cf.  IV, $\S$ 4, Theorem 4.1 and IV, $\S$ 2, Corollary 3.9 in \cite{DG}).

For preparation, we  state several lemmas.

\begin{Lem}\label{la1} (Lemma 6.3.3. in \cite{GY})

Let $1 \rightarrow X\rightarrow Y\rightarrow Z\rightarrow 1$ be an exact sequence of group schemes that are locally of finite type over $\kappa$,
where  $\kappa$ is a perfect field. Suppose that $X$ is smooth, connected, and unipotent.
Then $1 \rightarrow X(R)\rightarrow Y(R)\rightarrow Z(R)\rightarrow 1$ is exact for any $\kappa$-algebra $R$. 
\end{Lem}

Let $\tilde{M}$ be the special fiber of $\underline{M}^{\ast}$ and let $R$ be a $\kappa$-algebra.
Recall that we have described an element   and
the multiplication of elements of $\underline{M}(R)$  in Section  \ref{m}.
Based on these,
an element of $\tilde{M}(R)$ is
$$m= \begin{pmatrix} \pi^{max\{0,j-i\}}m_{i,j} \end{pmatrix} \mathrm{~with~}z_i^{\ast}, m_{i,i}^{\ast}, m_{i,i}^{\ast\ast}$$
satisfying the following:
\begin{itemize}
\item[(a)] If $i$ is even and $L_i$ is \textit{of type} $\textit{I}^o$ (resp. \textit{of type} $\textit{I}^e$), then
 $$m_{i,i}=\begin{pmatrix} s_i&\pi y_i\\ \pi v_i&1+\pi z_i \end{pmatrix} \textit{(resp.
$\begin{pmatrix} s_i&r_i&\pi t_i\\ \pi y_i&1+\pi x_i&\pi z_i\\ v_i&u_i&1+\pi w_i \end{pmatrix}$)},$$
where
 $s_i\in M_{(n_i-1)\times (n_i-1)}(B\otimes_AR)$ (resp.  $s_i\in M_{(n_i-2)\times (n_i-2)}(B\otimes_AR)$), etc.,
 and $s_i$ mod $\pi\otimes 1$ is invertible.

\item[(b)] If $i$ is odd and $L_i$ is \textit{free of type I}, then
 \[m_{i,i}=\begin{pmatrix} s_i&\pi r_i&t_i\\  y_i&1+\pi x_i& u_i\\\pi  v_i&\pi z_i&1+\pi w_i \end{pmatrix},\]
  where $s_i\in M_{(n_i-2)\times (n_i-2)}(B\otimes_AR)$, etc., and $s_i$ mod $\pi\otimes 1$ is invertible.
\item[(c)] For the remaining $m_{i,j}$'s except for the cases explained above, $m_{i,j}\in M_{n_i\times n_j}(B\otimes_AR)$
and $m_{i,i}$  mod $\pi\otimes 1$ is invertible.
\item[(d)]
Assume that $i$ is even and that $L_i$ is \textit{of type I}.  Then
$$z_i+\delta_{i-2}k_{i-2, i}+\delta_{i+2}k_{i+2, i}=\pi z_i^{\ast}$$
such that  $z_i^{\ast}\in B\otimes_AR$.
This equation is considered in $B\otimes_AR$ and $\pi$ stands for $\pi\otimes 1\in B\otimes_AR$.
 Here,
    \begin{itemize}
  \item[(i)] $z_i$ is an entry of $m_{i,i}$ as described in the above step (a).
\item[(ii)] $k_{i-2, i}$ (resp. $k_{i+2, i}$) is the $(n_{i-2}, n_i)^{th}$-entry (resp. $(n_{i+2}, n_i)^{th}$-entry) of the matrix $m_{i-2, i}$
(resp. $m_{i+2, i}$)
if $L_{i-2}$ (resp. $L_{i+2}$) is \textit{of type} $\textit{I}^o$.
\item[(iii)] $k_{i-2, i}$ (resp. $k_{i+2, i}$) is the $(n_{i-2}-1, n_i)^{th}$-entry (resp. $(n_{i+2}-1, n_i)^{th}$-entry) of the matrix $m_{i-2, i}$
(resp. $m_{i+2, i}$)
if $L_{i-2}$ (resp. $L_{i+2}$) is \textit{of type} $\textit{I}^e$.
\end{itemize}

\item[(e)] Assume that $i$ is odd and that $L_i$ is \textit{bound of type I}. Then
$$\delta_{i-1}v_{i-1}\cdot m_{i-1, i}+\delta_{i+1}v_{i+1}\cdot m_{i+1, i}=\pi m_{i,i}^{\ast}$$
such that  $m_{i,i}^{\ast} \in M_{1\times n_i}(B\otimes_AR)$.
This equation is considered in $B\otimes_AR$ and $\pi$ stands for $\pi\otimes 1\in B\otimes_AR$.
Here,
    \begin{itemize}
   \item[(i)] $v_{i-1}=(0,\cdots, 0, 1)$ (resp. $v_{i-1}=(0,\cdots, 0, 1, 0)$) of size $1\times n_{i-1}$
if $L_{i-1}$ is \textit{of type} $\textit{I}^o$ (resp. \textit{of type} $\textit{I}^e$).
   \item[(ii)] $v_{i+1}=(0,\cdots, 0, 1)$ (resp. $v_{i+1}=(0,\cdots, 0, 1, 0)$) of size $1\times n_{i+1}$
if $L_{i+1}$ is \textit{of type} $\textit{I}^o$ (resp. \textit{of type} $\textit{I}^e$).
 \end{itemize}

 \item[(f)] Assume that $i$ is odd and that $L_i$ is \textit{bound of type I}.
  Then
 $$ \delta_{i-1}v_{i-1}\cdot {}^tm_{i, i-1}+\delta_{i+1}v_{i+1}\cdot {}^tm_{i, i+1} = \pi m_{i,i}^{\ast\ast}$$
 such that $ m_{i,i}^{\ast\ast} \in M_{1\times n_i}(B\otimes_AR)$.
This equation is considered in $B\otimes_AR$ and $\pi$ stands for $\pi\otimes 1\in B\otimes_AR$.
Here,
 $v_{i-1}$ (resp. $v_{i+1}$)$=(0,\cdots, 0, 1)$  of size $1\times n_{i-1}$ (resp. $1\times n_{i+1}$).\\

\end{itemize}

Let 
  \[\tilde{M_i}= \mathrm{GL}_{B/\pi B}(B_i/Y_i) \textit{ for all $i$}.\]

Let $s_i=m_{i,i}$ if $L_i$ is \textit{of type II} or if $L_i$ is \textit{bound of type I} with $i$ odd in the above description of an element of $\tilde{M}(R)$.
Then $s_i$ mod $\pi\otimes 1$ is an element of $\tilde{M}_i(R)$.
 Therefore, we have a surjective morphism of algebraic groups
 $$r : \tilde{M} \longrightarrow \prod\tilde{M}_i, ~~~~~~~ m \mapsto \prod \left(\textit{$s_i$ mod $\pi\otimes 1$}\right)$$
 defined over $\kappa$. We now have the following easy lemma:
 \begin{Lem}\label{la2}
The kernel of $r$ is the unipotent radical $\tilde{M}^+$ of $\tilde{M}$, and $\prod\tilde{M}_i$ is the maximal reductive quotient of $\tilde{M}$.\\
\end{Lem}
\begin{proof}
Since $\prod\tilde{M}_i$ is a reductive group, we only have to show that the kernel of $r$ is a connected smooth unipotent group.
By the description of the morphism $r$ in terms of matrices explained above,
 the kernel of $r$ is isomorphic to an affine space  as an algebraic variety over $\kappa$.
Therefore, it is a connected smooth unipotent group by a theorem of Lazard which is stated at the beginning of Appendix \ref{App:AppendixA}.
\end{proof}

Recall that we have defined the morphism $\varphi$ in Section \ref{red}.
 The  morphism $\varphi$ extends to     an obvious morphism
 $$\tilde{\varphi} : \tilde{M} \longrightarrow  \prod_{i:even}\mathrm{GL}_{\kappa}(B_i/Z_i) \times  \prod_{i:odd}\mathrm{GL}_{\kappa}(B_i/Y_i) $$
 such that $\tilde{\varphi}|_{\tilde{G}}=\varphi $.
 Note that $Y_i\otimes_AR$, when $i$ is odd, is preserved by an element of $\underline{M}(R)$ for a flat $A$-algebra $R$ (cf. Lemma \ref{l42}).
 By using this,
the construction of $\tilde{\varphi}$ is similar to Theorems \ref{t43} and \ref{t44}  and thus we skip it.
 Let $R$ be a $\kappa$-algebra.
Based on the description of the morphism $\varphi_i$ explained in Section \ref{red},
$\mathrm{Ker~}\tilde{\varphi}(R)$ is the subgroup of  $\tilde{M}(R)$ defined by the following conditions:

\begin{enumerate}
\item[(a)]  If $i$ is even and $L_i$ is \textit{of type I},  $s_i=\mathrm{id}$ mod $\pi \otimes 1$.
\item[(b)]  If $i$ is even and $L_i$ is \textit{of type II}, $m_{i,i}=\mathrm{id}$ mod $\pi \otimes 1$.
\item[(c)] Let $i$ be even and $L_i$ be \textit{bound of type II}. Then $\delta_{i-1}^{\prime}e_{i-1}\cdot m_{i-1, i}+\delta_{i+1}^{\prime}e_{i+1}\cdot m_{i+1, i}+\delta_{i-2}e_{i-2}\cdot m_{i-2, i}+\delta_{i+2}e_{i+2}\cdot m_{i+2, i}=0$
    mod $\pi \otimes 1$. \\
Let $i$ be even and $L_i$ be \textit{of type I}. Then
$v_i(\mathrm{resp.~}(y_i+\sqrt{\bar{\gamma}_i}v_i))+(\delta_{i-2}e_{i-2}\cdot m_{i-2, i}+\delta_{i+2}e_{i+2}\cdot m_{i+2, i})\tilde{e_i}=0$
mod $\pi \otimes 1$
 if $L_i$ is \textit{of type} $\textit{I}^o$ (resp. \textit{of type} $\textit{I}^e$).

Here,
\begin{itemize}
\item[(i)] $
\delta_{j}^{\prime} = \left\{
  \begin{array}{l l}
  1    & \quad  \textit{if $j$ is odd and $L_j$ is \textit{free of type I}};\\
  0    &   \quad  \textit{otherwise}.
    \end{array} \right.
$
\item[(ii)] If $j$ is odd, then  $e_{j}=(0,\cdots, 0, 1)$  of size $1\times n_{j}$.
\item[(iii)] When $j$ is even, $e_{j}=(0,\cdots, 0, 1)$ (resp. $e_j=(0,\cdots, 0, 1, 0)$) of size $1\times n_{j}$
if $L_{j}$ is \textit{of type} $\textit{I}^o$ (resp. \textit{of type} $\textit{I}^e$).
\item[(iv)] $v_i$ and  $y_i$ are blocks of $m_{i, i}$ as explained in the description of an element of $\tilde{M}(R)$ above,
and $\bar{\gamma}_i$ is as explained in Remark \ref{r33}.(2).
\item[(v)] $\tilde{e_i}=\begin{pmatrix} \mathrm{id}\\0 \end{pmatrix}$ of size $n_i\times (n_{i}-1)$ (resp. $n_i\times (n_{i}-2)$), where $\mathrm{id}$ is the identity matrix of size $(n_i-1)\times (n_{i}-1)$ (resp. $(n_i-2)\times (n_{i}-2)$) if $L_{i}$ is \textit{of type} $\textit{I}^o$ (resp. \textit{of type} $\textit{I}^e$).
 \end{itemize}

\item[(d)]  If $i$ is odd and $L_i$ is \textit{of type II} or \textit{bound of type I}, $m_{i,i}=\mathrm{id}$  mod $\pi \otimes 1$.
\item[(e)] If $i$ is odd and $L_i$ is \textit{free of type I}, $s_i=\mathrm{id}$  mod $\pi \otimes 1$.\\
\end{enumerate}
It is obvious that $\mathrm{Ker~}\tilde{\varphi}$ is a closed subgroup scheme of the unipotent radical $\tilde{M}^+$ of $\tilde{M}$ and  is smooth and unipotent
 since it is isomorphic to an affine space as an algebraic variety over $\kappa$.

For completeness of the content in this section, we repeat the argument written in from the last paragraph of page 502 to page 503 in \cite{C2}.

Recall from Remark \ref{r31} that we  defined the functor $\underline{M}^{\prime}$ such that $(1+\underline{M}^{\prime})(R)=\underline{M}(R)$
inside $\mathrm{End}_{B\otimes_AR}(L \otimes_A R)$ for a flat $A$-algebra $R$.
Thus there is an isomorphism of set valued functors
$$1+ :  \underline{M}^{\prime} \longrightarrow \underline{M}, ~~~ m\mapsto 1+m,$$
where $m\in \underline{M}^{\prime}(R)$ for a flat $A$-algebra $R$.
We define a new operation $\star$ on $\underline{M}^{\prime}(R)$ such that $x\star y=x+y+xy$ for a flat $A$-algebra $R$.
Since $\underline{M}^{\prime}(R)$ is closed under addition and multiplication, it is also closed under the new operation $\star$.
Moreover, it has  $0$ as an identity element  with respect to $\star$.
Thus $\underline{M}^{\prime}$ may and shall be considered as a scheme of monoids with $\star$.
 We claim that the above morphism $1+$ is an isomorphism of monoid schemes.
 Namely, we claim  the following commutative diagram of schemes:
 \[\xymatrixcolsep{5pc}\xymatrix{
\underline{M}^{\prime}\times \underline{M}^{\prime} \ar[d]^{\star} \ar[r]^{(1+)\times (1+)} &\underline{M}\times
\underline{M}\ar[d]^{multiplication}\\
\underline{M}^{\prime} \ar[r]^{1+} &\underline{M}}\]
Since all schemes are irreducible and smooth, it suffices to check the commutativity of the diagram at the level of flat $A$-points
as explained in the third paragraph from below in Remark \ref{r32} of \cite{C2}, and this is obvious.

Since $\underline{M}^{\ast}$ is an open subscheme of $\underline{M}$,  $(1+)^{-1}(\underline{M}^{\ast})$ is an open subscheme of $\underline{M}^{\prime}$.
The composite of the following three morphisms
 \[\xymatrixcolsep{5pc}\xymatrix{
(1+)^{-1}(\underline{M}^{\ast})  \ar[r]^{(1+)} &\underline{M}^{\ast} \ar[r]^{inverse} &\underline{M}^{\ast}\ar[r]^{(1+)^{-1}}&
(1+)^{-1}(\underline{M}^{\ast})}\]
defines the inverse morphism on the scheme of monoids $(1+)^{-1}(\underline{M}^{\ast})$ with respect to the operation $\star$.
Thus we can see that $(1+)^{-1}(\underline{M}^{\ast})$ is a  group scheme with respect to $\star$ and
 the morphism $1+$ is   an isomorphism of group schemes between $(1+)^{-1}(\underline{M}^{\ast})$ and $\underline{M}^{\ast}$.

Let $R$ be a $\kappa$-algebra.
Since the morphism $1+$ is an isomorphism of monoid schemes between $\underline{M}^{\prime}$ and $\underline{M}$,
we can write each element of $\underline{M}(R)$ as $1+x$ with $x \in \underline{M}^{\prime}(R)$.
Here,  $1+x$ means the image of $x$ under the morphism $1+$ at the level of $R$-points.
Note that  $\underline{M}^{\prime}(R)$ is a $B\otimes_AR$-algebra for any $A$-algebra $R$ with respect to the original multiplication on it,
 not  the operation $\star$.
In particular,  $\underline{M}^{\prime}(R)$ is a $(B/2B)\otimes_AR$-algebra for any $\kappa$-algebra $R$.
Therefore, we consider the following two functors:
$$
\left\{
  \begin{array}{l}
\textit{the subfunctor $\underline{\pi M^{\prime}} : R \mapsto (\pi\otimes 1) \underline{M}^{\prime}(R)$ of $\underline{M}^{\prime}\otimes\kappa$};\\
\textit{the subfunctor $\tilde{M}^1:R\mapsto 1+\underline{\pi M^{\prime}}(R)$ of $\mathrm{Ker~}\tilde{\varphi}$}.
    \end{array} \right.
$$
Here, by $1+\underline{\pi M^{\prime}}(R)$, we mean the image of $\underline{\pi M^{\prime}}(R)$
inside $\underline{M}(R) (=\tilde{M}(R))$ under the morphism $1+$ at the level of $R$-points.
That   $1+\underline{\pi M^{\prime}}(R)$ is contained in $\mathrm{Ker~}\tilde{\varphi}(R)$ can easily be checked
 by observing the construction of $\tilde{\varphi}$.
The multiplication on $\tilde{M}^1$  is as follows:
for  two elements $1+\pi x$ and $1+\pi y$ in $\tilde{M}^1(R)$,
based on the above commutative diagram, the product of $1+\pi x$ and $1+\pi y$  is
$$(1+\pi x)\cdot(1+\pi y)=1+\pi x\star \pi y=1+(\pi (x+y)+\pi^2(xy))=1+\pi (x+y).$$
Here, $\pi$ stands for $\pi\otimes 1 \in B\otimes_AR$.
Then we have the following  lemma.
\begin{Lem}\label{la3}
(i) The functor $\tilde{M}^1$ is representable by a smooth, connected, unipotent group scheme over $\kappa$.
Moreover, $\tilde{M}^1$ is a closed normal subgroup of $\mathrm{Ker~}\tilde{\varphi} $.

(ii) The quotient group scheme $\mathrm{Ker~}\tilde{\varphi}/\tilde{M}^1$ represents the functor
$$R\mapsto \mathrm{Ker~}\tilde{\varphi}(R)/\tilde{M}^1(R)$$
by Lemma \ref{la1} and is smooth, connected, and unipotent.\\
\end{Lem}
\begin{proof}
The proof is the same as that of Lemma A.3 of \cite{C2} and so we skip it.
\end{proof}

 This paragraph is a reproduction of 6.3.6 in \cite{GY}.
Recall that there is a closed immersion $\tilde{G}\rightarrow \tilde{M}$.
Notice that $\mathrm{Ker~}\varphi$ is the kernel of the composition $\tilde{G}\rightarrow \tilde{M} \rightarrow \tilde{M}/ \mathrm{Ker~}\tilde{\varphi}$.
We define $\tilde{G}^1$ as the kernel of the composition
\[ \tilde{G}\rightarrow \tilde{M} \rightarrow \tilde{M}/ \tilde{M}^1.\]
 Then $\tilde{G}^1$ is the kernel of the morphism $\mathrm{Ker~}\varphi\rightarrow \mathrm{Ker~}\tilde{\varphi}/\tilde{M}^1$
 and, hence, is a closed normal subgroup of $\mathrm{Ker~}\varphi$.
 The induced morphism $\mathrm{Ker~}\varphi/\tilde{G}^1\rightarrow \mathrm{Ker~}\tilde{\varphi}/\tilde{M}^1$ is a monomorphism, and thus
 $\mathrm{Ker~}\varphi/\tilde{G}^1$ is a closed subgroup scheme of $ \mathrm{Ker~}\tilde{\varphi}/\tilde{M}^1$
 by (Exp. $\mathrm{VI_B},$ Corollary 1.4.2 in \cite{SGA3}).

 \begin{Thm}\label{ta4}
 $\tilde{G}^1$ is connected, smooth, and unipotent.
  Furthermore,   the underlying algebraic variety of $\tilde{G}^1$ over $\kappa$ is  an affine space of dimension
\[
\sum_{i<j}n_in_j+\sum_{i:\mathrm{even}}\frac{n_i^2+n_i}{2}+\sum_{i:\mathrm{odd}}\frac{n_i^2-n_i}{2}
+\#\{i:\textit{$i$ is odd and $L_i$ is free of type I}\} \]
\[-\#\{i:\textit{$i$ is even and $L_i$ is of type I}\}
+ \#\{i:\textit{$i$ is even, $L_i$ is  of type I and  $L_{i+2}$ is of type II}\}.\]
 \end{Thm}

 \begin{proof}
We prove this theorem by writing out a set of equations completely defining $\tilde{G}^1$
(after all there are so many different sets of equations defining $\tilde{G}^1$).
We first introduce the following trick.
Consider the polynomial ring $\kappa[x_1, \cdots, x_n]$ and its quotient ring $\kappa[x_1, \cdots, x_n]/(x_1+P(x_2, \cdots, x_n))$.
Then the quotient ring $\kappa[x_1, \cdots, x_n]/(x_1+P(x_2, \cdots, x_n))$
 is isomorphic to $\kappa[x_2, \cdots, x_n]$ and in this case we say that \textit{$x_1$ can be eliminated by $x_2, \cdots, x_n$}.

Let $R$ be a $\kappa$-algebra.
As explained in Remark \ref{r33}.(2), we consider the given hermitian form $h$ as an element of $\underline{H}(R)$
and write it as a formal matrix $h=\begin{pmatrix} \pi^{i}\cdot h_i\end{pmatrix}$ with $(\pi^{i}\cdot h_i)$
for the $(i,i)$-block and $0$ for the remaining blocks.
We also write $h$ as $(f_{i, j}, a_i\cdots f_i)$.
Recall that the notation $(f_{i, j}, a_i\cdots f_i)$ is defined and explained in Section \ref{h}
and explicit values of $(f_{i, j}, a_i\cdots f_i)$ for the  $h$ are given in Remark \ref{r33}.(2).

We choose an element $m=(m_{i,j}, s_i\cdots w_i)\in (\mathrm{Ker~}\tilde{\varphi})(R)$ with a formal matrix interpretation
$m= \begin{pmatrix} \pi^{max\{0,j-i\}}m_{i,j} \end{pmatrix}$,
where the notation  $(m_{i,j}, s_i\cdots w_i)$ is explained in Section \ref{m}.
Then $h\circ m$ is an element of $\underline{H}(R)$ and
$(\mathrm{Ker~}\varphi)(R)$ is the set of $m$ such that  $h\circ m=(f_{i, j}, a_i\cdots f_i)$.
The action $h\circ m$ is explicitly described in Remark \ref{r35}. Based on this,
we need to write the matrix product $h\circ m=\sigma({}^tm)\cdot h\cdot m$ formally.
To do that, we write each block of $\sigma({}^tm)\cdot h\cdot m$ as follows:

The diagonal $(i,i)$-block of  the formal matrix product $\sigma({}^tm)\cdot h\cdot m$ is the following:
\begin{multline}\label{ea1}
\pi^i\left(\sigma({}^tm_{i,i})h_im_{i,i}+\sigma(\pi)\cdot\sigma({}^tm_{i-1, i})h_{i-1}m_{i-1, i}+\pi\cdot\sigma({}^tm_{i+1, i})h_{i+1}m_{i+1, i}\right)+\\
\pi^i\left((\sigma\pi)^2\cdot\sigma({}^tm_{i-2, i})h_{i-2}m_{i-2, i}+
 \pi^2\cdot\sigma({}^tm_{i+2, i})h_{i+2}m_{i+2, i}+\pi^3(\ast)\right),
\end{multline}
where $0\leq i < N$ and $(\ast)$ is a certain formal polynomial.
\\
The $(i,j)$-block of the formal matrix product  $\sigma({}^tm)\cdot h\cdot m$, where $i<j$, is the following:
\begin{equation}\label{ea2}
\pi^j\left(\sum_{i\leq k \leq j} \sigma({}^tm_{k,i})h_km_{k,j}+\sigma(\pi)\cdot\sigma({}^tm_{i-1,i})h_{i-1}m_{i-1,j}+\pi\cdot\sigma({}^tm_{j+1,i})h_{j+1}m_{j+1,j}+\pi^2(\ast)\right),
\end{equation}
where $0\leq i, j < N$ and $(\ast)$ is a certain formal polynomial.
In the following computations, we always have in mind that $\sigma(\pi)=-\pi$.  as mentioned at the beginning of Section \ref{Notations}.

Before studying $\tilde{G}^1$,
we describe the conditions for an element $m\in \tilde{M}(R)$ as above to belong to the subgroup $\tilde{M}^1(R)$.
\begin{enumerate}
\item $m_{i,j}=\pi m_{i,j}^{\prime} \mathrm{~if~} i\neq j,$

\item If $i$ is even and $L_i$ is \textit{of type} $\textit{I}^o$, then
 $$m_{i,i}= \begin{pmatrix} s_i&\pi y_i\\ \pi v_i&1+\pi z_i \end{pmatrix}=\begin{pmatrix} \mathrm{id}+\pi s_i^{\prime}&\pi^2 y_i^{\prime}\\ \pi^2 v_i^{\prime}&1+\pi^2 z_i^{\prime} \end{pmatrix}.$$
\item If $i$ is even and $L_i$ is \textit{of type} $\textit{I}^e$, then 
$$ m_{i,i}=\begin{pmatrix} s_i&r_i&\pi t_i\\ \pi y_i&1+\pi x_i&\pi z_i\\ v_i&u_i&1+\pi w_i \end{pmatrix}=
 \begin{pmatrix} \mathrm{id}+\pi s_i^{\prime}&\pi r_i^{\prime}&\pi^2t_i^{\prime}\\ \pi^2y_i^{\prime}&1+\pi^2x_i^{\prime}&\pi^2z_i^{\prime}\\ \pi v_i^{\prime}&\pi u_i^{\prime}&1+\pi^2w_i^{\prime} \end{pmatrix}.$$
 \item If $i$ is even and $L_i$ is \textit{of type II}, then $m_{i,i}=\mathrm{id}+\pi m_{i,i}^{\prime}$.
\item If $i$ is even and $L_i$ is \textit{of type I}, then $z_i^{\ast}=\pi (z_i^{\ast})^{\prime}$.
This equation yields the following (formal) equation by condition (d) of the description of an element of $\tilde{M}(R)$
given at the paragraph following Lemma \ref{la1}:

\begin{equation}\label{52}
 z_i^{\prime}+\delta_{i-2}k_{i-2, i}^{\prime}+\delta_{i+2}k_{i+2, i}^{\prime}=0 \left( =\pi (z_i^{\ast})^{\prime} \right).
\end{equation}

Here, $k_{i-2, i}=\pi k_{i-2, i}^{\prime}$ and $k_{i+2, i}=\pi k_{i+2, i}^{\prime}$,
where $k_{i-2, i}$ and $k_{i+2, i}$ are as explained in (d) of the description of an element of $\tilde{M}(R)$.
\item If $i$ is odd and $L_i$ is \textit{free of type I}, then \[m_{i,i}=\begin{pmatrix} s_i&\pi r_i&t_i\\  y_i&1+\pi x_i& u_i\\\pi  v_i&\pi z_i&1+\pi w_i \end{pmatrix}
 =\begin{pmatrix} \mathrm{id}+\pi s_i^{\prime}&\pi^2 r_i^{\prime}&\pi t_i^{\prime}\\  \pi y_i^{\prime}&1+\pi^2 x_i^{\prime}& \pi u_i^{\prime}\\\pi^2  v_i^{\prime}&\pi^2 z_i^{\prime}&1+\pi^2 w_i^{\prime} \end{pmatrix}.\]
If $i$ is odd and $L_i$ is \textit{of type II} or \textit{bound of type I}, then $m_{i,i}=\mathrm{id}+\pi m_{i,i}^{\prime}$.

\item If $i$ is odd and $L_i$ is \textit{bound of type I}, then
\[m_{i,i}^{\ast}=\pi (m_{i,i}^{\ast})^{\prime}, ~~~~~~~ m_{i,i}^{\ast\ast}=\pi (m_{i,i}^{\ast\ast})^{\prime}.\]
These two equations yield the following (formal) equations
by conditions (e) and (f) of the description of an element of $\tilde{M}(R)$
given at the paragraph following Lemma \ref{la1}:
\begin{equation}\label{ea72}
\left\{ 
  \begin{array}{l}
 \delta_{i-1}v_{i-1}\cdot m_{i-1, i}^{\prime}+\delta_{i+1}v_{i+1}\cdot m_{i+1, i}^{\prime}=0
 ~~~ \left(=\pi (m_{i,i}^{\ast})^{\prime}\right);\\
 \delta_{i-1}v_{i-1}\cdot {}^tm_{i, i-1}^{\prime}+\delta_{i+1}v_{i+1}\cdot {}^tm_{i, i+1}^{\prime} =0
 ~~~ \left(=\pi (m_{i,i}^{\ast\ast})^{\prime}\right).\\
    \end{array} 
    \right.
    \end{equation}

Here,  notations follow from those of (e) and (f) in the description of an element of $\tilde{M}(R)$.
\end{enumerate}
Here, all matrices having ${}^{\prime}$ in the superscription are considered as matrices with entries in $R$.
 When $i$ is even and $L_i$ is \textit{of type} $\textit{I}$ or when $i$ is odd and $L_i$ is \textit{free of type} $\textit{I}$,
 we formally write $m_{i,i}=\mathrm{id}+\pi m_{i,i}^{\prime}$. 
Then $\tilde{G}^1(R)$ is the set of $m\in \tilde{M}^1(R)$ such that  $h\circ m=h=(f_{i, j}, a_i\cdots f_i)$.
Since $h\circ m$ is an element of $\underline{H}(R)$,  we can write $h\circ m$ as $(f_{i, j}', a_i'\cdots f_i')$.
In what follows,
we will write $(f_{i, j}', a_i'\cdots f_i')$ in terms of $h=(f_{i, j}, a_i\cdots f_i)$ and $m$, and will compare $(f_{i, j}', a_i'\cdots f_i')$ with
$(f_{i, j}, a_i\cdots f_i)$, in order to obtain a set of equations defining $\tilde{G}^1$.\\

If we put all these (1)-(7) into  (\ref{ea2}), then we obtain
$$\pi^j\left(\sigma(1+\pi\cdot {}^tm_{i,i}')h_i\pi m_{i,j}'+\sigma(\pi\cdot {}^tm_{j,i}')h_j(1+\pi m_{j,j}')+\pi^2(\ast))\right),$$
where $(\ast)$ is a certain formal polynomial.
Therefore,
\begin{equation}\label{ea3-}
f_{i,j}'=\left(\sigma(1+\pi\cdot {}^tm_{i,i}')h_i\pi m_{i,j}'+\sigma(\pi\cdot {}^tm_{j,i}')h_j(1+\pi m_{j,j}')+\pi^2(\ast)\right),
\end{equation}
where this equation is considered in $B\otimes_AR$ and $\pi$ stands for $\pi\otimes 1 \in B\otimes_AR$.
Thus each term having $\pi^2$ as a factor is $0$ and we have
\begin{equation}\label{ea3}
f_{i,j}'=h_i\pi m_{i,j}'+\sigma(\pi\cdot {}^tm_{j,i}')h_j, \textit{where $i<j$}.
\end{equation}
This equation is of the form $f_{i,j}'=X+\pi Y$ since it is an equation in $B\otimes_AR$.
By letting  $f_{i,j}'=f_{i,j}=0$,
 we obtain
 \begin{equation}\label{ea4}
 \bar{h}_i m_{i,j}'+{}^tm_{j,i}'\bar{h}_j=0, \textit{where $i<j$},\end{equation}
 where $\bar{h}_i$ (resp. $\bar{h}_j$) is obtained by letting each term in $h_i$ (resp. $h_j$) having $\pi$ as a factor be zero
 so that  this equation is considered in $R$.
Note that $\bar{h}_i$ and $\bar{h}_j$ are invertible as matrices with entries in $R$ by Remark \ref{r33}.
Thus $m_{i,j}'=\bar{h}_i^{-1}\cdot {}^tm_{j,i}'\cdot \bar{h}_j$.
This induces that each entry of $m_{i,j}'$ is expressed as a linear combination of the entries of $m_{j,i}'$.\\

When $i$ is odd and $L_i$ is \textit{bound of type I}, we consider the above equation (\ref{ea3-}) again because of the appearance of
$m_{i,i}^{\ast}, m_{i,i}^{\ast\ast}, f_{i,i}^{\ast}$.
We rewrite $f_{i-1,i}'$ and $f_{i,i+1}'$ formally as follows:
$$f_{i-1,i}'=\left(\sigma(1+\pi\cdot {}^tm_{i-1,i-1}')h_{i-1}\pi m_{i-1,i}'+\sigma(\pi\cdot {}^tm_{i,i-1}')h_i(1+\pi m_{i,i}')+\pi^3(\ast))\right),$$
$$f_{i, i+1}'=\left(\sigma(1+\pi\cdot {}^tm_{i,i}')h_i\pi m_{i,i+1}'+\sigma(\pi\cdot {}^tm_{i+1,i}')h_{i+1}(1+\pi m_{i+1,i+1}')+\pi^3(\ast\ast))\right).$$
Here,  $(\ast)$ and $(\ast\ast)$ are  certain formal polynomials with coefficients $\pi^3$, not $\pi^2$,
since $m_{j, j'}=\pi m_{j,j'}^{\prime}$  when $j\neq j'$.
We consider the following equations formally (cf. the second paragraph of Remark \ref{r35}):
\[
\left\{
  \begin{array}{l}
\pi (f_{i, i}^{\ast})'=\delta_{i-1}(0,\cdots, 0, 1)\cdot f_{i-1,i}'+\delta_{i+1}(0,\cdots, 0, 1)\cdot f_{i+1,i}';\\
 \sigma({}^tf_{i, i+1}')=f_{i+1, i}',\\
    \end{array}
    \right.
\]
where $(f_{i, i}^{\ast})'$ is a matrix of size $(1 \times n_i)$ with entries in $B\otimes_AR$.
The above formal equation involving $(f_{i, i}^{\ast})'$ should be interpreted as follows (cf. Remark \ref{r35}).
 We formally compute the right hand side by using the above formal expansions of $f_{i-1,i}'$ and $f_{i, i+1}'$,
  the equation $\sigma({}^tf_{i, i+1}')=f_{i+1, i}'$,
  and two formal equations (\ref{ea72}) involving $(m_{i,i}^{\ast})'$ and $(m_{i,i}^{\ast\ast})'$.
  We then get the following: 
\[\pi (f_{i, i}^{\ast})'=\pi^2\left((m_{i,i}^{\ast})'-(m_{i,i}^{\ast\ast})'h_i+\dag\right) \]  
after letting each term having $\pi^3$ as a factor be zero since $(f_{i, i}^{\ast})'$ is a matrix with entries in $B\otimes_AR$.
Here, $\dag$ is a polynomial of $m_{i-1,i-1}', m_{i-1,i}', m_{i,i-1}', m_{i,i}', m_{i,i+1}', m_{i+1,i}', m_{i+1,i+1}'$.
Thus
\begin{equation}\label{ea3'}
(f_{i, i}^{\ast})'=\pi\left((m_{i,i}^{\ast})'-(m_{i,i}^{\ast\ast})'h_i+\dag\right).
\end{equation}
Since this is an equation in $B\otimes_AR$, it is of the form $X+\pi Y=0$.
By letting $(f_{i, i}^{\ast})'=f_{i, i}^{\ast}=0$,
we obtain
\begin{equation}\label{ea4'}
(m_{i,i}^{\ast})'=(m_{i,i}^{\ast\ast})'\bar{h}_i+\bar{\dag},
\end{equation}
 where $\bar{h}_i$ (resp. $\bar{\dag}$) is obtained by letting each term in $h_i$ (resp. $\dag$) having $\pi$ as a factor be zero
 so that  this equation is considered in $R$.\\

 On the other hand, we apply Equation (\ref{ea4}) to the cases $(i-1, i)$ and $(i, i+1)$ and then we have
 \[\bar{h}_{i-1} m_{i-1,i}'+{}^tm_{i,i-1}'\bar{h}_i=0,\]
 \[ \bar{h}_i m_{i,i+1}'+{}^tm_{i+1,i}'\bar{h}_{i+1}=0.\]
When $i$ is odd and $L_i$ is \textit{bound of type I}, the above two equations with the first equation of Equation (\ref{ea72})
 yield the second equation of Equation (\ref{ea72}).
Thus by combining Equations (\ref{ea4}) and  (\ref{ea4'}) together with the first equation of Equation (\ref{ea72}),
we conclude that 
\[2\left(\sum_{\textit{$L_i$:bound of type I with i odd}}n_i\right)+\left(\sum_{i<j}n_in_j\right)-\textit{variables}\]
can be eliminated among
\[2\left(\sum_{\textit{$L_i$:bound of type I with i odd}}n_i\right)+2\left(\sum_{i<j}n_in_j\right)-\textit{variables }
\textit{$\{m_{i,j}'\}_{i\neq j}$, $(m_{i,i}^{\ast})'$, $(m_{i,i}^{\ast\ast})'$}.\]

\textit{   }\\

Next, we put (1)-(7) into  (\ref{ea1}).
Then we obtain
\begin{equation}\label{ea5}
\pi^i\left(\sigma(1+\pi\cdot {}^tm_{i,i}')h_i(1+\pi m_{i,i}')+\pi^3(\ast)\right).
\end{equation}
Here, $(\ast)$ is a certain formal polynomial. 
We interpret this so as to obtain equations defining $\tilde{G}^1$.
There are 6 cases, indexed by  (i) - (vi), according to types of $L_i$.\\

\begin{enumerate}
\item[(i)] Assume that $i$ is odd and that $L_i$ is \textit{of type II} or \textit{bound of type I}.
Then $\pi^ih_i=\xi^{(i-1)/2}\pi a_i$ as explained in Section \ref{h} and thus we have
\[a_i'=\sigma(1+\pi\cdot {}^tm_{i,i}')a_i(1+\pi m_{i,i}')+\pi^3(\ast).\]
Here, the nondiagonal entries of this equation are considered in $B\otimes_AR$ and each diagonal entry of $a_i'$ is of the form $\pi^3 x_i'$ with $x_i'\in R$.
Now,  the nondiagonal entries of $-\pi^2\cdot{}^tm_{i,i}'a_i m_{i,i}'+\pi^3(\ast)$ are all $0$
since they contain $\pi^2$ as a factor.
In addition, the diagonal entries of $\pi^3(\ast)$ are  $0$ since  they contain $\pi^5$ as  a factor, which can be verified 
 by using Equation (\ref{ea72}),
and the diagonal entries of $-\pi^2\cdot{}^tm_{i,i}'a_i m_{i,i}'$ are also $0$ since  they contain $\pi^4$ as  a factor.
Thus the above equation equals
\[a_i'=a_i+\sigma(\pi)\cdot {}^tm_{i,i}'a_i+ \pi\cdot a_i m_{i,i}'.\]
By letting $a_i'=a_i$,  we have the following equation
\[   \sigma(\pi)\cdot {}^tm_{i,i}'a_i+ \pi\cdot a_i m_{i,i}'=0.\]

Based on (3) of the description of an element of $\underline{H}(R)$ for a $\kappa$-algebra $R$, which is explained in Section \ref{h},
in order to investigate this equation,
we need to consider the nondiagonal entries of $\sigma(\pi)\cdot {}^tm_{i,i}'a_i+ \pi\cdot a_i m_{i,i}'$ as elements of $B\otimes_AR$
and the diagonal entries of $\sigma(\pi)\cdot {}^tm_{i,i}'a_i+ \pi\cdot a_i m_{i,i}'$ as of the form $\pi^3 x_i$ with $x_i\in R$.
Recall from Remark \ref{r33}.(2) that
$$a_i=\begin{pmatrix} \begin{pmatrix} 0&1\\-1&0\end{pmatrix}& &  \\ &\ddots &  \\ & & \begin{pmatrix} 0&1\\-1&0\end{pmatrix}\end{pmatrix}.$$
Note that ${}^ta_i=-a_i$ and  $\sigma(\pi)=-\pi$ so that
$\sigma(\pi)\cdot {}^tm_{i,i}'a_i+ \pi\cdot a_i m_{i,i}'=\pi({}^t(a_i m_{i,i}')+a_i m_{i,i}')$.
Then we can see that each diagonal entry  as well as
each nondiagonal (upper triangular) entry of $\sigma(\pi)\cdot {}^tm_{i,i}'a_i+ \pi\cdot a_i m_{i,i}'$ produces a linear equation.
Thus there are exactly $(n_i^2+n_i)/2$ independent linear equations and $(n_i^2-n_i)/2$ entries of $m_{i,i}'$ determine all entries of $m_{i,i}'$.

For example, let $m_{i,i}'=\begin{pmatrix} x&y\\z&w\end{pmatrix}$ and $a_i=\begin{pmatrix} 0&1\\-1&0\end{pmatrix}$.
Then $$\sigma(\pi)\cdot {}^tm_{i,i}'a_i+ \pi\cdot a_i m_{i,i}'=\pi\begin{pmatrix} 2z&-x+w\\-x+w&-2y\end{pmatrix}.$$
Thus there are three linear equations $-x+w=0, ~~~ z=0, ~~~ y=0$ and $x$ determines every other entry of $m_{i,i}'$.\\

\item[(ii)] Assume that $i$ is odd and that $L_i$ is \textit{free of type $I$}.
Then

$\pi^ih_i=\xi^{(i-1)/2}\cdot \pi\begin{pmatrix} a_i&\pi b_i& e_i\\ -\sigma(\pi \cdot {}^tb_i) &\pi^3f_i&1+\pi d_i \\
-\sigma({}^te_i) &-\sigma(1+\pi d_i) &\pi+\pi^3c_i \end{pmatrix}$
as explained in Section \ref{h} and we have
\begin{multline}\label{ea6}
\begin{pmatrix} a_i'&\pi b_i'& e_i'\\ -\sigma(\pi \cdot {}^tb_i') &\pi^3f_i'&1+\pi d_i' \\
-\sigma({}^te_i') &-\sigma(1+\pi d_i') &\pi+\pi^3c_i' \end{pmatrix}=\\
\sigma(1+\pi\cdot {}^tm_{i,i}')\cdot
\begin{pmatrix} a_i&\pi b_i& e_i\\ -\sigma(\pi \cdot {}^tb_i) &\pi^3f_i&1+\pi d_i \\
-\sigma({}^te_i) &-\sigma(1+\pi d_i) &\pi+\pi^3c_i \end{pmatrix}
\cdot(1+\pi m_{i,i}')+\pi^3(\ast).
\end{multline}
Here, the nondiagonal entries of $a_i'$ as well as the entries of $b_i', e_i', d_i'$ are considered in $B\otimes_AR$,
 each diagonal entry of $a_i'$ is of the form $\pi^3 x_i$ with $x_i\in R$,
  and $c_i', f_i'$ are in $R$.
In addition, $b_i=0, d_i=0, e_i=0, c_i=0, f_i=\bar{\gamma}_i$ as explained in Remark \ref{r33}.(2) and
$a_i$ is the diagonal matrix with $\begin{pmatrix} 0&1\\-1&0\end{pmatrix}$ on the diagonal.
In the above equation, we can cancel the term $\pi^3(\ast)$ since its nondiagonal entries contain $\pi^3$ as a factor and
its diagonal entries contain $\pi^5$ as a factor
since $L_i$ is \textit{free of type $I$} so that both $L_{i-1}$ and $L_{i+1}$ are \textit{of type II}.

Note that in this case,
$m_{i,i}'=\begin{pmatrix}  s_i^{\prime}& \pi r_i^{\prime}& t_i^{\prime}\\  y_i^{\prime}&\pi x_i^{\prime}&u_i^{\prime}\\
\pi v_i^{\prime}& \pi z_i^{\prime}&\pi w_i^{\prime} \end{pmatrix}$.
Compute $\sigma(\pi\cdot {}^tm_{i,i}')\cdot\begin{pmatrix} a_i&0&0\\ 0&\pi^3 \bar{\gamma}_i&1 \\ 0&-1 &\pi \end{pmatrix}\cdot(\pi m_{i,i}')$ formally and this equals
 $\sigma(\pi)\pi\begin{pmatrix} {}^ts_i'a_is_i'+\pi^2X_i & \pi Y_i &  Z_i
\\ \sigma( \pi \cdot {}^tY_i) &\pi^2X_i'&\pi Y_i' \\ \sigma({}^tZ_i)&\sigma(\pi\cdot  {}^tY_i') &{}^tt_i'a_it_i'+\pi^2 Z_i'  \end{pmatrix}$
for certain matrices $X_i, Y_i, Z_i, X_i', Y_i', Z_i'$ with suitable sizes.
Here, the diagonal entries of ${}^ts_i'a_is_i'$ and ${}^tt_i'a_it_i'$ are zero.
Thus we can ignore the contribution from
$\sigma(\pi\cdot {}^tm_{i,i}')\cdot\begin{pmatrix} a_i&0&0\\ 0&\pi^3 \bar{\gamma}_i&1 \\ 0&-1 &\pi \end{pmatrix}\cdot(\pi m_{i,i}')$
in Equation (\ref{ea6}) and so Equation (\ref{ea6}) equals
\begin{multline*}
\begin{pmatrix} a_i'&\pi b_i'& e_i'\\ -\sigma(\pi \cdot {}^tb_i') &\pi^3f_i'&1+\pi d_i' \\
-\sigma({}^te_i') &-\sigma(1+\pi d_i') &\pi+\pi^3c_i' \end{pmatrix}=
\begin{pmatrix} a_i&0&0\\ 0&\pi^3 \bar{\gamma}_i&1 \\ 0&-1 &\pi \end{pmatrix}+\\
-\pi\begin{pmatrix}  {}^ts_i'&{}^ty_i' &-\pi\cdot {}^t v_i'\\-\pi\cdot {}^t r_i'&-\pi\cdot x_i'&-\pi\cdot  z_i' \\
 {}^t t_i'&u_i'&-\pi\cdot w_i'\end{pmatrix}
\begin{pmatrix} a_i&0&0\\ 0&\pi^3 \bar{\gamma}_i&1 \\ 0&-1 &\pi \end{pmatrix}+
\pi \begin{pmatrix} a_i&0&0\\ 0&\pi^3 \bar{\gamma}_i&1 \\ 0&-1 &\pi\end{pmatrix}
\begin{pmatrix}  s_i^{\prime}& \pi r_i^{\prime}& t_i^{\prime}\\  y_i^{\prime}&\pi x_i^{\prime}&u_i^{\prime}\\
\pi v_i^{\prime}& \pi z_i^{\prime}&\pi w_i^{\prime} \end{pmatrix}.
\end{multline*}

We interpret each block of the above equation below:

\begin{enumerate}
\item Let us  consider the $(1,1)$-block.
The computation associated to this block is similar to that for the above case (i).
Hence  there are exactly $((n_i-2)^2+(n_i-2))/2$ independent linear equations and $((n_i-2)^2-(n_i-2))/2$ entries of $s_i'$ determine all entries of $s_i'$.

\item We consider the $(1,2)$-block.
We can ignore the contribution from ${}^ty_i'\bar{\gamma}_i$  since it contains $\pi^3$ as a factor.
Then the $(1,2)$-block is
\begin{equation}\label{ea7}
b_i'=\pi(- {}^tv_i'+a_ir_i').
\end{equation}
 By letting $b_i'=b_i=0$, we have
 \[\pi(- {}^tv_i'+a_ir_i')=0\]
as an equation in $B\otimes_AR$.
Thus  there are exactly $(n_i-2)$ independent linear equations  among the entries of $v_i'$ and $r_i'$.

\item The $(1,3)$-block is
\begin{equation}\label{ea8}
  e_i'=\pi(- {}^ty_i'+a_it_i').\end{equation}
This is an equation in $B\otimes_AR$. By letting $e_i'=e_i=0$,
 there are exactly $(n_i-2)$ independent linear equations  among the entries of $y_i'$ and $t_i'$.

\item The $(2,3)$-block is
\[1+\pi d_i'=1-\pi(-\pi x_i'-\pi^2z_i')+\pi(\pi^3\bar{\gamma}_iu_i'+\pi w_i').\]
By letting $d_i'=d_i=0$, we have
\begin{equation}\label{ea9}
d_i'=\pi(x_i'+ w_i')=0.\end{equation}
This is an equation in $B\otimes_AR$.
Thus there is  exactly one independent linear equation  between  $x_i'$ and $w_i'$.

\item The $(2,2)$-block is
\begin{equation}\label{ea10}
\pi^3 f_i'=\pi^3\bar{\gamma}_i-\pi(-\pi^4\bar{\gamma}_ix_i'+\pi z_i')+\pi(\pi^4\bar{\gamma}_i x_i'+\pi z_i').
\end{equation}
Since $-\pi(-\pi^4\bar{\gamma}_ix_i'+\pi z_i')+\pi(\pi^4\bar{\gamma}_i x_i'+\pi z_i')$ contains $2\pi^5$ as a factor,  by letting $f_i'=f_i=\bar{\gamma}_i$, this equation  is trivial.

\item The $(3,3)$-block is
\[\pi+\pi^3c_i'=\pi-\pi(u_i'-\pi^2w_i')+\pi(-u_i'+\pi^2 w_i').\]
By letting $c_i'=c_i=0$,
\begin{equation}\label{ea11}
c_i'=-u_i+2w_i'=0.
\end{equation}
This is an equation in $R$.
Thus $u_i=0$ is the only independent linear equation.\\
\end{enumerate}

By combining all six cases (a)-(f),
there are exactly $((n_i-2)^2+(n_i-2))/2+2(n_i-2)+2=(n_i^2+n_i)/2-1$ independent linear equations and $(n_i^2-n_i)/2+1$ entries of
$m_{i,i}'$ determine all entries of $m_{i,i}'$.\\

\item[(iii)] Assume that $i$ is even and that $L_i$ is \textit{of type II}. This case is parallel to the above case (i).
Then $\pi^ih_i=\xi^{i/2} a_i$ as explained in Section \ref{h} and thus we have
\[
a_i'=\sigma(1+\pi\cdot {}^tm_{i,i}')a_i(1+\pi m_{i,i}')+\pi^3(\ast).
\]
Here, the nondiagonal entries of this equation are considered in $B\otimes_AR$ and each diagonal entry of $a_i'$ is of the form $2 x_i$ with $x_i\in R$.
Thus we can cancel the term $\pi^3(\ast)$ since each entry contains $\pi^3$ as a factor.
In addition, we can  cancel the term $\sigma(\pi\cdot {}^tm_{i,i}')a_i(\pi m_{i,i}')$
since its nondiagonal entries contain $\pi^2$ as a factor and
its diagonal entries contain $\pi^4$ as  a factor.
Thus the above equation equals
\[ a_i'=a_i+\sigma(\pi)\cdot {}^tm_{i,i}'a_i+ \pi\cdot a_i m_{i,i}'.
\]
By letting $a_i'=a_i$,  we have the following equation
\[\sigma(\pi)\cdot {}^tm_{i,i}'a_i+ \pi\cdot a_i m_{i,i}'=0.\]

Based on (2) of the description of an element of $\underline{H}(R)$ for a $\kappa$-algebra $R$, which is explained in Section \ref{h},
in order to investigate this equation,
we need to consider the nondiagonal entries of $\sigma(\pi)\cdot {}^tm_{i,i}'a_i+ \pi\cdot a_i m_{i,i}'$ as elements of $B\otimes_AR$
and the diagonal entries of $\sigma(\pi)\cdot {}^tm_{i,i}'a_i+ \pi\cdot a_i m_{i,i}'$ as of the form $2 x_i$ with $x_i\in R$.
Recall from Remark \ref{r33}.(2) that
$$a_i=\begin{pmatrix} \begin{pmatrix} 0&1\\1&0\end{pmatrix}& & & \\ &\ddots & & \\ & &\begin{pmatrix} 0&1\\1&0\end{pmatrix}& \\ & & & \begin{pmatrix} 2\cdot 1&1\\1&2\cdot\bar{\gamma}_i\end{pmatrix} \end{pmatrix}.$$
Note that ${}^ta_i=a_i$ and  $\sigma(\pi)=-\pi$ so that
$\sigma(\pi)\cdot {}^tm_{i,i}'a_i+ \pi\cdot a_i m_{i,i}'=\pi\left( a_i m_{i,i}'-{}^t(a_i m_{i,i}')\right)$.

Then we can see that there is no contribution coming from diagonal entries
and each nondiagonal (upper triangular) entry  produces a linear equation.
Thus there are exactly $(n_i^2-n_i)/2$ independent linear equations and $(n_i^2+n_i)/2$ entries of $m_{i,i}'$ determine all entries of $m_{i,i}'$.

For example, let $m_{i,i}'=\begin{pmatrix} x&y\\z&w\end{pmatrix}$ and $a_i=\begin{pmatrix} 2&1\\1&2\bar{\gamma}_i\end{pmatrix}$.
Then $$\pi\left( a_i m_{i,i}'-{}^t(a_i m_{i,i}')\right)=\pi\begin{pmatrix} 0&w-x+2y-2\bar{\gamma}_iz\\w-x+2y-2\bar{\gamma}_iz&0\end{pmatrix}.$$
Thus there is only one linear equation $w-x=0$ and $x, y, z$ determine all entries of $m_{i,i}'$.\\

\item[(iv)] Assume that $i$ is even and that $L_i$ is \textit{of type $I^o$}.
Then $\pi^ih_i=\xi^{i/2} \begin{pmatrix} a_i&\pi b_i\\ \sigma(\pi\cdot {}^t b_i) &1 +2\bar{\gamma}_i+4c_i \end{pmatrix}$ as explained in Section \ref{h} and thus we have
\begin{equation}\label{ea12}
\begin{pmatrix} a_i'&\pi b_i'\\ \sigma(\pi\cdot {}^t b_i') &1 +2\bar{\gamma}_i+4c_i'  \end{pmatrix}=
\sigma(1+\pi\cdot {}^tm_{i,i}')\cdot\begin{pmatrix} a_i&\pi b_i\\ \sigma(\pi\cdot {}^t b_i) &1 +2\bar{\gamma}_i+4c_i  \end{pmatrix}\cdot(1+\pi m_{i,i}')+\pi^3(\ast).
\end{equation}
Here,  the nondiagonal entries of $a_i'$ as well as the entries of $b_i'$ are considered in $B\otimes_AR$,
 each diagonal entry of $a_i'$ is of the form $2 x_i$ with $x_i\in R$, and $c_i'$ is in $R$.
In addition, $b_i=0, c_i=0$ as explained in Remark \ref{r33}.(2) and
$a_i$ is the diagonal matrix with $\begin{pmatrix} 0&1\\1&0\end{pmatrix}$ on the diagonal.

Note that in this case, $m_{i,i}'=\begin{pmatrix} s_i'&\pi y_i'\\ \pi v_i' &\pi z_i' \end{pmatrix}$.
Compute $\sigma(\pi\cdot {}^tm_{i,i}')\cdot\begin{pmatrix} a_i&0\\ 0 &1 +2\bar{\gamma}_i \end{pmatrix}\cdot (\pi m_{i,i}')$ formally and
this equals $\sigma(\pi)\pi\begin{pmatrix} {}^ts_i'a_is_i'+\pi^2X_i &\pi Y_i\\ \sigma(\pi\cdot {}^tY_i) &-\pi^2(z_i')^2+\pi^4Z_i \end{pmatrix}$
for certain matrices $X_i, Y_i, Z_i$ with suitable sizes.

Thus we can ignore the $(1,1)$ and $(1,2)$-blocks of the term $\sigma(\pi\cdot {}^tm_{i,i}')\begin{pmatrix} a_i&0\\ 0 &1 +2\bar{\gamma}_i \end{pmatrix}(\pi m_{i,i}')$ in Equation (\ref{ea12}).
On the other hand, we should consider the $(2,2)$-block of this term because of the appearance of $\pi^4(z_i')^2$.
By the same reason, we can ignore the $(1,1)$ and $(1,2)$-blocks of  the term $\pi^3(\ast)$, whereas
the $(2,2)$-block of this term should  be considered.

We interpret each block of Equation (\ref{ea12})  below:

\begin{enumerate}
\item Firstly, we consider the $(1,1)$-block.
The computation associated to this block is similar to that for the above case  (iii).
Hence  there are exactly $((n_i-1)^2-(n_i-1))/2$ independent linear equations and $((n_i-1)^2+(n_i-1))/2$ entries of $s_i'$ determine all entries of $s_i'$.

\item Secondly, we consider the $(1, 2)$-block.
Then it equals
\begin{equation}\label{13}
\pi b_i'=-\pi^2 \cdot{}^tv_i'+\pi^2\cdot a_iy_i'.
\end{equation}
By letting $b_i'=b_i=0$, we have
\[-\pi \cdot{}^tv_i'+\pi \cdot a_iy_i'=0\]
as an equation in $B\otimes_AR$.
Thus there are exactly $(n_i-1)$ independent linear equations among the entries of $v_i'$ and $y_i'$
and the entries of $v_i'$ determine all entries of $y_i'$.

\item We postpone to consider the $(2, 2)$-block later in Step (vi).\\
\end{enumerate}

By combining two cases (a) and (b), there are exactly $((n_i-1)^2-(n_i-1))/2+(n_i-1)=n_i(n_i-1)/2$ independent linear equations.
Thus  $n_i^2-1-n_i(n_i-1)/2=n_i(n_i+1)/2-1$ entries of
$m_{i,i}'$  determine all entries of $m_{i,i}'$ except for $z_i'$ and $(z_i^{\ast})'$.\\

\item[(v)] Assume that $i$ is even and that $L_i$ is \textit{of type $I^e$}.
Then

$\pi^ih_i=\xi^{i/2}\begin{pmatrix} a_i&{}^tb_i&\pi e_i\\ \sigma(b_i) &1+2f_i&1+\pi d_i \\ \sigma(\pi \cdot {}^te_i) &\sigma(1+\pi d_i) &2\bar{\gamma}_i+4c_i \end{pmatrix}$ as explained in Section \ref{h} and thus we have
\begin{multline}\label{ea13}
\begin{pmatrix} a_i'&{}^tb_i'&\pi e_i'\\ \sigma(b_i') &1+2f_i'&1+\pi d_i' \\
 \sigma(\pi \cdot {}^te_i') &\sigma(1+\pi d_i') &2\bar{\gamma}_i+4c_i' \end{pmatrix}=\\
\sigma(1+\pi\cdot {}^tm_{i,i}')\cdot
\begin{pmatrix} a_i&{}^tb_i&\pi e_i\\ \sigma(b_i) &1+2f_i&1+\pi d_i \\ \sigma(\pi \cdot {}^te_i) &\sigma(1+\pi d_i) &2\bar{\gamma}_i+4c_i \end{pmatrix}
\cdot(1+\pi m_{i,i}')+\pi^3(\ast).
\end{multline}
Here, the nondiagonal entries of $a_i'$ as well as the entries of $b_i', e_i', d_i'$ are considered in $B\otimes_AR$,
 each diagonal entry of $a_i'$ is of the form $2 x_i$ with $x_i\in R$, and $c_i', f_i'$ are in $R$.
In addition, $b_i=0, d_i=0, e_i=0, f_i=0, c_i=0$ as explained in Remark \ref{r33}.(2) and
$a_i$ is the diagonal matrix with $\begin{pmatrix} 0&1\\1&0\end{pmatrix}$ on the diagonal.

Note that in this case, $m_{i,i}'=\begin{pmatrix}  s_i^{\prime}& r_i^{\prime}&\pi t_i^{\prime}\\ \pi y_i^{\prime}&\pi x_i^{\prime}&\pi z_i^{\prime}\\  v_i^{\prime}& u_i^{\prime}&\pi w_i^{\prime} \end{pmatrix}$.
Compute $\sigma(\pi\cdot {}^tm_{i,i}')\cdot\begin{pmatrix} a_i&0&0\\ 0&1&1 \\ 0&1 &2\bar{\gamma}_i \end{pmatrix}\cdot(\pi m_{i,i}')$ formally and this equals $\sigma(\pi)\pi\begin{pmatrix} {}^ts_i'a_is_i'+\pi^2X_i & Y_i & \pi Z_i
\\ \sigma( {}^tY_i) &{}^tr_i'a_ir_i'+\pi^2X_i'&\pi Y_i' \\ \sigma(\pi\cdot {}^tZ_i)&\sigma(\pi\cdot  {}^tY_i') &-\pi^2(z_i')^2+  \pi^4 Z_i'  \end{pmatrix}$
for certain matrices $X_i, Y_i, Z_i, X_i', Y_i', Z_i'$ with suitable sizes.

Thus we can ignore the $(1, 1), (1, 2), (1, 3), (2, 2),$ and $(2, 3)$-blocks of the term
$\sigma(\pi\cdot {}^tm_{i,i}')\cdot\begin{pmatrix} a_i&0&0\\ 0&1&1 \\ 0&1 &2\bar{\gamma}_i \end{pmatrix}\cdot(\pi m_{i,i}')$
in Equation (\ref{ea13}).
As in the above case (iv), we should consider the $(3,3)$-block of this term because of the appearance of $\pi^4(z_i')^2$.
By the same reason, we can ignore the $(1, 1), (1, 2), (1, 3), (2, 2),$ and $(2, 3)$-blocks of  the term $\pi^3(\ast)$, whereas
the $(3,3)$-block of this term should  be considered.

We interpret each block of Equation (\ref{ea13}) below:

\begin{enumerate}
\item Firstly, we consider the $(1,1)$-block.
The computation associated to this block is similar to that for the above case  (iii).
Hence  there are exactly $((n_i-2)^2-(n_i-2))/2$ independent linear equations and $((n_i-2)^2+(n_i-2))/2$ entries of $s_i'$ determine all entries of $s_i'$.

\item Secondly, we consider the $(1, 2)$-block.
Then it equals
\begin{equation}\label{ea14}
{}^tb_i'={}^tb_i+\pi( -{}^tv_i'+a_ir_i').
\end{equation}
This is an equation in $B\otimes_AR$. By letting $b_i'=b_i=0$,
 there are exactly $(n_i-2)$ independent linear equations  among the entries of $v_i'$ and $r_i'$.

\item The $(1, 3)$-block is
\[\pi e_i'=\pi e_i+\pi^2({}^ty_i'+a_it_i').\]
By letting $e_i'=e_i=0$,
we have
\begin{equation}\label{ea15}
e_i'=e_i+\pi({}^ty_i'+a_it_i')=0.\end{equation}
This is an equation in $B\otimes_AR$.
Thus there are exactly $(n_i-2)$ independent linear equations  among the entries of $y_i'$ and $t_i'$.

\item The $(2, 3)$-block is
\[1+\pi d_i'=1+\pi d_i+\pi^2(x_i'+z_i'+w_i').\]
By letting $d_i'=d_i=0$, we have
\begin{equation}\label{ea16}
d_i'=\pi(x_i'+ z_i'+w_i')=0.\end{equation}
This is an equation in $B\otimes_AR$.
Thus there is  exactly one independent linear equation  among  $x_i', z_i', w_i'$.

\item The $(2\times 2)$-block is
\[1+2 f_i'=1+2 f_i+2\pi(u_i'+\pi x_i'). \]
By letting $f_i'=f_i=0$, we have
\begin{equation}\label{ea17}
f_i'=f_i+\pi(u_i'+\pi x_i')=0. 
\end{equation}
This is an equation in $R$.
Thus this equation is trivial.

\item We postpone to consider the $(3, 3)$-block later in Step (vi).\\
\end{enumerate}

By combining all five cases (a)-(e),  there are exactly $((n_i-2)^2-(n_i-2))/2+2(n_i-2)+1=(n_i^2-n_i)/2$ independent linear equations.
 Thus $n_i^2-1-(n_i^2-n_i)/2=n_i(n_i+1)/2-1$   entries of
$m_{i,i}'$ determine all entries of $m_{i,i}'$ except for $z_i'$ and $(z_i^{\ast})'$.\\

\item[(vi)] Let $i$ be even and $L_i$ be \textit{of type I}.
 Finally, we consider the $(2, 2)$-block of Equation (\ref{ea12}) if $L_i$ is \textit{of type $I^o$} or
 the $(3, 3)$-block of Equation (\ref{ea13}) if $L_i$ is \textit{of type $I^e$}
 given below:
\begin{equation}\label{ea18}
 2\bar{\gamma}_i+4c_i'=2\bar{\gamma}_i+4c_i+ 2\pi^2z_i'+\pi^4(z_i')^2-\pi^4\delta_{i-2}k_{i-2, i}'-\pi^4\delta_{i+2}k_{i+2, i}'.
\end{equation}
Here, $k_{i-2, i}'$ and $k_{i+2, i}'$ are as explained in  the condition (5) given at the paragraph following (\ref{ea2}).
Note that  the condition (5) yields Equation (\ref{52}), $z_i'+\delta_{i-2}k_{i-2, i}'+\delta_{i+2}k_{i+2, i}'=0$, in $R$.
Thus, by letting $c_i'=c_i=0$, we have
\begin{equation}\label{ea19}
 0=c_i'=c_i+z_i'=z_i'
\end{equation}
as an equation in $R$.

Note that if both $L_i$ and $L_{i+2}$ are \textit{of type I}, then $k_{i+2, i}'=k_{i,i+2}'$ by Equation (\ref{ea4}).
We now choose an even integer $j$ such that $L_j$ is \textit{of type I} and $L_{j+2}$ is \textit{of type II}.
For such $j$, there is a nonnegative integer $m_j$ such that $L_{j-2l}$ is \textit{of type I}
for every $l$ with $0\leq l \leq m_j$ and $L_{j-2(m_j+1)}$ is \textit{of type II}.
As mentioned in the above paragraph, the condition (5) yields the following equation in $R$:
\[\mathcal{Z}_{j-2l}' : z_{j-2l}'+\delta_{j-2l-2}k_{j-2l-2, j-2l}'+\delta_{j-2l+2}k_{j-2l+2, j-2l}'=0.\]
Then the sum of equations
\[\sum_{0\leq l \leq m_j}\mathcal{Z}_{j-2l}'\]
is the same as
$$\sum_{0\leq l \leq m_j}z_{j-2l}'=0$$
since $k_{j-2l+2, j-2l}'=k_{j-2l,j-2l+2}'$.
Therefore, among Equations (\ref{ea19}) for $j-2m_j \leq i \leq j$, only one of them is redundant.
In conclusion, 
there are
$$\#\{i:\textit{$i$ is even and $L_i$ is of type I}\}-\#\{i:\textit{$i$ is even, $L_i$ is  of type I and  $L_{i+2}$ is of type II}\}$$
  independent linear equations of the form $z_i'=0$.\\
\end{enumerate}

We now combine all works done in this proof. Namely, we collect the above (i), (ii), (iii), (iv), (v), (vi) which are  the interpretations of Equation (\ref{ea5}),
together with
Equations (\ref{ea4}) and (\ref{ea4'}).
To simplify notation, we say just in this paragraph that Equation (\ref{ea4'}) is linear.
Then there are exactly
$$\sum_{i<j}n_in_j+\sum_{i:\mathrm{odd}}\frac{n_i^2+n_i}{2}-\#\{i:\textit{$i$ is odd and $L_i$ is free of type I}\}+
\sum_{i:\mathrm{even}}\frac{n_i^2-n_i}{2}+$$
$$\#\{i:\textit{$i$ is even and $L_i$ is of type I}\}-\#\{i:\textit{$i$ is even, $L_i$ is  of type I and  $L_{i+2}$ is of type II}\}$$
independent linear equations among the entries of $m\in \tilde{M}^1(R)$. 
 Furthermore, all coefficients of these equations are in $\kappa$.
Therefore, we consider $\tilde{G}^1$ as a subvariety of $\tilde{M}^1$ determined by these linear equations.
Since  $\tilde{M}^1$ is an affine space of dimension $n^2$,
  the underlying algebraic variety of $\tilde{G}^1$ over $\kappa$ is  an affine space of dimension
\[
\sum_{i<j}n_in_j+\sum_{i:\mathrm{even}}\frac{n_i^2+n_i}{2}+\sum_{i:\mathrm{odd}}\frac{n_i^2-n_i}{2}
+\#\{i:\textit{$i$ is odd and $L_i$ is free of type I}\} \]
\[-\#\{i:\textit{$i$ is even and $L_i$ is of type I}\}
+ \#\{i:\textit{$i$ is even, $L_i$ is  of type I and  $L_{i+2}$ is of type II}\}.\]
This completes the proof by using a theorem of Lazard which is stated at the beginning of Appendix \ref{App:AppendixA}.
\end{proof}
\textit{  }

Let $R$ be a $\kappa$-algebra.
We describe the functor of points of the scheme  $\mathrm{Ker~}\tilde{\varphi}/\tilde{M}^1$
by using points of the scheme $(\underline{M}'\otimes\kappa)/ \underline{\pi M}'$, based on Lemma \ref{la3}.
To do that, we take the argument  in  pages 511-512 of \cite{C2}.

Recall from two paragraphs before Lemma \ref{la3}
 that $(1+)^{-1}(\underline{M}^{\ast})$, which is an open subscheme of  $\underline{M}'$, is a group scheme with the operation $\star$.
Let $\tilde{M}'$ be the special fiber of $(1+)^{-1}(\underline{M}^{\ast})$.
Since $\tilde{M}^{1}$ is a closed normal subgroup of $\tilde{M} (=\underline{M}^{\ast}\otimes\kappa)$ (cf. Lemma \ref{la3}.(i)),
$\underline{\pi M'}$, which is the inverse image of $\tilde{M}^{1}$ under the isomorphism $1+$, is a closed normal subgroup of $\tilde{M}'$.
Therefore, the morphism $1+$ induces the following isomorphism of group schemes, which is also denoted by $1+$,
$$1+  : \tilde{M}'/\underline{\pi M'}\longrightarrow \tilde{M}/\tilde{M}^{1}.$$
Note that $\tilde{M}'/\underline{\pi M'}(R)=\tilde{M}'(R)/\underline{\pi M'}(R)$ by Lemma \ref{la1}.
Thus each element of $(\mathrm{Ker~}\tilde{\varphi}/\tilde{M}^1)(R)$ is uniquely written as $1+\bar{x}$,
 where $\bar{x}\in \tilde{M}'(R)/  \underline{\pi M}'(R)$.
 Here, by $1+\bar{x}$, we mean the image of $\bar{x}$ under the morphism $1+$ at the level of $R$-points.

We still need a better description of an element of $(\mathrm{Ker~}\tilde{\varphi}/\tilde{M}^1)(R)$ by using a point of the scheme $(\underline{M}'\otimes\kappa)/ \underline{\pi M}'$.
Note that  $(\underline{M}'\otimes\kappa)/ \underline{\pi M}'$ is a quotient of group schemes with respect to the addition, whereas
$\tilde{M}'/  \underline{\pi M}'$ is a quotient of group schemes with respect to the operation $\star$.

The open immersion $\iota : \tilde{M}' \rightarrow \underline{M}'\otimes\kappa, x\mapsto x$ induces a monomorphism of monoid schemes
preserving the operation $\star$:
$$\bar{\iota} : \tilde{M}'/  \underline{\pi M}' \rightarrow (\underline{M}'\otimes\kappa)/ \underline{\pi M}'.$$
Note that although  $(\underline{M}'\otimes\kappa)/ \underline{\pi M}'$ is a quotient of group schemes with respect to the addition,
the operation $\star$ is  well-defined on $(\underline{M}'\otimes\kappa)/ \underline{\pi M}'$.
For the proof, see  the first two paragraphs from below in page 511 and the first two paragraphs in page 512 in \cite{C2}.

To summarize, the morphism $1+  : \tilde{M}'/\underline{\pi M'}\longrightarrow \tilde{M}/\tilde{M}^{1}$ is an isomorphism of group schemes and
the morphism  $\bar{\iota} : \tilde{M}'/  \underline{\pi M}' \rightarrow (\underline{M}'\otimes\kappa)/ \underline{\pi M}'$ is a monomorphism
preserving the operation $\star$.
Therefore, each element of $(\mathrm{Ker~}\tilde{\varphi}/\tilde{M}^1)(R)$ is uniquely written as $1+\bar{x}$,
 where $\bar{x}\in (\underline{M}'\otimes\kappa)(R)/ \underline{\pi M}'(R)$.
Here, by $1+\bar{x}$, we mean $(1+)\circ \bar{\iota}^{-1}(\bar{x})$.
From now on to the end of this paper, we keep the notation $1+\bar{x}$ to express an element of $(\mathrm{Ker~}\tilde{\varphi}/\tilde{M}^1)(R)$
such that $\bar{x}$ is an element of $(\underline{M}'\otimes\kappa)(R)/ \underline{\pi M}'(R)$
which is a quotient of $R$-valued points of group schemes
with respect to addition.
Then the product of two elements $1+\bar{x}$ and $1+\bar{y}$ is the same as $1+\bar{x}\star \bar{y}$ $(=1+(\bar{x} + \bar{y}+\bar{x} \bar{y}))$.

\begin{Rmk}\label{ra5}
By the above argument, we  write an element of $(\mathrm{Ker~}\tilde{\varphi}/\tilde{M}^1)(R)$ formally as
$$m= \begin{pmatrix} \pi^{max\{0,j-i\}}m_{i,j} \end{pmatrix}\mathrm{~together~with~}z_i^{\ast}, m_{i,i}^{\ast}, m_{i,i}^{\ast\ast},$$
with $s_i,\cdots, w_i$ as in Section \ref{m} such that
 each entry of each of the matrices $(m_{i,j})_{i\neq j}, s_i, \cdots, w_i$ is in $(B\otimes_AR)/(\pi\otimes 1)(B\otimes_AR)\cong R$.
In particular, based on the description of $\mathrm{Ker~}\tilde{\varphi}(R)$ given at the paragraph following  Lemma \ref{la2},
we have the following conditions on $m$:
\begin{enumerate}
\item Assume that  $L_i$ is \textit{of type I} with $i$ even or  that  $L_i$ is \textit{free of type I} with $i$ odd.
 Then $s_i=\mathrm{id}$.

\item  Assume that  $L_i$ is \textit{of type II} with $i$ even, or  that  $L_i$ is \textit{bound of type I or type II} with $i$ odd.
 Then $m_{i,i}=\mathrm{id}$.

\item Let $i$ be even.
\begin{itemize}
\item If  $L_i$ is \textit{bound of type II},
then $\delta_{i-1}^{\prime}e_{i-1}\cdot m_{i-1, i}+\delta_{i+1}^{\prime}e_{i+1}\cdot m_{i+1, i}+\delta_{i-2}e_{i-2}\cdot m_{i-2, i}+\delta_{i+2}e_{i+2}\cdot m_{i+2, i}=0$.

\item If  $L_i$ is \textit{of type I},
then  $v_i(\mathrm{resp.~}(y_i+\sqrt{\bar{\gamma}_i}v_i))+(\delta_{i-2}e_{i-2}\cdot m_{i-2, i}+\delta_{i+2}e_{i+2}\cdot m_{i+2, i})\tilde{e_i}=0$ if $L_i$ is \textit{of type} $\textit{I}^o$ (resp. \textit{of type} $\textit{I}^e$).
\end{itemize}
Here, notations are as explained in Step  (c) of  the description of an element of $\mathrm{Ker~}\tilde{\varphi}(R)$ given at the paragraph following Lemma \ref{la2}.

\item If $i$ is even and $L_i$ is \textit{of type I}, then 
$$z_i+\delta_{i-2}k_{i-2, i}+\delta_{i+2}k_{i+2, i}=0 ~~~ \left(=\pi z_i^{\ast}  \right).$$
Here,  notations are as explained in Step  (d) of  the description of an element of $\tilde{M}(R)$ given at the paragraph following Lemma \ref{la1}.

\item If $i$ is odd and $L_i$ is \textit{bound of type I}, then 
$$\delta_{i-1}v_{i-1}\cdot m_{i-1, i}+\delta_{i+1}v_{i+1}\cdot m_{i+1, i}=0 ~~~ \left(=\pi m_{i,i}^{\ast}\right).$$
Here,  notations are as explained in Step  (e) of  the description of an element of $\tilde{M}(R)$ given at the paragraph following Lemma \ref{la1}.

\item If $i$ is odd and $L_i$ is \textit{bound of type I}, then 
$$\delta_{i-1}v_{i-1}\cdot {}^tm_{i, i-1}+\delta_{i+1}v_{i+1}\cdot {}^tm_{i, i+1}=0 ~~~ \left(=\pi m_{i,i}^{\ast\ast}\right).$$
Here,  notations are as explained in Step  (f) of  the description of an element of $\tilde{M}(R)$ given at the paragraph following Lemma \ref{la1}.

\end{enumerate}
\end{Rmk}

 \begin{Thm}\label{ta6}
$\mathrm{Ker~}\varphi/\tilde{G}^1 $ is isomorphic to $ \mathbb{A}^{l^{\prime}}\times (\mathbb{Z}/2\mathbb{Z})^{\beta}$ as a $\kappa$-variety,
 where $\mathbb{A}^{l^{\prime}}$ is an affine space of dimension $l^{\prime}$.  Here,
 \begin{itemize} 
  \item $l^{\prime}$ is such that \textit{$l^{\prime}$ + dim $\tilde{G}^1=l$.}
   Note that $l$ is defined in Lemma \ref{l46} and
  that the dimension of $\tilde{G}^1$ is given in Theorem \ref{ta4}.
\item  $\beta$ is the number of  integers $j$ such that $L_j$ is of type I and $L_{j+2}, L_{j+3}, L_{j+4}$ (resp. $L_{j-1},$ $L_{j+1},$ $L_{j+2}, L_{j+3}$) are of type II if $j$ is even (resp. odd).
  \end{itemize}
   \end{Thm}

   \begin{proof}
 Lemma \ref{la1} and Theorem \ref{ta4} imply that $\mathrm{Ker~}\varphi/\tilde{G}^1 $ represents the functor
    $R\mapsto \mathrm{Ker~}\varphi(R)/\tilde{G}^1(R)$.
    Recall that $\mathrm{Ker~}\varphi/\tilde{G}^1$ is a closed subgroup scheme of $\mathrm{Ker~}\tilde{\varphi}/\tilde{M}^1$
    as explained at the paragraph just before Theorem \ref{ta4}.
    Let $m= \begin{pmatrix} \pi^{max\{0,j-i\}}m_{i,j} \end{pmatrix} \mathrm{~with~}z_i^{\ast}, m_{i,i}^{\ast}, m_{i,i}^{\ast\ast}$
         be an element of  $(\mathrm{Ker~}\tilde{\varphi}/\tilde{M}^1)(R)$
     such that $m$ belongs to $(\mathrm{Ker~}\varphi/\tilde{G}^1)(R)$.
     We want to find equations which $m$ satisfies.
Note that the entries of $m$ involve $(B\otimes_AR)/(\pi\otimes 1)(B\otimes_AR)$ as explained in Remark \ref{ra5}.

 Recall that $h$ is the fixed hermitian form and we consider it as an element in $\underline{H}(R)$ as explained in Remark \ref{r33}.(2).
 We write it as a formal matrix $h=\begin{pmatrix} \pi^{i}\cdot h_i\end{pmatrix}$ with $(\pi^{i}\cdot h_i)$
for the $(i,i)$-block and $0$ for the remaining blocks.
    We choose a representative $1+x\in \mathrm{Ker~}\varphi(R)$ of $m$ so that $h\circ (1+x)=h$.
    Any other representative of $m$ in  $\mathrm{Ker~}\tilde{\varphi}(R)$ is
    of the form $(1+x)(1+\pi y)$ with $y \in \underline{M}'(R)$
    and we have  $h\circ (1+x)(1+\pi y)=h\circ (1+\pi y)$.
    Notice that $h\circ (1+\pi y)$ is an element of $\underline{H}(R)$ so we express it as $(f_{i,j}', a_i' \cdots f_i')$.
    We also let  $h=(f_{i,j}, a_i \cdots f_i)$.
    Here, we follow notation from Section \ref{h}, the paragraph just before Remark \ref{r33}.
    Recall that $h=(f_{i,j}, a_i \cdots f_i)$ is described explicitly in Remark \ref{r33}.(2).
Now, $1+\pi y$ is an element of $\tilde{M}^1(R)$ and so we can use our result (Equations (\ref{ea3}), (\ref{ea3'}), 
(\ref{ea7}), (\ref{ea8}),  (\ref{ea9}),   (\ref{ea10}), (\ref{13}), 
(\ref{ea14}), (\ref{ea15}),  (\ref{ea16}),  (\ref{ea17}), (\ref{ea19}))
  stated in the proof of Theorem \ref{ta4} in order to compute $h\circ (1+\pi y)$.
Based on this, we enumerate equations which $m$ satisfies as follows:

\begin{enumerate}
\item Assume that $i<j$. By Equation (\ref{ea3}) which involves an element of $\tilde{M}^1(R)$,
each entry of $f_{i,j}'$ has $\pi$ as a factor so that $f_{i,j}'\equiv f_{i,j} (=0)$
mod $(\pi\otimes 1)(B\otimes_AR)$.
In other words, the $(i,j)$-block of $h\circ (1+x)(1+\pi y)$ divided by $\pi^{max\{i, j\}}$ is $f_{i,j} (=0)$
modulo $(\pi\otimes 1)(B\otimes_AR)$, which is independent of the choice of $1+\pi y$.
Let $\tilde{m}\in \mathrm{Ker~}\tilde{\varphi}(R)$ be a lift of $m$.
Therefore, if we write the $(i, j)$-block of $\sigma({}^t\tilde{m})\cdot h\cdot \tilde{m}$ as $\pi^{max\{i, j\}}\mathcal{X}_{i,j}(\tilde{m})$,
where $\mathcal{X}_{i,j}(\tilde{m}) \in M_{n_i\times n_j}(B\otimes_AR)$,
then the image of $\mathcal{X}_{i,j}(\tilde{m})$ in $M_{n_i\times n_j}(B\otimes_AR)/(\pi\otimes 1)M_{n_i\times n_j}(B\otimes_AR)\cong M_{n_i\times n_j}(R)$
is independent of the choice of the lift $\tilde{m}$ of $m$.
Therefore, we may denote this image by $\mathcal{X}_{i,j}(m)$.
On the other hand, by Equation (\ref{ea2}), we have the following identity:
\begin{equation}\label{ea20}
\mathcal{X}_{i,j}(m)=\sum_{i\leq k \leq j} \sigma({}^tm_{k,i})\bar{h}_km_{k,j} \mathrm{~if~}i<j.
\end{equation}
We explain how to interpret the above equation.
We know that $\mathcal{X}_{i,j}(m)$ and $m_{k,k'}$ (with $k\neq k'$) are matrices with entries in $(B\otimes_AR)/(\pi\otimes 1)(B\otimes_AR)$,
whereas $m_{i,i}$ and $m_{j,j}$ are formal matrices as explained in Remark \ref{ra5}.
Thus we consider $\bar{h}_k$, $m_{i,i}$, and $m_{j,j}$ as matrices with entries in $(B\otimes_AR)/(\pi\otimes 1)(B\otimes_AR)$
by letting $\pi$ be zero in each entry of  formal matrices $h_k$, $m_{i,i}$, and $m_{j,j}$.
Then the right hand side is computed as a sum of products of matrices (involving the usual matrix addition and multiplication)
with entries in $(B\otimes_AR)/(\pi\otimes 1)(B\otimes_AR)$.
Thus, the assignment $m\mapsto \mathcal{X}_{i,j}(m)$ is a polynomial in $m$.
Furthermore, since $m$ actually belongs to $\mathrm{Ker~}\varphi(R)/\tilde{G}^1(R)$,
 we have the following equation by the argument made at the beginning of this paragraph:
  $$\mathcal{X}_{i,j}(m)=f_{i,j}\textit{ mod $(\pi\otimes 1)(B\otimes_AR)$}=0.$$
Thus we get an $n_i\times n_j$ matrix $\mathcal{X}_{i,j}$ of polynomials on $\mathrm{Ker~}\tilde{\varphi}/\tilde{M}^1$
defined by Equation (\ref{ea20}),
vanishing on the subscheme $\mathrm{Ker~}\varphi/\tilde{G}^1$.
For example, if $j=i+1$, then
\begin{equation}\label{ea21}
\sigma({}^tm_{i,i}) \bar{h}_i m_{i,i+1}+\sigma({}^tm_{i+1,i}) \bar{h}_{i+1} m_{i+1,i+1}=0.
\end{equation}

Before moving to  the following steps, we fix  notation.
Let $m$ be an element in $(\mathrm{Ker~}\tilde{\varphi}/\tilde{M}^1)(R)$ and $\tilde{m}\in \mathrm{Ker~}\tilde{\varphi}(R)$ be its lift.
For any block $x_i$  of $m$, $\tilde{x}_i$ is denoted by the corresponding block of $\tilde{m}$ whose reduction is $x_i$.
Since $x_i$ is a block of an element of $(\mathrm{Ker~}\tilde{\varphi}/\tilde{M}^1)(R)$,
it involves $(B\otimes_AR)/(\pi\otimes 1)(B\otimes_AR)$ as explained in Remark \ref{ra5}, whereas
$\tilde{x}_i$ involves $B\otimes_AR$.
In addition,
for a block $a_i$  of $h$, $\bar{a}_i$ is denoted by  the image  of $a_i$ in  $(B\otimes_AR)/(\pi\otimes 1)(B\otimes_AR)$.\\

\item Assume that  $i$ is odd and that $L_i$ is \textit{bound of type I}.
By Equation (\ref{ea3'}) which involves an element of $\tilde{M}^1(R)$,
each entry of $(f_{i,i}^{\ast})'$ has $\pi$ as a factor so that $(f_{i,i}^{\ast})'\equiv f_{i,i}^{\ast}=0$ mod $(\pi\otimes 1)(B\otimes_AR)$.
Let $\tilde{m}\in \mathrm{Ker~}\tilde{\varphi}(R)$ be a lift of $m$.
We write 
$$\pi\mathcal{X}_{i,i}^{\ast}(\tilde{m})=\delta_{i-1}(0,\cdots, 0, 1)\cdot \mathcal{X}_{i-1,i}(\tilde{m})+
\delta_{i+1}(0,\cdots, 0, 1)\cdot \mathcal{X}_{i+1, i}(\tilde{m}) $$
formally, where $\mathcal{X}_{i,i}^{\ast}(\tilde{m}) \in M_{1\times n_i}(B\otimes_AR)$.
This equation should be interpreted as follows. 
We formally compute the right hand side then it is of the form $\pi\cdot X$, 
where $X$ involves $\tilde{m}_{i,i}^{\ast}$ and $\tilde{m}_{i,i}^{\ast\ast}$.
The left hand side $\mathcal{X}_{i,i}^{\ast}(\tilde{m})$ is then defined to be $X$.
Then by using an argument similar to the paragraph   just before Equation (\ref{ea20}) of Step (1),
 the image of $\mathcal{X}_{i,i}^{\ast}(\tilde{m})$ in $M_{1\times n_i}(B\otimes_AR)/(\pi\otimes 1)M_{1\times n_i}(B\otimes_AR)$
is independent of the choice of the lift $\tilde{m}$ of $m$.
Therefore, we may denote this image by $\mathcal{X}_{i,i}^{\ast}(m)$.
As for Equation (\ref{ea20}) of Step (1), we need to express $\mathcal{X}_{i,i}^{\ast}(m)$ as matrices.
By using conditions (5) and (6) of the description of an element of $(\mathrm{Ker~}\tilde{\varphi}/\tilde{M}^1)(R)$ given in
Remark \ref{ra5}, we have that
\begin{equation}
\mathcal{X}_{i,i}^{\ast}(m)=m_{i,i}^{\ast}+  m_{i,i}^{\ast\ast}\cdot \bar{h}_i+\mathcal{P}^{\ast}_i.
\end{equation}
Here, $\mathcal{P}^{\ast}_i$ is a polynomial with variables in the entries  of $m$ not including $m_{i,i}^{\ast}$ and $m_{i,i}^{\ast\ast}$.
Since $m$ actually belongs to $\mathrm{Ker~}\varphi(R)/\tilde{G}^1(R)$,
 we have the following equation by the argument made at the beginning of this paragraph:
  \begin{equation}\label{ea22}
 \mathcal{X}_{i,i}^{\ast}(m)=m_{i,i}^{\ast}+  m_{i,i}^{\ast\ast}\cdot \bar{h}_i+\mathcal{P}^{\ast}_i=\bar{f}_{i,i}^{\ast}=0.
  \end{equation}
Thus we get  polynomials $\mathcal{X}_{i,i}^{\ast}(m)$ on $\mathrm{Ker~}\tilde{\varphi}/\tilde{M}^1$,
vanishing on the subscheme $\mathrm{Ker~}\varphi/\tilde{G}^1$.\\

\item Assume that  $i$ is odd and that $L_i$ is \textit{free of type I}.
By Equations (\ref{ea7}), (\ref{ea8}),  (\ref{ea9}), and  (\ref{ea10}) which involve an element of $\tilde{M}^1(R)$,
each entry of $b_i', e_i', d_i'$ has $\pi$ as a factor and $f_i-f_i'$  has $\pi^2$ as a factor so that
 $b_i'\equiv b_i=0, e_i'\equiv e_i=0, d_i'\equiv d_i=0$  mod
$(\pi\otimes 1)(B\otimes_AR)$, and $f_i'= f_i=\bar{\gamma}_i$.
Let $\tilde{m}\in \mathrm{Ker~}\tilde{\varphi}(R)$ be a lift of $m$.
By using an argument similar to the paragraph   just before Equation (\ref{ea20}) of Step (1),
if we write the $(1, 2), (1,3), (2,3), (2,2)$-blocks of the $(i, i)$-block of the formal matrix product
$\sigma({}^t\tilde{m})\cdot h\cdot \tilde{m}$ as $\pi\cdot\xi^{(i-1)/2}\cdot \pi\mathcal{X}_{i,1,2}(\tilde{m})$,
$\pi\cdot\xi^{(i-1)/2}\cdot \mathcal{X}_{i,1,3}(\tilde{m})$,
$\pi\cdot\xi^{(i-1)/2}\cdot (1+\pi\mathcal{X}_{i,2,3}(\tilde{m}))$,
$\pi\cdot\xi^{(i-1)/2}\cdot \pi^3\mathcal{X}_{i,2,2}(\tilde{m})$,
respectively, where
$\mathcal{X}_{i,1,2}(\tilde{m}), \mathcal{X}_{i,1,3}(\tilde{m}) \in M_{(n_i-2)\times 1}(B\otimes_AR)$ 
and $\mathcal{X}_{i,2,3}(\tilde{m}), \mathcal{X}_{i,2,2}(\tilde{m}) \in B\otimes_AR$,
then the images of $\mathcal{X}_{i,1,2}(\tilde{m}), \mathcal{X}_{i,1,3}(\tilde{m})$
in $M_{(n_i-2)\times 1}(B\otimes_AR)/(\pi\otimes 1)M_{(n_i-2)\times 1}(B\otimes_AR)$ and
the images of $\mathcal{X}_{i,2,3}(\tilde{m}), \mathcal{X}_{i,2,2}(\tilde{m})$  in $(B\otimes_AR)/(\pi\otimes 1)(B\otimes_AR)$
are independent of the choice of the lift $\tilde{m}$ of $m$.
Therefore, we may denote these images by $\mathcal{X}_{i,1,2}(m)$, $\mathcal{X}_{i,1,3}(m)$,  $\mathcal{X}_{i,2,3}(m)$,
and $\mathcal{X}_{i,2,3}(m)$, respectively.
Note that $\mathcal{X}_{i,2,2}(\tilde{m})$ is indeed contained in $R$.
Thus $\mathcal{X}_{i,2,2}(\tilde{m})$ is naturally identified with $\mathcal{X}_{i,2,2}(m)$.
As for Equation (\ref{ea20}) of Step (1), we need to express $\mathcal{X}_{i,1,2}(m)$, $\mathcal{X}_{i,1,3}(m)$, and $\mathcal{X}_{i,2,3}(m)$
, and $\mathcal{X}_{i,2,2}(m)$
as matrices.
Recall that
\[
\pi^ih_i=\xi^{(i-1)/2}\cdot\pi \begin{pmatrix} a_i&0&0\\ 0 &\pi^3\bar{\gamma}_i&1 \\ 0 &-1 &\pi \end{pmatrix}.
\]
We  write
\[
m_{i,i}=\begin{pmatrix} id&\pi r_i& t_i\\  y_i&1+\pi x_i&u_i\\ \pi v_i&\pi z_i&1+\pi w_i \end{pmatrix}
\mathrm{~and~}
\tilde{m}_{i,i}=\begin{pmatrix} \tilde{s}_i&\pi \tilde{r}_i& \tilde{t}_i\\  \tilde{y}_i&1+\pi \tilde{x}_i&\tilde{u}_i
\\ \pi \tilde{v}_i&\pi \tilde{z}_i&1+\pi \tilde{w}_i \end{pmatrix}
\]
such that $\tilde{s}_i=\mathrm{id}$ mod $\pi \otimes 1$.
Then
\begin{multline}\label{ea23}
\sigma({}^t\tilde{m}_{i,i})h_i\tilde{m}_{i,i}=(-1)^{(i-1)/2}
\begin{pmatrix}\sigma({}^t\tilde{s}_i)&\sigma( {}^t \tilde{y}_i)&\sigma(\pi\cdot {}^t \tilde{v}_i)\\
\sigma(\pi\cdot {}^t \tilde{r}_i)&1+\sigma(\pi \tilde{x}_i)&\sigma( \pi\cdot \tilde{z}_i)\\
\sigma( {}^t \tilde{t}_i)&\sigma( {}^t \tilde{u}_i)&1+\sigma(\pi \tilde{w}_i) \end{pmatrix}\\
\begin{pmatrix} a_i&0&0\\ 0 &\pi^3\bar{\gamma}_i&1 \\ 0 &-1 &\pi \end{pmatrix}
\begin{pmatrix} \tilde{s}_i&\pi \tilde{r}_i& \tilde{t}_i\\  \tilde{y}_i&1+\pi \tilde{x}_i&\tilde{u}_i
\\ \pi \tilde{v}_i&\pi \tilde{z}_i&1+\pi \tilde{w}_i \end{pmatrix}.
\end{multline}
Then the $(1,2)$-block of $\sigma({}^t\tilde{m}_{i,i})h_i\tilde{m}_{i,i}$ is
$(-1)^{(i-1)/2}\pi\left(a_i\tilde{r}_i+\sigma({}^t\tilde{v}_i)+\sigma({}^t\tilde{y}_i)\tilde{z}_i+\pi(\ast)\right)$,
 the $(1, 3)$-block is
$(-1)^{(i-1)/2}\left(a_i\tilde{t}_i+\sigma({}^t\tilde{y}_i)+\pi(\ast\ast)\right)$,
  the $(2, 3)$-block is
 $(-1)^{(i-1)/2}(1+$ $\pi(-\sigma({}^t\tilde{r}_i)a_i\tilde{t}_i-\sigma(\tilde{x}_i)+
 \sigma(\tilde{z}_i)\tilde{u}_i+\pi^2(\ast\ast\ast)))$,
   and the $(2,2)$-block is
$(-1)^{(i-1)/2}\left(\pi^3\bar{\gamma}_i+\pi^3\left(\tilde{z}_i+\tilde{z}_i^2+\pi^2(\ast\ast\ast\ast)\right)\right)$
 for  certain polynomials $(\ast), (\ast\ast), (\ast\ast\ast), (\ast\ast\ast\ast)$.
Therefore, by observing the $(1, 2), (1,3), (2,3), (2,2)$-blocks of  Equation (\ref{ea1}) again, we have
\[\left \{
  \begin{array}{l}
  \mathcal{X}_{i,1,2}(m)=\bar{a}_ir_i+{}^tv_i+{}^ty_iz_i+\mathcal{P}^i_{1, 2};\\
  \mathcal{X}_{i,1,3}(m)=\bar{a}_it_i+{}^ty_i;\\
  \mathcal{X}_{i,2,3}(m)={}^tr_i\bar{a}_it_i+x_i+z_iu_i+\mathcal{P}^i_{2, 3}.\\
  \mathcal{X}_{i,2,2}(m)=\bar{\gamma}_i+z_i+z_i^2+1/2\left({}^tm_{i-1, i}'h_{i-1}m_{i-1, i}'+{}^tm_{i+1, i}'h_{i+1}m_{i+1, i}'\right)+\notag\\
\textit{       }  ~~~~~~~~~~ \left(\delta_{i-2}'(m_{i-2, i}^{\#})^2+\delta_{i+2}'(m_{i+2, i}^{\#})^2\right)+
\left(\delta_{i-3}(m_{i-3, i}^{\natural})^2+\delta_{i+3}(m_{i+3, i}^{\natural})^2\right).
    \end{array} \right.\]
Here, $\mathcal{P}^i_{1, 2}, \mathcal{P}^i_{2, 3}$ are suitable polynomials with variables in the entries  of $m_{i-1, i}, m_{i+1, i}$ and
\begin{itemize}
\item $m_{i\pm 1,i}^{\prime}$ is the $(n_i-1)$-th column vector of the matrix $m_{i\pm 1,i}$.
\item  $\delta_{j}^{\prime} = \left\{
  \begin{array}{l l}
  1    & \quad  \textit{if $j$ is odd and $L_j$ is \textit{free of type I}};\\
  0    &   \quad  \textit{otherwise}.
    \end{array} \right.$
\item $m_{i\pm 2, i}^{\#}$ is the $(n_{i\pm 2}\times n_i-1)$-th entry of $m_{i\pm 2, i}$.
\item $m_{i\pm 3, i}^{\natural}= \left\{
  \begin{array}{l l}
 \textit{the $n_{i\pm 3}\times (n_i-1)$-th entry of $m_{i\pm 3, i}$}    & \quad  \textit{if $L_{i \pm 3}$ is of type $I^o$};\\
 \textit{the $(n_{i\pm 3}-1)\times (n_i-1)$-th entry of $m_{i\pm 3, i}$}    & \quad  \textit{if $L_{i \pm 3}$ is of type $I^e$}.    \end{array} \right.$
 \end{itemize}
 In the right hand side of the equation including $\mathcal{X}_{i,2,2}(m)$,
 the term $1/2({}^tm_{i-1, i}'h_{i-1}m_{i-1, i}'+{}^tm_{i+1, i}'h_{i+1}m_{i+1, i}')$ should be interpreted as follows.
 We formally compute $1/2({}^tm_{i-1, i}'h_{i-1}m_{i-1, i}'+{}^tm_{i+1, i}'h_{i+1}m_{i+1, i}')$ and it is of the form $1/2(2X)$.
 Then the term $1/2({}^tm_{i-1, i}'h_{i-1}m_{i-1, i}'+{}^tm_{i+1, i}'h_{i+1}m_{i+1, i}')$ is defined as the modified $X$ by letting each term having $\pi$ as a factor in $X$ be zero.

These equations are considered in $(B\otimes_AR)/(\pi\otimes 1)(B\otimes_AR)$.
Since $m$ actually belongs to $\mathrm{Ker~}\varphi(R)/\tilde{G}^1(R)$,
 we have the following equations by the argument made at the beginning of this step:
  \begin{equation}\label{24}
\left \{
  \begin{array}{l}
 \mathcal{X}_{i,1,2}(m)=\bar{a}_ir_i+{}^tv_i+{}^ty_iz_i+\mathcal{P}^i_{1, 2}=\bar{b}_i=0;\\
\mathcal{X}_{i,1,3}(m)=\bar{a}_it_i+{}^ty_i=\bar{e}_i=0; \\
\mathcal{X}_{i,2,3}(m)={}^tr_i\bar{a}_it_i+x_i+z_iu_i+\mathcal{P}^i_{2, 3}=\bar{d}_i=0
   \end{array} \right.
  \end{equation}
  and
  \begin{multline}\label{24'}
  \mathcal{X}_{i,2,2}(m)=\bar{\gamma}_i+z_i+z_i^2+1/2\left({}^tm_{i-1, i}'h_{i-1}m_{i-1, i}'+{}^tm_{i+1, i}'h_{i+1}m_{i+1, i}'\right)+\\
\textit{       }  ~~~~~~~~~~ \left(\delta_{i-2}'(m_{i-2, i}^{\#})^2+\delta_{i+2}'(m_{i+2, i}^{\#})^2\right)+
\left(\delta_{i-3}(m_{i-3, i}^{\natural})^2+\delta_{i+3}(m_{i+3, i}^{\natural})^2\right)=\bar{f}_i=\bar{\gamma}_i.
  \end{multline}

Thus we get  polynomials $\mathcal{X}_{i,1,2}, \mathcal{X}_{i,1,3}, \mathcal{X}_{i,2,3}, \mathcal{X}_{i,2,2}-\bar{\gamma}_i$
on $\mathrm{Ker~}\tilde{\varphi}/\tilde{M}^1$,
vanishing on the subscheme $\mathrm{Ker~}\varphi/\tilde{G}^1$.\\

\item Assume that $i$ is even and that $L_i$ is \textit{of type} $\textit{I}^o$.
By Equation (\ref{13}) which involves an element of $\tilde{M}^1(R)$,
each entry of $b_i'$ has $\pi$ as a factor so that $b_i'\equiv b_i=0$ mod $(\pi\otimes 1)(B\otimes_AR)$.
Let $\tilde{m}\in \mathrm{Ker~}\tilde{\varphi}(R)$ be a lift of $m$.
By using an argument similar to the paragraph   just before Equation (\ref{ea20}) of Step (1),
 if we write the $(1, 2)$-block of the $(i, i)$-block of the formal matrix product
$\sigma({}^t\tilde{m})\cdot h\cdot \tilde{m}$ as $\xi^{i/2}\cdot \pi\mathcal{X}_{i,1,2}(\tilde{m})$,
where $\mathcal{X}_{i,1,2}(\tilde{m}) \in M_{(n_i-1)\times 1}(B\otimes_AR)$,
then the image of $\mathcal{X}_{i,1,2}(\tilde{m})$ in $M_{(n_i-1)\times 1}(B\otimes_AR)/(\pi\otimes 1)M_{(n_i-1)\times 1}(B\otimes_AR)$
is independent of the choice of the lift $\tilde{m}$ of $m$.
Therefore, we may denote this image by $\mathcal{X}_{i,1,2}(m)$.
As for Equation (\ref{ea20}) of Step (1), we need to express $\mathcal{X}_{i,1,2}(m)$ as matrices.
Recall that $\pi^ih_i=\xi^{i/2} \begin{pmatrix} a_i&0\\ 0 &1 +2\bar{\gamma}_i \end{pmatrix}
=\pi^i\cdot(-1)^{i/2}\begin{pmatrix} a_i&0\\ 0 &1 +2\bar{\gamma}_i \end{pmatrix}$. 
We  write $m_{i,i}$ as $\begin{pmatrix} id&\pi y_i\\ \pi v_i&1+\pi z_i \end{pmatrix}$
and  $\tilde{m}_{i,i}$ as $\begin{pmatrix} \tilde{s}_i&\pi \tilde{y}_i\\ \pi \tilde{v}_i&1+\pi \tilde{z}_i \end{pmatrix}$
such that $\tilde{s}_i=\mathrm{id}$ mod $\pi \otimes 1$.
Then
\begin{equation}\label{ea25'}
\sigma({}^t\tilde{m}_{i,i})h_i\tilde{m}_{i,i}=(-1)^{i/2}\begin{pmatrix}\sigma({}^t\tilde{s}_i)&\sigma(\pi\cdot {}^t \tilde{v}_i)\\ \sigma(\pi\cdot {}^t \tilde{y}_i)&1+\sigma(\pi \tilde{z}_i) \end{pmatrix}
 \begin{pmatrix} a_i&0\\ 0 &1 +2\bar{\gamma}_i \end{pmatrix} \begin{pmatrix} \tilde{s}_i&\pi \tilde{y}_i\\ \pi \tilde{v}_i&1+\pi \tilde{z}_i \end{pmatrix}.
\end{equation}
The $(1,2)$-block of $\sigma({}^t\tilde{m}_{i,i})h_i\tilde{m}_{i,i}$ is $(-1)^{i/2}\pi(a_i\tilde{y}_i-\sigma({}^t\tilde{v}_i))+\pi^2(\ast)$ for a certain polynomial $(\ast)$.
Therefore, by observing the $(1, 2)$-block  of Equation (\ref{ea1}), we have
\[
\mathcal{X}_{i,1,2}(m)=\bar{a}_iy_i+{}^tv_i+\mathcal{P}^i_{1, 2}.
\]
 Here, $\mathcal{P}^i_{1, 2}$ is a polynomial with variables in the entries  of $m_{i-1, i}, m_{i+1, i}$.
Note that this is an equation in $(B\otimes_AR)/(\pi\otimes 1)(B\otimes_AR)$.
Since $m$ actually belongs to $\mathrm{Ker~}\varphi(R)/\tilde{G}^1(R)$,
 we have the following equation by the argument made at the beginning of this paragraph:
  \begin{equation}\label{ea25}
  \mathcal{X}_{i,1,2}(m)=\bar{a}_iy_i+{}^tv_i+\mathcal{P}^i_{1, 2}=\bar{b}_i=0.
  \end{equation}
Thus we get  polynomials $\mathcal{X}_{i,1,2}$ on $\mathrm{Ker~}\tilde{\varphi}/\tilde{M}^1$,
vanishing on the subscheme $\mathrm{Ker~}\varphi/\tilde{G}^1$.\\

\item  Assume that $i$ is even and that $L_i$ is \textit{of type} $\textit{I}^e$.
The argument used in this step is similar to that of the above Step (4).
By Equations (\ref{ea14}), (\ref{ea15}),  (\ref{ea16}), and (\ref{ea17}) which involve an element of $\tilde{M}^1(R)$,
each entry of $b_i', e_i', d_i', f_i'$ has $\pi$ as a factor so that $b_i'\equiv b_i=0, e_i'\equiv e_i=0, d_i'\equiv d_i=0, f_i'\equiv f_i=0$ mod
$(\pi\otimes 1)(B\otimes_AR)$.
Let $\tilde{m}\in \mathrm{Ker~}\tilde{\varphi}(R)$ be a lift of $m$.
By using an argument similar to the paragraph   just before Equation (\ref{ea20}) of Step (1),
if we write the $(1, 2), (1,3), (2,3), (2,2)$-blocks of the $(i, i)$-block of the formal matrix product
$\sigma({}^t\tilde{m})\cdot h\cdot \tilde{m}$ as $\xi^{i/2}\cdot \mathcal{X}_{i,1,2}(\tilde{m})$,
$\xi^{i/2}\cdot \pi\mathcal{X}_{i,1,3}(\tilde{m})$, $\xi^{i/2}\cdot (1+\pi\mathcal{X}_{i,2,3}(\tilde{m}))$,
$\xi^{i/2}\cdot (1+2 \mathcal{X}_{i,2,2}(\tilde{m}))$, respectively, where
$\mathcal{X}_{i,1,2}(\tilde{m}), \mathcal{X}_{i,1,3}(\tilde{m}) \in M_{(n_i-2)\times 1}(B\otimes_AR)$ 
and $\mathcal{X}_{i,2,3}(\tilde{m}), \mathcal{X}_{i,2,2}(\tilde{m}) \in B\otimes_AR$,
then the images of $\mathcal{X}_{i,1,2}(\tilde{m}), \mathcal{X}_{i,1,3}(\tilde{m})$ 
in $M_{(n_i-2)\times 1}(B\otimes_AR)/(\pi\otimes 1)M_{(n_i-2)\times 1}(B\otimes_AR)$
and the images of $\mathcal{X}_{i,2,3}(\tilde{m}), \mathcal{X}_{i,2,2}(\tilde{m})$ in $(B\otimes_AR)/(\pi\otimes 1)(B\otimes_AR)$
are independent of the choice of the lift $\tilde{m}$ of $m$.
Therefore, we may denote these images by $\mathcal{X}_{i,1,2}(m)$, $\mathcal{X}_{i,1,3}(m)$,  $\mathcal{X}_{i,2,3}(m)$, and $\mathcal{X}_{i,2,2}(m)$ respectively.
Note that $\mathcal{X}_{i,2,2}(\tilde{m})$ is indeed contained in $R$.
Thus $\mathcal{X}_{i,2,2}(\tilde{m})$ is naturally identified with $\mathcal{X}_{i,2,2}(m)$.

As for Equation (\ref{ea20}) of Step (1),
we need to express $\mathcal{X}_{i,1,2}(m)$, $\mathcal{X}_{i,1,3}(m)$, $\mathcal{X}_{i,2,3}(m)$, and $\mathcal{X}_{i,2,2}(m)$ as matrices.
Recall that $\pi^ih_i=\xi^{i/2} \begin{pmatrix} a_i&0&0\\ 0 &1&1 \\ 0 &1 &2\bar{\gamma}_i \end{pmatrix}
=\pi^i\cdot(-1)^{i/2}\begin{pmatrix} a_i&0&0\\ 0 &1&1 \\ 0 &1 &2\bar{\gamma}_i \end{pmatrix}$.
We  write $m_{i,i}$ as $\begin{pmatrix} id&r_i&\pi t_i\\ \pi y_i&1+\pi x_i&\pi z_i\\ v_i&u_i&1+\pi w_i \end{pmatrix}$
and $\tilde{m}_{i,i}$ as $\begin{pmatrix} \tilde{s}_i&\tilde{r}_i&\pi \tilde{t}_i\\ \pi \tilde{y}_i&1+\pi \tilde{x}_i&\pi \tilde{z}_i\\ \tilde{v}_i&\tilde{u}_i&1+\pi \tilde{w}_i \end{pmatrix}$ such that $\tilde{s}_i=\mathrm{id}$ mod $\pi \otimes 1$.
Then
\begin{multline}\label{ea26}
\sigma({}^t\tilde{m}_{i,i})h_i\tilde{m}_{i,i}=(-1)^{i/2}\begin{pmatrix}\sigma({}^t\tilde{s}_i)&\sigma(\pi\cdot {}^t \tilde{y}_i)&\sigma( {}^t \tilde{v}_i)\\
\sigma({}^t \tilde{r}_i)&1+\sigma(\pi \tilde{x}_i)&\sigma(\tilde{u}_i)\\\sigma(\pi\cdot {}^t \tilde{t}_i)&\sigma(\pi\cdot {}^t \tilde{z}_i)&1+\sigma(\pi \tilde{w}_i) \end{pmatrix}\\
 \begin{pmatrix} a_i&0&0\\ 0 &1&1 \\ 0 &1 &2c_i \end{pmatrix}
\begin{pmatrix} \tilde{s}_i&\tilde{r}_i&\pi \tilde{t}_i\\ \pi \tilde{y}_i&1+\pi \tilde{x}_i&\pi \tilde{z}_i\\ \tilde{v}_i&\tilde{u}_i&1+\pi \tilde{w}_i \end{pmatrix}.
\end{multline}
The  $(1,2)$-block of $\sigma({}^t\tilde{m}_{i,i})h_i\tilde{m}_{i,i}$ is
$(-1)^{i/2}(a_i\tilde{r}_i+\sigma({}^t\tilde{v}_i))+\pi(\ast)$,
 the $(1, 3)$-block is $(-1)^{i/2}\pi(a_i\tilde{t}_i-\sigma({}^t\tilde{y}_i)+\sigma({}^t\tilde{v}_i)\tilde{z}_i)
 +\pi^2(\ast\ast)$,
  the $(2, 3)$-block is $(-1)^{i/2}(1+\pi(\sigma({}^t\tilde{r}_i)a_i\tilde{t}_i-\sigma(\tilde{x}_i)+\tilde{z}_i+\tilde{w}_i+\sigma(\tilde{u}_i)\tilde{z}_i)
 +\pi^2(\ast\ast\ast))$,
 and the  $(2, 2)$-block is
$(-1)^{i/2}(1+\sigma({}^t\bar{r}_i)a_i\bar{r}_i+2\bar{u}_i+2\bar{\gamma}_i\bar{u}_i^2-2\bar{x}_i^2$ $+\pi^4(\ast\ast\ast\ast))$
 for  certain polynomials $(\ast), (\ast\ast), (\ast\ast\ast), (\ast\ast\ast\ast)$.
Therefore, by observing the $(1, 2), (1,3), (2,3), (2,2)$-blocks of  Equation (\ref{ea1}) again, we have
\[\left \{
  \begin{array}{l}
  \mathcal{X}_{i,1,2}(m)=\bar{a}_ir_i+{}^tv_i;\\
  \mathcal{X}_{i,1,3}(m)=\bar{a}_i t_i+{}^ty_i+{}^tv_iz_i+\mathcal{P}^i_{1, 3};\\
  \mathcal{X}_{i,2,3}(m)={}^tr_i\bar{a}_it_i+x_i+z_i+ w_i+u_iz_i+\mathcal{P}^i_{2, 3};\\
  \mathcal{X}_{i,2,2}(m)=u_i+\bar{\gamma}_iu_i^2+x_i^2+1/2\cdot{}^tr_i\bar{a}_ir_i+\left(\delta_{i-2}(m_{i-2, i}^{\natural})^2+\delta_{i+2}(m_{i+2, i}^{\natural})^2\right).
    \end{array} \right.
    \]
Here, $\mathcal{P}^i_{1, 3}, \mathcal{P}^i_{2, 3}$ are suitable polynomials with variables in the entries  of $m_{i-1, i}, m_{i+1, i}$, and
$$m_{i\pm 2, i}^{\natural}= \left\{
  \begin{array}{l l}
 \textit{the $n_{i\pm 2}\times (n_i-1)$-th entry of $m_{i\pm 2, i}$}    & \quad  \textit{if $L_{i \pm 2}$ is of type $I^o$};\\
 \textit{the $(n_{i\pm 2}-1)\times (n_i-1)$-th entry of $m_{i\pm 2, i}$}    & \quad  \textit{if $L_{i \pm 2}$ is of type $I^e$}.    \end{array} \right.$$
 In the right hand side of the equation including $\mathcal{X}_{i,2,2}(m)$,
 the term $1/2\cdot{}^tr_i\bar{a}_ir_i$ should be interpreted as follows.
 We formally compute $1/2\cdot{}^tr_i\bar{a}_ir_i$ and it is of the form $1/2(2X)$.
 Then the term $1/2\cdot{}^tr_i\bar{a}_ir_i$ is defined as the modified $X$ by letting each term having $\pi$ as a factor in $X$ be zero.

These equations are considered in $(B\otimes_AR)/(\pi\otimes 1)(B\otimes_AR)$.
Since $m$ actually belongs to $\mathrm{Ker~}\varphi(R)/\tilde{G}^1(R)$,
 we have the following equations by the argument made at the beginning of this step:
  \begin{equation}\label{ea27}
\left \{
  \begin{array}{l}
 \mathcal{X}_{i,1,2}(m)=\bar{a}_ir_i+{}^tv_i=\bar{b}_i=0;\\
\mathcal{X}_{i,1,3}(m)=\bar{a}_i t_i+{}^ty_i+{}^tv_iz_i+\mathcal{P}^i_{1, 3}=\bar{e}_i=0; \\
\mathcal{X}_{i,2,3}(m)={}^tr_i\bar{a}_it_i+x_i+z_i+ w_i+u_iz_i+\mathcal{P}^i_{2, 3}=\bar{d}_i=0;\\
\mathcal{X}_{i,2,2}(m)=u_i+\bar{\gamma}_iu_i^2+x_i^2+1/2\cdot{}^tr_i\bar{a}_ir_i+\left(\delta_{i-2}(m_{i-2, i}^{\natural})^2+\delta_{i+2}(m_{i+2, i}^{\natural})^2\right)=\bar{f}_i=0.
    \end{array} \right.
  \end{equation}
Thus we get  polynomials $\mathcal{X}_{i,1,2}, \mathcal{X}_{i,1,3}, \mathcal{X}_{i,2,3}, \mathcal{X}_{i,2,2}$
on $\mathrm{Ker~}\tilde{\varphi}/\tilde{M}^1$, vanishing on the subscheme $\mathrm{Ker~}\varphi/\tilde{G}^1$.\\




\item   Assume that $i$ is even and that $L_i$ is \textit{of type I}.
By Equation (\ref{ea19}) which involves an element of $\tilde{M}^1(R)$,
we have $c_i' = c_i+z_i'$.
Since $c_i'\not\equiv c_i$, we cannot follow the argument used in the previous steps in the case of
the $(2, 2)$-block (when $L_i$ is \textit{of type $I^o$}) or the $(3,3)$-block (when $L_i$ is \textit{of type $I^e$})
 of the $(i, i)$-block of $h\circ \tilde{m}=\sigma({}^t\tilde{m})\cdot h\cdot \tilde{m}$.
Note that $c_i'$ and $z_i'$ are indeed contained in $R$ and $R$ is a $\kappa$-algebra.
Thus $c_i'$ and $z_i'$ mod $(\pi\otimes 1)(B\otimes_AR)$ are naturally identified with themselves, respectively.

 On the other hand,  we choose an even integer $j$ such that $L_j$ is \textit{of type I} and $L_{j+2}$ is \textit{of type II}.
For such $j$, there is a nonnegative integer $m_j$ such that $L_{j-2l}$ is \textit{of type I}
for every $l$ with $0\leq l \leq m_j$ and $L_{j-2(m_j+1)}$ is \textit{of type II} (cf. Step (vi) of the proof of Theorem \ref{ta4}).
Then we have
\begin{equation}\label{ea28}
\sum_{0\leq l \leq m_j}c_{j-2l}'= \sum_{0\leq l \leq m_j}\left(c_{j-2l}+z_{j-2l}'\right)
= \sum_{0\leq l \leq m_j}c_{j-2l} =0
\end{equation}
since the sum of equations $\sum_{0\leq l \leq m_j}\mathcal{Z}_{j-2l}'$ equals $\sum_{0\leq l \leq m_j}z_{j-2l}'=0$ as mentioned in Step (vi) of the proof of Theorem \ref{ta4} and $c_i=0$ by Remark \ref{r33}.(2).

We now apply the argument used in the previous steps to the above equation (\ref{ea28}).
Let $\tilde{m}\in \mathrm{Ker~}\tilde{\varphi}(R)$ be a lift of $m$.
By using an argument similar to the paragraph   just before Equation (\ref{ea20}) of Step (1),
if we write the $(2, 2)$-block (when $L_i$ is \textit{of type $I^o$}) or the $(3,3)$-block (when $L_i$ is \textit{of type $I^e$})
 of the $(i, i)$-block of $h\circ \tilde{m}=\sigma({}^t\tilde{m})\cdot h\cdot \tilde{m}$
as $\xi^{i/2}\cdot (1+2\bar{\gamma}_i+4\mathcal{F}_i(\tilde{m}))$ or
$\xi^{i/2}\cdot (2\bar{\gamma}_i+4\mathcal{F}_{i}(\tilde{m}))$ respectively, then
the image of $\sum_{0\leq l \leq m_j}\mathcal{F}_{_{j-2l}}(\tilde{m})$ in
 $(B\otimes_AR)/(\pi\otimes 1)(B\otimes_AR)$
is independent of the choice of the lift $\tilde{m}$ of $m$.
Therefore, we may denote this image by $\sum_{0\leq l \leq m_j}\mathcal{F}_{_{j-2l}}(m)$.
Note that $\sum_{0\leq l \leq m_j}\mathcal{F}_{_{j-2l}}(\tilde{m})$ is indeed contained in $R$.
Thus $\sum_{0\leq l \leq m_j}\mathcal{F}_{_{j-2l}}(\tilde{m})$ is naturally identified with
$\sum_{0\leq l \leq m_j}\mathcal{F}_{_{j-2l}}(m)$.

As for Equation (\ref{ea20}) of Step (1), we need to express $\sum_{0\leq l \leq m_j}\mathcal{F}_{_{j-2l}}(m)$ precisely.
Each entry $\tilde{x}_i$ of $(\tilde{m_{i,j}}, \tilde{s_i} \cdots \tilde{w_i})$
is expressed  as $\tilde{x}_i=(\tilde{x}_i)_1+\pi (\tilde{x}_i)_2$.
Then we have
\begin{multline}\label{ea30}
 \mathcal{F}_i(\tilde{m}) = \left\{\begin{array}{l l}
 \left(1/2\cdot{}^t(\tilde{y}_i)_1a_i(\tilde{y}_i)_1+\bar{\gamma}_i(\tilde{z}_i)_1^2\right)  & \quad  \textit{if $L_i$ is of type $I^o$};\\
 \left(1/2\cdot{}^t(\tilde{t}_i)_1a_i(\tilde{t}_i)_1+\bar{\gamma}_i(\tilde{w}_i)_1^2+(\tilde{z}_i)_1(\tilde{w}_i)_1\right)
  & \quad  \textit{if $L_i$ is of type $I^e$}  \end{array}\right.\\
 +(\tilde{z}_i)_2+(\tilde{z}_i)_2^2+(\tilde{z}_i)_1(\delta_{i-2}(\tilde{k}_{i-2, i})_1+\delta_{i+2}(\tilde{k}_{i+2,i})_1)+
 \delta_{i-2}\delta_{i+2}(\tilde{k}_{i-2, i})_1(\tilde{k}_{i+2, i})_1+\\
 \left({}^t(\tilde{m}_{i-1, i}')_1\cdot  h_{i-1}\cdot (\tilde{m}_{i-1, i}')_2+{}^t(\tilde{m}_{i+1, i}')_1\cdot  h_{i+1}\cdot (\tilde{m}_{i+1, i}')_2\right)+\\
 \frac{1}{2}\left({}^t(\tilde{m}_{i-2, i})_1'\cdot  a_{i-2}\cdot (\tilde{m}_{i-2, i})_1'+{}^t(\tilde{m}_{i+2, i})_1'\cdot a_{i+2}\cdot (\tilde{m}_{i+2, i})_1'\right)\\
  +\left\{\begin{array}{l l}
 \delta_{i-2}\left(\bar{\gamma}_{i-2}(\tilde{k}_{i-2, i})_1^2+(\tilde{k}_{i-2, i})_2^2\right) & \quad \textit{if $L_{i-2}$ is of type $I^o$}; \\
    \delta_{i-2}\left(\bar{\gamma}_{i-2}(\tilde{k}_{i-2, i})_1'^2+ (\tilde{k}_{i-2, i})_2^2+  (\tilde{k}_{i-2, i})_1\cdot (\tilde{k}_{i-2, i})_1'\right)
 & \quad \textit{if $L_{i-2}$ is of type $I^e$} \end{array}\right.\\
+ \left\{\begin{array}{l l}
 \delta_{i+2}\left(\bar{\gamma}_{i+2}(\tilde{k}_{i+2, i})_1^2+(\tilde{k}_{i+2, i})_2^2\right) & \quad  \textit{if $L_{i+2}$ is of type $I^o$};\\
 \delta_{i+2}\left(\bar{\gamma}_{i+2}((\tilde{k}_{i+2, i})_1')^2+ (\tilde{k}_{i+2, i})_2^2+  (\tilde{k}_{i+2, i})_1\cdot (\tilde{k}_{i+2, i})_1'\right)
 & \quad  \textit{if $L_{i+2}$ is of type $I^e$} \end{array}\right.\\
+ \left(\delta_{i-3}'(\tilde{k}_{i-3, i})_1^2+ \delta_{i+3}'(\tilde{k}_{i+3, i})_1^2\right)+
 \left(\delta_{i-4}(\tilde{k}_{i-4, i})_1^2+ \delta_{i+4}(\tilde{k}_{i+4, i})_1^2\right).
 \end{multline}
 Here,
 \begin{itemize}
 \item $1/2\cdot{}^t(\tilde{y}_i)_1a_i(\tilde{y}_i)_1$ in the first line should be interpreted as follows.
 We formally compute $1/2\cdot{}^t(\tilde{y}_i)_1a_i(\tilde{y}_i)_1$   and it is of the form $1/2(2X)$.
 Then the term $1/2\cdot{}^t(\tilde{y}_i)_1a_i(\tilde{y}_i)_1$ is defined as the modified $X$ by letting each term having $\pi^2$ as a factor in $X$ be zero.
 We interpret all terms having $1/2$ as a factor appeared below in this manner.

 \item $\tilde{m}_{i\pm 1, i}'$ is the last column vector of $\tilde{m}_{i\pm 1, i}$.

 \item $\tilde{m}_{i\pm 2, i}', \tilde{k}_{i\pm 2, i}, \tilde{k}_{i\pm 2, i}'$ are such that the last column vector of  $\tilde{m}_{i\pm 2, i}$ is
 \[ \left\{
  \begin{array}{l l}
\tilde{m}_{i\pm 2, i}'  & \quad \textit{if $L_{i\pm 2}$ is of type II};\\
 {}^t({}^t\tilde{m}_{i\pm 2, i}', \tilde{k}_{i\pm 2, i})  & \quad \textit{if $L_{i\pm 2}$ is of type $I^o$};\\
 {}^t({}^t\tilde{m}_{i\pm 2, i}', \tilde{k}_{i\pm 2, i}, \tilde{k}_{i\pm 2, i}') & \quad \textit{if $L_{i\pm 2}$ is of type $I^e$}.
      \end{array} \right. \]
\item $\delta_{i\pm 3}'=1$ if $L_{i\pm 3}$ is \textit{free of type I} and $0$ otherwise.

\item  $\tilde{k}_{j, i}$  is the $(n_{j}, n_i)^{th}$-entry (resp. $(n_{j}-1, n_i)^{th}$-entry) of the matrix $\tilde{m}_{j, i}$
if $L_{j}$ is \textit{of type} $\textit{I}^o$ with $j$ even or \textit{free of type I with j odd} (resp. \textit{of type $I^e$ with j even}).
\end{itemize}

Note that Condition (d) of the description of an element of $\tilde{M}(R)$ given at the paragraph following Lemma \ref{la1},
$\tilde{z}_i+\delta_{i-2}\tilde{k}_{i-2, i}+\delta_{i+2}\tilde{k}_{i+2, i}=\pi \tilde{z}_i^{\ast}$,
yields the following two formal equations:
\[ \left\{
  \begin{array}{l}
(\tilde{z}_i)_1+   \delta_{i-2}(\tilde{k}_{i-2, i})_1+\delta_{i+2}(\tilde{k}_{i+2, i})_1=2 (\tilde{z}_i^{\ast})_2;\\
(\tilde{z}_i)_2+ \delta_{i-2}(\tilde{k}_{i-2, i})_2+\delta_{i+2}(\tilde{k}_{i+2, i})_2= (\tilde{z}_i^{\ast})_1.
      \end{array} \right. \]
Using these,
 the sum of $\mathcal{F}_{j-2l}(\tilde{m})$ is the following:
\begin{multline}\label{ea31}
\sum_{l=0}^{m_j}\mathcal{F}_{j-2l}(\tilde{m})=\sum_{l=0}^{m_j}\left((\tilde{z}_{j-2l}^{\ast})_1+(\tilde{z}_{j-2l}^{\ast})_1^2   \right)+\\
 \left({}^t(\tilde{m}_{j-2m_j-1, j-2m_j}')_1\cdot  h_{j-2m_j-1}\cdot (\tilde{m}_{j-2m_j-1, j-2m_j}')_2+{}^t(\tilde{m}_{j+1, j}')_1\cdot  h_{j+1}\cdot (\tilde{m}_{j+1, j}')_2\right)+\\
 \frac{1}{2}\left({}^t(\tilde{m}_{j-2m_j-2, j-2m_j})_1'\cdot  a_{j-2m_j-2}\cdot (\tilde{m}_{j-2m_j-2, j-2m_j})_1'+{}^t(\tilde{m}_{j+2, j})_1'\cdot a_{j+2}\cdot (\tilde{m}_{j+2, j})_1'\right)\\
+ \left(\delta_{j-2m_j-3}'(\tilde{k}_{j-2m_j-3, j-2m_j})_1^2+ \delta_{j+3}'(\tilde{k}_{j+3, j})_1^2\right)+
 \left(\delta_{j-2m_j-4}(\tilde{k}_{j-2m_j-4, j-2m_j})_1^2+ \delta_{j+4}(\tilde{k}_{j+4, j})_1^2\right)+\\
 \sum_{\textit{$L_{j-2l}$ : of type $I^e$,
 $l=0$}}^{m_j}\bar{\gamma}_{j-2l}\left((\tilde{x}_i)_1^2+1/2\cdot{}^t(\tilde{r}_i)_1a_i(\tilde{r}_i)_1+
 \left(\delta_{i-2}(\tilde{m}_{i-2, i}^{\natural})_1^2+\delta_{i+2}(\tilde{m}_{i+2, i}^{\natural})_1^2\right)\right).
\end{multline}
Here, $(\tilde{m}_{i\pm 2, i}^{\natural})_1$ is as explained in the above Step (5).
Note that both $(\tilde{m}_{j-2m_j-1, j-2m_j}')_1$ and $(\tilde{m}_{j+1, j}')_1$ are zero in $R$
because of Condition (f) of the description of an element of $\tilde{M}(R)$ given at the paragraph following Lemma \ref{la1}
since $L_{j-2m_j-2}$ and $L_{j+2}$ are \textit{of type II}.

The proof of the above $\sum_{l=0}^{m_j}\mathcal{F}_{j-2l}$ is basically similar to that of Lemma A.7 of \cite{C2} and we skip it.
It is mainly based on Equation (\ref{ea2}), especially when $j-i=1\textit{ or }2$.
We also have to incorporate Equations (\ref{ea25}) and (\ref{ea27})
together with the two formal equations about $(\tilde{z_i^{\ast}})_1$ and $(\tilde{z_i^{\ast}})_2$ given just before Equation (\ref{ea31}),
and Conditions (e), (f) of the description of an element of $\tilde{M}(R)$ given at the paragraph following Lemma \ref{la1}.\\

Since $\sum_{l=0}^{m_j}\mathcal{F}_{j-2l}(\tilde{m})$ does not include any factor of type $(\tilde{x}_i)_2$,
we finally obtain $\sum_{l=0}^{m_j}\mathcal{F}_{j-2l}(m)$ as follows:
\begin{multline}\label{ea32}
 0=\sum_{0\leq l \leq m_j}c_{j-2l}=\sum_{l=0}^{m_j}\mathcal{F}_{j-2l}(m)=\sum_{l=0}^{m_j}\left(z_{j-2l}^{\ast}+(z_{j-2l}^{\ast})^2   \right)+\\
 \sum_{\textit{$L_{j-2l}$ : of type $I^e$, $l=0$}}^{m_j}\bar{\gamma}_{j-2l}\left(x_i^2+1/2\cdot{}^tr_i\bar{a}_ir_i+\left(\delta_{i-2}(m_{i-2, i}^{\natural})^2+\delta_{i+2}(m_{i+2, i}^{\natural})^2\right)\right)+\\
 \frac{1}{2}\left({}^t(m_{j-2m_j-2, j-2m_j})'\cdot  a_{j-2m_j-2}\cdot (m_{j-2m_j-2, j-2m_j})'+{}^t(m_{j+2, j})'\cdot a_{j+2}\cdot (m_{j+2, j})'\right)+\\
 \left(\delta_{j-2m_j-3}'(k_{j-2m_j-3, j-2m_j})^2+ \delta_{j+3}'(k_{j+3, j})^2\right)+
 \left(\delta_{j-2m_j-4}(k_{j-2m_j-4, j-2m_j})^2+ \delta_{j+4}(k_{j+4, j})^2\right).
\end{multline}
Here, notations as as explained in Equations (\ref{ea30}) and (\ref{ea31}).
Thus we get  polynomials $\sum_{l=0}^{m_j}\mathcal{F}_{j-2l}$
on $\mathrm{Ker~}\tilde{\varphi}/\tilde{M}^1$,
vanishing on the subscheme $\mathrm{Ker~}\varphi/\tilde{G}^1$.\\

\item
This step is similar to Step (5) in the proof of Theorem A.6 of \cite{C2}.
Recall that   $\mathcal{B}$ is the set of  integers $j$ such that $L_j$ is \textit{of type I} and $L_{j+2}, L_{j+3}, L_{j+4}$
 (resp. $L_{j-1}, L_{j+1},$ $L_{j+2}, L_{j+3}$) are \textit{of type II} if $j$ is even (resp. odd).
 We choose $j\in \mathcal{B}$.

For   $j\in \mathcal{B}$, there is a nonnegative integer $k_j$ such that $L_{j-k_j}$ is \textit{of type I},
$j-l\notin\mathcal{B}$ for all $l$ with $0 < l \leq k_j$, and
$L_{j-k_j-2}, L_{j-k_j-3}, L_{j-k_j-4}$ (resp. $L_{j-k_j+1},$ $L_{j-k_j-1},$ $L_{j-k_j-2},$ $L_{j-k_j-3}$) are \textit{of type II}
if $j-k_j$ is even (resp. odd).

We  denote  $\mathcal{X}_{i,2,2}-\bar{\gamma}_i$ by $\mathcal{F}_{i}$   when $i$ is odd and $L_i$ is \textit{free of type I} 
(cf. Equation (\ref{24'})),
and denote $\mathcal{X}_{i,2,2}$ by $\mathcal{E}_{i}$ when $i$ is even and $L_i$ is \textit{of type $I^e$} (cf. Equation (\ref{ea27})).
Then, based on Equations (\ref{24'}) and (\ref{ea32}) and $\mathcal{X}_{i,2,2}$ in Equation (\ref{ea27}), the sum of equations
$$\sum_{l=0}^{k_j}\mathcal{F}_{j-l}(m)+
\sum_{\textit{ $L_{j-l}$ : type $\textit{I}^e$, }l=0}^{k_j}
\bar{\gamma}_{j-l} \cdot \mathcal{E}_{j-l}(m)$$
equals
\begin{multline}\label{32'}
\sum_{l=0}^{k_j}\left(z_{j-l}^{\ast}+(z_{j-2l}^{\ast})^2   \right)+
\sum_{l=0}^{k_j}
\left(\bar{\gamma}_{j-l}u_{j-l}^{\ast}+(\bar{\gamma}_{j-l}u_{j-l}^{\ast})^2\right)=\\
\left(\sum_{l=0}^{k_j}z_{j-l}^{\ast}+\sum_{l=0}^{k_j}
\bar{\gamma}_{j-l}u_{j-l}^{\ast}   \right)
\left(\left(\sum_{l=0}^{k_j}z_{j-l}^{\ast}+\sum_{l=0}^{k_j}
\bar{\gamma}_{j-l}u_{j-l}^{\ast}\right)+1\right)=0.
\end{multline}

Here,
\[z_{j-l}^{\ast} = \left\{
  \begin{array}{l l}
  z_{j-l}^{\ast}    & \quad  \textit{if $j-l$ is even and $L_{j-l}$ is \textit{of type I}};\\
    z_{j-l}    & \quad  \textit{if $j-l$ is odd and $L_{j-l}$ is \textit{free of type I}};\\
     0    & \quad  \textit{otherwise},
          \end{array} \right.\]
and
\[u_{j-l}^{\ast} = \left\{
  \begin{array}{l l}
  u_{j-l}    & \quad  \textit{if $j-l$ is even and $L_{j-l}$ is \textit{of type $I^e$}};\\
     0    & \quad  \textit{otherwise}.
          \end{array} \right.\]


The proof of the above is also   basically similar to that of Lemma A.7 in \cite{C2} and we skip it.
It is mainly based on Equation (\ref{ea20}), especially when $j-i=2\textit{ or }3$.\\
\end{enumerate}

Let $G^{\ddag}$ be the subfunctor of $ \mathrm{Ker~}\tilde{\varphi}/\tilde{M}^1$ consisting of those $m$ 
satisfying Equations (\ref{ea20}), (\ref{ea22}), (\ref{24}), (\ref{24'}), (\ref{ea25}), (\ref{ea27}), and (\ref{ea32}). 
Note that such $m$ also satisfies Equation (\ref{32'}).
In Lemma \ref{la8} below, we will prove that
$G^{\ddag}$ is represented by a smooth  closed subscheme of $ \mathrm{Ker~}\tilde{\varphi}/\tilde{M}^1$
and is isomorphic to $ \mathbb{A}^{l^{\prime}}\times (\mathbb{Z}/2\mathbb{Z})^{\beta}$ as a $\kappa$-variety,
 where $\mathbb{A}^{l^{\prime}}$ is an affine space of dimension
\begin{align*}
 l^{\prime}=\sum_{i<j}n_in_j +\sum_{i:\mathrm{even~and~} L_i:\textit{of type }I^e}(n_i-1) +
 \sum_{i:\mathrm{odd~and~}L_i:\textit{free of type I}}(2n_i-2) \notag \\
 - \sum_{i:\mathrm{even~and~} L_i:\textit{bound of type II}}n_i +\#\{i:\textit{$i$ is even and $L_i$ is of type I}\} ~~~~~~~~~~~~~~~\notag \\ 
-\#\{i:\textit{$i$ is even, $L_i$ is  of type I and  $L_{i+2}$ is of type II}\}. ~~~~~~~~~~~~~~
\end{align*}

 For ease of notation, let $G^{\dag}=\mathrm{Ker~}\varphi/\tilde{G}^1$.
 Since $G^{\dag}$ and $G^{\ddag}$ are both closed subschemes of $ \mathrm{Ker~}\tilde{\varphi}/\tilde{M}^1$
 and $G^{\dag}(\bar{\kappa})\subset G^{\ddag}(\bar{\kappa})$, $(G^{\dag})_{\mathrm{red}}$ is a closed subscheme of $(G^{\ddag})_{\mathrm{red}}=G^{\ddag}$.
 It is easy to check that $\mathrm{dim~} G^{\dag} = \mathrm{dim~}G^{\ddag}$
 since  $\mathrm{dim~} G^{\dag} =\mathrm{dim~}\mathrm{Ker~}\varphi - \mathrm{dim~}\tilde{G}^1=l-\mathrm{dim~}\tilde{G}^1$
 and $\mathrm{dim~}G^{\ddag}=l'=l-\mathrm{dim~}\tilde{G}^1$.
 Here, $\mathrm{dim~}\mathrm{Ker~}\varphi = l$ is given in Lemma \ref{l46}
 and dim $\tilde{G}^1$ is given in Theorem \ref{ta4}.\\

We claim that $(G^{\dag})_{\mathrm{red}}$ contains  at least one (closed) point of each connected component of $G^{\ddag}$.
Choose an integer $j$ such that $L_j$ is \textit{of type I} and $L_{j+2}, L_{j+3}, L_{j+4}$
 (resp. $L_{j-1}, L_{j+1},$ $L_{j+2}, L_{j+3}$) are \textit{of type II} if $j$ is even (resp. odd).
Consider the closed subgroup scheme $F_j$ of $\tilde{G}$ defined by the following equations:
\begin{itemize}
\item $m_{i,k}=0$ \textit{if $i\neq k$};
\item $m_{i,i}=\mathrm{id}, z_i^{\ast}=0, m_{i,i}^{\ast}=0, m_{i,i}^{\ast\ast}=0$ \textit{if $i\neq j$};
\item and for $m_{j,j}$,
\[\left \{
  \begin{array}{l l}
  s_j=\mathrm{id~}, y_j=0, v_j=0, z_j=\pi z_j^{\ast} & \quad  \textit{if $i$ is even and $L_i$ is \textit{of type} $\textit{I}^o$};\\
  s_j=\mathrm{id~}, r_j=t_j=y_j=v_j=u_j=w_j=0, z_j=\pi z_j^{\ast}  & \quad \textit{if $i$ is even and  $L_i$ is \textit{of type} $\textit{I}^e$};\\
  s_j=\mathrm{id~}, r_j=t_j=y_j=v_j=u_j=w_j=0 & \quad \textit{if $i$ is odd and  $L_i$ is \textit{free of type I}}.\\
    \end{array} \right.\]
\end{itemize}
We will prove in Lemma \ref{la9} below
that each element of $F_j(R)$ for a $\kappa$-algebra $R$ satisfies $(z_j^{\ast})_1+(z_j^{\ast})_1^2=0$ (if $j$ is even)
or $(z_j)_1+(z_j)_1^2=0$ (if $j$ is odd),
where $z_j^{\ast}=(z_j^{\ast})_1+\pi (z_j^{\ast})_2$ and $z_j=(z_j)_1+\pi (z_j)_2$, and
that $F_j$ is isomorphic to $ \mathbb{A}^{1} \times \mathbb{Z}/2\mathbb{Z}$ as a $\kappa$-variety,
where $\mathbb{A}^{1}$ is an affine space of dimension $1$.

Notice that $F_j$ and $F_{j^{\prime}}$ commute with each other for all  integers  $j\neq j^{\prime}$, where $j, j^{\prime}\in \mathcal{B}$,
in the sense that $f_j\cdot f_{j^{\prime}}=f_{j^{\prime}}\cdot f_j$, where $f_j\in F_j $ and $ f_{j^{\prime}}\in F_{j^{\prime}}$.
Let $F=\prod_{j}F_j$.
Then $F$ is  smooth and is a closed subgroup scheme of $\mathrm{Ker~}\varphi$ as mentioned in the proof of Theorem \ref{t411}.
If $F^{\dag}$  is the image of $F$ in $G^{\dag}$, then it is smooth and thus  a closed subscheme of $(G^{\dag})_{\mathrm{red}}$.
By observing Equation (\ref{32'}) and $(z_j^{\ast})_1+(z_j^{\ast})_1^2=0$ (if $j$ is even)
or $(z_j)_1+(z_j)_1^2=0$ (if $j$ is odd) above,
we can easily see that $F^{\dag}$ contains at least one (closed) point of each connected component of $G^{\ddag}$ and this proves our claim.\\

Combining this fact with dim $G^{\dag}$ = dim $G^{\ddag}$, we conclude that $(G^{\dag})_{\mathrm{red}}\simeq G^{\ddag}$,
and hence, $G^{\dag}=G^{\ddag}$ because $G^{\dag}$ is a subfunctor of $G^{\ddag}$.
This completes the proof.
\end{proof}

\begin{Lem}\label{la8}
Let $G^{\ddag}$ be the subfunctor of $ \mathrm{Ker~}\tilde{\varphi}/\tilde{M}^1$ consisting of those $m$ 
satisfying Equations (\ref{ea20}), (\ref{ea22}), (\ref{24}), (\ref{24'}), (\ref{ea25}), (\ref{ea27}), and (\ref{ea32}). 
Note that such $m$ also satisfies Equation (\ref{32'}).
Then  $G^{\ddag}$ is represented by a smooth  closed subscheme of $ \mathrm{Ker~}\tilde{\varphi}/\tilde{M}^1$
and is isomorphic to $ \mathbb{A}^{l^{\prime}}\times (\mathbb{Z}/2\mathbb{Z})^{\beta}$ as a $\kappa$-variety,
 where $\mathbb{A}^{l^{\prime}}$ is an affine space of dimension $l^{\prime}$.
 Here,
\begin{align}
 l^{\prime}=\sum_{i<j}n_in_j +\sum_{i:\mathrm{even~and~} L_i:\textit{of type }I^e}(n_i-1) +
 \sum_{i:\mathrm{odd~and~}L_i:\textit{free of type I}}(2n_i-2) \notag \\
 - \sum_{i:\mathrm{even~and~} L_i:\textit{bound of type II}}n_i +\#\{i:\textit{$i$ is even and $L_i$ is of type I}\} ~~~~~~~~~~~~~~~\notag \\ 
-\#\{i:\textit{$i$ is even, $L_i$ is  of type I and  $L_{i+2}$ is of type II}\}. ~~~~~~~~~~~~~~
\end{align}
\end{Lem}

\begin{proof}
Let $\mathcal{B}$ be the set of  integers $j$ such that $L_j$ is \textit{of type I} and $L_{j+2}, L_{j+3}, L_{j+4}$
 (resp. $L_{j-1}, L_{j+1},$ $L_{j+2}, L_{j+3}$) are \textit{of type II} if $j$ is even (resp. odd).
Equation (\ref{32'}) implies that $G^{\ddag}$ is disconnected with at least $2^\beta$ connected components (Exercise 2.19 of \cite{H}).
Here $\beta=\# \mathcal{B}$.

Let $\mathcal{B}_1$ and $\mathcal{B}_2$ be a pair of two (possibly empty) subsets  of $\mathcal{B}$
such that $\mathcal{B}$ is the disjoint union of $\mathcal{B}_1$ and $\mathcal{B}_2$.
Let  $\widetilde{G}^{\ddag}_{\mathcal{B}_1, \mathcal{B}_2}$ be the subfunctor of $ \mathrm{Ker~}\tilde{\varphi}/\tilde{M}^1$ consisting of those $m$   satisfying Equations(\ref{ea20}), (\ref{ea22}), (\ref{24}), (\ref{24'}), (\ref{ea25}), (\ref{ea27}), and (\ref{ea32}),
 the equations $\sum_{l=0}^{k_j}z_{j-l}^{\ast}+\sum_{l=0}^{k_j}
\bar{\gamma}_{j-l}u_{j-l}^{\ast}=0$ for any $j\in \mathcal{B}_1$,
 and  the equations $\sum_{l=0}^{k_j}z_{j-l}^{\ast}+\sum_{l=0}^{k_j}
\bar{\gamma}_{j-l}u_{j-l}^{\ast}=1$ for any $j\in \mathcal{B}_2$.
Here $k_j$ is the integer associated to $j$ defined in the paragraph before Equation (\ref{32'}).
We claim that
 $\widetilde{G}^{\ddag}_{\mathcal{B}_1, \mathcal{B}_2}$ is represented by a smooth  closed subscheme of $ \mathrm{Ker~}\tilde{\varphi}/\tilde{M}^1$
and is isomorphic to $ \mathbb{A}^{l^{\prime}}$.
Since the scheme $G^{\ddag}$ is a direct product of $\widetilde{G}^{\ddag}_{\mathcal{B}_1,
\mathcal{B}_2}$'s for any such pair of $\mathcal{B}_1, \mathcal{B}_2$ by Exercise 2.19 of \cite{H},
the lemma follows from this claim.

It is obvious that $\widetilde{G}^{\ddag}_{\mathcal{B}_1, \mathcal{B}_2}$ is represented by a closed subscheme of $ \mathrm{Ker~}\tilde{\varphi}/\tilde{M}^1$
since the equations defining $\widetilde{G}^{\ddag}_{\mathcal{B}_1, \mathcal{B}_2}$ as a subfunctor of $ \mathrm{Ker~}\tilde{\varphi}/\tilde{M}^1$ are all polynomials.
Thus it suffices to show that $\widetilde{G}^{\ddag}_{\mathcal{B}_1, \mathcal{B}_2}$ is isomorphic to an affine space $ \mathbb{A}^{l^{\prime}}$.
Our strategy to show this is that 
 the coordinate ring of $\widetilde{G}^{\ddag}_{\mathcal{B}_1, \mathcal{B}_2}$ is isomorphic to a polynomial ring, which is also used in the proof of Lemma A.8 in \cite{C2}.
To do that, we use the following trick over and over.
We consider the polynomial ring $\kappa[x_1, \cdots, x_n]$ and its quotient ring $\kappa[x_1, \cdots, x_n]/(x_1+P(x_2, \cdots, x_n))$.
Then the quotient ring $\kappa[x_1, \cdots, x_n]/(x_1+P(x_2, \cdots, x_n))$
 is isomorphic to $\kappa[x_2, \cdots, x_n]$ and in this case we say that \textit{$x_1$ can be eliminated by $x_2, \cdots, x_n$}.

By the description of an element of $(\mathrm{Ker~}\tilde{\varphi}/\tilde{M}^1)(R)$ in Remark \ref{ra5},
we see that $\mathrm{Ker~}\tilde{\varphi}/\tilde{M}^1$ is isomorphic to an affine space of dimension
$$2\sum_{i<j}n_in_j-\sum_{\textit{i:even and $L_i$:bound of type II}}n_i+ \sum_{\textit{i:even and $L_i$:of type $I^o$}}n_i+$$
$$\sum_{\textit{i:even and $L_i$:of type $I^e$}}(3n_i-2)
 +\sum_{\textit{i:odd and $L_i$:free of type $I$}}(4n_i-4)$$
  with variables $(m_{i,j})_{i\neq j}, (y_i, v_i, z_i, z_i^{\ast})_{\textit{i:even and $L_i$:of type $I^o$}},
 (r_i, t_i, y_i, v_i, x_i, z_i, u_i, w_i, z_i^{\ast})_{\textit{i:even and $L_i$:of type $I^e$}}$,
$ (r_i, t_i, y_i, v_i, x_i, z_i, u_i, w_i)_{\textit{i:odd and $L_i$:free of type I}}$,
$ (m_{i,i}^{\ast}, m_{i,i}^{\ast\ast})_{\textit{i:odd and $L_i$:bound of type I}}$, such that

\begin{itemize}
\item Let $i$ be even and $L_i$ be \textit{bound of type II}.
Then $\delta_{i-1}^{\prime}e_{i-1}\cdot m_{i-1, i}+\delta_{i+1}^{\prime}e_{i+1}\cdot m_{i+1, i}+\delta_{i-2}e_{i-2}\cdot m_{i-2, i}+\delta_{i+2}e_{i+2}\cdot m_{i+2, i}=0$.
Here, notations are as explained in Step  (c) of  the description of an element of $\mathrm{Ker~}\tilde{\varphi}(R)$ given at the paragraph following Lemma \ref{la2}.

\item Let $i$ be even and $L_i$ be \textit{of type I}.
Then  $v_i(\mathrm{resp.~}(y_i+\sqrt{\bar{\gamma}_i}v_i))+(\delta_{i-2}e_{i-2}\cdot m_{i-2, i}+\delta_{i+2}e_{i+2}\cdot m_{i+2, i})\tilde{e}_i=0$ if $L_i$ is \textit{of type} $\textit{I}^o$ (resp. \textit{of type} $\textit{I}^e$).
Here, notations are as explained in Step  (c) of  the description of an element of $\mathrm{Ker~}\tilde{\varphi}(R)$ given at the paragraph following Lemma \ref{la2}.

\item If $i$ is even and $L_i$ is \textit{of type I}, then
$z_i+\delta_{i-2}k_{i-2, i}+\delta_{i+2}k_{i+2, i}=0.$
Here,  notations are as explained in Step  (d) of  the description of an element of $\tilde{M}(R)$ given at the paragraph following Lemma \ref{la1}.

\item If $i$ is odd and $L_i$ is \textit{bound of type I}, then
$\delta_{i-1}v_{i-1}\cdot m_{i-1, i}+\delta_{i+1}v_{i+1}\cdot m_{i+1, i}=0.$
Here,  notations are as explained in Step  (e) of  the description of an element of $\tilde{M}(R)$ given at the paragraph following Lemma \ref{la1}.

\item If $i$ is odd and $L_i$ is \textit{bound of type I}, then
$\delta_{i-1}\cdot m_{i, i-1}'+\delta_{i+1}\cdot m_{i, i+1}'=0.$
Here,  $m_{i, i-1}'$ (resp. $m_{i, i+1}'$) is the last column vector of $m_{i, i-1}$ (resp. $m_{i, i+1}$).
\end{itemize}

Thus we can see that $v_i, z_i$ (resp. $y_i, z_i$) can be eliminated by other variables
when $L_i$ is \textit{of type} $\textit{I}^o$ (resp. \textit{of type} $\textit{I}^e$), and so on.

From now on, we eliminate suitable variables based on Equations (\ref{ea20}), (\ref{ea22}), (\ref{24}), (\ref{24'}), (\ref{ea25}), (\ref{ea27}), and (\ref{ea32}),
 the equations $\sum_{l=0}^{k_j}z_{j-l}^{\ast}+\sum_{l=0}^{k_j}
\bar{\gamma}_{j-l}u_{j-l}^{\ast}=0$ for all $j\in \mathcal{B}_1$,
 and  the equations $\sum_{l=0}^{k_j}z_{j-l}^{\ast}+\sum_{l=0}^{k_j}
\bar{\gamma}_{j-l}u_{j-l}^{\ast}=1$ for all $j\in \mathcal{B}_2$.

\begin{enumerate}
\item We first consider Equation (\ref{ea20}).
This case is  similar to the case (1) in the proof of Lemma A.8 in \cite{C2}.
Thus, all lower triangular blocks $m_{j,i}$ can be eliminated by upper triangular blocks $m_{i,j}$ together with
other variables arose in diagonal blocks.

On the other hand, if $i$ is odd and $L_i$ is \textit{bound of type I},
then we have two equations 
involving $m_{i-1, i}, m_{i+1, i}, m_{i, i-1}', m_{i, i+1}'$ as explained above in this proof.
Since these two equations involve lower triangular blocks, we should figure it out in terms of upper triangular blocks.
By Equation (\ref{ea21}) we have that
$$v_{i-1}\cdot m_{i-1, i}={}^tm_{i, i-1}'h_i \textit{   and   } v_{i+1}\cdot m_{i+1, i}={}^tm_{i, i+1}'h_i.$$
Therefore, these two equations  involving $m_{i-1, i}, m_{i+1, i}, m_{i, i-1}', m_{i, i+1}'$ explained above in this proof
 are identical and they are the same as the following single equation:
$$\delta_{i-1}v_{i-1}\cdot m_{i-1, i}+\delta_{i+1}\cdot {}^tm_{i, i+1}'h_i=0.$$ \\

\item By considering Equation (\ref{ea22}), we see that $m_{i,i}^{\ast}$ can be eliminated by $m_{i,i}^{\ast\ast}$ when
$L_i$ is \textit{bound of type I} with $i$ odd.\\

\item We consider Equation (\ref{24}).
If  $L_i$ is \textit{free of type $I$} with $i$ odd, then $v_i$ can be eliminated by $\left(y_iz_i, r_i, m_{i-1,i}, m_{i,i+1}\right)$,
$y_i$  can be eliminated by $\left(t_i\right)$, and $x_i$ can be eliminated by $\left(z_iu_i, r_i, t_i, m_{i-1,i}, m_{i,i+1}\right)$.\\

\item We  consider Equation (\ref{ea25}).
If  $L_i$ is \textit{of type $I^o$}, then $v_i$ can be eliminated by $y_i$ and $m_{i-1,i}, m_{i,i+1}$.\\

\item We consider Equation (\ref{ea27}).
If  $L_i$ is \textit{of type $I^e$}, then $v_i$ can be eliminated by $\left(r_i\right)$,
$y_i$ can be eliminated by $\left(t_i, v_iz_i, m_{i-1,i}, m_{i,i+1}\right)$, and
 $w_i$ can be eliminated by $\left(r_i, t_i, z_i, x_i, u_i, m_{i-1,i}, m_{i,i+1}\right)$.\\

\item Finally, we consider Equations (\ref{24'}) and (\ref{ea32}), $\mathcal{X}_{i,2,2}$ in Equation (\ref{ea27}),  together with
 the equations $\sum_{l=0}^{k_j}z_{j-l}^{\ast}+\sum_{l=0}^{k_j}
\bar{\gamma}_{j-l}u_{j-l}^{\ast}=0$ for any $j\in \mathcal{B}_1$,
 and  the equations $\sum_{l=0}^{k_j}z_{j-l}^{\ast}+\sum_{l=0}^{k_j}
\bar{\gamma}_{j-l}u_{j-l}^{\ast}=1$ for any $j\in \mathcal{B}_2$.

We first consider $\mathcal{X}_{i,2,2}$ in Equation (\ref{ea27}).
If $\gamma_i$ is not a unit, then  $\bar{\gamma}_i=0$ and so $u_i$ can be eliminated.
If $\gamma_i$ is  a unit so that $\bar{\gamma}_i\neq 0$, then we add $\sqrt{\bar{\gamma}_i}x_i$
to both sides of $\bar{\gamma}_i\mathcal{X}_{i,2,2}=0$.
Then we have
$$\left(\bar{\gamma}_iu_i+\sqrt{\bar{\gamma}_i}x_i\right)+\left(\sqrt{\bar{\gamma}_i}x_i+\bar{\gamma}_iu_i\right)^2+
\bar{\gamma}_i/2{}^tr_ia_ir_i+\bar{\gamma}_i\left(\delta_{i-2}(m_{i-2, i}^{\natural})^2+\delta_{i+2}(m_{i+2, i}^{\natural})^2\right)=\sqrt{\bar{\gamma}_i}x_i.$$
If we let $\tilde{u}_i=\bar{\gamma}_iu_i+\sqrt{\bar{\gamma}_i}x_i$, then $u_i$ can be eliminated by $\tilde{u}_i, x_i$.
In addition, the above equation yields that $x_i$ can be eliminated by $\tilde{u}_i, r_i, m_{i-2, i}^{\natural}, m_{i+2, i}^{\natural}$.
Therefore, since we introduce one new variable $\tilde{u}_i$ and eliminate two variables $u_i$ and $x_i$,
the effect of $\mathcal{X}_{i,2,2}$ in Equation (\ref{ea27}) is to eliminated one variable.

Similarly,  each of Equations (\ref{24'}) and (\ref{ea32}) eliminates one variable.
The proof of this is similar to the above or the argument of Step (4) in the proof of Lemma A.8 of \cite{C2}.
Thus we skip it.

The equation $\sum_{l=0}^{k_j}z_{j-l}^{\ast}+\sum_{l=0}^{k_j}
\bar{\gamma}_{j-l}u_{j-l}^{\ast}=0$ for any $j\in \mathcal{B}_1$,
 and  the equation $\sum_{l=0}^{k_j}z_{j-l}^{\ast}+\sum_{l=0}^{k_j}
\bar{\gamma}_{j-l}u_{j-l}^{\ast}=1$ for any $j\in \mathcal{B}_2$ yield that
for each $j\in \mathcal{B}$,
only one of equations of type Equations (\ref{24'}) or (\ref{ea32}), or $\mathcal{X}_{i,2,2}$ in Equation (\ref{ea27})
is reduntant, and the associated one variable,
say $z_{j}^{\ast}$ for simplicity, can be eliminated by other $z_{j-l}^{\ast}, u_{j-l}^{\ast}$.\\
\end{enumerate}

We now combine all cases (1)-(6) observed above.
\begin{enumerate}
\item[(a)] By (1) and (2), we eliminate $\sum_{i<j}n_in_j$ variables. 
\item[(b)] By (3), we eliminate $\sum_{\textit{i:odd and $L_i$:free of type $I$}}(2n_i-3)$ variables.
\item[(c)] By (4), we eliminate $\sum_{\textit{i:even and $L_i$:of type $I^o$}}(n_i-1)$ variables.
\item[(c)] By (5), we eliminate $\sum_{\textit{i:even and $L_i$:of type $I^e$}}(2n_i-3)$ variables.
\item[(d)] By (6), we eliminate
$$\#\{\textit{i:odd such that $L_i$ is free of type I}\}+\#\{\textit{i:even such that $L_i$ is of type $I^e$}\}+$$
$$\textit{$\#$\{i:even such that $L_i$ is of type I and $L_{i+2}$ is of type II\}$+\beta-\beta$ variables.}$$
Here $\beta=\# \mathcal{B}$. \\
\end{enumerate}

Recall  from the beginning of the proof that  $ \mathrm{Ker~}\tilde{\varphi}/\tilde{M}^1$ is isomorphic to an affine space of dimension
$$2\sum_{i<j}n_in_j-\sum_{\textit{i:even and $L_i$:bound of type II}}n_i+ \sum_{\textit{i:even and $L_i$:of type $I^o$}}n_i+$$
$$\sum_{\textit{i:even and $L_i$:of type $I^e$}}(3n_i-2)
 +\sum_{\textit{i:odd and $L_i$:free of type $I$}}(4n_i-4).$$
Therefore, the dimension of $\widetilde{G}^{\ddag}_{\mathcal{B}_1, \mathcal{B}_2}$ is
\begin{multline}\label{ea40}
 l^{\prime}=\sum_{i<j}n_in_j +\sum_{i:\mathrm{even~and~} L_i:\textit{of type }I^e}(n_i-1) +
 \sum_{i:\mathrm{odd~and~}L_i:\textit{free of type I}}(2n_i-2) \\
 - \sum_{i:\mathrm{even~and~} L_i:\textit{bound of type II}}n_i +\#\{i:\textit{$i$ is even and $L_i$ is of type I}\}  \\
-\#\{i:\textit{$i$ is even, $L_i$ is  of type I and  $L_{i+2}$ is of type II}\},
\end{multline}
which finishes the proof.
\end{proof}

\begin{Lem}\label{la9}
Let $F_j$ be  the closed subgroup scheme  of $\tilde{G}$ defined by the following equations:
\begin{itemize}
\item $m_{i,k}=0$ \textit{if $i\neq k$};
\item $m_{i,i}=\mathrm{id}, z_i^{\ast}=0, m_{i,i}^{\ast}=0, m_{i,i}^{\ast\ast}=0$ \textit{if $i\neq j$};
\item and for $m_{j,j}$,
\[\left \{
  \begin{array}{l l}
  s_j=\mathrm{id~}, y_j=0, v_j=0, z_j=\pi z_j^{\ast} & \quad  \textit{if $i$ is even and $L_i$ is \textit{of type} $\textit{I}^o$};\\
  s_j=\mathrm{id~}, r_j=t_j=y_j=v_j=u_j=w_j=0, z_j=\pi z_j^{\ast}  & \quad \textit{if $i$ is even and  $L_i$ is \textit{of type} $\textit{I}^e$};\\
  s_j=\mathrm{id~}, r_j=t_j=y_j=v_j=u_j=w_j=0 & \quad \textit{if $i$ is odd and  $L_i$ is \textit{free of type I}}.\\
    \end{array} \right.\]
\end{itemize}
Then  $F_j$ is isomorphic to $ \mathbb{A}^{1} \times \mathbb{Z}/2\mathbb{Z}$ as a $\kappa$-variety,
where $\mathbb{A}^{1}$ is an affine space of dimension $1$, and has exactly two connected components.
\end{Lem}
\begin{proof}
A matrix form of an element $m$ of $F_j(R)$ for a $\kappa$-algebra $R$ is
\[\begin{pmatrix} id&0& & \ldots& & &0\\ 0&\ddots&& & & &\\ & &id& & & & \\  \vdots & & &m_{j,j} & & &\vdots
\\ & & & & id & & \\  & & & & &\ddots &0 \\ 0& & &\ldots & &0 &id \end{pmatrix}$$
such that $$m_{j,j}=\left\{
\begin{array}{l l}
\begin{pmatrix}id&0\\0&1+2 z_j^{\ast} \end{pmatrix} & \quad \textit{if $j$ is even and $L_j$ is of type $I^o$};\\
\begin{pmatrix}id&0&0\\0&1+\pi x_j&2 z_j^{\ast}\\0&0&1 \end{pmatrix} & \quad \textit{if $j$ is even and  $L_j$ is of type $I^e$};\\
\begin{pmatrix}id&0&0\\0&1+\pi x_j&0\\0&\pi z_j&1 \end{pmatrix} & \quad \textit{if $j$ is odd and  $L_j$ is free of type $I$}.
\end{array}\right.\]
We emphasize that we have $2z_j^{\ast}$, not $\pi z_j$, when $i$ is even.

To prove the lemma, we consider the matrix equation $\sigma({}^tm)\cdot h\cdot m=h$.
Recall that $h$, as an element of $\underline{H}(R)$, is as explained in Remark \ref{r33}.(2).
Based on Equations (\ref{ea1}) and (\ref{ea2}),
the diagonal $(i,i)$-blocks of $\sigma({}^tm)\cdot h\cdot m=h$ with $i\neq j$ are trivial
and the nondiagonal blocks of $\sigma({}^tm)\cdot h\cdot m=h$ are also trivial.
The $(j, j)$-block of $\sigma({}^tm)\cdot h\cdot m$ is
\[\left\{
\begin{array}{l l}
\pi^j\cdot\begin{pmatrix}a_j&0\\0&(1+\sigma(2 z_j^{\ast}))\cdot (1+2\bar{\gamma}_j)\cdot (1+2 z_j^{\ast}) \end{pmatrix}
& \quad \textit{$L_j$ : of type $I^o$ with j even};\\
\pi^j\cdot\begin{pmatrix}a_j&0&0\\0&(1+\sigma(\pi x_j))(1+\pi x_j)&(1+\sigma(\pi x_j))(1+2 z_j^{\ast})\\
0&(1+\sigma(2 z_j^{\ast}))(1+\pi x_j)&(1+2 z_j^{\ast})\sigma(2 z_j^{\ast})+2 z_j^{\ast}+2\bar{\gamma}_j \end{pmatrix} & \quad
\textit{$L_j$ : of type $I^e$ with j even};\\
\pi^j\cdot\begin{pmatrix}a_j&0&0\\0&\pi^3\bar{\gamma}_j+\pi^3\left((z_j)_1+(z_j)_1^2\right)&1+\sigma(\pi x_j)+\pi\cdot\sigma(\pi z_j)  \\
0&-1-\pi x_j+\pi^2 z_j&\pi \end{pmatrix} & \quad
\textit{$L_j$ : free of type $I$ with j odd}.
\end{array}\right.\]
We write $z_j^{\ast}=(z_j^{\ast})_1+\pi (z_j^{\ast})_2$, $x_j=(x_j)_1+\pi (x_j)_2$, and $z_j=(z_j)_1+\pi (z_j)_2$, where $(z_j^{\ast})_1, (z_j^{\ast})_2,$ $(x_j)_1,$ $(x_j)_2,$ $(z_j)_1,$ $(z_j)_2 \in R \subset R\otimes_AB$ and $\pi$ stands for $1\otimes \pi\in R\otimes_AB$.
When $L_j$ is \textit{of type $I^o$}, by considering the $(2, 2)$-block of the matrix above, we obtain the equation
\[(z_j^{\ast})_1+(z_j^{\ast})_1^2=0.\]
Therefore, in this case, $F_j$ is isomorphic to $ \mathbb{A}^{1} \times \mathbb{Z}/2\mathbb{Z}$ as a $\kappa$-variety.

When $L_j$ is \textit{of type $I^e$}, by considering the $(2, 2)$-block of the matrix above, we obtain the equation
\[ (x_j)_1^2=0.\]
We also  consider the $(2, 3)$-block of the matrix above, and we obtain two equations
\[(x_j)_1=0, ~~~ (x_j)_2+(z_j^{\ast})_1=0.   \]
By considering the $(3, 3)$-block of the matrix above, we obtain the equation
\[(z_j^{\ast})_1+(z_j^{\ast})_1^2=0.\]
By combining all these, we see that $F_j$ is isomorphic to $ \mathbb{A}^{1} \times \mathbb{Z}/2\mathbb{Z}$ as a $\kappa$-variety.

By using a similar argument used above, when  $L_j$ is \textit{free of type $I$} with $i$ odd, we obtain the equations
\[(z_j)_1+(z_j)_1^2=0, ~~~ (x_j)_1=0, ~~~(z_j)_1+(x_j)_2=0.  \]
By combining all these, we see that $F_j$ is isomorphic to $ \mathbb{A}^{1} \times \mathbb{Z}/2\mathbb{Z}$ as a $\kappa$-variety.
\end{proof}
\textit{  }

We finally prove Lemma \ref{l46}.
\begin{proof}
 We start with the following short exact sequence
\[1\rightarrow \tilde{G}^1 \rightarrow \mathrm{Ker~}\varphi\rightarrow\mathrm{Ker~}\varphi/\tilde{G}^1\rightarrow 1.\]
It is obvious that $\mathrm{Ker~}\varphi$ is smooth  by Theorems \ref{ta4} and  \ref{ta6}.
$\mathrm{Ker~}\varphi$ is also unipotent since it is a subgroup of a unipotent group $\tilde{M}^+$.
Since $\tilde{G}^1$ is connected by Theorem \ref{ta4}, the  component group of $\mathrm{Ker~}\varphi$ is the same as that of $\mathrm{Ker~}\varphi/\tilde{G}^1$ by  Lemma A.10 of \cite{C2}.
 Moveover, the dimension of $\mathrm{Ker~}\varphi$ is the sum of the dimension of $\tilde{G}^1$
and the dimension of $\mathrm{Ker~}\varphi/\tilde{G}^1$. This completes the proof.
\end{proof}

\section{Examples} \label{App:AppendixB}
In this appendix, we provide an example with a unimodular lattice $(L, h)$ of rank 1.
The structure of this appendix is parallel to that of Appendix B of \cite{C2} and thus
many sentences of loc. cit. are repeated without comment.
Let $L$ be $B\textit{e}$ of rank 1 hermitian lattice with hermitian form $h(le, l'e)=\sigma(l)l'$.
With this lattice, we construct the smooth integral model and its special fiber and compute the local density.

\subsection{Naive construction (without using our technique)}\label{nc}
We first construct the smooth integral model and its special fiber, without using any technique introduced in this paper.
If we write an element of $L$ as $x+\pi y$ where $x, y\in A$,
then it is easy to see that a naive integral model $\underline{G}'$ is $\mathrm{Spec~}A[x,y]/(x^2+(\pi +\sigma(\pi))xy+\pi\sigma(\pi)y^2-1)$.
As mentioned in Section \ref{Notations}, we may assume that $\pi +\sigma(\pi)=0$ and $\pi\sigma(\pi)=-2\delta$ for a unit $\delta\in A$
such that $\delta\equiv 1 \mathrm{~mod~}2$.
We remark that $\underline{G}'$ is smooth if $p\neq 2$, and in this case its special fiber is
$\mathrm{Spec~}\kappa[x,y]/(x^2-1)=\mathbb{A}^1\times \mu_2$ as a $\kappa$-variety.
However, if $p=2$, then its special fiber is no longer smooth since $\kappa[x,y]/(x^2-1)=\kappa[x,y]/(x-1)^2$ is nonreduced.
Some of the difficulty in the case $p=2$ arises from this.
The associated smooth integral model is
obtained by a finite sequence of \textit{dilatations} (at least once) of $\underline{G}'$ (cf. \cite{BLR}).

On the other hand, the difficulty can also be explained in terms of quadratic forms.
Namely, the smoothness of any scheme over $A$ should be closely related to the smoothness of its special fiber.
If we define a function $q : L\longrightarrow A, l\mapsto h(l,l)$,
then $q$  mod 2   is a quadratic form over $\kappa$.
Therefore, the associated smooth integral model should contain information about this quadratic form,
which is  more subtle than quadratic forms over a field of characteristic not equal 2.

To construct the smooth integral model, we observe the characterization of $\underline{G}$ such that $\underline{G}(R)=\underline{G}'(R)$
for an \'etale $A$-algebra $R$.
Thus any element of \underline{G}(R) is of the form $x+\pi y$ such that $x^2-2\delta y^2=1$.
Therefore, $(x-1)^2$ is contained in the ideal $(2)$ of $R$ so that we can rewrite $x=1+2x'$ since $R$ is \'etale over $A$.
With this, any element of \underline{G}(R) is of the form $1+2x'+\pi y$ such that
$2(x'+(x')^2)-\delta y^2=0$.
This equation also yields that $y^2$ is contained in the prime ideal $(2)$ of $R$
so that we can rewrite $y=2y'$ since $R$ is \'etale over $A$.
With this, any element of \underline{G}(R) is of the form $1+2x'+\pi\cdot 2y'$ such that
$x'+(x')^2-2\delta (y')^2=0$.

We consider the affine scheme $\mathrm{Spec~}A[x,y]/(x+x^2-2\delta y^2)$.
 Its special fiber  is then reduced and smooth.
 Thus, this affine scheme is the desired smooth integral model $\underline{G}$.
 Furthermore, its special fiber $\mathrm{Spec~}\kappa[x,y]/(x+x^2)$ is isomorphic to $\mathbb{A}^1\times \mathbb{Z}/2\mathbb{Z}$
 as a $\kappa$-variety so that the number of rational points is $2f$, where $f$ is the cardinality of $\kappa$.

\subsection{Construction following our technique}\label{cfot}
Let $q$ be the function defined over $L$ such that
$$q : L\longrightarrow A, l\mapsto h(l,l).$$
If we write $l=x+\pi y$ such that $x, y \in A$, then $q(l)=h(x+\pi y, x+\pi y)=x^2-2\delta y^2$.
Thus $q$  mod 2  is an additive polynomial over $\kappa$.
Let $B(L)$ be the sublattice of $L$ such that $B(L)/\pi L$ is the kernel of the additive polynomial  $q$  mod 2 on  $L/\pi L$.
In this case, $B(L)=\pi L$.

We define another  sublattice $Z(L)$ of $L$ such that
$Z(L)/\pi B(L)$ is the radical of the quadratic form $\frac{1}{2}q$ mod 2 on $B(L)/\pi B(L)$.
In this case, $Z(L)=\pi B(L)=2L$.

For an    \'etale $A$-algebra $R$ with  $g\in \mathrm{Aut}_{B\otimes_AR}(L\otimes_AR, h\otimes_AR)$,
 it is easy to see that $g$ induces the identity on $L/B(L)=L/\pi L$.
 An element $g$ also induces the identity on $L/Z(L)=L/2L$ since
 $$q(gw-w, gw-w)=2q(w)-\left(h(gw, w)+h(w, gw)\right)=$$
 $$2q(w)-\left(h(w+x, w)+h(w, w+x)\right) =-\left(h(x, w)+h(w, x)\right)$$
 where $w\in L$ and $x\in B(L)=\pi L$ such that $gw=w+x$, and thus $h(x, w)\in \pi B$ and $\left(h(x, w)+h(w, x)\right)\in 4 B$.

Based on this, we construct the following functor from the category of commutative flat $A$-algebras to the category of monoids as follows. For any commutative flat $A$-algebra $R$, set
    $$\underline{M}(R) = \{m \in \mathrm{End}_{B\otimes_AR}(L \otimes_A R)\} ~|~ \textit{$m$ induces the identity on $L\otimes_A R/ Z(L)\otimes_A R$}\}.$$
    This functor $\underline{M}$ is then representable by a polynomial ring and has the structure of a scheme of monoids.
    Let $\underline{M}^{\ast}(R)$ be the set of invertible elements in  $\underline{M}(R)$ for any commutative $A$-algebra $R$.
Then    $\underline{M}^{\ast}$ is representable by a group scheme which is an open subscheme of $\underline{M}$ (Section \ref{m}).
    Thus  $\underline{M}^{\ast}$ is smooth.
    Indeed, in our case, $\underline{M}^{\ast}=\underline{M}$.
    As a matrix, each element of $\underline{M}^{\ast}(R)$ for a flat $A$-algebra $R$ can be written as $\begin{pmatrix} 1+2 z \end{pmatrix}$.

    We define another functor from the category of commutative flat $A$-algebras to the category of sets as follows. For any commutative flat $A$-algebra $R$,
    let $\underline{H}(R)$ be the set of hermitian forms $f$ on $L\otimes_{A}R$ (with values in $B\otimes_AR$)
    such that
    $f(a,a)$ mod 4 = $h(a, a)$ mod 4, where $a \in L \otimes_{A}R$.
    As a matrix, each element of $\underline{H}(R)$ for a flat $A$-algebra $R$ is $\begin{pmatrix} 1+4 c \end{pmatrix}$.

    Then for any flat $A$-algebra $R$, the group $\underline{M}^{\ast}(R)$ acts on the right of $\underline{H}(R)$
    by $f\circ m = \sigma({}^tm)\cdot f\cdot m$ and
    this action is represented by an action morphism (Theorem \ref{t34})
     \[\underline{H} \times \underline{M}^{\ast} \longrightarrow \underline{H} .\]
        Let $\rho$ be the morphism $\underline{M}^{\ast} \rightarrow \underline{H}$ defined by $\rho(m)=h \circ m$,
  which is obtained from the above action morphism.
  As a matrix, for a flat $A$-algebra $R$,  $$\rho(m)=\rho(\begin{pmatrix} 1+2 z \end{pmatrix})
     =\begin{pmatrix} 1+2 z+\sigma(2 z)+4\cdot z\sigma(z) \end{pmatrix}.$$

  Then $\rho$ is smooth of relative dimension 1 (Theorem \ref{t36}).
   Let $\underline{G}$ be the stabilizer of $h$ in $\underline{M}^{\ast}$.
 The group scheme $\underline{G}$ is smooth, and $\underline{G}(R)=\mathrm{Aut}_{B\otimes_AR}(L\otimes_A R,h\otimes_A R)$ for any \'{e}tale $A$-algebra $R$ (Theorem \ref{t38}).\\

We now describe the structure of the special fiber $\tilde{G}$ of $\underline{G}$.
For a $\kappa$-algebra $R$,
each element of $\underline{M}(R)$ (resp. $\underline{H}(R)$)  can be written as a formal matrix $m=\begin{pmatrix} 1+2 z \end{pmatrix}$
(resp. $f=\begin{pmatrix} 1+4 c \end{pmatrix}$).
Firstly, it is obvious that  the morphism $\varphi$ in Section \ref{red} is trivial
since the dimension of $B(L)/Z(L)=\pi L/2L$ is $1$ as a $\kappa$-vector space
so that the associated reduced orthogonal group is trivial.

For the component groups,
as explained in Theorem \ref{t411}, there is a surjective morphism from $\tilde{G}$ to  $\mathbb{Z}/2\mathbb{Z}$.
 Let us  describe this morphism explicitly below.
It is easy to see that  $L^0=M_0=L$ and $C(L^0)=M_0'=L$.
Here, we follow notation of Section \ref{cg}.
Since $M_0=L$ is \textit{of type $I^o$}, there exists a morphism from the special fiber $\tilde{G}$ $(=G_0)$ to the special fiber of the smooth integral model associated to
$M_0'\oplus C(L^0)= L\oplus L$ \textit{of type $I^e$} as explained in the argument (2) just before Remark \ref{r410}.
Remark \ref{r410} tells us how to describe this morphism as formal matrices.
Let $(e_1, e_2)$ be a basis for $L\oplus L$ so that the associated Gram matrix of   the hermitian lattice $L\oplus L$
with respect to this basis is $\begin{pmatrix} 1& 0 \\ 0& 1  \end{pmatrix}$.
Then we consider the basis $(e_1, e_1+e_2)$, with respect to which the morphism described in Remark \ref{r410} is given as
\[\begin{pmatrix} 1+2 z \end{pmatrix} \longrightarrow \begin{pmatrix} 1& -2 z \\ 0& 1+2 z \end{pmatrix}.\]
We now construct a morphism from the special fiber of the smooth integral model associated to
$M_0'\oplus C(L^0)= L\oplus L$ to $\mathbb{Z}/2\mathbb{Z}$
and describe the image of $\begin{pmatrix} 1& -2 z \\ 0& 1+2 z \end{pmatrix}$ in $\mathbb{Z}/2\mathbb{Z}$.

Let $R$ be a $\kappa$-algebra.
The Gram matrix for the hermitian lattice $L\oplus L$ with respect to the basis $(e_1, e_1+e_2)$  is $\begin{pmatrix} 1& 1 \\ 1& 2 \end{pmatrix}$.
Since $L\oplus L$  is \textit{unimodular of type $I^e$},
an $R$-point of the special fiber associated to $L\oplus L$ with respect to this basis is expressed as the  formal matrix
$\begin{pmatrix} 1+\pi x'& 2 z' \\ u'& 1+\pi w' \end{pmatrix}$,  as explained in Section \ref{m}.
Based on the argument (1) following Definition \ref{d49},
 the morphism mapping to $\mathbb{Z}/2\mathbb{Z}$ factors through the special fiber associated to $Y(C(L\oplus L))$,
 composed with the Dickson invariant associated to the corresponding orthogonal group.
Note that $C(L\oplus L)$ is  generated by $(\pi e_1, e_1+e_2)$  and the corresponding Gram matrix is
$\begin{pmatrix} -2\delta& \pi \\ -\pi& 2 \end{pmatrix}$.
Then $Y(C(L\oplus L))$ (in this case it is the same as $B(C(L\oplus L))$)  is generated by $(2 e_1, \pi e_1+e_1+e_2)$ and the corresponding Gram matrix is
$\begin{pmatrix} 4\delta^2& 2\delta(\pi-1) \\ 2\delta(-\pi-1)& 2(1-\delta) \end{pmatrix}$.
Note that a method to choose the above basis is also explained in the argument (1) following Definition \ref{d49}.
Since $\delta\equiv 1$ mod 2, the lattice $Y(C(L\oplus L))$ is $\pi^2$-modular \textit{of type II}.
Thus there is no congruence condition on an element of the smooth integral model associated to $Y(C(L\oplus L))$
as explained in Section  \ref{m}.

We write $x'=x'_1+\pi x'_2$, $y'=y'_1+\pi y'_2$, and $z'=z'_1+\pi z'_2$.
The image of $\begin{pmatrix} 1+\pi x'& 2 z' \\ u'& 1+\pi w' \end{pmatrix}$ in the special fiber associated to $Y(C(L\oplus L))$, with respect to the above basis $(2 e_1, \pi e_1+e_1+e_2)$,
is  $$\begin{pmatrix} 1+\pi (x'_1+u'_1)&  x'+u'+z'+w' \\ 0 & 1+\pi (x'_1+u'_1) \end{pmatrix}.$$
Since $Y(C(L\oplus L))$ is $\pi^2$-modular \textit{of type II} with rank 2, there is a morphism from the special fiber
associated to $Y(C(L\oplus L))$ to the orthogonal group associated to $Y(C(L\oplus L))/\pi Y(C(L\oplus L))$,
as described in Theorem \ref{t43} or Remark \ref{r47}.
Then the image of $\begin{pmatrix} 1+\pi (x'_1+u'_1)&  x'+u'+z'+w' \\ 0 & 1+\pi (x'_1+u'_1) \end{pmatrix}$ in this orthogonal group is
$$\begin{pmatrix} 1&  x'_1+u'_1+z'_1+w'_1 \\ 0& 1 \end{pmatrix}.$$
The Dickson invariant of the above matrix is $x'_1+u'_1+z'_1+w'_1$
 as mentioned in Step (1) of the proof of Theorem \ref{t411}.

In conclusion, the image of $\begin{pmatrix} 1+2 z \end{pmatrix}$, which is an element of $\tilde{G}(R)$ for a $\kappa$-algebra $R$,
 in $\mathbb{Z}/2\mathbb{Z}$ is $z_1$, where we write $z=z_1+\pi z_2$.
On the other hand, the equation defining $\tilde{G}$ is  $z_1+z_1^2=0$.
Thus, the morphism from $\tilde{G}$ to $\mathbb{Z}/2\mathbb{Z}$ is surjective.
Therefore the maximal reductive quotient of $\tilde{G}$  is  $\mathbb{Z}/2\mathbb{Z}$  and using Remark 5.3 of \cite{C2},
$$\#(\tilde{G}(\kappa))=\#(\mathbb{Z}/2\mathbb{Z})\cdot \#(\mathbb{A}^1)=2f,$$
where $f$ is the cardinality of $\kappa$.
Based on Theorem \ref{t52}, the local density is
$$\beta_L=f^{-1}\cdot 2f=2.$$

\begin{Rmk}[Correction]\label{correction}
In the last line of Appendix B of \cite{C2}, 
$$\beta_L=f^{0}\cdot 2f=2f$$
should be fixed by  
$$\beta_L=f^{-1}\cdot 2f=2.$$
\end{Rmk}


\begin{thebibliography}{99}


    \bibitem[BLR90]{BLR} S. Bosch, W. L$\ddot{\mathrm{u}}$tkebohmert, and M. Raynaud,  \textit{N$\acute{\mathrm{e}}$ron Models}, Ergeb. Math. Grenzgeb.(3) 21, Springer, Berlin, 1990
 \bibitem[Cho15]{C1} S. Cho, \textit{Group schemes and local densities of quadratic lattices in residue characteristic 2}, Compositio Math.
 Vol. 151, 793-827, 2015
 \bibitem[Cho16]{C2} S. Cho, \textit{Group schemes and local densities of ramified hermitian lattices in residue characteristic 2 Part I}, Algebra \& Number Theory, 10-3,  451-532, 2016
\bibitem[DG70]{DG} M. Demazure, P. Gabriel, \textit{Groupes $Alg\acute{e}briques$}, Tome I, Masson et Cie, 1970
\bibitem[SGA3]{SGA3} M. Demazure and A. Grothendieck, \textit{Sch\'emas en groups, I, II, III},
S\'eminaire de G\'eom\'etrie Alg\'ebrique du Bois-Marie 1962/64 (SGA 3), Lecture Notes in Math. 151, 152, 153, Springer, Berlin, 1970
\bibitem[GY00]{GY}  W. T. Gan and J.-K. Yu,  \textit{Group schemes and local densities}, Duke Math. J. 105, 497-524, 2000
\bibitem[Har77]{H}  R. Hartshorne, \textit{Algebraic Geometry}, Grad. Texts in Math. 52, Springer, New York, 1977
\bibitem[Jac62]{J} R. Jacobowitz, Hermitian forms over local fields, American J. Math. 84, 551-465, 1962
\bibitem[Sah60]{Sa} C.-H. Sah, \textit{Quadratic forms over fields of characteristic 2}, American J. Math. 82, 812-830, 1960





\end{thebibliography}
\end{document}